\documentclass[a4paper,12pt]{amsart}

\usepackage{amsmath}
\usepackage{amsthm}
\usepackage{amsfonts}
\usepackage{amssymb}
\usepackage{amscd}     %( for creating only rectangular diagram with labeled arrows, not for diagonal arrows)

\usepackage{euscript}   %( Euler script is the full form and is used for better fonts)
\usepackage{bbm}        % ( for mathbb command)
\usepackage{mathrsfs}  % ( for mathscr command)
\usepackage{verbatim}   % ( useful for creating an article according to the writer's style)

\usepackage[all,cmtip]{xy} % (for a large variety of arrows including curved arrows)
\usepackage{epsfig}
\usepackage{xcolor}

\usepackage{graphicx}
\DeclareMathAlphabet{\mathpzc}{OT1}{pzc}{m}{it}
\allowdisplaybreaks
\usepackage{tikz}

\usepackage[pagebackref,hypertexnames=false, colorlinks, citecolor=red,linkcolor=blue, urlcolor=red]{hyperref}

\makeatletter
\@namedef{subjclassname@2020}{%
  \textup{2020} Mathematics Subject Classification}
\makeatother

\DeclareSymbolFont{SY}{U}{psy}{m}{n}
\DeclareMathSymbol{\emptyset}{\mathord}{SY}{'306}

\theoremstyle{plain}

\newtheorem{thm}{Theorem}[section]
\newtheorem{cor}[thm]{Corollary}
\newtheorem{lem}[thm]{Lemma}
\newtheorem{prop}[thm]{Proposition}
\newtheorem{rem}[thm]{Remark}
\newtheorem{defn}[thm]{Definition}

\theoremstyle{definition}
\newtheorem{ex}[thm]{Example}

\numberwithin{equation}{section}

\def\B{{\mathbb B}}
\def\C{{\mathbb C}}
\def\D{{\mathbb D}}

\def\M{\mathbb{M}}
\def\N{\mathbb{N}}

\def\T{\mathbb T}

\def \CA { \mathcal A }                                                        \def \CY { \mathcal Y }
\def \CB { \mathcal B }                                    \def\CR { \mathcal R }                    \def \CZ { \mathcal Z }
\def \CC { \mathcal C }                 \def\CK { \mathcal K }                   \def\CS { \mathcal S }
\def \CD { \mathcal D }                 \def\CL { \mathcal L }                    \def\CT { \mathcal T }
\def \CE { \mathcal E }                  \def\CM { \mathcal M }                 \def\CU { \mathcal U }
\def \CF { \mathcal F }                   \def\CN { \mathcal N }                 \def\CV { \mathcal V }
\def \CG { \mathcal G }                  \def\CO { \mathcal O }                \def \CW { \mathcal W }
\def \CH { \mathcal H }                                    \def \CX { \mathcal X }

\def \SA { \mathsf A }                   \def \SI { \mathsf  I }                     \def\SQ { \mathsf Q }                  \def \SY { \mathsf Y }
\def \SB { \mathsf B }                  \def \SJ { \mathsf J }                                        
\def \SC { \mathsf C }                                     \def\SS { \mathsf S }
                                      \def\ST { \mathsf T }
\def \SE { \mathsf E }

                  \def\SP { \mathsf P }                  \def \SX { \mathsf X }

\def \se { \mathsf e}
\def \ev { \mathsf{EV} }
\def \sev { \mathsf{ev} }

\def \TA { \mathtt A }

\newcommand { \be } { \begin {equation} }
\newcommand { \ee } { \end {equation} }
\newcommand { \bea } { \begin {eqnarray} }
\newcommand { \eea } { \end {eqnarray} }
\newcommand { \Bea } { \begin {eqnarray*} }
\newcommand { \Eea } { \end {eqnarray*} }

\usepackage{xpatch}
\makeatletter   
\xpatchcmd{\@tocline}
{\hfil\hbox to\@pnumwidth{\@tocpagenum{#7}}\par}
{\ifnum#1<0\hfill\else\dotfill\fi\hbox to\@pnumwidth{\@tocpagenum{#7}}\par}
{}{}
\makeatother

\makeatletter
 \def\l@subsection{\@tocline{2}{0pt}{4pc}{6pc}{}}
\def\l@subsubsection{\@tocline{3}{0pt}{8pc}{8pc}{}}
 \makeatother

\def\norm#1{\left\|{#1}\right\|}

\makeatletter
\@namedef{subjclassname@2020}{%
  \textup{2020} Mathematics Subject Classification}
\makeatother

\title{Sections and Chapters}

\begin{document}

%---------------------------------------------------TITLE AND ACKNOWLEDGEMENT------------------------------------------------------------

\title[nc Cowen-Douglas class]{Towards a noncommutative theory of Cowen-Douglas class of noncommuting operators}

%Authors' information------------------------------------------------------------------------------

\author[P. Deb]{Prahllad Deb}
\address[P. Deb]{Department of Mathematics, Indraprastha Institute of Information Technology Delhi, New Delhi - 110020, India}
%\curraddr{Department of Mathematics, Ben-Gurion University in the Negev,  Beer-Sheva, Israel - 84105}
\email{prahllad@iiitd.ac.in}
\thanks{P. Deb gratefully acknowledges the postdoctoral fellowship supported by Israel Science Foundation (ISF) Grant No. 2123/17 in the Department of Mathematics at Ben-Gurion University of the Negev, as well as the professional development fund (PDA) and the initial research grant (IRG) provided by IIIT-Delhi, for supporting his research.}

\author[V. Vinnikov]{Victor Vinnikov}
\address[V. Vinnikov]{Department of Mathematics, Ben-Gurion University in the Negev,  Beer-Sheva, Israel - 84105}
%\curraddr
\email{vinnikov@bgu.ac.il}
\thanks{V. Vinnikov acknowledges Israel Science Foundation (ISF) Grant No. 2123/17 for supporting this work. He would also like to thank Milken Families Foundation Chair in Mathematics for a partial support of his research.}



%%%%%%%%%%%%%%%%%%%%%%%%%%%%%%%%%%%%%%%%%%%%%%%%%%%%%%%%%%%%%

%-------------------------------------------------------------------ABSTRACT----------------------------------------------------------------------------

\begin{abstract} 
The classical Cowen-Douglas class of (commuting tuples of) operators possessing an open set of (joint) eigenvalues of finite constant multiplicity was introduced by Cowen and Douglas, generalizing the backward shifts. Their unitary equivalence classes are determined by the equivalence classes of certain hermitian holomorphic vector bundles associated with them on this set. 
%As emphasized in the work of Curto and Salinas, they are modelled by the adjoints of the multiplication operators by the independent variable(s) on a reproducing kernel Hilbert space.

This article develops a free noncommutative analogue of Cowen-Douglas theory to explore the notion of vector bundles in the setting of free noncommutative function theory. We define the noncommutative Cowen-Douglas class using matricial joint eigenvalues, as envisioned by Taylor, and show via the Taylor-Taylor series that the associated joint eigenspaces naturally form such a vector bundle, what we call a noncommutative hermitian holomorphic vector bundle.

A key result is that the unitary equivalence class of a tuple in this class is completely determined by the equivalence class of its associated noncommutative vector bundle. This work lays the groundwork of the noncommutative hermitian geometry, which investigates noncommutative analogues of complex manifolds, vector bundles, and hermitian metrics by drawing on ideas from both complex hermitian geometry and operator theory.

We also examine noncommutative reproducing kernel Hilbert space models and introduce the noncommutative Gleason problem, showing that elements of the noncommutative Cowen-Douglas class are essentially (up to unitary equivalence) adjoints of left multiplication operators by noncommuting independent variables in a noncommutative reproducing kernel Hilbert space.
\end{abstract}

%%%%%%%%%%%%%%%%%%%%%%%%%%%%%%%%%%%%%%%%%%%%%%%%%%%%%%%%%%%%%

\maketitle

\tableofcontents

%%%%%%%%%%%%%%%%%%%%%%%%%%%%%%%%%%%%%%%%%%%%%%%%%%%%%%%%%%%%%

\section{Introduction} \label{Introduction}

A central problem in the study of operators or tuples of operators on a Hilbert space $ \CH $ is the identification of unitary invariants. Finding a complete set of invariants that are both comprehensive and computationally feasible is a challenging task, with only a few successful instances in practice. One significant contribution in this direction was made by Cowen and Douglas, who developed a framework for a broad class of non-normal operators or tuples of commuting operators on $ \CH $ -- known as the Cowen-Douglas class (\cite{Cowen-Douglas1, Cowen-Douglas2}). However, this theory does not extend to tuples of non-commuting operators. 

The purpose of this article is to initiate a systematic exploration of tuples of noncommuting operators -- reminiscent of the tuple formed by the adjoints of right creation operators on the $ \ell^2 $ space over a free monoid with finitely many letters -- and their associated families of joint matricial eigenspaces by adapting the classical Cowen-Douglas theory to the modern context of free noncommutative (nc for short) function theory. This endeavor lays the foundation for a more general paradigm applicable to a broad class of such operator tuples, which we refer to as the nc Cowen-Douglas class.  Our approach seeks to provide a comprehensive framework for analyzing tuples of bounded linear operators (not necessarily commuting) while deepening our understanding of nc function theory on the nc joint spectrum associated with such operator tuples.

We begin with the classical example: The backward shift operator $ S $ on the Hilbert space $ \ell^2 ( \N ) $ of square integrable complex sequences. Note that each $ \lambda $ in the open unit disc $ \D $ is an eigenvalue of $ S $ of multiplicity $ 1 $ with a canonical choice of eigenvector $ ( 1, \lambda, \lambda^2, \hdots ) \in \ell^2 ( \N ) $. Further, the operator $ S - \lambda I_{ \CH } $ is surjective for every $ \lambda \in \D $. Generalizing the properties of backward shift, Cowen and Douglas, in their seminal paper \cite{Cowen-Douglas1} (and its follow-up \cite{Cowen-Douglas2}), introduced a class of (tuples of commuting) operators on a Hilbert space. More precisely, an $ m $-tuple $ T = ( T_1, \hdots, T_m ) $ of commuting operators on $ \CH $ is in the Cowen-Douglas class $ \mathrm B_r ( \CD ) $ over a bounded domain $ \CD \subset \C^m $ of rank $ r $ if every $ z \in \CD $ is a joint eigenvalue of $ T $ of multiplicity $ r $ (that is, $ \dim \bigcap_{ i = 1 }^m \ker ( T_i - z_i \mathrm {Id}_{ \CH } ) = r $), the linear span of these joint eigenspaces $ \bigcap_{ i = 1 }^m \ker ( T_i - z_i \mathrm {Id}_{ \CH } ) $ is dense in $ \CH $ and the associated tuples $ ( T_1 - z_1, \hdots, T_m - z_m ) $ have closed range in $ \CH \oplus \cdots \oplus \CH $ ($ m $-times) for all $ z = ( z_1, \hdots, z_m ) \in \CD $.

The key insight of Cowen and Douglas is that such a tuple $ T $ naturally defines a hermitian holomorphic vector bundle $ E_T $ over an open subset $ \CD_0 $ of $ \CD $ -- the fibre of $ E_T $ over a point $ ( z_1, \hdots, z_m ) \in \CD_0 $ is the joint eigenspace $ \bigcap_{ i = 1 }^m \ker ( T_i - z_i \mathrm {Id}_{ \CH } ) $. One of their striking results establishes a one-to-one correspondence between the unitary equivalence classes of these tuples $ T $ and the (local) equivalence classes of the associated vector bundles. Through this correspondence, the invariants of these vector bundles, like the curvature, second fundamental form and others, serve as unitary invariants for the tuple $ T $. A complete set of invariants in this case are described in \cite{Cowen-Douglas-connection}. However, identifying a complete and computationally accessible set of invariants remains a formidable challenge except for the bundles whose fibres are $ 1 $ dimensional.

Crucial in any study of such a class is the problem of finding a canonical model of a generic element in the class. In \cite{Curto-Salinas}, Curto and Salinas obtained a canonical model for a tuple in $ \mathrm B_r ( \CD ) $ as the adjoint of the $ m $-tuple of multiplication operators $ ( M_{ z_1 }^*, \hdots, M_{ z_m }^* ) $ by the coordinate functions on a reproducing kernel Hilbert space of $ \C^r $ valued holomorphic functions on the conjugated domain $ \CD_0^* = \{ z : \overline{ z } \in \CD_0 \} $. They further identified sufficient conditions on a reproducing kernel $ K $ defined on $ \CD^* $ ensuring that the tuple $ ( M_{ z_1 }^*, \hdots, M_{ z_m }^* ) $ on the reproducing kernel Hilbert space with the reproducing kernel $ K $ belongs to $ \mathrm B_r ( \CD ) $. 

We now provide a brief overview of free noncommutative function theory and refer interested readers to Section \ref{Preliminaries} for an elaborate discussion and the monograph \cite{Verbovetskyi-Vinnikov} for a comprehensive treatment. One way to understand this theory is as an extension of the theory of holomorphic functions of several commuting complex variables $ z = ( z_1, \hdots, z_d ) $ to the study of functions of matrix tuples $ Z = ( Z_1, \hdots, Z_d ) $ of freely noncommuting $ n \times n $ complex matrices $ Z_1, \hdots, Z_d $ for all $ n \in \N $. In this framework, the set of all matrices of any size over a vector space (or more generally, over any module over a commutative ring) serves as the analogue of $ \C^d $ in commutative function theory -- we refer to this set as the nc space over the vector space. A nc set is a subset of this nc space, which is closed under direct sum (cf. Subsection \ref{NC domains}), and a nc function is defined as a graded function on a nc set that takes values in a nc space and respects both direct sum and similarities. It turns out that any function defined on a nc set that is graded and respects intertwinings qualifies as a nc function (cf. Subsection \ref{NC functions}). When a nc space is equipped with an appropriate topology, there is a well defined notion of nc differentiability for nc functions  (cf. Subsection \ref{nc analyticity}). All nc spaces in this article are equipped with the uniformly-open topology (cf. Subsection \ref{uniform topology}; this topology is named as ``fat topology" in \cite{Agler-McCarthy2}) -- open sets in this topology are referred to as uniformly open sets -- and, any nc analytic function in this topology is called the uniformly analytic nc function (cf. Definition \ref{uniformly analyticity}).

In \cite{Taylor1}, Taylor introduced the theory of nc functions as a component of his broader framework for functional calculus, designed to apply to tuples of noncommuting operators (see also \cite{Taylor2}). This theory was later expanded and referred to as the ``Theory of Fully Matricial Functions" in the work of Voiculescu and his collaborators, particularly in the context of free probability and random matrix theory (cf. \cite{Voiculescu-free-probability}, and also \cite{Voiculescu-I, Voiculescu-II}). The ``Fully Matricial Sets" in this theory correspond to nc sets in our setting, with the additional requirement that they are closed under similarities at all levels. This theory has been applied in the research of Helton, Klep, McCullough, and their collaborators, specially on topics such as Linear-Matrix Inequalities in the free noncommutative set-up (see \cite{Helton-Klep-McCullough-free-proper, Helton-Klep-McCullough-convexity-semidefinite-programming, Helton-Klep-McCullough-Slinglend-nc-ball-maps, Helton-Klep-McCullough-Volcic-spectrahedra}) and in noncommutative convexity (\cite{Helton-McCullough-Putinar-Vinnikov, Helton-McCullough-Vinnikov}) as well as noncommutative real semialgebraic geometry (\cite{Helton-Klep-McCullough-convex-Positivstellensatz, Helton-Klep-McCullough-Free-convex-algebraic-geometry, Helton-McCullough-convex-free-basic-semi-algebraic-set}).

Rational functions and formal power series in $ d $ noncommuting indeterminates naturally emerge in the theory of automata and formal languages, as well as in the study of recognizable formal power series \cite{Berstel-Reutenauer, Eilenberg, Kalman-Falb-Arbib, Kleene, Schutzenberger1, Schutzenberger2}, and as transfer functions in multidimensional systems evolving along the free monoid \cite{Alpay-Verbovetskyi, Ball-Groenewald-Malakorn1, Ball-Groenewald-Malakorn2, Ball-Groenewald-Malakorn3, Ball-Verbovetskyi, Ball-Vinnikov-formal-kernel, Ball-Vinnikov}. Further relevant work includes the contributions by Hadwin \cite{Hadwin} and Hadwin-Kaonga-Mathes \cite{Hadwin-Kaonga-Mathes}, Popescu \cite{Popescu1, Popescu2, Popescu3, Popescu4, Popescu5, Popescu6, Popescu7, Popescu8} on general noncommuting tuples of operators, as well as the work by Muhly and Solel \cite{Muhly-Solel1, Muhly-Solel2, Muhly-Solel3, Muhly-Solel4} on generalized Hardy algebras associated with a $ W^* $-correspondence.

In the monograph \cite{Verbovetskyi-Vinnikov}, Kaliuzhnyi-Verbovetskyi and the second author of this article have developed the theory of noncommutative functions, which round off Taylor's theory, keeping an eye on further study of noncommutative function theory. In particular, it is shown how the ``invariant under direct sum and similarities" combined with some mild additional assumptions -- namely, the local boundedness -- lead to Taylor-Taylor series developments (at least locally) for a general noncommutative function. The recent work of Agler-McCarthy and collaborators \cite{Agler-McCarthy1, Agler-McCarthy2, Agler-McCarthy-Young1,Agler-McCarthy-Young2} have made further progress in the theory of noncommutative functions.

As mentioned earlier, the goal of the present article is to study the tuples of noncommuting operators possessing a ``nc open set" of joint eigenvalues on which the dimension of the corresponding joint eigenspaces becomes a ``nc constant function". It turns out that the family of joint nc eigenspaces of such a tuple exhibits a structure analogous to that of a hermitian holomorphic vector bundle in the classical setting. This investigation thus leads to a plausible notion of ``nc hermitian holomorphic vector bundles" which could be of independent interest for further research. 

Our starting point is the noncommuting tuple $ ( R_1^*, \hdots, R_d^* ) $ of adjoint of the right creation operators on the $ \ell^2 $ space $ \ell^2 ( \CG^d ) $ over the free monoid $ \CG^d $ (unital semi-group) of all words in $ d $ letters with the empty word as the unit element (cf. Example \ref{nc Hardy space}). For every $ d $ tuple of $ n \times n $ matrices $ ( X_1, \hdots, X_d ) $ with $ \norm{ X_1 X_1^* + \cdots + X_d X_d^* } < 1 $ (matrix norm), the following identity holds:
$$ \dim \left( \bigcap_{ j = 1 }^d \ker \left( R_j^*\otimes \mathrm{Id}_{ \C^{ n \times n } } - \mathrm{Id}_{ \ell^2 ( \CG^d ) } \otimes R_{ X_j } \right) \right) = n^2, $$
for all sizes $ n \in \N $, where $ R_X : \C^{ n \times n } \rightarrow \C^{ n \times n } $ is the right multiplication operator. Furthermore, each of these joint kernels is a free left module over $ \C^{ n \times n } $ of rank $ 1 $. Modeled after this example, we identify a class of noncommuting tuples of operators, referred to as the nc Cowen-Douglas class, denoted by $ \mathrm B_r ( \Omega )_{ nc } $, where $ r \in \N $ and $ \Omega \subset \C^d_{ nc } $ is a uniformly open nc domain. For $ \boldsymbol{ T } = ( T_1, \hdots, T_d ) \in \mathrm B_r ( \Omega )_{ nc } $ with $ T_1, \hdots, T_d \in \CB ( \CH ) $ for some complex separable Hilbert space $ \CH $, we set 
$$ D_{ \boldsymbol{T} - W } := ( T_1 \otimes \mathrm{Id}_{ \C^{ m \times m } } - \mathrm{Id}_{ \CH } \otimes R_{ W_1 }, \hdots, T_d \otimes \mathrm{Id}_{ \C^{ m \times m } } - \mathrm{Id}_{ \CH } \otimes R_{ W_d } ),  $$
for $ W = ( W_1, \hdots, W_d ) \in \Omega \cap ( \C^d )^{ m \times m } $,  $ m \in \N $. In this definition of nc Cowen-Douglas operator, the family $ \{ \ker D_{ \boldsymbol{T} - W } : W \in \Omega \} $ of matricial eigenspaces is required to satisfy a suitable density property. Additionally, $ \ker D_{ \boldsymbol{T} - W } $ must have complex dimension $ m r^2 $ as vector subspace of $ \CH^{ m \times m } $, for $ W \in \Omega \cap ( \C^d )^{ m \times m } $, and the range of $ D_{ \boldsymbol{T} - W } $ needs to be closed in $ \CH^{ m \times m } \oplus \cdots \oplus \CH^{ m \times m } $ ($ d $ times) for every $ m \in \N $ (see Definition \ref{Definition of nc CD class}).  

After establishing a plausible definition of the nc Cowen-Douglas class over a uniformly open set $ \Omega $, a systematic study of such operator tuples requires examining the nature of the subspace-valued mapping $ W \mapsto \ker D_{ \boldsymbol{T} - W } $ on $ \Omega $.  In particular, analogous to the classical case, we look for a uniformly analytic nc function $ \gamma : \Omega_0 \to \CH_{ nc } $ on possibly a smaller open subset $ \Omega_0 \subset \Omega $, satisfying $ \gamma ( W ) \in \ker D_{ \boldsymbol{T} - W } $. In other words, one is required to look for a uniformly analytic function $ \gamma : \Omega_0 \to \CH_{ nc } $ such that $ \CT ( W ) ( \gamma ( W ) ) = 0 $ for $ W \in \Omega_0 $, where $ \CT : \Omega \to \CB ( \CH )_{ nc } $ is a uniformly analytic nc function given by $ \CT ( W ) : = D_{ \boldsymbol{ T } - W } $.

To address this, it is necessary to establish the local existence of a uniformly analytic nc function $ \CG $ around a given point $ Y \in \Omega $ that takes values in the nc space over an operator space $ \CY $ and satisfies the equation 
$$ \CT ( X ) \star \CG ( X ) = \CF ( X ) , $$
for given uniformly analytic nc functions $ \CT $ and $ \CF $ on $ \Omega $ taking values in the nc space over operator spaces $ \CX $ and $ \CZ $, respectively, equipped with a multiplication $ \star : \CX \times \CY \to \CZ $, provided that $ \CG $ is known at some initial point. 

This result, Theorem \ref{closed range operators general version}, constitutes one of the main contributions of this article and may be of independent interest (cf. Remark \ref{rem on solving nc eqn}). The first half of Section \ref{solving nc equation} is devoted to proving this result, and in the latter half discusses various special cases. Notably, it is shown that some of the technical hypotheses on $ \CT $ in Theorem \ref{closed range operators general version} are automatically satisfied in specific settings (see Theorem \ref{closed range operators}). For instance, when $ \CX $ is replaced by the space $ \CB ( \CH ) $ of all bounded linear operators on a Hilbert space $ \CH $, equipped with the natural operator space structure, and $ \CY = \CZ = \CH $ are equipped with either the row, or column operator space structure, those hypotheses are inherently met.  

In Section \ref{nc reproducing kernel Hilbert space}, we begin by briefly revisiting the notion of the completely positive (cp) nc reproducing kernel Hilbert space $ \CH ( K ) $ of uniformly analytic nc functions on a uniformly open nc domain $ \Omega $ (cf. \cite{Ball-Marx-Vinnikov-nc-rkhs}, see also \cite{Ball-Vinnikov-formal-kernel, Popa-Vinnikov}). A key result in this section is the computation of the nc difference-differential of the kernel elements, which are shown to remain kernel elements within the same space. We derive a formula for reproducing the nc difference-differential of any function in this Hilbert space. 

The novel contributions of this section, however, lie in the study of the nc space over  $ \CH ( K ) $, as explored in Subsection \ref{coordinate free presentation}. Within this nc space, we identify certain special matrices whose entries are kernel elements in $ \CH ( K ) $, referred to as the \textbf{generalized kernel elements} (cf. Section \ref{coordinate free presentation}). Using the generalized kernel elements and their nc different-differentials, respectively, we establish identities for reproducing both the functional values and the values of the nc difference-differential of any $ F \in \CH \otimes \C^{ m \times m } $, $ m \in \N $. Additionally, the action of the adjoint of left multiplication operators by the nc coordinate functions on nc difference-differential of the \textbf{generalized kernel elements} has been computed (see Proposition \ref{properties of generalized kernel elements}). This computation is then utilized to describe the structure of certain joint invariant subspaces of the tuple formed by the adjoints of left multiplication operators by the nc coordinate functions on the given nc reproducing kernel Hilbert space (see Corollary \ref{nc generalized eigenspaces for mult op}). 

An important observation made along the way is that each level of this nc space, $ \CH ( K ) \otimes \C^{ m \times m } $, $ m \in \N $, admits the canonical Hilbert $ C^* $ module structure over $ \C^{ m \times m } $ induced from the Hilbert space structure on $ \CH ( K ) $. Leveraging this Hilbert $ C^* $ module structure on $ \CH \otimes \C^{ m \times m } $, we derive several equivalent criteria for two important properties of the kernel $ K $: (1) Non-degeneracy: $ \ker K ( W, W ) \cap \{ P \in \C^{ m \times m } : P \geq 0 \} = \{ 0 \} $ and (2) Strict positivity: $ K ( W, W ) ( P ) > 0 $ whenever $ P > 0 $, for $ W \in \Omega \cap ( \C^d )^{ m \times m } $. These results are formalized in Theorem \ref{equivalent criterion for non-degeneracy} and in Theorem \ref{equivalent criterion for strict positivity}, respectively. This analysis, we believe, may be of independent interest, offering insights into the intrinsic geometry of the reproducing kernel Hilbert space with a cp nc kernel. Later in Section \ref{model and local operators}, we make use of the results in this section to construct and study the model space of a nc Cowen-Douglas operator tuple and the associated local operators.

The contents of Section \ref{Noncommutative Cowen-Douglas class and vector bundles} constitute the nucleus of the present article. We begin this section by introducing the nc Cowen-Douglas class $ \mathrm B_r ( \Omega )_{ nc } $ of noncommuting tuples of operators (see Definition \ref{nc CD class}). As mentioned earlier, for $ \boldsymbol{ T } \in \mathrm B_r ( \Omega )_{ nc } $, the kernel of $ D_{ \boldsymbol{T} - W } $ is required to be of complex dimension $ m r^2 $ as vector subspace of $ \CH^{ m \times m } \oplus \cdots \oplus \CH^{ m \times m } $ ($ d $ times) for every $ m \in \N $. However, with the help of Flanders' theorem (see \cite[Theorem 1]{Flanders}), we prove that the kernel of $ D_{ \boldsymbol{T} - W } $ is, in fact, a free left submodule of $ \CH^{ m \times m } $ of free rank $ r $ over $ \C^{ m \times m } $ for each $ m \in \N $. This key observation plays a significant role in understanding such operator tuples and, alongside other properties of the family $ \{ \ker D_{ \boldsymbol{T} - W } : W \in \Omega \} $, sheds light on what a plausible notion of vector bundles in the free noncommutative category would be. 

Based on these ideas, we lay the foundation for nc hermitian holomorphic vector bundles in the latter half of this section. We first define what a trivial nc holomorphic vector bundle is and then restrict out attention to those nc vector bundles which arise as holomorphic subbundles of a trivial nc vector bundle (see Definition \ref{trivial nc bundle} and \ref{nc vector bundle}). The nc sections, bundle homomorphisms and transition data of nc vector bundles are explained in Subsections \ref{section and homomorphism} and \ref{transition data}. Along the way, we establish the nc analogue of (Čech) cocycle conditions, which closely resemble their classical counterparts. Notably, we show that nc vector bundles over certain nc domains, whose nc cocycles are cohomologous to the trivial nc cocycle, are isomorphic to trivial nc vector bundles. 

With the foundation of nc holomorphic vector bundles in place, we proceed to define a hermitian structure for these bundles (see Definition \ref{hermitian metric}). Specifically, a nc trivial bundle with the nc space over a Hilbert space as its fibre naturally inherits a canonical hermitian structure. This structure arises from the matrix sesquilinear form associated with Pisier's OH matrix norms (see \cite[Chapter 7, pg. 122]{Pisier}, \cite[Proposition 3.5.2]{Effros-Ruan} for OH matrix norm). Subbundles of this nc trivial bundle are then equipped with the hermitian structures induced from the one defined on the ambient bundle. 

It is worth noting that nc functions on a nc domain do not generally form a sheaf which is a crucial obstruction to carry over the classical notion of vector bundles in the realm of nc function theory. However, we believe that these embedded subbundles represent an important first step in understanding the structure of general nc vector bundles and their associated differential geometric properties -- a direction we plan to explore in a follow-up article. Moreover, this restricted class of nc vector bundles aligns well with the objectives of this article, as demonstrated later in this subsection. Specifically, we show that any operator tuple in $ \mathrm B_r ( \Omega )_{ nc } $ gives rise to such a vector bundle over a uniformly open nc subdomain of $ \Omega $ with the help of Theorem \ref{closed range operators general version}. Finally, we conclude this subsection by establishing a relationship between the unitary invariants of an element in $ \mathrm B_r ( \Omega )_{ nc } $ and those of the associated vector bundle (see Theorem \ref{unitary equivalence}). 

Following the work of Curto and Salinas in \cite{Curto-Salinas} on the construction of the models for classical Cowen-Douglas tuples, we develop a model for an element in the nc Cowen-Douglas class $ \mathrm B_r ( \Omega )_{ nc } $ in Section \ref{model}. It is shown that such a tuple can be realized as the adjoint of the $ d $-tuple of left multiplication operators by the nc coordinate functions on a nc reproducing kernel Hilbert space. This Hilbert space consists of uniformly analytic nc functions on some nc uniformly open subset of $ \Omega^* $ ($ \Omega^* : = \{ W : W^* \in \Omega \} $), taking values in the nc space over $ \CL ( \C, \C^r ) \cong \C^r $. Furthermore, we note that these nc reproducing kernels determine the $ d $-tuples of operators in $ \mathrm B_r ( \Omega )_{ nc } $ up to unitary equivalence. 

In Subsection \ref{local operators}, we generalize the notion of local operators, initially introduced for commuting operator tuples in the classical Cowen-Douglas class (\cite[\S 1.5, pp. 190]{Cowen-Douglas1}), to the noncommutative setting. We define a tuple of nc local operators of order $ \ell \in \N $ in Equation \eqref{local operator}, denoted by $ N^{ ( \ell ) }_{ \boldsymbol{ T }, W } $, associated to an element $ \boldsymbol{ T } \in \mathrm B_r ( \Omega )_{ nc } $ at the point $ W \in \Omega $. Informally, these tuples of operators are the restriction of the tuple $ D_{ \boldsymbol{ T } - W } $ to its ``generalized eigenspace at $ W $ in nc setting". Using the Taylor-Taylor series of uniformly analytic nc functions, we show that if the tuples of nc local operators of all orders associated to $ \boldsymbol{ T } $ and $ \widetilde{ \boldsymbol{ T } } $ in $ \mathrm B_r ( \Omega )_{ nc } $ are unitarily equivalent, then so is $ \boldsymbol{ T } $ and $ \widetilde{ \boldsymbol{ T } } $. It would be interesting to investigate if the unitary equivalence of $ N^{ ( \ell ) }_{ \boldsymbol{ T }, W } $ and $ N^{ ( \ell ) }_{ \widetilde{ \boldsymbol{ T } }, W } $ up to finite order would be sufficient for unitary equivalence for $ \boldsymbol{ T } $ and $ \widetilde{ \boldsymbol{ T } } $. We plan to pursue this project in the sequel.

After developing a model for elements in the nc Cowen-Douglas class, a natural question arises: Under what conditions is the tuple of adjoints of the left multiplication operators on a cp nc reproducing kernel Hilbert space of uniformly analytic nc functions itself in the nc Cowen-Douglas class? In addressing this question, we introduce the notion of the ``noncommutative Gleason problem", which we explore in Section \ref{On nc Gleason problem}.

The classical Gleason problem was first studied by Andrew Gleason in \cite{Gleason}, where he examined the maximal ideals of a commutative Banach algebra. In particular, he proved that if the maximal ideal of functions vanishing at the origin in the Banach algebra $\CA(\B(0,1))$ is finitely generated, then it must be generated by the coordinate functions. Here, $\CA(\B(0,1))$ denotes the Banach algebra of holomorphic functions on the open unit ball $\B(0,1)$ in $\C^n$ that extend continuously to the boundary equipped with the supremum norm. The question of whether the coordinate functions generate the maximal ideals in an algebra of holomorphic functions is now known as the \textit{Gleason problem}. 

In the noncommutative setting, we show that the nc Gleason problem (see Equation \eqref{Gleason solution}) is always globally solvable when the domain is a uniformly open nc set that is right admissible (see Theorem \ref{constructive solution for nc Gleason problem-algebraic version}). This is in contrast to the classical case, where solutions are not always guaranteed globally. Moreover, the solution to the nc Gleason problem is uniquely determined by the given function (see Theorem \ref{uniqueness of Gleason solution}), unlike the classical case, where uniqueness holds only for planar domains.

Finally, in Section \ref{application of nc Gleason problem}, we show that any nc Cowen-Douglas tuple can be realized as the tuple of adjoint of the left multiplication operators on a cp nc reproducing kernel Hilbert space of uniformly analytic nc functions where the nc Gleason problem is solvable (see Theorem \ref{complete description of CD tuple}).

\section{Preliminaries} \label{Preliminaries}

In this section, we review some foundational concepts from noncommutative function theory, which will serve as a basis for the discussions in subsequent sections. Our presentation largely follows the comprehensive treatment provided in \cite{Verbovetskyi-Vinnikov}, a definitive resource on the subject of noncommutative functions, to which we refer interested readers for a more detailed discussion and proofs.

First, let us fix some notations. For a set $ \CS $ and $ n, m \in \N $, denote $ \CS^{ n \times m } $ as the set of all $ n \times m $ matrices with entries in $ \CS $. When $ \CS =\C $, we use the usual notation $ \mathrm{GL} ( n, \C ) $ for the set of all $ n \times n $ invertible matrices. As mentioned earlier, the shorthand notation ``nc" is used for the word ``noncommutative" in what follows.

\subsection{Noncommutative domains} \label{NC domains} Let $ \CR $ be a commutative ring with identity. For a module $ \CM $ over $ \CR $, we define the \textit{nc space over} $ \CM $
\be \label{nc space} \CM_{ nc } : = \coprod_{ n =1 }^{ \infty } \CM^{ n \times n } . \ee
A subset $ \Omega \subset \CM $ is called a \textit{nc set} if it is closed under the direct sum: For $ m, n \in \N $,
$$ \Omega_n \oplus \Omega_m : = \left\{ A \oplus B : = \left[ \begin{smallmatrix} A & 0 \\ 0 & B \end{smallmatrix} \right] : A \in \Omega_n,~ B \in \Omega_m \right\} $$
where $ \Omega_j = \Omega \cap \CM^{ j \times j } $, $ j = n, m $. Evidently, for any $ X \in \Omega_s $ and $ m \in \N $, the $ ms \times ms $ matrix $ X^{ \oplus m } : = \left[ \begin{smallmatrix} X & & \\  & \ddots &  \\  &  & X \end{smallmatrix} \right] $ is in $ \Omega $. We call this matrix $ X^{ \oplus m } $ as the \textit{$m$-th amplification of $ X $}. 

Observe that the matrices over $ \CR $ act from the left and the right on the matrices over $ \CM $ by the usual matrix multiplication obtained from the action of $ \CR $ on $ \CM $. Namely, for $ A \in \CM^{ n \times m } $, $ \SX \in \CR^{ p \times n } $ and $ \SY \in \CR^{ m \times q } $, the matrix $ \SX  A \SY $ makes sense as an element in $ \CM^{ p \times q } $. In the case of $ \CM = \CR^d $, we identify the elements in $ \CM^{ n \times m } $ as $ d $-tuples of $ n \times m $ matrices over $ \CM $. Under this identification, the direct sum of these tuples and the action of the matrices over $ \CR $ on $ \CM^{ n \times m } $ happen to be entry-wise of these tuples. 

While nc sets are the natural domains for defining nc functions, an additional condition is required to develop a robust difference-differential calculus for these functions. Specifically, a nc set must be closed under the formation of upper (or lower) triangular block matrices with arbitrary upper (or lower, respectively) blocks. However, this condition is too restrictive to hold in general—it fails even for fundamental examples such as the nc ball and the nc poly-disc. The appropriate notion, as it turns out, is more subtle and is introduced as follows. 

A nc set $ \Omega \subset \CM_{ nc } $ is said to be \textit{right admissible} (\textit{left admissible}) if for every $ A \in \Omega_n $, $ B \in \Omega_m $ and $ C \in \CM^{ n \times m } $ ($ C \in \CM^{ m \times n } $, respectively) there exists an invertible element $ \mathsf{r} \in \CR $ such that $ \left[ \begin{smallmatrix} A & \mathsf{r} C \\ 0 & B \end{smallmatrix} \right] \in \Omega_{ n + m } $ ($ \left[ \begin{smallmatrix} A & 0 \\ \mathsf{r} C & B \end{smallmatrix} \right] \in \Omega_{ n + m } $, respectively).

\subsubsection{Full noncommutative sets.} \label{bifull} 

Later, on several occasions, we will need additional structure on the nc sets to study the nc functions defined on them. For our purpose, it is customary to consider the nc space $ \CV_{ nc } $ over a complex vector space $ \CV $.

Following \cite[Definition 2.4]{Ball-Marx-Vinnikov-interpolation} recall that a nc set $ \Omega \subset \CV_{ nc } $ is said to be \textit{left-full} if $ \Omega $ is \textit{invariant under left injective intertwining}: For $ n \geq m $, $ A \in \Omega_n $ and $ X \in \CV^{ m \times m } $, if there exists an injective $ \SI \in \C^{ n \times m } $ such that $ A \SI = \SI X $, then $ X \in \Omega_m $.

Analogously, we say that $ \Omega $ is \textit{right-full} if it is \textit{invariant under right surjective intertwining}: For $ n \geq m $, $ A \in \Omega_n $ and $ X \in \CV^{ m \times m } $, if there exists an surjective $ \SJ \in \C^{ m \times n } $ such that $ \SJ A = X \SJ $, then $ X \in \Omega_m $.

So we define a nc set $ \Omega $ is \textit{bi-full} if it is invariant under both the \textit{left injective intertwining} as well as the \textit{right surjective intertwining}.

Notice that any left or right full nc set is invariant under the similarity. Further, a similar argument as in \cite[$\S$2.2, pp 15]{Ball-Marx-Vinnikov-interpolation} shows that a nc set $ \Omega $ is \textit{bi-full} if and only if it is \textit{closed under the restriction to both invariant and co-invariant subspaces}: For $ n, m \in \N $ and $ A \in \Omega_{ n + m } $, whenever there exists an invertible $ \SQ \in \C^{ ( n+ m ) \times ( n + m ) } $ such that $ \SQ A \SQ^{ - 1 } = \left[ \begin{smallmatrix} X & Z \\ 0 & Y \end{smallmatrix} \right] $ with $ X \in \CV^{ n \times n }, ~ Y \in \CV^{ m \times m }, ~ Z \in \CV^{ n \times m } $, then not only is $ \SQ A \SQ^{ -1 } $ in $ \Omega_{ n + m } $, but also $ X \in \Omega_n $ and $ Y \in \Omega_m $. 

\subsubsection{Operator spaces and complete boundedness} \label{operator spaces}

To perform analysis on a nc set, it is essential to fix a topology on these sets. This is achieved by equipping the nc space $ \CV_{nc} $ over a vector space $ \CV $ with a topology induced by an operator space structure on $ \CV $. Before proceeding, we first recall the definition of an operator space, following \cite[Chapter 2]{Effros-Ruan}; see also \cite{Paulsen, Pisier} for additional references.

A normed linear space $ \CV $ is said to be an \textit{operator space} if it is equipped with a family of norms $ \| \cdot \|_n $ on $ \CV^{ n \times n } $ for $ n \in \N $ and there is a linear map $ \varphi : \CV \rightarrow \CL ( \CX ) $ for some Hilbert space $ \CX $, such that for every $ n \in \N $, the map
\be \label{amplification of linear map}
 \varphi^{ ( n ) } : = \varphi \otimes \mbox{Id}_{ \C^{ n \times n } } : \CV^{ n \times n } \rightarrow \CL ( \CX )^{ n \times n } ~~~ \mbox{ defined by } ~~~ ( \! ( v_{ i j } ) \! )_{ i, j = 1 }^n \mapsto ( \! ( \varphi ( v_{ i j } ) ) \! )_{ i, j = 1 }^n
\ee
is a linear isometry. Such a space can be represented as a system of norms $ \{ \| \cdot \|_n \} $ on $ \CV^{ n \times n } $, $ n \in \N $, satisfying the following two conditions:
\begin{itemize}
\item[(\textbf{R1})] $ \| A \oplus B \|_{ n + m } = \mbox{max} \{ \| A \|_n, \| B \|_m \} $ , $ A \in \CV^{ n \times n } $, $ B \in \CV^{ m \times m } $;
\item[(\textbf{R2})] $ \| \SX A \SY \|_m  \leq  \| \SX \| ~ \| A \|_n  ~ \| \SY \| $, $ A \in \CV^{ n \times n } $, $ \SX \in \C^{ m \times n } $, $ \SY \in \C^{ n \times m } $,
\end{itemize}
as proved by Ruan (see \cite[Proposition 2.3.6]{Effros-Ruan}).

Let $ \CV $ and $ \CW $ be operator spaces and $ T : \CV \rightarrow \CW $ be a linear mapping. Then $ T $ is said to be \textit{completely bounded} if 
$$ \norm{ T }_{ \text{cb} } : = \sup_{ n \in \N } \norm{ T \otimes \mbox{Id}_{ \C^{ n \times n } } } < \infty , $$
and \textit{completely contractive} if $  \norm{ T }_{ \text{cb} } \leq 1 $ where $ \norm{ T \otimes \mbox{Id}_{ \C^{ n \times n } } } $ is the norm of the operator $  T \otimes \mbox{Id}_{ \C^{ n \times n } } : \CV^{ n \times n } \rightarrow \CW^{ n \times n } $.

\subsubsection{Uniformly-open topology} \label{uniform topology}

Let $ \CV $ be an operator space and $ s \in \N $. For $ X \in \CV^{ s \times s } $ and $ r > 0 $, we define a \textit{nc ball} centred at the point $ X $ with radius $ r $ as 
$$ B_{ nc } ( X, r ) : = \coprod_{ m = 1 }^{ \infty } B ( X^{ \oplus m }, r ) = \coprod_{ m = 1 }^{ \infty } \left\{ A \in \CV^{ m s \times m s } : \left\| A - X^{ \oplus m } \right\| < r \right\} . $$
It can be seen (see \cite[Proposition 7.12]{Verbovetskyi-Vinnikov}) that the \textit{nc balls} of the above kind form a basis for a topology on $ \CV_{ nc } $. This topology is known as the \textit{uniformly-open topology}. We call the open sets in this topology as \textit{uniformly open}. The uniformly-open topology is never Hausdorff ( cf. \cite[Remark 7.13]{Verbovetskyi-Vinnikov}). However, the multiplication by a scalar as a mapping $ \C \times \CV_{ nc } \rightarrow \CV_{ nc } $ is continuous in the uniformly-open topology on $ \CV_{ nc } $.  

\begin{defn} \label{nc domain}
Let $ \CV $ be an operators space and the nc space $ \CV_{ nc } $ be equipped with the uniformly-open topology. Then a nc set $ \Omega \subset \CV_{ nc } $ is said to be a \textit{nc domain} if it is uniformly open and for each $ n \in \N $, $ \Omega_n \subset \CV^{ n \times n } $ is connected. 
\end{defn}

\subsection{Noncommutative functions} \label{NC functions}

Let $ \CV $ and $ \CW $ be operator spaces and $ \Omega \subset \CV_{ nc } $ be a nc set. A map $ f : \Omega \rightarrow \CW_{ nc } $ is said to be a \textit{nc function} if it satisfies the following conditions:
\begin{itemize}
\item[(1)] $ f $ is graded: $ f ( \Omega_n ) \subset \CW^{ n \times n } $ for all $ n \in \N $;
\item[(2)] $ f $ respects direct sums: $ f ( X \oplus Y ) = f ( X ) \oplus f( Y ) $ for all $ X, Y \in \Omega $;
\item[(3)] $ f $ respects similarities: $ f ( \SA X \SA^{ - 1 } ) = \SA f ( X ) \SA^{ - 1 } $ for all $ X \in \Omega_n $ and $ \SA \in \mbox{GL} ( n, \C ) $ such that $ \SA X \SA^{ -1 } \in \Omega_n $, $ n \in \N $.
\end{itemize} 
It is well known (see \cite[Proposition 2.1]{Verbovetskyi-Vinnikov}) that the conditions $ ( 2 ) $ and $ ( 3 ) $ in the definition above can be replaced by a single condition, namely, a graded mapping $ f : \Omega \rightarrow \CW_{ nc } $ respects the direct sums and similarities if and only if it obeys the \textit{intertwining}: For $ X \in \Omega_n $, $ Y \in \Omega_m $ and $ \SA \in \C^{ n \times m } $,
\be \label{itertwining condition} 
 f ( X ) \SA = \SA  f ( Y ) \quad \text{whenever} \quad X \SA = \SA Y . 
\ee
J. Taylor (cf. \cite{Taylor2} ) used this condition to define a nc function. A nc function of the above kind is called a \textit{nc function of order $ 0 $}. For a more detailed study of nc functions and their various applications in other fields of mathematics, we refer the readers to \cite{Agler-McCarthy2, Verbovetskyi-Vinnikov, Voiculescu-free-probability}. 

\subsubsection{Examples} \label{examples}

(a) A prototypical example of nc functions is a nc polynomial. In particular, note that the nc polynomial $ p : \C^2_{ nc } \rightarrow \C_{ nc } $ of degree $ 6 $ defined by
$$ p ( X_1, X_2 ) =  \mathrm{Id} + 5 X_1 + 8 X_1 X_2 + 11 X_2 X_1 + 7 X_1^3 X_2^2 + 3 X_2 X_1^2 X_2 X_1 + 2 X_2^ 6 $$
is a nc function where $ \mathrm{Id} $ denotes the identity matrix of appropriate size. For instance, if $ X_1, X_2 \in \C^{ n \times n } $ then $ \mbox{Id} $ is the $ n \times n $ identity matrix.

(b) Let $ T : \CV \rightarrow \CW $ be a linear mapping of operator spaces $ \CV $ and $ \CW $. Then note that $ T $ gives rise to a nc function $ \widetilde{ T } : \CV_{ nc } \rightarrow \CW_{ nc } $ such that 
$$ T_n : = \widetilde{ T }|_{ \CV^{ n \times n } } : = \widetilde{ T } \left( ( \! ( v_{ i j } ) \! )_{ i, j = 1 }^n \right) = ( \! ( T ( v_{ i j } ) ) \! )_{ i, j = 1 }^n $$ is the linear mapping from $ \CV^{ n \times n } $ to $ \CW^{ n \times n } $, that is, $ T_n = T \otimes \mbox{Id}_{ \C^{ n \times n } } $. 

\subsubsection{Higher order nc functions}

 There is a notion of higher order nc functions, which are vaguely mappings from the cartesian product of nc sets to a certain space of multi-linear mappings with the property that on each variable, such a mapping becomes a nc function in the sense mentioned above. For instance, if $ f :  \Omega^0 \rightarrow \C_{ nc } $ and $ g : \Omega^1 \rightarrow \C_{ nc } $ are nc functions of order $ 0 $ on nc sets $ \Omega^0 $ and $ \Omega^1 $, respectively, then the function $ F : \Omega^0 \times \Omega^1 \rightarrow \coprod_{ m, n \geq 1 } \CL ( \C^{ n \times m } ) $ defined by 
$$ F ( Z, W ) : = f ( Z ) \otimes g ( W ) : \C^{ n \times m } \rightarrow \C^{ n \times m } $$
is a nc function of order $ 1 $ in the sense of the definition given in Section 3.1 in \cite{Verbovetskyi-Vinnikov} (also, cf. \cite[Remark 3.6]{Verbovetskyi-Vinnikov}). For $ k + 1 $ vector spaces $ \CW_0, \hdots, \CW_k $, we denote $ \CT^k ( \Omega, \CW_{ 0, nc }, \hdots, \CW_{ k, nc } ) $ as the space of all $ k $-th order nc functions $ f $ on $ \Omega \times \cdots \times \Omega $ ($ k + 1 $ times) with 
$$ f ( \Omega_{ n_0 }, \hdots, \Omega_{ n_k } ) \subset  \mathrm{Hom} ( \CW_1^{ n_0 \times n_1 } \otimes \cdots \otimes \CW_k^{ n_{ k - 1 } \times n_k }, \CW_0^{ n_0 \times n_k } ) . $$
For a detailed discussion on higher order nc functions, we refer the readers to \cite[Chapter 3]{Verbovetskyi-Vinnikov}. However, the definition of nc functions of order $ 1 $ will be given at the beginning of Section \ref{nc reproducing kernel Hilbert space} in the context of nc reproducing kernels.

\subsection{Noncommutative analyticity} \label{nc analyticity} The objective of the present subsection is to introduce noncommutative analytic functions with respect to the uniformly-open topology. 

\subsubsection{Noncommutative difference-differential operators}

Let $ \Omega $ be a right admissible nc set and $ f : \Omega \rightarrow \CW_{ nc } $ be a nc function. Then \textit{$ \ell $-th order right nc difference-differential operator} is defined as the linear map 
$$ \Delta_R^{ \ell } : \CT^0 ( \Omega, \CW_{ nc } ) \rightarrow \CT^{ \ell } ( \Omega, \CW_{ nc }, \underbrace{ \CV_{ nc }, \hdots, \CV_{ nc } }_{ \ell - \mbox{times} } ) $$
where for $ f \in \CT^0 ( \Omega, \CW_{ nc } ) $ and for every $ X^0 \in \Omega_{ n_0 }, \hdots, X^{ \ell } \in \Omega_{ n_{ \ell } } $, $ Z^1 \in \CV^{ n_0 \times n_1 }, \hdots, Z^{ \ell } \in \CV^{ n_{ \ell - 1 } \times n_{ \ell } } $ such that 
$$ \left[ \begin{smallmatrix} 
X^0 & Z^1 & 0 & \cdots & 0 \\
0 & X^1 & Z^2 & \ddots & 0 \\
\vdots & \ddots & \ddots & \ddots & 0 \\
\vdots &   & \ddots & X^{ \ell - 1 } & Z^{ \ell} \\
0 & \cdots & \cdots & 0 & X^{ \ell} \\
\end{smallmatrix} \right] \in \Omega_{ n_0 + n_1 + \cdots + n_{ \ell } },
$$
$ \Delta_R^{ \ell } f $ is given by the formula:

\begin{multline} \label{nc difference-differential operator}
f \left( \left[ \begin{smallmatrix} 
X^0 & Z^1 & 0 & \cdots & 0 \\
0 & X^1 & Z^2 & \ddots & 0 \\
\vdots & \ddots & \ddots & \ddots & 0 \\
\vdots &   & \ddots & X^{ \ell - 1 } & Z^{ \ell} \\
0 & \cdots & \cdots & 0 & X^{ \ell} \\
\end{smallmatrix} \right] \right) = \\
\left[ \begin{smallmatrix} 
f ( X^0 ) & \Delta^1_R f ( X^0, X^1 ) ( Z^1 ) & \cdots & \cdots & \Delta^{ \ell }_R f ( X^0, \hdots, X^{ \ell } ) ( Z^1, \hdots, Z^{ \ell } ) \\
0 & f ( X^1 ) & \Delta^1_R f ( X^1, X^2 ) ( Z^2 ) & \cdots & \Delta^{ \ell - 1 }_R f ( X^1, \hdots, X^{ \ell } ) ( Z^2, \hdots, Z^{ \ell } ) \\
\vdots & \ddots & \ddots & \ddots & \vdots \\
\vdots &   & \ddots & f ( X^{ \ell - 1 } ) & \Delta^1_R f ( X^{ \ell - 1 }, X^{ \ell } ) ( Z^{ \ell} ) \\
0 & \cdots & \cdots & 0 & f ( X^{ \ell} ) \\
\end{smallmatrix} \right] .
\end{multline}
The block structure in this equation is preserved due to $ f $ being nc (cf. \cite[Theorem 3.11]{Verbovetskyi-Vinnikov}). When $ X^0 = X^1 = \cdots = X^{ \ell } = Y \in \Omega_s $, we get \textit{$ \ell $-th order right nc differential operator} at $ Y $, which is the $ \ell $-linear mapping 
\be \label{nc differential operator} 
\Delta^{ \ell }_R f ( Y, \hdots, Y ) : ( \CV^{ s \times s } )^{ \ell } \rightarrow \CW^{ s \times s } 
\ee
defined by Equation \eqref{nc difference-differential operator}. In a similar manner, one can define the \textit{left nc difference-differential operator} $ \Delta_L f $ by evaluating $ f $ on a lower triangular bi-diagonal block matrix. It follows that the left and right nc difference-differential operators are related via the following identity (see \cite[Proposition 2.8]{Verbovetskyi-Vinnikov})
\begin{align} \label{left vs right dd operators}
& \Delta^{ \ell }_L f ( X^0, X^1, \hdots, X^{ \ell - 1 }, X^{ \ell } ) ( Z^1, Z^2, \hdots, Z^{ \ell - 1 }, Z^{ \ell } ) \\
\nonumber & = \Delta^{ \ell }_R f ( X^{ \ell }, X^1, \hdots, X^{ \ell - 1 }, X^0 ) ( Z^{ \ell }, Z^2, \hdots, Z^{ \ell -1 }, Z^1 ) .
\end{align}

When $ \CV = \C^d $, we define, for all $ n, m \in \N $, $ X \in \Omega_n $, $ Y \in \Omega_m $, the \textit{$ j $-th right partial difference-differential operator} $ \Delta_{ R, j } f ( X, Y ) : \C^{ n \times m } \rightarrow \CW^{ n \times m } $ by 
\be \label{nc j-th partial difference-differential operator}
\Delta_{ R, j } f ( X, Y ) ( Z ) = \Delta_R f ( X, Y ) ( Z \varepsilon_j )
\ee
where $ \{ \varepsilon_j \}_{ j =1 }^d $ is the standard ordered basis for $ \C^d $. By linearity, we then have 
\be \label{nc partial difference-differential operator}
\Delta_R f ( X, Y ) ( Z ) = \sum_{ j = 1 }^d \Delta_{ R, j } f ( X, Y ) ( Z_j ).
\ee
Furthermore, $ \Delta_R $ satisfies the following nc finite difference formula: For $ X $ and $ Y $ as above and for $ \SS \in \C^{ m \times n } $,
\be \label{finite difference formula}
\SS f ( X ) - f ( Y ) \SS = \Delta_R f ( Y, X ) ( \SS X - Y \SS ) = \sum_{ j = 1 }^d \Delta_{ R, j } f ( Y, X ) ( \SS X_j - Y_j \SS ) .
\ee
Now using the nc finite difference formula for the higher order nc functions (see \cite[Theorem 3.19, pp 54]{Verbovetskyi-Vinnikov}), one obtains the following Taylor-Taylor formula (see \cite[Theorem 4.2, pp 62]{Verbovetskyi-Vinnikov}).
\begin{thm}
    Let $ f \in \CT^0 ( \Omega, \CW_{ nc } ) $ and $ Y \in \Omega_s $. Then for each $ N \in \N $ and arbitrary $ m \in \N $ and $ X \in \Omega_{ m s } $,
    \begin{multline} \label{TT formula}
        f ( X ) = \sum_{ \ell = 0 }^N \sum_{ | \beta | = \ell } \big( X - Y^{ \oplus m } \big)^{ \odot_s \beta } \Delta^{ \beta^{ \top } }_R f ( Y, \hdots, Y ) \\
        + \sum_{ | \beta | = N + 1 } \big( X - Y^{ \oplus m } \big)^{ \odot_s \beta } \Delta^{ \beta^{ \top } }_R f ( \underbrace{ Y, \hdots, Y}_{ N + 1 ~ \text{times}}, X ).
    \end{multline} 
\end{thm}

Note that Equations \eqref{nc partial difference-differential operator}, \eqref{finite difference formula} and \eqref{TT formula} are held by $ \Delta_L $ as well.

\subsubsection{Amplifying operators} \label{amplifying operators}

An $ \ell $-linear map $ T : \CV^{ \ell } \rightarrow \CW $ can be viewed as a linear transformation $ T : \CV^{ \otimes \ell } \rightarrow \CW $. For every $ n, m \in \mathbb{N} $, this map can be extended to a linear mapping $ ( \CV^{ \otimes \ell } )^{ n \times m } \rightarrow \CW^{ n \times m } $ by applying $ T $ entry-wise. Composing this extension with the canonical identification $ ( \CV^{ n \times m } )^{ \otimes \ell } \simeq ( \CV^{ \otimes \ell } )^{ n \times m } $, we obtain an $ \ell $-linear map $ T_{n,m} : ( \CV^{ n \times m } )^{ \ell } \rightarrow \CW^{ n \times m } $, which we refer to as the \textit{amplification} of $ T $. When $ n = m $, we denote this map as $ T_n $ and call it the $ n $-th amplification of $ T $. For $ m = \ell = 1 $, this reduces to the standard notion of the \textit{$ n $-th amplification} of a matrix.

In the special case where $ \CV = ( \C^d )^{ s \times s } \simeq ( \C^{ s \times s } )^d $, an $ \ell $-linear map $ T : \left( ( \C^{ s \times s } )^d \right)^{ \ell } \rightarrow \C^{ s \times s } $ extends to an $ \ell $-linear map $ T_n : ( ( \C^d )^{ ns \times ns } )^{ \ell } \rightarrow \C^{ ns \times ns} $ as above using the identification $ ( \C^{ s \times s } )^{ n \times n } \cong \C^{ ns \times ns} $. Following \cite{Verbovetskyi-Vinnikov}, this amplified map can be obtained as 
$$ T_n ( Z^1, \hdots, Z^{ \ell } ) = ( Z^1 \odot_s \cdots \odot_s Z^{ \ell } ) T $$
for $ Z^1, \hdots, Z^{ \ell } \in ( \C^{ ns \times ns } )^d = ( ( \C^d )^{ s \times s } )^{ n \times n } $ using the faux product $ \odot_s $ for $ n \times n $ matrices over the tensor algebra $ \T ( ( \C^d )^{ s \times s} ) $.

The amplification of the nc differential operators acts on the amplifying points in the nc set according to the formula below (cf. \cite[Proposition 3.3]{Verbovetskyi-Vinnikov}):
$$ ( \Delta^{ \ell }_R f ( Y^{ \oplus n }, \hdots, Y^{ \oplus n } ) ) ( Z^1, \hdots, Z^{ \ell } ) =  \Delta^{ \ell }_R f ( Y, \hdots, Y )_n ( Z^1, \hdots, Z^{ \ell } ) . $$

\subsubsection{Canonical intertwining conditions}
The nc differential operators of a nc function obey certain intertwining conditions due to the property mentioned in Equation \eqref{itertwining condition}. We are now about to discuss those conditions.

\begin{defn} \label{canonical intertwining conditions of finite order}
Let $ \CV $ be a vector space, $ Y \in \CV^{ s \times s } $ and $ k \in \N $. A sequence $ \{ f_{ \ell } \}_{ \ell = 0 }^k $ of $ \ell $-linear mappings satisfies the \textit{canonical intertwining conditions of order $ k $} at $ Y $ (in short $ \mbox{CIC}_k ( Y ) $) if for all $ Z^1, \hdots, Z^{ \ell - 1 } \in \CV^{ s \times s } $,
$$ f_1 ( [ \SS, Y ] ) = [ \SS, f_0 ], $$
and 
\begin{align*}
& f_{ \ell } ( [ \SS, Y ], Z^1, \hdots, Z^{ \ell -1 } ) = \SS f_{ \ell -1 } ( Z^1, \hdots, Z^{ \ell - 1 } ) - f_{ \ell - 1 } ( \SS Z^1, Z^2, \hdots, Z^{ \ell -1 } ),\\
& f_{ \ell } ( Z^1, \hdots, Z^j, [ \SS, Y ], Z^{ j + 1 }, \hdots, Z^{ \ell -1 } ) =  f_{ \ell -1 } ( Z^1, \hdots, Z^{ j - 1 }, Z^j \SS, Z^{ j + 1 }, \hdots,  Z^{ \ell - 1 } )\\
& - f_{ \ell - 1 } ( Z^1, \hdots, Z^j, \SS Z^{ j + 1}, Z^{ j + 2 }, \hdots, Z^{ \ell -1 } ), \\
& f_{ \ell } ( Z^1, \hdots, Z^{ \ell - 1 }, [ \SS, Y ] ) = f_{ \ell - 1 } ( Z^1, \hdots, Z^{ \ell - 2 }, Z^{ \ell - 1 } \SS ) - f_{ \ell - 1 } ( Z^1, \hdots, Z^{ \ell - 1 } ) \SS
\end{align*}
for all $ 2 \leq \ell \leq k $, $ 1 \leq j \leq \ell - 2 $, $ \SS \in \C^{ s \times s } $, and
\begin{align*}
& \SS f_k ( Z^1, \hdots, Z^{ \ell - 1 } ) = f_k ( \SS Z^1, \hdots, Z^{ \ell - 1 } ),\\
& f_k ( Z^1, \hdots, Z^{ j - 1 }, Z^j \SS, Z^{ j + 1 }, \hdots,  Z^{ \ell - 1 } ) = f_k ( Z^1, \hdots, Z^j, \SS Z^{ j + 1}, Z^{ j + 2 }, \hdots, Z^{ \ell -1 } ), \\
& f_k ( Z^1, \hdots, Z^{ \ell - 2 }, Z^{ \ell - 1 } \SS ) = f_k ( Z^1, \hdots, Z^{ \ell - 1 } ) \SS,
\end{align*}
for all $ 1 \leq j \leq k - 1 $, $ \SS \in \{ \SS \in \C^{ s \times s } : \SS Y = Y \SS \} $.
\end{defn}
We say that $ f_0 $ satisfies $ \mbox{CIC}_0 ( Y ) $ if $ [ f_0, \SS ] = 0 $ for all $ \SS \in \{ \SS \in \C^{ s \times s } : \SS Y = Y \SS \} $.

\begin{defn} \label{canonical intertwining conditions}
Let $ \CV $ be a vector space, $ Y \in \CV^{ s \times s } $ and $ k \in \N $. A sequence $ \{ f_{ \ell } \}_{ \ell = 0 }^{ \infty } $ of $ \ell $-linear mappings satisfies the \textit{canonical intertwining conditions} at $ Y $ (in short $ \mbox{CIC} ( Y ) $) if $ \{ f_{ \ell } \}_{ \ell = 0 }^k $ satisfies $ \mbox{CIC}_k ( Y ) $ for all $ k \in \N \cup \{ 0 \} $.
\end{defn}

Note that a sequence $ \{ f_{ \ell } \}_{ \ell = 0 }^{ \infty } $ of $ \ell $-linear mappings satisfies $ \mbox{CIC} ( Y ) $ in this sense if and only if it satisfies the conditions given in \cite[Reark 4.3, pp 62]{Verbovetskyi-Vinnikov}.
\subsubsection{Uniformly analytic nc functions}

Let $ \CV $ and $ \CW $ be operator spaces and $ \Omega \subset \CV_{ nc } $ be a uniformly open subset. A nc function $ f : \Omega \rightarrow \CW_{ nc } $ is said to be \textit{Gâteaux (G-) differentiable} if for every $ n \in \N $ the function $ f |_{ \Omega_n } $ is G-differentiable, that is, for $ X \in \Omega_n $ and $ Z \in \CV^{ n \times n } $, the G-derivative of $ f $ at $ X $ along the direction $ Z $ defined by 
\be \label{G - derivative}
\delta f ( X ) ( Z ) = \lim_{ t \rightarrow 0 } \frac{ f ( X + t Z ) - f ( X ) }{ t } = \frac{ d }{ d t } f ( X + t Z ) |_{ t = 0 }
\ee
exists. Also, recall that $ f $ is called \textit{uniformly locally bounded} if for each $ s \in \N $ and $ Y \in \Omega_s $, there exists $ \varrho  > 0 $ (depending on $ Y $) such that the nc ball $ B_{ nc } ( Y, \varrho ) \subset \Omega $ and $ f $ is bounded on $ B_{ nc } ( Y, \varrho ) $, that is, there is $ M > 0 $ such that 
$$ \| f ( X ) \|_{ sm } \leq M ~~~ \mbox{for all} ~~~ m \in \N ~~~ \mbox{and} ~~~ X \in B_{ nc } ( Y, \varrho )_{ s m } . $$

\begin{defn} \label{uniformly analyticity}
Let $ \CV $ and $ \CW $ be operator spaces and $ \Omega \subset \CV_{ nc } $ be a uniformly open subset. A nc function $ f : \Omega \rightarrow \CW_{ nc } $ is said to be \textit{uniformly analytic} if $ f $ is uniformly locally bounded and G-differentiable on $ \Omega $. 
\end{defn}

Recall from Example (b) (\ref{examples}) above that any linear mapping $ T : \CV \rightarrow \CW $ gives rise to a nc function $ \widetilde{T} : \CV_{ nc } \rightarrow \CW_{ nc } $. It follows that $ \widetilde{T} $ is uniformly locally bounded if and only if $ \widetilde{T} $ is continuous with respect to the uniformly-open topologies if and only if $ T $ is completely bounded. In general, the first two equivalences happen to be true for any nc functions. Moreover, we have the following striking result relating the local boundedness with the nc analyticity of a nc function.

\begin{thm} \cite[Corollary 7.28]{Verbovetskyi-Vinnikov}\label{locally bounded = nc analytic}
Let $ \Omega \subset \CV_{ nc } $ be a uniformly open nc set. Then a nc function $ f : \Omega \rightarrow \CV_{ nc } $ is uniformly locally bounded if and only if $ f $ is continuous with respect to the uniformly-open topologies if and only if $ f $ is uniformly analytic.
\end{thm}

\subsubsection{Taylor-Taylor series}

Likewise, in the classical complex analysis, an equivalent condition for nc analyticity is the local existence of a power series for nc functions, which we describe below. First, following \cite[Section 7.4, pp. 114]{Verbovetskyi-Vinnikov}, we introduce the notion of bounded and completely bounded multi-linear mappings.

For Banach spaces $ \CW_0, \CW_1, \hdots, \CW_k $, a $ k $-linear mapping $ \omega : \CW_1 \times \cdots \times \CW_k \rightarrow \CW_0 $ is called \textit{bounded} if 
\be \label{bounded multi-linear mapping}
\| \omega \|_{ \CL^k } : = \sup_{ \| w^1 \| = \cdots = \| w^k \| = 1 } \| \omega ( w^1, \hdots, w^k ) \| < \infty . 
\ee
We use the notation $ \CL^k ( \CW_1, \hdots, \CW_k, \CW_0 ) $ for the Banach space of bounded $ k $-linear mappings. 

When $ \CW_0, \CW_1, \hdots, \CW_k $ are equipped with operator space structures, we define, for $ n_0, n_1, $ $ \hdots, n_k \in \N $, a $ k $-linear mapping $ \omega^{ ( n_0, \hdots, n_k ) } : \CW_1^{ n_0 \times n_1 } \times \cdots \times \CW_k^{ n_{ k - 1 } \times n_k } \rightarrow \CW_0^{ n_0 \times n_k } $ as 
\be \label{definition of amplified multi-linear mapping}
\omega^{ ( n_0, \hdots, n_k ) } ( W^1, \hdots, W^k ) = ( W^1 \odot \cdots \odot W^k ) \omega, 
\ee
where $ \odot $ is the faux product as mentioned in \eqref{amplifying operators} and $ \omega $ acts on the matrix $ ( W^1 \odot \cdots \odot W^k ) \in ( \CW_1 \otimes \cdots \otimes \CW_k )^{ n_0 \times n_k } $ from right entry-wise. The $ k $-linear mapping $ \omega $ is called \textit{completely bounded} if 
\be \label{completely bounded multi-linear mapping}
\| \omega \|_{ \CL^k_{ \text{cb} } } : = \sup_{ n_0, \hdots, n_k \in \N } \| \omega^{ ( n_0, \hdots, n_k ) } \|_{ \CL^k } < \infty .
\ee
From now on the Banach space of such completely bounded $ k $-linear mappings will be denoted by $ \CL^k_{ \text{cb} } ( \CW_1 \times \cdots \times \CW_k, \CW_0 ) $. Following \cite[Chapter 17]{Paulsen} or \cite[Chapter 5]{Pisier}, we observe that $ \| \omega \|_{ \CL^k_{ \text{cb} } } $ turns out to be the completely bounded norm of $ \omega $ viewed as a linear mapping $ \omega : \CW_1 \otimes \cdots \otimes \CW_k \rightarrow \CW_0 $ where the vector spaces $ ( \CW_1 \otimes \cdots \otimes \CW_k )^{ n_0 \times n_k } $ with $ n_0, n_k \in \N $ are endowed with the \textit{Haagerup norm}, $ \| \cdot \|_{ \text{H}, n_0, n_k } $:
\be \label{Haagerup norm}
\| W \|_{ \text{H}, n_0, n_k } : = \inf_{ n_1, \hdots, n_{ k - 1 } } \{ \| W^1 \|_{ n_0, n_1 } \cdots \| W^k \|_{ n_{ k - 1 }, n_k } : W = W^1 \odot \cdots \odot W^k \} .
\ee
In this setting, the following result relates the uniformly locally boundedness of a nc function $ f $ to the convergence of the Taylor-Taylor series in the uniformly-open topologies.

\begin{thm}\cite[Theorem 7.21 and 8.11]{Verbovetskyi-Vinnikov} \label{convergence of TT series} 
Let $ \CV $ and $ \CW $ be operator spaces, $ \Omega \subset \CV_{ nc } $ be a uniformly open subset and $ Y \in \Omega_s $ for some $ s \in \N $. If a nc function $ f : \Omega \rightarrow \CW_{ nc } $ is uniformly locally bounded, then there exists $  r > 0 $ (depending on $ Y $) such that 
\be \label{TT series of nc function}
f ( X ) = \sum_{ \ell = 0 }^{ \infty } ( X - Y^{ \oplus m } )^{ \odot_s \ell } \Delta^{ \ell }_R f ( Y, \hdots, Y ) 
\ee
for all $ X \in B_{ nc } ( Y, r )_{ ms } $ and $ m \in \N $, where the series in the equation above converges absolutely and uniformly and $ \limsup_{ \ell \rightarrow \infty } \sqrt[ \ell ]{ \| \Delta^{ \ell }_R f ( Y, \hdots, Y ) } \|_{ \CL^{ \ell }_{ \text{cb} } } < \infty $.

Conversely, let $ \{ f_{ \ell } \}_{ \ell = 0 }^{ \infty } $ be a sequence of multi-linear mappings satisfying the canonical intertwining conditions $ \mbox{CIC} ( Y ) $ and 
$$ \mu_{ \text{cb} } : = \limsup_{ \ell \rightarrow \infty } \sqrt[ \ell ]{ \| f_{ \ell } \|_{ \CL^{ \ell }_{ \text{cb} } } } < \infty  . $$
Then the Taylor-Taylor series
\be \label{TT series}
\sum_{ \ell = 0 }^{ \infty } ( X - Y^{ \oplus m } )^{ \odot_s \ell } f_{ \ell }
\ee
converges absolutely and uniformly for all $ X \in B_{ nc } ( Y, \varrho ) $ for $ 0 < \varrho < \frac{ 1 }{ \mu_{ \text{cb} } } $, and the function defined by the formula given in \eqref{TT series} is an uniformly locally bounded nc function on  $ X \in B_{ nc } ( Y, \varrho ) $.
\end{thm}

\section{Solving a parametrized family of noncommuting equations} \label{solving nc equation}

Let $ \CV $, $ \CX $, $ \CY $ and $ \CZ $ be operator spaces, $ \Omega \subset \CV_{ nc } $ be a uniformly open bounded nc set. Consider a bilinear mapping $ \mathfrak{ m } : \CX \times \CY \rightarrow \CZ $ denoted by 
$$
\mathfrak{ m } ( f, g ) : = f \star g .
$$
For each $ n \in \N $, we define the $n$-th amplification $ \mathfrak{ m }_n : \CX^{ n \times n } \times \CY^{ n \times n }\rightarrow \CZ^{ n \times n } $ of $ \mathfrak{ m } $ as: For $ F = ( \! ( f_{ i j } ) \! )_{ i, j = 1}^n $ and $ G = ( \! ( g_{ i j } ) \! )_{ i, j = 1 }^n $,
\be \label{matrix multiplication}
\mathfrak{ m }_n ( F, G )_{ i j } = ( F \star G )_{ i j } : = \sum_{ k = 1 }^n f_{ i k } \star g_{ k j }, ~~~ 1 \leq i, j \leq n .
\ee
Since $ \mathfrak{ m } $ is a bilinear mapping, $ \mathfrak{ m }_n $ is invariant under the canonical left $ \C^{ n \times n } $ module action on $ \CX $ as well as right $ \C^{ n \times n } $ module action on $ \CY $ for every $ n \in \N $. Moreover, $ \mathfrak{ m }_n $ satisfies
$$ \mathfrak{ m }_n ( F \SA, G ) = \mathfrak{ m }_n ( F, \SA G ), ~~ \SA \in \C^{ n \times n } . $$

Given operator spaces $ \CX $, $ \CY $ and $ \CZ $ equipped with such a multiplication $ \star $ and uniformly analytic nc functions $ \CT : \Omega \rightarrow \CX_{ nc } $ and $ \CF : \Omega \rightarrow \CZ_{ nc } $, the objective of the present section is to solve a noncommutative equation of the form 
\be \label{parametrised linear equation} ( \CT  \star \CG ) ( X ) : = \CT ( X ) \star \CG ( X ) = \CF ( X ) \ee
on a possibly a smaller uniformly open subset of $ \Omega $ for a uniformly analytic nc function $ \CG $, provided that the initial data at some point is known (cf. Theorem \ref{closed range operators general version}). In the first half of this section, we prove the existence of a solution to the equation above in the general set-up, while in the latter half, we discuss its applications in various special cases. 

\subsection{General case} We now proceed to prove one of the main results, Theorem \ref{closed range operators general version}. Heuristically, the idea behind the proof is to construct a family of $ \ell $-linear mappings $ \CG_{ \ell } : ( \CV^{ s \times s } )^{ \ell } \rightarrow \CY^{ s \times s } $ for $ \ell \geq 0 $ with a prescribed growth condition on their completely bounded norm satisfying the canonical intertwining conditions CIC$ ( Y ) $ at $ Y $ as well as an algebraic identity obtained from Equation \eqref{parametrised linear equation} so that the associated Taylor-Taylor series at $ Y $ converges uniformly on a nc ball around $ Y $ yielding the required uniformly analytic nc function $ \CG $. Before going into the details, we wish to recall some terminologies from algebra and present an algebraic lemma -- which might be of independent interest -- demonstrating the relation between the Taylor-Taylor coefficients of $ \CT, \CG $ and $ \CF $ induced from their canonical intertwining conditions together with Equation \eqref{parametrised linear equation}. 

First, we introduce some useful algebraic notations. Let $ \CV $ be an operator space, $ Y  \in \CV^{ s \times s } $, $ \mathfrak{ c } : \C^{ s \times s } \rightarrow \CV^{ s \times s } $ be the linear mapping defined by $ \mathfrak{ c } ( \SS ) := [ \SS, Y ] $ and denote $ \ker \mathfrak{ c } = \CC ( Y ) $, $ \mbox{ran } \mathfrak{ c } = \CC_1$. Since $ \SC_1 [ \SS, Y ] \SC_2 = [ \SC_1 \SS \SC_2, Y ] $ for any $ \SS \in \C^{ s \times s } $ and $ \SC_1, \SC_2 \in \CC ( Y ) $, $ \CC_1 $ is a sub-bi-module in $ \CV^{ s \times s } $ over $ \CC ( Y ) $. For $ k \in \N $, the subspaces $ \CC_k \subset ( \CV^{ s \times s } )^{ \otimes k } $ are defined as
\[ \CC_k = \sum_{ i = 0 }^{ k - 1 } ( \CV^{ s \times s } )^{ \otimes i } \otimes { \CC }_1 \otimes
( \CV^{ s \times s } )^{ \otimes k - 1 - i } \]
\begin{comment}
Let $ \CC_k' : = ( \CC_1' )^{ \otimes k } $ and observe that $ \CC_k' $ is complementary to the subspace $ \CC_k $ in $( \CV^{ s \times s } )^{ \otimes k } $ for each $ k \in \N $. Since $ \CC_1 $ is a sub-bi-module in $ \CV^{ s \times s } $ over $ \CC ( Y ) $, it is evident that both $ \CC_k $ and $ \CC'_k $ are $ \CC ( Y ) $ - bi-modules.
\end{comment}

%Next, we introduce three identities to simplify the following lemma notationally. 
For $ Y, Z^1, \hdots, $ $ Z^{\ell} \in \CV^{ s \times s } $, $ \SS \in \C^{ s \times s } $, and a family of multi-linear mappings $ \{ \CG_{\ell} : ( \CV^{ s \times s } )^{\ell} \rightarrow \CV_1^{ s \times s } : $ $ \ell \geq 0 \} $ for some operator space $ \CV_1 $, we denote
\begin{multline*}
\mathrm{CIC}_{ \ell, 0 } ( \CG, Y ) ( Z^1, \hdots, Z^{\ell} )  =  \CG_{ \ell + 1 } ( [ \SS, Y ],  Z^1, \hdots, Z^{\ell} ) - \SS \CG_{\ell} ( Z^1, \hdots, Z^{\ell} ) + \CG_{\ell} ( \SS Z^1, \hdots, Z^{\ell} ) \\
\mathrm{CIC}_{ \ell, j } ( \CG, Y ) ( Z^1, \hdots, Z^{\ell} )  =  \CG_{ \ell + 1 } ( Z^1, \hdots, Z^j, [ \SS, Y ], Z^{ j + 1 }, \hdots, Z^{\ell} )
-  \CG_{\ell} ( Z^1, \hdots, Z^{ j - 1 }, Z^j \SS, \\
Z^{ j + 1 },  \hdots, Z^{\ell} ) +  \CG_{\ell} ( Z^1, \hdots, Z^j, \SS Z^{ j + 1 }, Z^{ j + 2 }, \hdots, Z^{\ell} ) \\
 \mathrm{CIC}_{ \ell, \ell } ( \CG, Y ) ( Z^1, \hdots, Z^{\ell} )   =  \CG_{ \ell + 1 } ( Z^1, \hdots, Z^{\ell}, [ \SS, Y ] ) - \CG_{\ell} ( Z^1, \hdots, Z^{\ell} \SS ) + \CG_{\ell} ( Z^1, \hdots, Z^{\ell} ) \SS
\end{multline*}

\begin{lem} \label{algebraic lemma}

Let $ \CV, \CX $, $ \CY $ and $ \CZ $ be operator spaces and $ \Omega \subset \CV_{ nc } $ be a uniformly open nc set. Let $ \mathfrak{ m } : \CX \times \CY \rightarrow \CZ $ be a bilinear mapping: $ ( f, g ) \mapsto f \star g $. Fix $ s \in \N $, $ Y \in \Omega_s $ and $ H \in \CY^{ s \times s } $. Consider uniformly analytic nc functions $ \CT : \Omega \rightarrow \CX_{ nc } $ and $ \CF : \Omega \rightarrow \CZ_{ nc } $ satisfying $ \CT ( Y ) \star H  = \CF ( Y ) $. Assume that  $ \CG_0 = H $ and $ \CG_{ \ell } : ( \CV^{ s \times  s } )^{ \ell } \rightarrow \CY^{ s \times s } $ are multi-linear mappings for $ \ell \geq 1 $. 
\begin{enumerate}
\item[(i)] If the family of multi-linear mappings $ \CG_{\ell} $, $ \ell \geq 0 $, satisfy the canonical intertwining conditions, then for each $ k \geq 0 $, they satisfy the equation
\be \label{ cauchy identity } 
\CF_k  = \sum_{ j = 0 }^k \CT_{ k - j } \star \CG_j 
\ee
on $ \CC_k $ where $ \{ \CT_k \}_{ k =0 }^{ \infty } $ and $ \{ \CF_k \}_{ k = 0 }^{ \infty } $ are Taylor-Taylor coefficients of $ \CT $ and $ \CF $, respectively,  at $ Y $.
\item[(ii)] If the identity given in part (i) holds on $ \CC_k $ for $ k \geq 0 $, then 
$$ [ \SS, \CG_0 ] - \CG_1 ( [ \SS, Y ] ), ~ \mathrm{CIC}_{ \ell, 0 } ( \CG, Y ) ( Z^1, \hdots, Z^{\ell} ) \in \ker \big( \mathfrak{ m } ( \CT_0, \cdot ) \big) , $$
for $ \ell \geq 1 $,  $ Z^1, \hdots, Z^{\ell} \in \CV^{ s \times s } $ and $ \SS \in \C^{ s \times s } $. Here, $ \mathfrak{ m } ( \CT_0, \cdot ) $ is the linear mapping $ \mathfrak{ m } ( \CT_0, \cdot ) : \CY \rightarrow \CZ $. Furthermore, for $ \ell \geq 1 $, $ Z^1, \hdots, Z^{\ell} \in \CV^{ s \times s } $ and $ \SS \in \C^{ s \times s } $,
\begin{align}
\label{ CICj } & \CT_0 \star \mathrm{CIC}_{ \ell, j }  ( \CG, Y ) ( Z^1,\hdots, Z^{\ell} )  + \CT_j ( Z^1, \hdots, Z^j ) \star \mathrm{CIC}_{ \ell - j, 0 }  ( \CG, Y ) ( Z^{ j + 1 }, \hdots, Z^{\ell} )  \\
\nonumber & +  \sum_{ i = \ell - j + 1 }^{ \ell - 1 }  \CT_{ \ell - i } ( Z^1, \hdots, Z^{ \ell - i } ) \star \mathrm{CIC}_{ i, j }  ( \CG, Y ) ( Z^{ \ell - i + 1 }, \hdots, Z^{\ell} )  = 0, ~ \text{and} \\
\label{ CICl } &  \CT_0 \star  \mathrm{CIC}_{ \ell, \ell }  ( \CG, Y ) ( Z^1, \hdots, Z^{\ell} ) + \CT_{\ell} ( Z^1, \hdots, Z^{\ell} ) \star \CG_1 ( [ \SS, Y ] ) - [ \SS, \CG_0 ] ) \\
\nonumber & + \sum_{ i = 1 }^{\ell}  \CT_{ \ell - i } ( Z^1, \hdots, Z^{ \ell - i } ) \star  \mathrm{CIC}_{ i + 1, i + 1 }  ( \CG, Y ) ( Z^{ \ell - i + 1 }, \hdots, Z^{\ell} ) = 0.
\end{align}
\end{enumerate}
\end{lem}

\begin{proof}

%(i) Let $ \CC_1 : = \{ S Y - Y S : S \in \C^{ s \times s } \} \subset \CV^{ s \times s } $ and 
%$$ \CC_k = \sum_{ i = 0 }^{ k - 1 } ( \CV^{ s \times s } )^{ i } \times \CC_1 \times
%( \CV^{ s \times s } )^{ k - 1 - i } \subset ( \CV^{ s \times s } )^{ k }. $$ 
(i) Since each of $ \CT_{\ell} $, $ \CG_{\ell} $ and $ \CF_{\ell} $ is multi-linear mapping on $ ( \CV^{ s \times s } )^{\ell} $, it is enough to prove that the desired identity holds on $ ( \CV^{ s \times s } )^{ i } \times \CC_1 \times ( \CV^{ s \times s } )^{  k - 1 - i } $ individually, for every $ i = 0, \hdots, k - 1 $. We prove this with the help of mathematical induction on $ k $. Observe from the hypothesis that the identity \eqref{ cauchy identity } holds for $ k = 0 $. So we assume that the identity \eqref{ cauchy identity } holds true for all $ 0 \leq j \leq k - 1 $. 

For $ i = 0 $ and $ ( [ \SS, Y ], Z^1, \hdots, Z^{ k - 1 } ) \in \CC_1 \times ( \CV^{ s \times s } )^{ k - 1 } $, using the canonical intertwining conditions for each of $ \CF_k $, $ \CT_{ k - j } $ and $ \CG_j $ for $ j = 0, \hdots, k $, note that
\begin{align*}
& \sum_{ j = 0 }^k \CT_{ k - j } ( [ \SS, Y ], Z^1, \hdots, Z^{ k - j - 1 } ) \star \CG_j ( Z^{ k - j }, \hdots, Z^{ k - 1 } )   \\
& = \sum_{ j = 0 }^{ k - 1 } \CT_{ k - j } ( [ \SS, Y ], Z^1, \hdots, Z^{ k - j - 1 } ) \star \CG_j ( Z^{ k - j }, \hdots, Z^{ k - 1 } ) +  \CT_0 \star \CG_k ( [ \SS, Y ], Z^1, \hdots, Z^{ k - 1 } ) \\
& = \sum_{ j = 0 }^{ k - 1 }  \left[ \SS \CT_{ k - j - 1 } ( Z^1, \hdots, Z^{ k - j - 1 } ) - \CT_{ k - j - 1 } ( \SS Z^1, \hdots, Z^{ k - j - 1 } ) \right] \star \CG_j ( Z^{ k - j }, \hdots, Z^{ k - 1 } ) \\
& \hspace{0.2in} +  \CT_0 \star \left[ \SS \CG_{ k - 1 } ( Z^1, \hdots, Z^{ k - 1 } ) - \CG_{ k - 1 } ( \SS Z^1, \hdots, Z^{ k - 1 } ) \right] \\
& = \sum_{ j = 0 }^{ k - 2 } \left[ \SS \CT_{ k - j - 1 } ( Z^1, \hdots, Z^{ k - j - 1 } ) - \CT_{ k - j - 1 } ( \SS Z^1, \hdots, Z^{ k - j - 1 } ) \right] \star \CG_j ( Z^{ k - j }, \hdots, Z^{ k - 1 } ) \\
& \hspace{0.2in}  + ( \SS \CT_0 - \CT_0 \SS ) \star \CG_{ k - 1 } ( Z^1, \hdots, Z^{ k - 1 } ) + \CT_0 \star \SS \CG_{ k - 1 } ( Z^1, \hdots, Z^{ k - 1 } ) \\
& \hspace{0.2in}  - \CT_0 \star \CG_{ k - 1 } ( \SS Z^1, \hdots, Z^{ k - 1 } ) \\
& = \sum_{ j = 0 }^{ k - 1 } \left[ \SS \CT_{ k - 1 - j } ( Z^1, \hdots, Z^{ k - 1 - j } ) -  \CT_{ k - 1 - j } ( \SS Z^1, \hdots, Z^{ k - j - 1 } )\right] \star \CG_j ( Z^{ k - j }, \hdots, Z^{ k - 1 } ) \\
& = \SS \CF_{ k - 1 } ( Z^1, \hdots, Z^{ k - 1 } ) - \CF_{ k - 1 } ( \SS Z^1, \hdots, Z^{ k - 1 } ) \\
& = \CF_k ( [ \SS, Y ], Z^1, \hdots, Z^{ k - 1} ) .
\end{align*}

A similar computation also yields that the desired identity holds on the summand $ ( \CV^{ s \times s } )^{ k - 1 } \times \CC_1 $ for $ i = k - 1 $. Finally, for $ 1 \leq i \leq k - 2 $, and $ ( Z^1, \hdots, Z^i, [ \SS, Y ], Z^{ i + 1 }, \hdots, $ $ Z^{ k - 1 } ) \in ( \CV^{ s \times s } )^{ i } \times \CC_1 \times ( \CV^{ s \times s } )^{ k - 1 - i } $, using the canonical intertwining conditions for each of $ \CF_k $, $ \CT_{ k - j } $ and $ \CG_j $ for $ j = 0, \hdots, k $, we have that
\begin{align*}
& \sum_{ j = 0 }^k \CT_{ k - j } \star \CG_j ( Z^1, \hdots, Z^i, [ \SS, Y ], Z^{ i + 1 }, \hdots, Z^{ k - 1 } ) \\
& = \sum_{ j = 0 }^{ k - i - 2 } \CT_{ k - j } \star \CG_j ( Z^1, \hdots, Z^i, [ \SS, Y ], Z^{ i + 1 }, \hdots, Z^{ k - 1 } ) \\
& \hspace{0.2in}   +  \CT_i ( Z^1, \hdots, Z^i ) \star \CG_{ k - i } ( [ \SS, Y ], Z^{ i + 1 }, \hdots, Z^{ k - 1 } ) \\
& \hspace{0.2in}   + \CT_{ i + 1 } ( Z^1, \hdots, Z^i, [ \SS, Y ] ) \star \CG_{ k - i - 1 } ( Z^{ i + 1 }, \hdots, Z^{ k - 1 } ) \\
& = \sum_{ j = 0 }^{ k - i - 2 } \CT_{ k - j } ( Z^1, \hdots, Z^i, [ S, Y ], Z^{ i + 1 }, \hdots, Z^{ k - j - 1 } ) \star \CG_j ( Z^{ k - j }, \hdots, Z^{ k - 1 } ) \\
& \hspace{0.2in}   + \CT_{ i + 1 } ( Z^1, \hdots, Z^i S ) \star \CG_{ k - i - 1 } ( Z^{ i + 1 }, \hdots, Z^{ k - 1 } ) \! - \! \CT_i ( Z^1, \hdots, Z^i ) \star \CG_{ k - i } ( S Z^{ i + 1 }, \hdots, Z^{ k - 1 } ) \\
& \hspace{0.2in}   + \sum_{ j = k - i + 1 }^k \CT_{ k - j } ( Z^1, \hdots, Z^{ k - j } ) \star \CG_j ( Z^{ k - j + 1 }, \hdots, Z^i, [ S, Y ], Z^{ i + 1 }, \hdots, Z^{ k - 1 } ) \\
& = \sum_{ j = 0 }^{ k - i - 2 } \CT_{ k - j - 1 } ( Z^1, \hdots, Z^i S, Z^{ i + 1 }, \hdots, Z^{ k - j - 1 } ) \star \CG_j ( Z^{ k - j }, \hdots, Z^{ k - 1 } ) \\
& \hspace{0.2in}   - \sum_{ j = 0 }^{ k - i - 2 } \CT_{ k - j - 1 } ( Z^1, \hdots, Z^i, S Z^{ i + 1 }, \hdots, Z^{ k - j - 1 } ) \star \CG_j ( Z^{ k - j }, \hdots, Z^{ k - 1 } ) \\
& \hspace{0.2in}   + \CT_{ i + 1 } ( Z^1, \hdots, Z^i S ) \star \CG_{ k - i - 1 } ( Z^{ i + 1 }, \hdots, Z^{ k - 1 } ) \! - \! \CT_i ( Z^1, \hdots, Z^i ) \star \CG_{ k - i } ( S Z^{ i + 1 }, \hdots, Z^{ k - 1 } ) \\
& \hspace{0.2in}   + \sum_{ j = k - i + 1 }^k \CT_{ k - j } ( Z^1, \hdots, Z^{ k - j } ) \star \CG_{ j - 1 } ( Z^{ k - j + 1 }, \hdots, Z^i S, Z^{ i + 1 }, \hdots, Z^{ k - 1 } ) \\
& \hspace{0.2in}   - \sum_{ j = k - i + 1 }^k \CT_{ k - j } ( Z^1, \hdots, Z^{ k - j } ) \star \CG_{ j - 1 } ( Z^{ k - j + 1 }, \hdots, Z^i, S Z^{ i + 1 }, \hdots, Z^{ k - 1 } ) \\
%& = \sum_{ j = 0 }^{ k - 1 } \CT_{ k - j - 1 } \CG_j ( Z^1, \hdots, Z^i S, Z^{ i + 1 }, \hdots, Z^{ k - 1 } ) - \sum_{ j = 0 }^{ k - 1 } \CT_{ k - j - 1 } \CG_j ( Z^1, \hdots, Z^i, S Z^{ i + 1 }, \hdots, Z^{ k - 1 } ) \\
& = \CF_{ k - 1 } ( Z^1, \hdots, Z^i S, Z^{ i + 1 },  \hdots, Z^{ k - 1 } ) - \CF_{ k - 1 } ( Z^1, \hdots, Z^i, S Z^{ i + 1 }, \hdots, Z^{ k - 1 } ) \\
& = \CF_k ( Z^1, \hdots, Z^i, [ S, Y ], Z^{ i + 1 }, \hdots, Z^{ k - 1 } ).
\end{align*}

(ii) We begin by pointing out that the given identity turns out to be $ \CF_1 = \CT_1 \CG_0 + \CT_0 \CG_1 $, and consequently, 
$$ \CT_0 \star \CG_1 ( [ \SS, Y ] )  = \CF_1 ( [ \SS, Y ] ) - \CT_1 ( [ \SS, Y ] ) \star \CG_0  = [ \SS, \CF_0 ] - [ \SS, \CT_0 ] \star \CG_0 =  \CT_0 \star [ \SS, \CG_0 ] $$ 
where the last equality holds since $ \CT_0 \star \CG_0  = \CF_0 $. Now for $ \ell \geq 1 $, $ Z^1, \hdots, Z^{\ell} \in \CV^{ s \times s } $ and $ \SS \in \C^{ s \times s } $, note that
\begin{align*}
& \CT_0 \star  \CG_{ \ell + 1 } ( [ \SS, Y ], Z^1, \hdots, Z^{\ell} )  \\
& = \CF_{ \ell + 1 } ( [ \SS, Y ], Z^1, \hdots, Z^{\ell} ) - \sum_{ i = 0 }^{\ell} \CT_{ \ell + 1 - i } ( [ \SS, Y ], Z^1, \hdots, Z^{ \ell - i } ) \star \CG_i ( Z^{ \ell - i + 1 }, \hdots, Z^{\ell} ) \\
& = \SS \CF_{\ell} ( Z^1, \hdots, Z^{\ell} ) - \CF_{\ell} ( \SS Z^1, \hdots, Z^{\ell} ) \\
& \hspace{0.2in} - \sum_{ i = 0 }^{\ell}  \CT_{ \ell + 1 - i } ( [ \SS, Y ], Z^1, \hdots, Z^{ \ell - i } ) \star \CG_i ( Z^{ \ell - i + 1 }, \hdots, Z^{\ell} ) \\
& = \SS \CF_{\ell} ( Z^1, \hdots, Z^{\ell} ) - \CF_{\ell} ( \SS Z^1, \hdots, Z^{\ell} ) - \sum_{ i = 0 }^{\ell} [ \SS \CT_{ \ell - i } ( Z^1, \hdots, Z^{ \ell - i } ) \\
& \hspace{0.2in}   - \CT_{ \ell - i } ( \SS Z^1, \hdots, Z^{ \ell - i } ) ] \star \CG_i ( Z^{ \ell - i + 1 }, \hdots, Z^{\ell} ) \\
& = \SS \bigg( \CF_{\ell} ( Z^1, \hdots, Z^{\ell} ) - \sum_{ i = 0 }^{\ell} \CT_{ \ell - i } ( Z^1, \hdots, Z^{ \ell - i } ) \star \CG_i ( Z^{ \ell - i + 1 }, \hdots, Z^{\ell} ) \bigg) \\
& \hspace{0.2in}   - \CF_{\ell} ( \SS Z^1, \hdots, Z^{\ell} ) + \sum_{ i = 0 }^{ \ell - 1 } \CT_{ \ell - i } ( \SS Z^1, \hdots, Z^{ \ell - i } ) \star \CG_i ( Z^{ \ell - i + 1 }, \hdots, Z^{\ell} ) \! + \! \CT_0 \star \SS \CG_{\ell} ( Z^1, \hdots, Z^{\ell} )  \\
& = \CT_0 \star \SS \CG_{\ell} ( Z^1, \hdots, Z^{\ell} ) - \CG_{\ell} ( \SS Z^1, \hdots, Z^{\ell} )
\end{align*}
verifying that $\mbox{CIC}_{ \ell, 0 } ( \CG ) $ is in the kernel of $ \mathfrak{ m } ( \CT_0, . ) $.

Next, we show that the identity \eqref{ CICj } holds for any $ \ell \geq 1 $, $ Z^1, \hdots, Z^\ell \in \CV^{ s \times s } $ and $ \SS \in \C^{ s \times s } $. For $ \ell \geq 1 $, $ Z^1, \hdots, Z^{\ell} \in \CV^{ s \times s } $ and $ \SS \in \C^{ s \times s } $, we use the notation $$ ( Z^1, \hdots, [ \SS, Y ], \hdots, Z^{ \ell } ) : = ( Z^1, \hdots, Z^j, [\SS, Y ], Z^{ j + 1 }, \hdots, Z^{\ell} ) , $$
and compute
\begin{multline*}
\CT_0 \star  \CG_{ \ell + 1 } ( Z^1, \hdots, [ \SS, Y ], \hdots, Z^{ \ell } ) = \bigg( \CF_{ \ell + 1 } - \sum_{ i = 0 }^{ \ell } \CT_{ \ell + 1 - i } \star \CG_i \bigg)(  Z^1, \hdots, [ \SS, Y ], \hdots, Z^{ \ell } )\\
= \CF_{ \ell + 1 } ( Z^1, \hdots, [ \SS, Y ], \hdots, Z^{ \ell } ) - \sum_{ i = 0 }^{\ell} \CT_{ \ell + 1 - i } \star \CG_i ( Z^1, \hdots, [ \SS, Y ], \hdots, Z^{ \ell } )  \\
= \sum_{ i = \ell - j + 1 }^{\ell} \!\!\! \CT_{ \ell - i } \star \CG_i ( Z^1, \hdots, Z^j \SS, Z^{ j + 1 }, \hdots, Z^{\ell} ) - \!\!\! \sum_{ i = \ell - j }^{\ell}  \CT_{ \ell - i } \star \CG_i ( Z^1, \hdots, Z^j, \SS Z^{ j + 1 }, \hdots, Z^{\ell} ) \\
+ \CT_j ( Z^1, \hdots, Z^j ) \star \SS \CG_{ \ell - j } ( Z^{ j + 1 }, \hdots, Z^{\ell} ) - \!\!\! \sum_{ i = \ell - j + 1 }^{\ell} \!\!\! \CT_{ \ell + 1 - i } \star \CG_i  ( Z^1, \hdots, [ \SS, Y ], \hdots, Z^{ \ell } ) \\
= -  \CT_j ( Z^1, \hdots, Z^j ) \star \text{CIC}_{ \ell - j, 0 } ( Z^{ j + 1 }, \hdots, Z^{\ell} ) + \!\!\! \sum_{ i = \ell - j + 1 }^{\ell}  \CT_{ \ell - i } \star \CG_i ( Z^1, \hdots, Z^j \SS, Z^{ j + 1 }, \hdots, Z^{\ell} ) \\
- \sum_{ i = \ell - j + 1 }^{\ell} \CT_{ \ell - i } \star \CG_i ( Z^1, \hdots, Z^j, \SS Z^{ j + 1 }, \hdots, Z^{\ell} ) - \sum_{ i = \ell - j + 2 }^{\ell} \CT_{ \ell + 1 - i } \star \CG_i ( Z^1, \hdots, [ \SS, Y ], \hdots, Z^{ \ell } ) \\
 = - \CT_j ( Z^1, \hdots, Z^j ) \star \text{CIC}_{ \ell - j, 0 } ( Z^{ j + 1 }, \hdots, Z^{\ell} )  +  \sum_{ i = \ell - j + 1 }^{\ell} \!\!\! \CT_{ \ell - i } \star \CG_i  ( Z^1, \hdots, Z^j \SS, Z^{ j + 1 }, \hdots, Z^{\ell} ) \\
- \!\!\! \sum_{ i = \ell - j + 1 }^{\ell} \CT_{ \ell - i } \star \CG_i ( Z^1, \hdots, Z^j, \SS Z^{ j + 1 }, \hdots, Z^{\ell} ) - \sum_{ i = \ell - j + 1 }^{ \ell - 1 } \CT_{ \ell - i } \star \CG_{ i + 1 } ( Z^1, \hdots, [ \SS, Y ], \hdots, Z^{ \ell } ) \\
=  \CT_0 \star  \left[ \CG_{\ell} ( Z^1, \hdots, Z^j \SS, Z^{ j + 1 }, \hdots, Z^{\ell} ) - \CG_{\ell} ( Z^1, \hdots, Z^j, \SS Z^{ j + 1 }, \hdots, Z^{\ell} ) \right] 
- \CT_j ( Z^1, \hdots,\\ Z^j ) \star \text{CIC}_{ \ell - j, 0 } ( Z^{ j + 1 }, \hdots, Z^{\ell} )  -  \sum_{ i = \ell - j + 1 }^{ \ell - j } \CT_{ \ell - i } ( Z^1, \hdots, Z^{ \ell - i } ) \star \text{CIC}_{ i, j } ( Z^{ \ell - i + 1 }, \hdots, Z^{\ell} ),
\end{multline*}
verifying Equation \eqref{ CICj }. Finally, an analogous computation as above for  $ \ell \geq 1$, $ Z^1, \hdots, $ $ Z^{\ell} \in \CV^{ s \times s } $ and $ \SS \in \C^{ s \times s } $ ensures the validation of Equation \eqref{ CICl }.
\end{proof}

%We are now at the position in presenting the main result in this section. Recall that two matrices $ X \in \CV^{ s \times s } $ and $ Y \in \CV^{ t \times t } $ are \textit{similar} if $ s = t $ and there exists an invertible matrix $ \SQ \in \C^{ s \times } $ such that $ Y = \SQ X \SQ^{ - 1 } $. A matrix $ Y \in \CV^{ s \times s } $ is said to be \textit{irreducible} if $ Y $ does not have nontrivial invariant subspace viewing $ Y $ as an operator in $ \CB ( \C^s, \CV^s ) $. More generally, we say that $ Y \in \CV^{ s \times s } $ is  \textit{semi-simple} if it is similar to a direct sum of irreducible matrices over $ \CV $. 

We now present the main theorem of this section for which we need the following technical definition. Also, note that the notations and the terminologies used in the following theorem have been introduced in the beginning of this subsection.

\begin{defn} \label{ completely complemented}
A closed subspace $ \CS $ of an operator space $ \CV $ is said to be \textit{completely complemented} if there exists a closed subspace $ \CS' \subset \CV $ and a completely bounded projection $ Q : \CV \rightarrow \CV $ onto $ \CS' $ with $ \ker Q = \CS $.

%For $ s \in \N $ and a subring $ R \subset \C^{ s \times s } $, a closed sub-bi-module $ \CS \subset \CV^{ s \times s } $ over $ R $ is \textit{completely complemented} if there exists a closed sub-bi-module $ \CS' \subset \CV^{ s \times s } $ over $ R $ and a completely bounded projection $ Q : \CV^{ s \times s } \rightarrow \CV^{ s \times s } $ with $ Q ( \CV^{ s \times s } ) = \CS' $ and $ \ker Q = \CS $ which is a bi-module homomorphism over $ R $ as well.
\end{defn} 

\begin{rem}
Note that if $ \CV $ is a Hilbert space equipped with the row operator space structure then every closed subspace of it is completely complemented as shown in Lemma \ref{nearest point approximation}. A similar proof in the same lemma can be used to show that so is true for any column operator Hilbert spaces as well.
\end{rem}

%\begin{thm} \label{closed range operators}
%
%Let $ \CV $ be an operator space, and Hilbert spaces $ \CH $ and $ \CK $ be equipped with the column operator space structure -- denoted by $ \CH_c $ and $ \CK_c $, respectively. Let $ \Omega \subset \CV_{ nc } $ be a bounded, uniformly open nc set which is right admissible and fix $ s \in \N $, a semi-simple point $ Y \in \Omega_s $ and $ H \in \CH_c^{ s \times s } $. Consider uniformly analytic nc functions 
%$$ \CT : \Omega \rightarrow \CB ( \CH, \CK )_{ nc } ~\mbox{and}~ \CF : \Omega \rightarrow ( \CK_c )_{ nc } $$ 
%satisfying $ \CF ( X ) \in \mbox{ran} ( \CT ( X ) ) $ for $ X \in \Omega $ and $ \CT ( Y ) H = \CF ( Y ) $ where $ \CB ( \CH, \CK ) $ is equipped with the natural operator space structure on it. Assume that $ \CT ( Y ) : \CH^{ \oplus s } \rightarrow \CK^{ \oplus s } $ satisfies the following condition: for any $ m \in \N $, whenever $ X_n \in \Omega_{ ms } $ with $ X_n \rightarrow Y^{ \oplus m } $, $ K_n \in \mbox{ran } \CT ( X_n ) $ and $ K_n \rightarrow K $, then $ K \in \mbox{ran } \CT ( Y^{ \oplus m } ) $ (in particular, $ \CT ( Y )^{ \oplus m } $ has closed range for each $ m \in \N $). Then there exists a uniformly open nc neighbourhood $ U $ of $ Y $ and a uniformly analytic function $ \CG : U \rightarrow ( \CH_c )_{ nc } $ such that 
%$$ \CG ( Y ) = H ~\mbox{and}~ \CT ( X ) \CG ( X ) = \CF ( X ) ~\mbox{for all}~ X \in U. $$
%\end{thm}

\begin{thm} \label{closed range operators general version}

Let $ \CV, \CX, \CY $ and $ \CZ $ be operator spaces and $ \mathfrak{ m } : \CX \times \CY \rightarrow \CZ $ be a completely contractive bilinear mapping: $ ( f, g ) \mapsto f \star g $. Let $ \Omega \subset \CV_{ nc } $ be a bounded, uniformly open nc set which is right admissible and fix $ s \in \N $ and $ Y \in \Omega_s $ so that $ \CC_1 $ is a complemented sub-bi-module of $ \CV^{ s \times s } $ over $ \CC ( Y ) $. Consider uniformly analytic nc functions 
$$ \CT : \Omega \rightarrow \CX_{ nc } \quad \mbox{and} \quad \CF : \Omega \rightarrow \CZ_{ nc } $$ 
satisfying $ \CF ( X ) \in \mathrm{ran} ( \mathfrak{ m } ( \CT ( X ), . ) ) $ for $ X \in \Omega $ and $ \mathfrak{ m} ( \CT ( Y ), H ) = \CF ( Y ) $ for some  $ H \in \CY^{ s \times s } $. Assume that $ \mathfrak{ m } ( \CT ( Y ), . ) : \CY^{ s \times s } \rightarrow \CZ^{ s \times s } $ satisfies the following conditions: 
\begin{itemize}
\item[(a)] $ \ker \mathfrak{ m } ( \CT ( Y ), . ) $ is a completely complemented subspace of $ \CY^{ s \times s } $ with the associated projection $ Q : \CY^{ s \times s } \rightarrow \CY^{ s \times s } $ to be a $ \CC ( Y ) $ bi-module homomorphism, and 
\item[(b)] for any $ m \in \N $, if $ \{ X_n \} \subset \Omega_{ ms } $, $ X_n $ converges to $ Y^{ \oplus m } $, $ K_n \in \text{ran } ( \mathfrak{ m } ( \CT ( X_n ), . ) ) $ and $ K_n $ converges to $ K $, then $ K \in \text{ran } ( \mathfrak{ m } ( \CT ( Y^{ \oplus m } ), . ) ) $ (in particular, $ \mathfrak{ m } ( \CT ( Y )^{ \oplus m }, . ) $ has closed range for each $ m \in \N $). 
\end{itemize}
Then there exist a uniformly open neighbourhood $ U $ of $ Y $ and a uniformly analytic nc function $ \CG : U \rightarrow \CY_{ nc } $ such that 
$$ \CG ( Y ) = H \quad \text{and} \quad  \CT ( X ) \star \CG ( X )  = \CF ( X ) ~\quad \text{for all} ~~ X \in U. $$
\end{thm}

\begin{proof}

Let $ \CC_1' $ be a closed sub-bi-module of $ \CV^{ s \times s } $ over $ \CC ( Y ) $ complementary to $ \CC_1 $ in $ \CV^{ s \times s } $, and $ \varrho_1 : \CV^{ s \times s } \rightarrow \CC_1 $, $ \varrho'_1 : \CV^{ s \times s } \rightarrow \CC_1' $ be the projections so that $ I_{ \CV^{ s \times s } } = \varrho_1 + \varrho'_1 $ where $ I_{ \CV^{ s \times s } } $ is the identity operator on $ \CV^{ s \times s } $. Since the dimension of $ \CC_1 $ is finite, $ \varrho_1 $ is completely bounded and hence so is $ \varrho_1' $. For $ k \in \N $, denote $ \CC_k' : = ( \CC_1' )^{ \otimes k } $ and observe that $ \CC_k $ is complementary to the subspace
$$ \CC_k = \sum_{ i = 0 }^{ k - 1 } ( \CV^{ s \times s } )^{ \otimes i } \otimes { \CC }_1 \otimes
( \CV^{ s \times s } )^{ \otimes k - 1 - i } $$ 
in $ ( \CV^{ s \times s } )^{ \otimes k } $. Evidently, $ \CC_k $ and $ \CC'_k $ are also $ \CC ( Y ) $ - bi-modules as well as $ I_{ (\CV^{ s \times s } )^{ \otimes k } } = \varrho_k + \varrho_k' $ where $ \varrho_k $ and $ \varrho_k' $ are the projections corresponding to the complementary subspaces $\CC_k $ and $ \CC_k' $, respectively. Then $ \| \varrho_k' \|_{ \text{cb} } \leq \| \varrho_1' \|_{ \text{cb} }^k $, and consequently, $ \| \varrho_k \|_{ \text{cb} } \leq 1 + \| \varrho_1' \|_{ \text{cb} }^k $. Also, note that the sub-bi-modules $ \CC_1 $ and $ \CC_1' $ give rise to the completely bounded right inverse $ \varphi : \CC_1 \rightarrow \C^{ s \times s } $ to the linear mapping $ \mathfrak{ c } : \C^{ s \times s } \rightarrow \CV^{ s \times s } $ defined by $ \SS \mapsto [ \SS, Y ] $.
 
Recall that both $ \CT $ and $ \CF $, being uniformly analytic nc functions, are locally uniformly bounded. Let $ r > 0 $ be such that both $ \CT $ and $ \CF $ are uniformly bounded on $ \B_{ nc } ( Y, 2 \norm{ \varrho_1' }_{ \text{cb} }  ( 1 + 2 \norm{ \varrho_1' }_{ \text{cb} } ) r ) $ and for all $ m \in \N $, $ X \in \B_{nc} ( Y, 2 \norm{ \varrho_1' }_{ \text{cb} } ( 1 + 2 \norm{ \varrho_1' }_{ \text{cb} } ) r ) $, 
$$ \CT ( X ) = \sum_{ \ell = 0 }^{ \infty } ( X - Y^{ \oplus m } )^{ \odot_s \ell } \CT_{ \ell } \quad\text{and} \quad \CF ( X ) = \sum_{ \ell = 0 }^{ \infty } ( X - Y^{ \oplus m } )^{ \odot_s \ell } \CF_{ \ell } $$ 
where $ \CT_{ \ell } : ( \CV^{ s \times s } )^{ \ell } \rightarrow \CX^{ s \times s } $ and $ \CF_{ \ell } : ( \CV^{ s \times  s } )^{ \ell } \rightarrow \CZ^{ s \times s } $ are $ \ell $-linear mappings satisfying the canonical intertwining conditions. Denote 
$$ \tilde{ r } = 2 \norm{ \varrho_1' }_{ \text{cb} } ( 1 + 2 \norm{ \varrho_1' }_{ \text{cb} } ) r \quad \text{and} \quad \tilde{ \hat{ r } } = 2 \norm{ \varrho_1' }_{ \text{cb} } ( 1 + 2 \norm{ \varrho_1' }_{ \text{cb} } ) \hat{ r } , $$ 
where $ \hat{ r } $ is defined as $ 0 < \hat{ r } \leq \text{min} \bigg\{ r, \frac{ 1 }{ 2 \norm{ \varrho_1' }_{ \text{cb} } \big( 1 + 2 \norm{ \varrho'_1 }_{ \text{cb} } \big) } \bigg\} $. Note that $ \hat{ r } \leq \tilde{ \hat{ r } } \leq \tilde{ r } ~\text{and}~ \tilde{ \hat{ r } } \leq 1 $ since being a projection, $ \norm{ \varrho'_1 }_{ \text{cb} } \geq 1$. Let 
$$ M = \text{max} \big\{ 1, \norm{ \varphi \circ \varrho_1 }_{ \text{cb} }, \sup\{ \norm{ \CT ( X ) }_{ ms, \CX }, \norm{ \CF ( X ) }_{ ms, \CZ } : X \in \B_{ nc } ( Y, \hat{ r } )_m, m \in \N \} \big\}. $$
Then it follows that 
\be \label{bound of Taylor coefficients}  
\norm{ \CT_{ \ell } }_{ \CL^{ \ell }_{ \text{cb} } } \leq \frac{ M }{ \tilde{ \hat{ r } }^{ \ell } } \leq \frac{ M }{ \hat{ r }^{ \ell } } \quad \text{and} \quad \norm{ \CF_{ \ell } }_{ \CL^{ \ell }_{ \text{cb} } } \leq \frac{ M }{ \tilde{ \hat{ r } }^{ \ell } } \leq \frac{ M }{ \hat{ r }^{ \ell } }, \quad \ell \geq 0. 
\ee

Now finding a uniformly analytic nc function $ \CG $ on some uniformly open nc neighbourhood of $ Y $ taking values in $ \CY_{ nc } $ amounts to find $ \ell $-linear mappings $ \CG_{ \ell } : ( \CV^{ s \times s } )^{ \ell } \rightarrow \CY^{ s \times s } $, for $ \ell \geq 0 $, satisfying the canonical intertwining conditions and a positive real number $ 0 < r' \leq r $ such that 
$$ \norm{ \CG_{ \ell } }_{ \CL^{ \ell }_{ \text{cb} } } \leq \frac{ M' }{ r'^{ \ell } }, \quad \ell \in \N \cup \{ 0 \} $$ 
for some $ M' \geq 0 $ so that the Taylor-Taylor series 
\be \label{required TT series} 
\sum_{ \ell = 0 }^{ \infty } ( X - Y^{ \oplus m } )^{ \odot_s \ell } \CG_{ \ell } 
\ee 
converges uniformly for all $ X \in \B_{ nc } ( Y, r' )_{ ms } $, $ m \in \N $, and the uniformly analytic nc function $ \CG $ associated to the Taylor-Taylor series in \eqref{required TT series} satisfies the identity 
\be \label{given equation} 
\CT ( X ) \star \CG ( X ) = \CF ( X ), ~ X \in \B_{ nc } ( Y, r' ). 
\ee

%consider the mapping $ Q : \CH_c^{ n \times n } \rightarrow \CH_c^{ n \times n } $ defined by \be \label{definition of Q}
%Q ( ( \! ( h_{ i j } ) \! )_{ i, j = 1 }^n ) : = ( \! ( h_{ i j } ) \! )_{ i, j = 1 }^n - ( \! ( \eta^0_{ i j } ) \! )_{ i, j = 1 }^n 
%\ee 
%where $ ( \! ( \eta^0_{ i j } ) \! )_{ i, j = 1 }^n $ is the unique element in $ \ker L_{ \CT ( Y )^{ \oplus n } } $ such that 
%$$ \inf \left\{ \norm{ ( \! ( h_{ i j } ) \! )_{ i, j = 1 }^n - ( \! ( \eta_{ i j } ) \! )_{ i, j = 1 }^n }_{ \CH_c^{ n \times n } } : ( \! ( \eta_{ i j } ) \! )_{ i, j = 1 }^n \in \ker L_{ \CT ( Y )^{ \oplus n } } \right\} = \norm{ ( \! ( h_{ i j } ) \! )_{ i, j = 1 }^n - ( \! ( \eta^0_{ i j } ) \! )_{ i, j = 1 }^n }_{ \CH_c^{ n \times n } } . $$ 

%Recall from Lemma \ref{nearest point approximation} that $ Q $ is linear. Since $ \CT ( Y ) $ has closed range as a linear mapping $ \CT ( Y ) : \CH^{ \oplus s } \rightarrow \CK^{ \oplus s } $, it follows from Lemma \ref{closed range implies bounded below} that

For every $ m \in \N $, it follows from the conditions (a) and (b) in the hypothesis that there exists $ C > 0 $ such that 
\be \label{bounded below inequality}
\norm{ \mathfrak{ m } \left( \CT ( Y )^{ \oplus m },  Q_n ( G ) \right) }_{ ms, \CZ } \geq C \norm{ Q_m ( G ) }_{ ms, \CY } 
\ee 
where $ G = ( \! ( G_{ i j } ) \! )_{ i, j = 1 }^m \in \CY^{ ms \times ms } $ with $ G_{ i j } \in \CY^{ s \times s } $ for $ i, j = 1, \hdots, m $. Set $ \alpha = \text{min} \{ 1, C \} $, 
$$ r' = \frac{ \hat{ r } \alpha }{ 2 M } \quad \text{and} \quad \tilde{ r' } = \frac{ \tilde{ \hat{ r } } \alpha }{ 2 M } . $$ 
Then note that $ r' \leq  \hat{ r } \leq \tilde{ \hat{ r } } $, $ r' \leq \tilde{ r }' $ and $ \B_{ nc } ( Y, r' ) \subset \B_{ nc } ( Y, \hat{ r } ) $.

We now construct the $ \ell $-linear mappings $ \CG_{ \ell } : ( \CV^{ s \times s } )^{ \ell } \rightarrow \CY^{ s \times s } $,  $ \ell \in \N \cup \{ 0 \} $ satisfying the canonical intertwining conditions as well as Equation \eqref{given equation}, or equivalently, for each $ \ell \in \N \cup \{ 0 \} $, 
\be \label{given equation of TT coefficients}
\sum_{ j = 0 }^{ \ell } \CT_{ \ell - j } \star \CG_j = \CF_{ \ell } \ee
and $ \norm{ \CG_{ \ell } }_{ \CL^{ \ell }_{ cb } } \leq r'^{ - { \ell } } \mathrm{max} \big\{ 1, \norm{ H }_{ s, \CY } \big\} $ with the help of mathematical induction on $ \ell $. Note that $ \CT_0 = \CT ( Y ) $. 

Assume that we have defined the desired $ \ell $-linear mappings $ \CG_{ \ell } $ satisfying Equation \eqref{given equation of TT coefficients} and 
$$ \norm{ \CG_{ \ell } }_{ \CL^{ \ell }_{ cb } } \leq \tilde{ r }'^{ - \ell } \mathrm{max} \big\{ 1, \norm{ H }_{ s, \CY } \big\} \leq  r'^{ - \ell } \mathrm{max} \big\{ 1, \norm{ H }_{ s, \CY } \big\} , $$ 
for $ \ell = 0, 1, \hdots, k - 1 $ so that the canonical intertwining conditions hold up to order $ k - 1 $ (see Definition \ref{canonical intertwining conditions of finite order}). We then explicitly construct a $ k $-linear mapping $ \CG_k : ( \CV^{ s \times s } )^k \rightarrow \CY^{ s \times s } $ with the prescribed growth on it's complete norm -- $ \norm{ \CG_k }_{ \CL^k_{ cb } } \leq r'^{ - k } \mathrm{max} \big\{ 1, \norm{ H }_{ s, \CY } \big\} $ -- satisfying \eqref{given equation of TT coefficients} and the canonical intertwining conditions up to order $ k $.

We describe this $ k $-linear mapping in two steps by defining it on $ \CC_k $ and $ \CC'_k $ separately, and then adding them to obtain the desired $ k $-linear mapping on $ \CV^{ s \times s } $.

Define the $ k $-linear mapping $ \CG_k \vert_{ \CC_k } : \CC_k \rightarrow \CY^{ s \times s } $ as follows
\begin{align*}
& \CG_k \vert_{ \CC_k } ( Z^1, \hdots, Z^k ) \\
& = ( \varphi \circ \varrho_1 ) ( Z^1 ) \CG_{ k - 1 } ( \varrho_1' ( Z^2 ), \hdots, \varrho_1' ( Z^k ) ) - \CG_{ k - 1 } ( ( \varphi \circ \varrho_1 ) ( Z^1 ) \varrho_1' ( Z^2 ), \hdots, \varrho_1' ( Z^k ) ) \\
& + \sum_{ j = 2 }^{ k - 1 } [ \CG_{ k - 1 } ( \varrho_1' ( Z^1 ), \hdots, \varrho_1' ( Z^{ j - 1 } ) ( \varphi \circ \varrho_1 ) ( Z^j ), \varrho_1' ( Z^{ j + 1 } ), \varrho_1' ( Z^{ j + 2 } ), \hdots, \varrho_1' ( Z^k ) ) \\
& - \CG_{ k - 1 } ( \varrho_1' ( Z^1 ), \hdots, \varrho_1' ( Z^{ j - 1 } ), ( \varphi \circ \varrho_1 ) ( Z^j ) \varrho_1' ( Z^{ j + 1 } ), \varrho_1' ( Z^{ j + 2 } ), \hdots, \varrho_1' ( Z^k ) ) ] \\
& + \CG_{ k - 1 } ( \varrho_1' ( Z^1 ), \hdots, \varrho_1' ( Z^{ k - 1 } ) ( \varphi \circ \varrho_1 ) ( Z^k ) ) - \CG_{ k - 1 }( \varrho_1' ( Z^1 ), \hdots, \varrho_1' ( Z^{ k - 1 } ) ) ( \varphi \circ \varrho_1 ) ( Z^k ).
\end{align*}
It turns out from this definition of $ \CG_k \vert_{ \CC_k } $ together with the fact that each of $ \varphi, \varrho_1 $ and $\varrho_1' $ is $ \CC ( Y ) $-bi-module homomorphism that the canonical intertwining conditions hold up to order $ k $. Further, from part (i) of Lemma \ref{algebraic lemma} it follows that $ \CG_k \vert_{ \CC_k } $ satisfies Equation  \eqref{given equation of TT coefficients}. So it remains to estimate the complete norm of $ \CG_k \vert_{ \CC_k } $. 

Note from the definition of $ \CG_k \vert_{ \CC_k } $ above that 
\begin{align*}
& \norm{ \CG_k \vert_{ \CC_k } }_{ \CL^k_{ cb } }  \leq  2 \norm{ \varphi \circ \varrho_1 }_{ \text{cb} } \norm{ \CG_{ k - 1 } }_{ \text{cb} } \norm{ \varrho_1' }^{ k - 1 }_{ \text{cb} } \\
& +  \sum_{ j = 2 }^{ k - 1 } 2 \norm{ \varphi \circ \varrho_1 }_{ \text{cb} } \norm{ \CG_{ k - 1 } }_{ \text{cb} } \norm{ \varrho_1' }^{ k - 1 }_{ \text{cb} } +  2 \norm{ \varphi \circ \varrho_1 }_{ \text{cb} } \norm{ \CG_{ k - 1 } }_{ \text{cb} } \norm{ \varrho_1' }^{ k - 1 }_{ \text{cb} } \\
& =  2 k \norm{ \varphi \circ \varrho_1 }_{ \text{cb} } \norm{ \CG_{ k - 1 } }_{ \text{cb} } \norm{ \varrho_1' }^{ k - 1 }_{ \text{cb} } \leq  2^k \norm{ \varrho_1' }^k_{ \text{cb} } \norm{ \varphi \circ \varrho_1 }_{ \text{cb} } \norm{ \CG_{ k - 1 } }_{ \text{cb} }.
\end{align*}

Next, we define the $ k $-linear mapping $ \widetilde{ \CG }_k \vert_{ \CC'_k } : \CC'_k \rightarrow \CY^{ s \times s } $ satisfying the equation 
$$ \CT_0 \star \widetilde{ \CG }_k \vert_{ \CC'_k } ( Z_1, \hdots, Z_k ) = \CF_k ( Z_1, \hdots, Z_k ) - \sum_{ j = 0 }^{ k - 1 } \CT_{ k - j } \star \CG_j ( Z_1, \hdots, Z_k ) , $$
for any $ Z_1, \hdots, Z_k \in \CC'_k $, as follows. Fix a basis $ \CB $ of $ \CC'_1 $ and observe that $ \CB^{ \otimes k } $ is a basis for $ \CC'_k $. For any $ \xi_1, \hdots, \xi_k \in \CB $, we first define $ \widetilde{ \CG }_k \vert_{ \CC'_k } ( \xi_1, \hdots, \xi_k ) $ and then extend it to $ \CC'_k $ by linearity. So let $ \xi_1, \hdots, \xi_k \in \CB $ and consider the sequence $ \{ X_n \}_{ n = 1 }^{ \infty } \subset \Omega_{ ( k + 1 ) s } $ defined by 
$$ X_n : = \begin{bmatrix}
Y & t_n \xi_1 & 0 & \cdots & 0 \\
0      & Y & t_n \xi_2 & \cdots & 0 \\
\vdots & \vdots & \ddots & \ddots & 0 \\
0      &       0     & \cdots   & Y & t_n \xi_k \\
0      &       0     & \cdots   &  0            &   Y
\end{bmatrix} $$ 
where $ \{ t_n \}_{ n = 1 }^{ \infty } $ is a sequence of positive real numbers converging to $ 0 $. Then note that $ \{ X_n \} $ converges to $ Y^{ \oplus ( k + 1 ) } $ and consequently, there exists $ N_0 \in \N $ such that for all $ n \geq N_0 $, $ X_n \in \B_{ nc } ( Y, r' ) $. It turns out from the Taylor-Taylor series of $ \CT $ around $ Y $ for $ 0 \leq j \leq k - 1 $, that 
\begin{align*}
& \begin{bmatrix}
0 & \hdots & 0 & \CT_{ k - j } ( \xi_1, \hdots, \xi_{ k - j } ) \star \CG_j ( \xi_{ k - j + 1 }, \hdots, \xi_k ) \\
0 & \hdots & 0 & 0 \\
\vdots & \hdots & \vdots & \vdots \\
0 & \hdots & 0 & 0
\end{bmatrix} \\
& = \frac{ 1 }{ t_n^{ k - j } } \left[ \CT ( X_n ) - \CT ( Y^{ \oplus ( k + 1 ) } ) - \sum_{ i = 1 }^{ k - j - 1 } ( X_n - Y^{ \oplus  ( k + 1 ) } )^{ \odot_i } \CT_i \right] \star A^j ( \xi_1, \hdots, \xi_k )
\end{align*}
where $ A^j ( \xi_1, \hdots, \xi_k ) = ( \! ( A^j_{ s t } ) \! )_{ s, t = 0 }^k $ is the $ ( k + 1 ) \times ( k + 1 ) $ block matrix with $ s \times s $ blocks $ A^j_{ s t } $ defined to be $ 0 $ whenever $ t - s \neq j $ and for $ s = 1, \hdots, k + 1 - j $, $ A^j_{ s, s + j } = \CG_j ( \xi_s, \hdots, \xi_{ j + s - 1 } ) $. Now for each $ \ell = 0, 1, \hdots, k-1 $, we use the Taylor-Taylor series of $ \CF $ around $ Y $ and Equation \eqref{given equation of TT coefficients} to obtain
\begin{align*}
& K : = \begin{bmatrix}
0 & \hdots & 0 & \CF_k ( \xi_1, \hdots, \xi_k ) - \sum_{ j = 0 }^{ k - 1 } \CT_{ k - j } \star \CG_j ( \xi_1, \hdots, \xi_k ) \\
0 & \hdots & 0 & 0 \\
\vdots & \hdots & \vdots & \vdots \\
0 & \hdots & 0 & 0
\end{bmatrix} \\
& = \frac{ 1 }{ t_n^k } \left[ \CF ( X_n ) - \CT ( X_n ) \star H^{ \oplus ( k + 1 ) } - \sum_{ i = 1 }^{ k - 1 } t_n^i \CT ( X_n ) \star A^i ( \xi_1, \hdots, \xi_k ) \right].
\end{align*}
Denote 
$$ K_n = \frac{ 1 }{ t_n^k } \left[ \CF ( X_n ) - \CT ( X_n ) \star H^{ \oplus ( k + 1 ) } - \sum_{ i = 1 }^{ k - 1 }t_n^i \CT ( X_n ) \star A^i ( \xi_1, \hdots, \xi_k ) \right] $$ and observe that $ K_n \in \mathrm{ ran } ~\mathfrak{ m } ( \CT ( X_n ), \cdot ) $. Moreover, it follows from the equation above that $ K_n $ converges to $ K $. Therefore, $ K \in \mathrm{ ran } ~ \mathfrak{ m } \left( \CT (Y^{ \oplus ( k + 1 ) } ) \right) $ implying that there exists $ \widetilde{ \CG }_k \vert_{ \CC'_k } ( \xi_1, \hdots, \xi_k ) \in \CY^{ s \times s } $ such that 
$$ \CT_0 ( \widetilde{ \CG }_k \vert_{ \CC'_k } ( \xi_1, \hdots, \xi_k ) ) = K . $$

Now define 
$$ \CG_k \vert_{ \CC'_k } ( Z_1, \hdots, Z_k ) : = Q ( \widetilde{ \CG }_k \vert_{ \CC'_k } ( Z_1,  \hdots, Z_k ) ), ~ Z_1, \hdots, Z_k \in \CC'_k $$ 
and observe from the condition (b) in the hypothesis together with the equation above that $ \CG_k \vert_{ \CC'_k } $ satisfies Equation \eqref{given equation of TT coefficients}. Also, observe from part (ii) in Lemma \eqref{algebraic lemma} that $ \SS \CG_k \vert_{ \CC'_k } ( Z^1, \hdots, Z^k ) - \CG_k \vert_{ \CC'_k } ( \SS Z^1, \hdots, Z^k ) $, $ \CG_k \vert_{ \CC'_k } ( Z^1, \hdots, Z^{ k - 1 }, Z^k \SS ) - \CG_k \vert_{ \CC'_k } ( Z^1, \hdots, Z^k ) \SS $ and $ \CG_k \vert_{ \CC'_k } ( Z^1, \hdots, Z^{ j - 1 } $, $ Z^j \SS $, $ Z^{ j + 1 }, \hdots, Z^k ) - \CG_k \vert_{ \CC'_k } ( Z^1, \hdots, Z^j, \SS Z^{ j + 1 }, Z^{ j + 2 }, \hdots, Z^k ) $ are in the $ \ker \mathfrak{ m } ( \CT_0, \cdot ) $ for all $ Z^1, \hdots, Z^k \in \CC'_k $ and $ \SS \in \CC ( Y ) $. Then another application of condition (b) yields that
\begin{eqnarray}
\label{CC1} \SS \CG_k \vert_{ \CC'_k } ( Z^1, \hdots, Z^k ) & = & \CG_k \vert_{ \CC'_k } ( \SS Z^1, \hdots, Z^k ), \\
\label{CC2}  \CG_k \vert_{ \CC'_k } ( Z^1, \hdots, Z^{ k - 1 }, Z^k \SS ) & = & \CG_k \vert_{ \CC'_k } ( Z^1, \hdots, Z^k ) \SS, \\
\label{CC3} \hspace{0.5in} \CG_k \vert_{ \CC'_k } ( Z^1, \hdots, Z^{ j - 1 },  Z^j \SS,  Z^{ j + 1 }, \hdots, Z^k ) & = &  \CG_k \vert_{ \CC'_k } ( Z^1, \hdots, Z^j, \SS Z^{ j + 1 }, Z^{ j + 2 }, \hdots, Z^k ),
\end{eqnarray}
for all $ Z^1, \hdots, Z^k \in \CC'_k $ and $ \SS \in \CC ( Y ) $.

To estimate the complete norm $ \norm{ \CG_k \vert_{ \CC'_k} }_{ \CL_{ \text{cb} }^k } $, we first estimate the norm $ \norm{ \CT_0 \star \CG_k \vert_{ \CC'_k} }_{ \CL_{ \text{cb} }^k } $ as follows.
\Bea
\norm{ \CT_0 \star \CG_k \vert_{ \CC'_k} }_{ \CL_{ \text{cb} }^k } & = & \sup_{ n_0, n_1, \hdots, n_k } \norm{ ( \CT_0 \star \CG_k \vert_{ \CC'_k } )^{ ( n_0, n_1, \hdots, n_k ) } }_{ \CL^k } \\
& = & \sup_{ n_0, n_1, \hdots, n_k } \bigg[ \sup_{ \norm{ Z_1 }_{ n_0, n_1 } = \cdots = \norm{ Z_k }_{ n_{ k - 1 }, n_k }= 1 } \norm{ ( Z_1 \odot \cdots \odot Z_k ) \CT_0 \star \CG_k \vert_{ \CC'_k } } \bigg] \\
& \geq & C \sup_{ n_0, n_1, \hdots, n_k } \bigg[ \sup_{ \norm{ Z_1 }_{ n_0, n_1 } = \cdots = \norm{ Z_k }_{ n_{ k - 1 }, n_k } = 1 } \norm{ ( Z_1 \odot \cdots \odot Z_k ) \CG_k \vert_{ \CC'_k } } \bigg] \\
& \geq & \alpha \norm{ \CG_k \vert_{ \CC'_k } }_{ \CL_{ \text{cb} }^k } .
\Eea
Here, we use Equation \eqref{bounded below inequality} to obtain the inequality in the third equation and $ \alpha = \mbox{min} \{ 1, C \} $ as mentioned earlier in this proof. Consequently, we have that
\begin{multline*}
\norm{ \CG_k \vert_{ \CC'_k} }_{ \CL_{ \text{cb} }^k } \leq \frac{ 1 }{ \alpha } \norm{ \CT_0 \star \CG_k \vert_{ \CC'_k } }_{ \CL_{ \text{cb} }^k } = \frac{ 1 }{ \alpha } \norm{ \CF_k - \sum_{ j = 0 }^{ k - 1 } \CT_{ k - j } \star \CG_j }_{ \CL_{ \text{cb} }^k }\\ 
%& \leq & \frac{ 1 }{ \alpha } \left[ \norm{ \CF_k }_{ \CL_{ \text{cb} }^k } + \sum_{ j = 0 }^{ k - 1 } \norm{ \CT_{ k - j } \CG_j }_{ \CL_{ \text{cb} }^k } \right]  
\leq   \frac{ 1 }{ \alpha } \bigg[ \norm{ \CF_k }_{ \CL_{ \text{cb} }^k } + \sum_{ j = 0 }^{ k - 1 } \norm{ \CT_{ k - j } }_{ \CL_{ \text{cb} }^{ k - j } } \norm{ \CG_j }_{ \CL_{ \text{cb} }^j } \bigg] 
%& \leq & \frac{ M }{ \alpha \tilde{ \hat{ r } }^k } + \sum_{ j = 0 }^{ k - 1 } \frac{ M }{ \tilde{ \hat{ r } }^{ k - j } } \times \frac{ 2^j M^j }{ \tilde{ \hat{ r } }^j \alpha^{ j + 1 } } \text{max} \{ 1, \norm{ H }_{ s, \CH_c } \} \\
\leq  \frac{ M }{ \alpha } \bigg[ \frac{ 1 }{ \tilde{ \hat{ r } }^k } + \sum_{ j = 0 }^{ k - 1 } \frac{ 2^j M^j }{ \alpha^j  \tilde{ \hat{ r } }^k } \mathrm{max} \big\{ 1, \norm{ H }_{ s, \CY } \big\} \bigg] \\
\leq  \frac{ M^k }{ \alpha^k \tilde{ \hat{ r } }^k } \bigg[ 1 + \sum_{ j = 0 }^{ k - 1 } \frac{ 2^j \alpha^{ k - j } }{ M^{ k - j } }\mathrm{max} \big\{ 1, \norm{ H }_{ s, \CY } \big\} \bigg] \leq   \frac{ M^k }{ \alpha^k \tilde{ \hat{ r } }^k } \bigg[ 1 + \sum_{ j = 0 }^{ k - 1 } 2^j \bigg] \mathrm{max} \big\{ 1, \norm{ H }_{ s, \CY } \big\} \\
\leq   \frac{ 2^k M^k }{ \alpha^k \tilde{ \hat{ r } }^k } \mathrm{max} \big\{ 1, \norm{ H }_{ s, \CY } \big\} =  \tilde{ r }'^{ - k } \mathrm{max} \big\{ 1, \norm{ H }_{ s, \CY } \big\} \hspace{2.42in}
\end{multline*}

Thus the desired $ k $-linear mapping $ \CG_k : ( \CV^{ s \times s } )^{ \otimes k } \rightarrow \CY^{ s \times s } $ is defined as 
$$ \CG_k ( Z^1, \hdots, Z^k ) : = \CG_k \vert_{ \CC_k} ( \varrho_k ( Z^1, \hdots, Z^k ) ) + \CG_k \vert_{ \CC'_k } ( \varrho'_k ( Z^1, \hdots, Z^k ) ) . $$ 
Since $ \CG_k \vert_{ \CC'_k } $ satisfies Equations \eqref{CC1}, \eqref{CC2} and \eqref{CC3} as well as $ \CG_k \vert_{ \CC_k } $ together with $ \CG_0, \CG_1, \hdots, \CG_{ k - 1 } $ satisfy the canonical intertwining conditions up to order $ k $, it follows that the canonical intertwining conditions hold for $ \CG_0, \hdots, \CG_k $ up to order $ k $. Finally, we show that $ \norm{ \CG_k }_{ \CL^k_{ \text{cb} } } $ is bounded above by $ r'^{ - k } \mathrm{max} \big\{ 1, \norm{ H }_{ s, \CY } \big\} $ with the help of the bounds on $ \norm{ \CG_k \vert_{ \CC_k } \circ \varrho_k }_{ \CL_\text{cb}^k } $ and $ \norm{ \CG_k \vert_{ \CC'_k } \circ \varrho'_k }_{ \CL_\text{cb}^k } $ as demonstrated below.
\Bea
\norm{ \CG_k }_{ \CL_\text{cb}^k } & \leq & \norm{ \CG_k \vert_{ \CC_k } \circ \varrho_k }_{ \CL_\text{cb}^k } + \norm{ \CG_k \vert_{ \CC'_k } \circ \varrho'_k }_{ \CL_\text{cb}^k } \\
& \leq & \norm{ \CG_k \vert_{ \CC_k } }_{ \CL_\text{cb}^k } \norm{ \varrho_k }_{ \text{cb} } + \norm{ \CG_k \vert_{ \CC'_k} }_{ \CL_\text{cb}^k } \norm{ \varrho'_k }_{ \text{cb} } \\
& \leq & \norm{ \CG_k \vert_{ \CC_k} }_{ \CL_\text{cb}^k } ( 1 + \norm{ \varrho'_1 }^k_{ \text{cb} } ) + \norm{ \CG_k \vert_{ \CC'_k } }_{ \CL_\text{cb}^k } \norm{ \varrho'_1 }^k_{ \text{cb} } \\
& = & \norm{ \CG_k \vert_{ \CC_k} }_{ \CL_\text{cb}^k } + \left( \norm{ \CG_k \vert_{ \CC_k } }_{ \CL_\text{cb}^k } +   \norm{ \CG_k \vert_{ \CC'_k } }_{ \CL_\text{cb}^k } \right) \norm{ \varrho'_1 }^k_{ \text{cb} } \\
& \leq & 2^k \norm{ \varrho_1' }^k_{ \text{cb} } \norm{ \varphi \circ \varrho_1 }_{ \text{cb} } \norm{ \CG_{ k - 1 } }_{ \text{cb} } \\
& + & \left( 2^k \norm{ \varrho_1' }^k_{ \text{cb} } \norm{ \varphi \circ \varrho_1 }_{ \text{cb} } \norm{ \CG_{ k - 1 } }_{ \text{cb} } + \tilde{ r }'^{ - k } \mathrm{max} \big\{ 1, \norm{ H }_{ s, \CY } \big\} \right) \norm{ \varrho'_1 }^k_{ \text{cb} } \\
& \leq & 2^k \norm{ \varrho_1'}^k_{ \text{cb} } \norm{ \varphi \circ \varrho_1 }_{ \text{cb} } \tilde{ r }'^{ - k } \mathrm{max} \big\{ 1, \norm{ H }_{ s, \CY } \big\} \tilde{ r }' \\
& + & \left( 2^k \norm{ \varrho_1' }^k_{ \text{cb} } \norm{ \varphi \circ \varrho_1 }_{ \text{cb} } \tilde{ r }'^{ - k } \mathrm{max} \big\{ 1, \norm{ H }_{ s, \CY } \big\} \tilde{ r }' + \tilde{ r }'^{ - k } \text{max} \{ 1, \norm{ H }_{s, \CY } \} \right) \norm{ \varrho'_1 }^k_{ \text{cb} } \\
& \leq & 2^k \norm{ \varrho_1' }^k_{ \text{cb} } \tilde{ r }'^{ - k } \mathrm{max} \big\{ 1, \norm{ H }_{ s, \CY } \big\} \\ & + & \left( 2^k \norm{ \varrho_1' }^k_{ \text{cb} } \tilde{ r }'^{ - k } \mathrm{max} \big\{ 1, \norm{ H }_{ s, \CY } \big\} + \tilde{ r }'^{ - k } \mathrm{max} \big\{ 1, \norm{ H }_{ s, \CY } \big\} \right) \norm{ \varrho'_1 }^k_{ \text{cb} } \\
& \leq & 2^k \norm{ \varrho_1' }^k_{ \text{cb} } \tilde{ r }'^{ - k } \mathrm{max} \big\{ 1, \norm{ H }_{ s, \CY } \big\} ( 1 + 2 \norm{ \varrho'_1 }^k_{ \text{cb} } ) \\
& \leq &  [ 2 \norm{ \varrho_1' }_{ \text{cb} } ( 1 + 2 \norm{ \varrho'_1 }_{ \text{cb} } ) ]^k \tilde{ r }'^{ - k } \text{max}\{ 1, \norm{ H }_{ s, \CY } \} \\
& \leq & r'^{ - k } \mathrm{max} \big\{ 1, \norm{ H }_{ s, \CY } \big\}.
\Eea
This completes the proof.
\end{proof}

\begin{rem} \label{rem on solving nc eqn}
(1) When the dimension of $ \CV $ is finite and $ Y \in \Omega_s \subset \Omega \subset \CV_{ nc } $ is a semi-simple point, $ \CC ( Y ) $-bi-modules turn out to be semi-simple. Consequently, being a $\CC ( Y ) $-bi-module in a natural way, $ \CV^{ s \times s } $ is semi-simple. Also, note that the map $ \mathfrak{ c } : \C^{ s \times s } \rightarrow \CV^{ s \times s } $ -- defined by $ \SS \mapsto [ \SS, Y ] $ -- admits a $ \CC ( Y ) $-bi-module right inverse $ \varphi : \CC_1 \rightarrow \C^{ s \times s } $. Thus, $ \varphi $ gives rise to a $ \CC ( Y ) $-sub-bi-module $ \CC_1' : = \ker \varphi $ complementary to $ \CC_1 $ in $ \CV^{ s \times s } $.

(2) Given Hilbert spaces $ \CV, \CW, \CH $ and $ \CK $, set $ \CX = \CB ( \CH, \CK ) $, $ \CY = \CB ( \CW, \CH ) $ and $ \CZ = \CB ( \CW, \CK ) $ equipped with the natural operator space structure. Let $ \mathfrak{ m } : \CX \times \CY \rightarrow \CZ $ be the multiplication operator obtained by the composition operator and fix $ H \in \CB ( \CW, \CH )^{ s \times s } $. Then for any uniformly analytic nc functions 
$$ \CT : \Omega \rightarrow \CB ( \CH, \CK )_{ nc } \quad \mbox{and} \quad \CF : \Omega \rightarrow \CB ( \CW, \CK )_{ nc } $$
satisfying the hypothesis of Theorem \ref{closed range operators general version}, there exist a uniformly open nc neighbourhood $ U $ of $ Y $ and a uniformly analytic function $ \CG : U \rightarrow \CB ( \CW, \CH )_{ nc } $ such that 
$$ \CG ( Y ) = H \quad \text{and} \quad \CT ( X) \circ \CG ( X ) = \CF ( X ), \quad \text{for all} ~~ X \in U . $$ Observe from CSPS theorem \cite[Theorem 1.5.7, pp. 32]{Blecher-Merdy} that this is essentially the Theorem \ref{closed range operators general version} when $ \CZ = \CB ( \CW, \CK ) $ for some Hilbert spaces $ \CW $ and $ \CK $. 

In particular, when $ \CW = \C $, $ \CY $ and $ \CZ $ becomes $ \CH $ and $ \CK $, respectively, equipped with the column operator space structure. In this case, it can be seen as in Lemma \ref{nearest point approximation} and Lemma \ref{closed range implies bounded below} in the following subsection that the hypothesis on $ \mathfrak{ m } $ and $ \CT $ in Theorem \ref{closed range operators general version} are automatic.
  
(3) As a special case taking $ \CW = \CH $ and $ \CK = \C $, we see that $ \CX = \CZ = \CH_r $ equipped with the row operator space structure and $ \CY = \CB ( \CH ) $. This particular case will be of interest in Section \ref{Noncommutative Cowen-Douglas class and vector bundles} and the hypothesis on $ \mathfrak{ m } $ and $ \CT $ in Theorem \ref{closed range operators general version} for this case are automatically satisfied as we show in the following subsection.
\end{rem}

\subsection{Application to the row operator Hilbert space} \label{column operator space}

In this subsection, we briefly recall the notion of the \textit{row operator space structure} on a Hilbert space as well as present an \textit{open mapping type result} (see Proposition \ref{closed range implies bounded below}) in a row operator space, which we then use to obtain a simpler version of Theorem \ref{closed range operators general version} in the case of row Hilbert operator spaces.

Recall that a Hilbert space $ \CH $ gives rise to a canonical isometry between the conjugate Hilbert space and its Banach dual given by 
$$ \theta : \overline{ \CH } \rightarrow \CH^* , \quad \overline{ \xi } \mapsto f_{ \xi }, \quad f_{ \xi } ( \eta ) : = \langle \eta, \xi \rangle, $$
and by Riesz representation theorem we have that $ \norm{ f_{ \xi } } = \norm{ \xi } $. We now define the \textit{Row isometry} $ \mathsf{ Row } : \CH \rightarrow \CH^{ * * } ( = \CB ( \CH^*, \C ) = \CB ( \overline{ \CH }, \C ) ) $ by
$$ \xi \mapsto \mathsf{ Row } ( \xi ) ( \overline{ \eta } ) : = \theta ( \overline{ \xi } ) ( \eta ) = \langle \eta, \xi \rangle. $$
Being an operator space, $ \CB ( \overline{ \CH }, \C ) $ induces an operator space structure on $ \CH $, which is known as the \textit{row operator space} structure. 

For any $ s \in \N $, we define the $ s $-th matrix norm on $ \CH^{ s \times s } $ by
$$ \norm{ H }_{ s, \CH_r } : = \norm{ \mathsf{ Row }_s ( H ) }, \quad \mathsf{ Row }_s : \CH^{ s \times s } \to \CB \left( \overline{ \CH }^s, \C^s \right) $$
is the amplification of $ \mathsf{ Row } $ (cf. Subsection \ref{amplifying operators}). Note that $  \norm{ H }_{ s, \CH_r }  $ is the operator norm of $ H $ viewed as an operator in $ \CB ( \overline{\CH}^s, \C^s ) $. Thus this family $ \{ \| \cdot \|_{ s, \CH_r } \}_{ s = 1 }^{ \infty } $ of norms on $ \CH $ indeed satisfy the conditions (\textbf{R1}) and (\textbf{R2}) mentioned in Subsection \ref{operator spaces} making $ \CH $ an operator space known as the \textit{row Hilbert operator space}. We denote it by 
$$ \CH_r = \left( \CH, \left\{ \| \cdot \|_{ s, \CH_r } \right\}_{ s \in \N } \right) . $$
Also, note that the adjoint operator of the row isometry acts as 
$$ \mathsf{ Row } ( \xi )^* : \C \rightarrow \overline{ \CH }, \quad \alpha \mapsto \alpha \overline{ \xi } ,$$
and consequently, its $ s $-th amplification acts as
$$ \mathsf{ Row }_s ( H )^* : \C^s \rightarrow \overline{ \CH }^s, \quad \left[ \begin{smallmatrix} \alpha_1 & \cdots & \alpha_s \end{smallmatrix} \right] \mapsto \left[ \begin{smallmatrix} \alpha_1 & \cdots & \alpha_s \end{smallmatrix} \right] \overline{ H } , ~~ H \in \CH^{ s \times s } . $$

In general, for a given operator space $ \CV $, the space of matrices $ \CV^{ n \times n } $ over $ \CV $ is a Banach space for each $ n \in \N $. Consequently, there is no canonical notion of the ``orthogonal projections" onto a closed subspace of $ \CV^{ n \times n } $ for $ n \in \N $. However, when $ \CV $ is a Hilbert space $ \CH $ endowed with the \textit{row operator space structure}, the Hilbert space structure on $ \CH $ canonically induces a notion of orthogonal projections onto a closed subspace of $ \CH^{ n \times n } $ with $ n \in \N $ as we point out in the following lemma.

\begin{lem} \label{nearest point approximation}

Let $ s \in \N $ and $ \CN \subset \CH_r^{ 1 \times s } $ be a closed subspace. 
\begin{itemize}
\item[(i)] For any $ H \in \CH_r^{ s \times s } $, there exists unique $ N_0 \in \CN^{ s \times 1 } $ such that 
$$ \inf \left\{ \norm{ H - N }_{ s, \CH_r } : N \in \CN^{ s \times 1 } \right\} = \norm{ H - N_0 }_{ s, \CH_r }. $$

\item[(ii)] For any $ n \in \N $ and $ ( \! ( H_{lk} ) \! )_{ l, k = 1 }^n \in \CH_r^{ ns \times ns } $ with $ H_{ l k } \in \CH_r^{ s \times s } $, 
\begin{align*} 
& \inf \left\{ \norm{ ( \! ( H_{lk} ) \! )_{ l, k = 1 }^n - ( \! ( N_{ l k } ) \! )_{ l, k = 1 }^n }_{ ns, \CH_r } : ( \! ( N_{ l k } ) \! )_{ l, k = 1 }^n \in ( \CN^{ s \times 1 } )^{ n \times n } \right\} \\
& = \norm{ ( \! ( H_{ l k } ) \! )_{ l, k = 1 }^n - ( \! ( N^0_{ l k } ) \! )_{ l, k = 1 }^n }_{ ns, \CH_r } 
\end{align*}
where for each $ l, k = 1, \hdots, n $, $ N^0_{ l k } $ is the unique element in $ \CN^{ s \times 1 } $ such that 
$$ \inf \left\{ \norm{ H_{ l k } - N_{ l k } }_{ s, \CH_r } : N_{ l k } \in \CN^{ 1 \times s } \right\} = \norm{ H_{ l k } - N^0_{ l k } }_{ s,\CH_r }. $$

\item[(iii)] The mapping $ Q: \CH_r^{ s \times s } \rightarrow \CH_r^{ s \times s } $ defined by $ Q ( H ) := H - N_0 $ is linear where $ N_0 $ is as defined in (i). 
\end{itemize}
\end{lem}

\begin{proof}
(i) We begin by pointing out that 
\begin{align*}
& \inf \left\{ \norm{ H - N }_{ s, \CH_r } : N \in \CN^{ s \times 1 } \right\} \\
& = \inf \left\{ \sup \left\{ \norm{ \mu \overline{ H } - \mu \overline{ N } }_{ \CH_r^{ 1 \times s } } : \mu \in \C^{ 1 \times s }, \norm{ \mu } \leq 1 \right \} : N \in \CN^{ s \times 1 } \right\} \\
& = \sup \left\{ \inf \left\{ \norm{ \mu \overline{ H } - \mu \overline{ N } }_{ \CH_r^{ 1 \times s } } : N \in \CN^{ s \times 1 } \right\} : \mu \in \C^{ 1 \times s }, \norm{ \mu } \leq 1 \right\}.
\end{align*}
Since $ \CH_r^{ 1 \times s } $ (isometrically isomorphic to the direct sum $ \CH \oplus \cdots \oplus \CH $ of $ s $ copies of $ \CH$) is a Hilbert space, for each $ \mu \in \C^{ 1 \times s } $ with $ \norm{ \mu } \leq 1 $, there exists unique $ n^0_{ \mu } \in \CN $ such that 
$$ \inf \left\{ \norm{ \mu \overline{ H } - \mu \overline{ N } }_{ \CH_r^{ 1 \times s } } : N \in \CN^{ s \times 1 } \right\} = \norm{ \mu \overline{ H } - \bar{ n }^0_{ \mu } }_{ \CH_r^{ 1 \times s } }. $$
Denote $ n^0_{ \varepsilon_j } = \left[ \begin{smallmatrix} n^0_{ j 1 } & n^0_{ j 2 } & \hdots & n^0_{ j s } \end{smallmatrix}\right] \in \CN $, where $ \{ \varepsilon_1, \hdots, \varepsilon_s \} $ is the standard ordered basis for $ \C^{ 1 \times s } $. We now show, for any $ \mu \in \C^{ 1 \times s } $, that there exists $ ( \! ( \bar{ n }^0_{ i j } ) \! )_{ i, j = 1 }^s \in \CN^{ s \times 1 } $ satisfying 
$$ \bar{ n }^0_{ \mu } = \mu ( \! ( \bar{ n }^0_{ i j } ) \! )_{ i, j = 1 }^s . $$

Note that for any $ \mu \in \C^{ 1 \times s } $, the subspace $ \CN $ can be realized as 
$$ \overline{ \CN } = \{ \mu \overline{ N } : N \in \CN^{ s \times 1 } \} $$ 
where $ \overline{ \CN } = \{ \bar{ n } : n \in \CN \} $. Indeed, the set in the right hand side is contained in $ \overline{ \CN } $ and for any $ \bar{ n } = \left[ \begin{smallmatrix} \bar{ n }_1 & \hdots & \bar{ n }_s \end{smallmatrix} \right] \in \overline{ \CN } $ and non-zero $ \mu = ( \mu_1, \hdots, \mu_s ) \in \C^{ 1 \times s } $ with $ \mu_l \neq 0 $ for some $ l $, the vector $ \bar{ n } $ can be written as $ \mu ( \! ( \bar{ n }'_{ i j } ) \! )_{ i, j = 1 }^s $ where $ \bar{ n }'_{ i j } $ is $ 0 $ if $ i \neq l $ and $ \bar{ n }'_{ l j } = \frac{ \bar{ n }_j }{ \mu_l } $. Consequently, for any fixed $ \mu \in \C^{ 1 \times s } $, we have that
$$ \inf \left\{ \norm{ \mu \overline{ H } - \mu \overline{ N } }_{ \CH_r^{ 1 \times s } } : N \in \CN^{ s \times 1 } \right\} = \inf \left\{ \norm{ \mu \overline{ H } - \bar{ n } }_{ \CH_r^{ 1 \times s } } : n \in \CN \right\} = \norm{ \overline{ P } ( \mu \overline{ H } ) } $$
where $ P : \CH_r^{ 1 \times s } \rightarrow \CN^{ \perp } $ is the orthogonal projection and $ \overline{ P } : \overline{ \CH }^s \rightarrow \overline{ \CN }^{ \perp } $. Next, we compute this infimum as follows
\begin{multline*}
\norm{ \overline{ P } ( \mu \overline{ H } ) }  =  \norm{ \overline{ P } \left( \sum_{ l = 1 }^s \mu_l \varepsilon_l \overline{ H } \right) } =  \norm{ \sum_{ l = 1 }^s \mu_l \overline{ P } ( \varepsilon_l \overline{ H } ) } 
= \norm{ \sum_{ l = 1 }^s \mu_l ( \varepsilon_l \overline{ H } - \bar{ n }_{ \varepsilon_l } ) }\\ 
=  \norm{ \sum_{ l = 1 }^s \mu_l [ \varepsilon_l \overline{ H } - \varepsilon_l ( \! ( \bar{ n }^0_{ i j } ) \! )_{ i, j = 1 }^s  ] } 
=  \norm{ \mu \overline{ H } - \mu ( \! ( \bar{ n }^0_{ i j } )\! )_{ i, j = 1 }^s }.
\end{multline*}
From the uniqueness of $ n_{ \mu } $ it then follows that $ \bar{ n }_{ \mu } = \mu ( \! ( \bar{ n }^0_{ i j } ) \! )_{ i, j = 1 }^s $ verifying the claim. Now recalling the identities obtained at the beginning of the proof, we have that 
\begin{align*}
&\inf \left\{ \norm{ H - N }_{ s, \CH_r } : N \in \CN^{ s \times 1 } \right\} \\
&= \inf \left\{ \sup \left\{ \norm{ \mu \overline{ H } - \mu \overline{ N } }_{ \CH_r^{ 1 \times s } } : \mu \in \C^{ 1 \times s }, \norm{ \mu } \leq 1 \right \} : N \in \CN^{ s \times 1 } \right\} \\
&=  \sup \! \left\{ \norm{ \mu \overline{ H } - \mu ( \! ( \bar{ n }^0_{ i j } ) \! )_{ i, j = 1 }^s } : \mu \in \C^{ 1 \times s }, \norm{ \mu } \leq 1 \right\} 
=  \norm{ H - (\! ( n^0_{ i j } ) \! )_{ i, j = 1 }^s }_{ s, \CH_r }
\end{align*}
completing the proof of (i) with $ N_0 = (\! ( n^0_{ i j } ) \! )_{ i, j = 1 }^s $.

(ii) Following the same line of argument as in part (i) we have that
\begin{multline*}
\inf \left\{ \norm{ ( \! ( H_{lk} ) \! )_{ l, k = 1 }^n - ( \! ( N_{ l k } ) \! )_{ l, k = 1 }^n }_{ ns, \CH_r } : ( \! ( N_{ l k } ) \! )_{ l, k = 1 }^n \in ( \CN^{ s \times 1 } )^{ n \times n } \right\} \\
= \sup \left\{ \inf \left\{ \norm{ \left[ \begin{smallmatrix} \nu_1 & \cdots &  \nu_n \end{smallmatrix} \right]  ( \! ( \overline{ H }_{ l k } ) \! )_{ l, k = 1 }^n - \left[ \begin{smallmatrix} \bar{ \eta }_1 & \cdots & \bar{ \eta }_n \end{smallmatrix} \right] }_{ \CH_r^{ 1 \times ns } } : \left[ \begin{smallmatrix} \eta_1 & \cdots & \eta_n \end{smallmatrix} \right] \in \CN^{ \oplus n } \right\} \right. \\ \left. : ( \nu_1, \hdots, \nu_n ) \in \overline{ B_{ ns } ( 0, 1 ) }  \right\}
\end{multline*}
and for $ ( \nu_1, \hdots, \nu_n ) \in \overline{ B ( 0, 1 ) }_{ ns } $ (here, $ \overline{ B ( 0, 1 ) }_{ ns } $ denotes the closed unit ball in $ \C^{ 1 \times ns } $) with $ \nu_1, \hdots, \nu_n \in \C^{ 1 \times s } $, there exists a unique element $ \left[ \begin{smallmatrix} \eta^0_1 & \cdots & \eta^0_n \end{smallmatrix} \right]_{ \nu } \in \CN^{ \oplus n } $ such that 
\begin{align*} 
&\inf \left\{ \norm{ \left[ \begin{smallmatrix} \nu_1 & \cdots &  \nu_n \end{smallmatrix} \right]  ( \! ( \overline{ H }_{ l k } ) \! )_{ l, k = 1 }^n - \left[ \begin{smallmatrix} \bar{ \eta }_1 & \cdots & \bar{ \eta }_n \end{smallmatrix} \right] }_{ \CH_r^{ 1 \times ns } } : \left[ \begin{smallmatrix} \eta_1 & \cdots & \eta_n \end{smallmatrix} \right] \in \CN^{ \oplus n } \right\} \\
&= \norm{ \left[ \begin{smallmatrix} \nu_1 & \cdots & \nu_n \end{smallmatrix} \right] ( \! ( \overline{ H }_{ l k } ) \! )_{ l, k = 1 }^n  - \left[ \begin{smallmatrix} \bar{ \eta }^0_1 & \cdots & \bar{ \eta }^0_n \end{smallmatrix} \right]_{ \nu } }_{ \CH_r^{ 1 \times ns } } . 
\end{align*}
Thus, it remains to prove that
$ \left[ \begin{smallmatrix} \bar{ \eta }^0_1 & \cdots & \bar{ \eta }^0_n \end{smallmatrix} \right]_{ \nu } =  \left[ \begin{smallmatrix} \nu_1 & \cdots & \nu_n \end{smallmatrix} \right] ( \! ( \overline{ N }^0_{ l k } ) \! )_{ l, k = 1 }^n $
for some $ ( \! ( N ^0_{ l k } ) \! )_{ l, k = 1 }^n $ in $ ( \CN^{ s \times 1 } )^{ n \times n } $.

Let $ P : \CH_r^{ 1 \times s } \rightarrow \CN^{ \perp } $ be the orthogonal projection and $ P^{ ( n ) } : ( \CH_r^{ 1 \times  s } )^{ \oplus n } \rightarrow ( \CN^{ \perp } )^{ \oplus n } $ be the amplification of $ P $ defined by
$$ P^{ ( n ) } \left( \left[ \begin{smallmatrix} H_1 & \cdots & H_n \end{smallmatrix} \right] \right) = \left[ \begin{smallmatrix} P ( H_1 ) & \cdots & P ( H_n ) \end{smallmatrix} \right]. $$
Let $ \overline{ P } $ and $ \overline{ P^{ ( n ) } } $ be the orthogonal projections between the corresponding conjugated spaces, and note that $ \overline{ P }^{ ( n ) } = \overline{ P^{ ( n ) } } $.
Then observe that
\begin{align*}
& \inf \left\{ \norm{ \left[ \begin{smallmatrix} \nu_1 & \cdots &  \nu_n \end{smallmatrix} \right]  ( \! ( \overline{ H }_{ l k } ) \! )_{ l, k = 1 }^n - \left[ \begin{smallmatrix} \bar{ \eta }_1 & \cdots & \bar{ \eta }_n \end{smallmatrix} \right] }_{ \CH_r^{ 1 \times ns } } : \left[ \begin{smallmatrix} \eta_1 & \cdots & \eta_n \end{smallmatrix} \right] \in \CN^{ \oplus n } \right\} \\
& = \norm{ \overline{ P }^{ ( n ) } \left( \left[ \begin{smallmatrix} \nu_1 & \cdots & \nu_n \end{smallmatrix} \right] ( \! ( \overline{ H }_{ l k } ) \! )_{ l, k = 1 }^n \right) }_{ \CH_r^{ 1 \times ns } }\\
%& =  \norm{ P^{ ( n ) } \left( \begin{smallmatrix} \sum_{ k = 1 }^n H_{ 1 k } \nu_k & \cdots & \sum_{ k = 1 }^n H_{ n k } \nu_k \end{smallmatrix} \right)^{ \dagger } }_{ \CH_c^{ ns \times 1 } } \\
& =  \norm{ \left[ \begin{smallmatrix} \overline{ P } \left( \sum_{ k = 1 }^n \overline{ H }_{ k 1 } \nu_k \right) & \cdots & \overline{ P } \left( \sum_{ k = 1 }^n \overline{ H }_{ k n } \nu_k \right) \end{smallmatrix} \right] }_{ \CH_r^{ 1 \times ns } } \\
& =  \norm{ \left[ \begin{smallmatrix} \sum_{ k = 1 }^n \overline{ P } ( \overline{ H }_{ k 1 } \nu_k ) & \cdots & \sum_{ k = 1 }^n \overline{ P } ( \overline{ H }_{ k n } \nu_k ) \end{smallmatrix} \right] }_{ \CH_r^{ 1 \times ns } } \\
& =  \norm{ \left[ \begin{smallmatrix} \sum_{ k = 1 }^n ( \overline{ H }_{ k 1 } \nu_k - \overline{ N }^0_{ k 1 } \nu_k ) & \cdots & \sum_{ k = 1 }^n ( \overline{ H }_{ k n } \nu_k - \overline{ N }^0_{ k n } \nu_k \end{smallmatrix} \right] }_{ \CH_r^{ 1 \times ns } } \\
& =  \norm{ \left[ \begin{smallmatrix} \nu_1 &  \cdots & \nu_n \end{smallmatrix} \right] \left[ ( \! ( \overline{ H }_{ l k } ) \! )_{ l, k = 1 }^n - ( \! ( \overline{ N }^0_{ l k } ) \!)_{ l, k = 1 }^n \right] }_{ \CH_r^{ 1 \times ns } }
\end{align*}
where the second last equality holds from part (i) and, for each $ l, k = 1, \hdots, n $, $ N^0_{ l k } $ is the unique element in $ \CN^{ s \times 1 } $ such that 
$$ \inf \left\{ \norm{ H_{ l k } - N_{ l k } }_{ s, \CH_r } : N_{ l k } \in \CN^{ s \times 1 } \right\} = \norm{ H_{ l k} - N^0_{ l k } }_{ s, \CH_r } . $$
Consequently, by the uniqueness of $ \left[ \begin{smallmatrix} \eta^0_1 & \cdots &  \eta^0_n \end{smallmatrix} \right]_{ \nu } $ we have that 
$$ \left[ \begin{smallmatrix} \bar{ \eta }^0_1 & \cdots & \bar{ \eta }^0_n \end{smallmatrix} \right]_{ \nu } = \left[ \begin{smallmatrix} \nu_1 & \cdots & \nu_n \end{smallmatrix} \right]  ( \! ( \overline{ N }^0_{ l k } ) \!)_{ l, k = 1 }^n $$
verifying the claim and hence, it completes the proof of part (ii) as in part (i).

(iii) Let $ H^1, H^2 \in \CH_r^{ s \times s } $ and $ Q( H^l ) = H^l - N^l $ where $ N^l \in  \CN^{ s \times 1 } $ are unique elements such that 
$$ \inf \left\{ \norm{ H^l - N }_{ s, \CH_r } : N \in \CN^{ s \times 1 } \right\} = \norm{ H^l - N^l }_{ s, \CH_r },~ l = 1, 2. $$

Observe that it is enough to show for $ \lambda_1, \lambda_2 \in \C $, that 
$$ \inf \left\{ \norm{ \lambda_1 H^1 + \lambda_2 H^2 - N }_{ s, \CH_r } : N \in \CN^{ s \times 1 } \right\}  = \norm{ \lambda_1 [ H^1 - N^1 ] + \lambda_2 [ H^2 - N^2 ] }_{ s, \CH_r }. $$

For $ \mu \in \C^{ 1 \times s } $, recall from the proof of part (i) that there exists unique $ N_0 \in \CN^{ s \times 1 } $ satisfying
\begin{align} \label{ inf = sup }
& \inf \left\{ \norm{ \lambda_1 H^1 + \lambda_2 H^2 - N }_{ s, \CH_r } : N \in \CN^{ s \times 1 } \right\} \\
\nonumber & = \sup \left\{ \norm{ \mu [ \bar{ \lambda }_1 \overline{ H }^1 + \bar{ \lambda }_2 \overline{ H }^2 - \overline{ N }_0 ] }_{ \CH_r^{ 1 \times s } } : \mu \in \C^{ 1 \times s }, \norm{ \mu } \leq  1 \right\}. 
\end{align}
Also, note that for any $ \mu \in \C^{ 1 \times s } $, 
$$  \norm{ \mu [ \bar{ \lambda }_1 \overline{ H }^1 + \bar{ \lambda }_2 \overline{ H }^2 - \overline{ N }_0 ] }_{ \CH_r^{ 1 \times s } } = \norm{ \mu [ \bar{ \lambda }_1 ( \overline{ H }^1 - \overline{ N }^1 ) ] + \mu [ \bar{ \lambda }_2 ( \overline{ H }^2 - \overline{ N }^2 ) }. $$
since $ \CH_r^{ 1 \times s } $ is a Hilbert space. Consequently, the proof of part (iii) follows from Equation \eqref{ inf = sup }
\end{proof}

\begin{lem} \label{closed range implies bounded below}

Let $ \CH $ and $ \CK $ be Hilbert spaces over $ \C $ equipped with the row operator space structure -- denoted by $ \CH_r $ and $ \CK_r $, respectively. Let $ \CB ( \CH, \CK ) $ be the space of bounded linear operators from $ \CH $ to $\CK $ equipped with the natural operator space structure on it. For any $ s \in \N $, assume that an element $ T = ( \! ( T_{ i j } ) \! )_{ i, j = 1 }^s $ in $ \CB (\CH, \CK )^{ s \times s } $ has closed range as a linear mapping $ T : \CH^{ \oplus s } \rightarrow \CK^{ \oplus s } $. Then there exists $ C > 0 $ such that for any $ n \in \N $ and $ H = ( \! ( H_{ i j } ) \! )_{ i, j = 1 }^n \in \CH_r^{ ns \times ns } $ with $ H_{ i j } \in \CH_r^{ s \times s } $, $ i, j = 1, \hdots, n $,
$$ \norm{ R_{ T^{ \oplus n } } ( Q^{ \oplus n } ( H ) ) }_{ ns, \CK_r } \geq  C \norm{ Q^{ \oplus n } ( H ) }_{ ns, \CH_r } $$ where $ Q:\CH_r^{ s \times s } \rightarrow \CH_r^{ s \times s } $ is the linear mapping defined in (iii) in Lemma \ref{nearest point approximation} with $ N = \ker R_{ T^{ \oplus n } } $. Here, $ R_{ T } : \CH_r^{ s \times s } \rightarrow \CK_r^{ s \times s } $ is defined by the right multiplication by $ T $.
\end{lem}

\begin{proof}
First, observe that $ \CH^{ \oplus s } \cong \CH_r^{ 1 \times s } $ and  $ \CK^{ \oplus s } \cong \CK_r^{ 1 \times s } $. Since $ T : \CH^{ \oplus s } \rightarrow \CK^{ \oplus s } $ acting from right on the row vectors over $ \CH $ has closed range, there exists $ C > 0 $ such that for $ h \in \CH^{ \oplus s } $,
\be \label{ bounded below inequality } 
\norm{ T ( h ) }_{ \CK^{ \oplus s } } \geq C \norm{ [ h ] }_{ \CH^{ \oplus s } / \ker T }.
\ee
Also, for $n\in\N$, we have 
\begin{multline*}
\norm{ \left[ \begin{smallmatrix} \! h_1 & \cdots & h_n \! \end{smallmatrix} \right] T^{ \oplus n } }^2_{ \CK_r^{ 1 \times ns } } =  \sum_{ j = 1 }^n \norm{ T ( h_j ) }^2_{ \CK^{ \oplus s } } 
\geq C^2 \sum_{ j = 1 }^n \norm{ [ h_j ] }^2_{ \CH^{ \oplus s } / \ker T } \\
= C^2 \sum_{ j = 1 }^n \inf \left\{ \norm{ h_j - n_j }^2_{ \CH^{ \oplus s } } : n_j \in \ker T \right\} 
= C^2 \inf \left \{ \sum_{ j = 1 }^n \norm{ h_j - n_j }^2_{ \CH^{ \oplus s } } : n_j \in \ker T \right\} \\
= C^2 \inf \left\{ \norm{ \left[ \begin{smallmatrix} \! h_1 & \cdots & h_n \! \end{smallmatrix} \right] - \left[ \begin{smallmatrix} \! \eta_1 & \cdots & \eta_n \! \end{smallmatrix} \right] }^2_{ \CH_r^{ 1 \times ns } } : \left[ \begin{smallmatrix} \! \eta_1 & \cdots & \eta_n \! \end{smallmatrix} \right] \in \ker T^{ \oplus n } \right\} 
= C^2 \norm{ \left[ \begin{smallmatrix} \! h_1 & \cdots & h_n \! \end{smallmatrix} \right] }^2_{ \CH_r^{ 1 \times ns } / \ker T^{ \oplus n } }
\end{multline*}
where the second equality holds from Equation \eqref{ bounded below inequality }. 

Now for $ H = ( \!( H_{ i j } ) \! )_{ i, j = 1 }^n \in \CH_r^{ ns \times ns } $, we compute
\begin{multline*}
\norm{ R_{ T^{ \oplus n } } ( Q^{ \oplus n } ( H ) ) }_{ ns, \CK_r } 
%& = \norm{ T^{ \oplus n } \cdot ( \! ( H_{ i j } ) \! )_{ i, j = 1 }^n }_{ ns, \CK_c } \\
= \sup \left\{ \norm{ \left( \left[ \begin{smallmatrix} \! v_1 & \cdots &  v_n  \! \end{smallmatrix} \right] \overline{ H } \right) T^{ \oplus n } }_{ \CK_r^{ 1 \times ns } } : ( v_1, \hdots, v_n ) \in \overline{ B ( 0, 1 ) }_{ns} \right\} \\
\geq  C \sup \left\{ \norm{ \left[ \begin{smallmatrix} \! v_1 & \cdots &  v_n \! \end{smallmatrix} \right] \overline{ H } }_{ \CH_r^{ 1 \times ns } / \ker T^{ \oplus n } } : ( v_1, \hdots, v_n ) \in \overline{ B ( 0, 1 ) }_{ ns } \right\} \\
= C \sup \! \left\{ \inf \! \left\{ \norm{ \left[ \begin{smallmatrix} \! v_1 & \cdots &  v_n \! \end{smallmatrix} \right] \overline{ H } - \left[ \begin{smallmatrix} \! \bar{ \eta }_1 & \cdots & \bar{ \eta }_n \! \end{smallmatrix} \right] }_{ \CH_r^{ 1 \times ns } } 
: \left[ \begin{smallmatrix} \! \eta_1 & \cdots & \eta_n \! \end{smallmatrix} \right] \in \ker T^{ \oplus n } \right\} : \right. \\  
\left. ( v_1, \hdots, v_n ) \in \overline{ B ( 0, 1 ) }_{ ns } \right\} = C \sup \left\{ \! \inf \! \left\{ \! \norm{ \left[ \begin{smallmatrix} \! v_1 & \cdots &  v_n \! \end{smallmatrix} \right] \overline{ H } \! - \! \left[ \begin{smallmatrix} \! v_1 & \cdots & v_n \! \end{smallmatrix} \right] \overline{ N } }_{ \CH_r^{ 1 \times ns } } \! \! \! : \! N \in \ker R_{ T^{ \oplus n } } \right\} \! : \right. \\
\left. \! ( v_1, \hdots, v_n ) \in \! \overline{ B ( 0, 1 ) }_{ ns } \! \right\} = C \inf \left\{ \norm{ H - N }_{ ns, \CH_r } : N \in \ker R_{ T^{ \oplus n } } \right\} = C \norm{ Q^{ \oplus n } ( H ) }_{ ns, \CH_r }
\end{multline*}
where the first equality holds from the definition of $ Q^{ \oplus n } : ( \CH_r^{ s \times s } )^{ n \times n } \rightarrow ( \CH_r^{ s \times s } )^{ n \times n } $ and the last one is obtained by using Lemma \ref{nearest point approximation} with $ \CN = \ker T $ and taking account the fact that
$$ \ker T^{ \oplus n } = \underbrace{ \ker T \oplus \cdots \oplus \ker T }_{ n - \text{times} } \subset \CH_r^{ 1 \times ns } \quad \text{and} \quad \ker R_{ T^{ \oplus n } } = ( \ker T^{ \oplus n } )^{ ns \times 1 } \subset \CH_r^{ ns \times ns }, $$ $ n \in \N $. This completes the proof.
 \end{proof}

We conclude this section with the following theorem -- a special case of Theorem \ref{closed range operators general version} in the case of row Hilbert operator space -- which we will use later in Section \ref{Noncommutative Cowen-Douglas class and vector bundles}.

\begin{thm} \label{closed range operators}
Let $ \CV $ be an operator space, and a Hilbert space $ \CH $ be equipped with the row operator space structure -- denoted by $ \CH_r $. Let $ \Omega \subset \CV_{ nc } $ be a bounded, uniformly open nc set which is right admissible and fix $ s \in \N $, a point $ Y \in \Omega_s $ so that $ \CC_1 $ is a complemented sub-bi-module of $ \CV^{ s \times s } $ over $ \CC ( Y ) $ and $ H \in \CH_r^{ s \times s } $. Consider uniformly analytic nc functions 
$$ \CT : \Omega \rightarrow \CB ( \CH )_{ nc } \quad \mbox{and} \quad \CF : \Omega \rightarrow ( \CH_r )_{ nc } $$ 
satisfying $ \CF ( X ) \in \mathrm{ ran } ~ \CT ( X ) $ for $ X \in \Omega $ and $ R_{ \CT ( Y ) } ( H )  = \CF ( Y ) $ where $ \CB ( \CH ) $ is equipped with the natural operator space structure on it. Assume that $ \CT ( Y ) : \CH^{ \oplus s } \rightarrow \CH^{ \oplus s } $ satisfies the following condition: for any $ m \in \N $, if $ X_n \in \Omega_{ ms } $ with $ X_n $ converging to $ Y^{ \oplus m } $, $ K_n \in \mathrm{ ran } ~ \CT ( X_n ) $ and $ K_n $ converges to $ K $, then $ K \in \mbox{ran } \CT ( Y^{ \oplus m } ) $ (in particular, $ \CT ( Y )^{ \oplus m } $ has closed range for each $ m \in \N $). Then there exist a uniformly open nc neighbourhood $ U $ of $ Y $ and a uniformly analytic function $ \CG : U \rightarrow ( \CH_r )_{ nc } $ such that 
$$ \CG ( Y ) = H \quad \mbox{and} \quad R_{ \CT ( X ) } ( \CG ( X ) ) = \CF ( X ), \quad \mbox{for all} ~ X \in U. $$
\end{thm}

\begin{proof}
The proof follows from Lemma \ref{closed range implies bounded below} and Theorem \ref{closed range operators general version}.
\end{proof}

\section{Noncommutative reproducing kernel Hilbert spaces} \label{nc reproducing kernel Hilbert space}

In this section, we explore fundamental properties of noncommutative reproducing kernels that will be used in the subsequent sections. We begin by revisiting the concept of completely positive (cp) noncommutative (nc) reproducing kernel Hilbert spaces and their complete characterization, as presented in \cite{Ball-Marx-Vinnikov-nc-rkhs}. Following this, we examine the nc space constructed over such a Hilbert space, where we introduce a notion of \emph{generalized kernel elements}. Finally, we investigate the Hilbert $ C^* $ module structure on the matrix spaces over a cp nc reproducing kernel Hilbert space.  

\subsection{Global and noncommutative kernels}

Following \cite{Ball-Marx-Vinnikov-nc-rkhs}, we begin this subsection by reviewing the definitions and some elementary properties of nc kernels. We then examine the behavior of the derivatives of nc kernels in Proposition \ref{properties of kernel elements}, where it is shown that, unlike the classical case, the derivative of a nc kernel element remains a kernel element.

In what follows, we consider operator spaces $ \CV $ and $ \CW_0, \CW_1 $, and use the notation
$$  \CL ( \CW_{ 0, nc }, \CW_{ 1, nc } ) : = \coprod_{ m, n \geq 1 } \CL ( \CW_0^{ n \times m }, \CW_1^{ n \times m } ) . $$

\subsubsection{Global kernels} \label{global kernel}

A \textit{global kernel} on an uniformly open subset $ \Omega $ of $ \CV_{ nc } $ is a function $ K : \Omega \times \Omega \rightarrow \CL ( \CW_{ 0, nc }, \CW_{ 1, nc } ) $ satisfying following two conditions:
\begin{itemize}
\item $ K $ is graded: For all $ n, m \in \N $,
\be \label{graded kernel}
K ( Z, W ) \in \CL ( \CW_0^{ n \times m }, \CW_1^{ n \times m } ), \quad \text{whenever} \quad Z \in \Omega_n ,  W \in \Omega_m;
\ee
\item $ K $ respects direct sums: For all $ n_1, n_2, m_1, m_2 \in \N $ and $ Z_i \in \Omega_i $ with $ i = 1, 2 $, $ W_j \in \Omega_j $ with $ j = 1, 2 $, $ P_{ i j } \in \CV^{ n_i \times m_j } $ with $ i, j = 1, 2 $,
\be \label{kernel respects direct sum}
K \left( \left[ \begin{smallmatrix} Z_1 & 0 \\ 0 & Z_2 \end{smallmatrix} \right],  \left[ \begin{smallmatrix} W_1 & 0 \\ 0 & W_2 \end{smallmatrix} \right] \right) \left( \left[ \begin{smallmatrix} P_{ 1 1 } & P_{ 1 2 } \\ P_{ 2 1 } & P_{ 2 2 } \end{smallmatrix} \right] \right) = \left[ \begin{smallmatrix} K ( Z_1, W_1 ) ( P_{ 1 1 } ) & K ( Z_1, W_2 ) ( P_{ 1 2 } ) \\ K ( Z_2, W_1 ) ( P_{ 2 1 } ) & K ( Z_2, W_2 ) ( P_{ 2 2 } ) \end{smallmatrix} \right] .
\ee
\end{itemize}
Note that the second condition above can be extended to evaluations of $ K $ on block diagonal matrices with arbitrary many blocks on the diagonal.

\subsubsection{Noncommutative kernels} \label{nc kernel}

Suppose that $ K : \Omega \times \Omega \rightarrow \CL ( \CW_{ 0, nc }, \CW_{ 1, nc } ) $ is a global kernel. Then $ K $ is said to be \textit{nc affine kernel} if it respects the following \textit{intertwining conditions}: For any $ n, m, \widetilde{ n }, \widetilde{ m } \in \N $, $ Z \in \Omega_n $, $ \widetilde{ Z } \in \Omega_{ \widetilde{ n } } $, $ \SA \in \C^{ \widetilde{ n } \times n } $ such that $ \SA Z = \widetilde{ Z } \SA $, and $ W \in \Omega_m $, $ \widetilde{ W } \in \Omega_{ \widetilde{ m } } $, $ \SB \in \C^{ m \times \widetilde{ m } } $ such that $ W \SB = \SB \widetilde{ W } $,
\be \label{intertwining condition for kernel}
\SA K ( Z, W ) ( P ) \SB = K ( \widetilde{ Z }, \widetilde{ W } ) ( \SA P \SB ), \quad \text{for all} ~~ P \in \CW_0^{ n \times m }.
\ee
Or equivalently, a global kernel $ K $ as above is a nc affine kernel if it respects the similarities in the following sense: For any $ n, m \in \N $, $ Z, \widetilde{ Z } \in \Omega_n $, $ \SA \in \text{GL} ( n, \C ) $ such that $ \widetilde{ Z } = \SA Z \SA^{ - 1 } $, and $ W, \widetilde{ W } \in \Omega_m $, $ \SB \in \text{GL} ( m, \C ) $ such that $ \widetilde{ W } = \SB^{ -1 } W \SB $,
\be \label{kernel respects similarities}
K ( \widetilde{ Z }, \widetilde{ W } ) ( P ) = \SA K ( Z, W ) ( \SA^{ -1 } P \SB^{ - 1 } ) \SB, \quad \text{for all} ~~ P \in \CW_0^{ n \times m } .
\ee

Observe from the definition of nc functions, for any nc function $ f : \Omega \rightarrow \CW_{ nc } $, the right difference-differential operator $ \Delta_R f : \Omega \times \Omega \rightarrow \CL ( \CV_{ nc }, \CW_{ nc } ) $ turns out to be a nc affine kernel (cf. \cite[Proposition 2.15]{Verbovetskyi-Vinnikov}).

In the present article, our main interest, however, is in a variant of the affine kernels, which we call \textit{nc kernels}. A \textit{nc kernel} is a function $ K : \Omega \times \Omega \rightarrow \CL ( \CW_{ 0, nc }, \CW_{ 1, nc } ) $ satisfying 
\begin{itemize}
\item $ K $ is \textit{graded} in the sense of Equation \eqref{graded kernel} and 
\item $ K $ respects \textit{intertwinings}, namely, for any $ n, m, \widetilde{ n }, \widetilde{ m } \in \N $, $ Z \in \Omega_n $, $ \widetilde{ Z } \in \Omega_{ \widetilde{ n } } $, $ \SA \in \C^{ \widetilde{ n } \times n } $ such that $ \SA Z = \widetilde{ Z } \SA $, and $ W \in \Omega_m $, $ \widetilde{ W } \in \Omega_{ \widetilde{ m } } $, $ \SB \in \C^{ m \times \widetilde{ m } } $ such that $ \SB W = \widetilde{ W } \SB $,
\be \label{intertwining condition for kernel*}
\SA K ( Z, W ) ( P ) \SB^* = K ( \widetilde{ Z }, \widetilde{ W } ) ( \SA P \SB^* ), \quad \text{for all} ~~ P \in \CW_0^{ n \times m } .
\ee
\end{itemize}
Or equivalently, 
\begin{itemize}
\item $ K $ is \textit{graded} in the sense of Equation \eqref{graded kernel},
\item $ K $ respects \textit{direct sums} in the sense of Equation \eqref{kernel respects direct sum}, and
\item for any $ n, m \in \N $, $ Z, \widetilde{ Z } \in \Omega_n $, $ \SA \in \text{GL} ( n, \C ) $ such that $ \widetilde{ Z } = \SA Z \SA^{ - 1 } $, and $ W, \widetilde{ W } \in \Omega_m $, $ \SB \in \text{GL} ( m, \C ) $ such that $ \widetilde{ W } = \SB W \SB^{ -1 } $,
\be \label{kernel respects similarities*}
K ( \widetilde{ Z }, \widetilde{ W } ) ( P ) = \SA K ( Z, W ) ( \SA^{ -1 } P ( \SB^{ - 1 } )^* ) \SB^* , \quad \text{for all} ~~ P \in \CW_0^{ n \times m }  .
\ee
\end{itemize}
Notice that if $ K $ is nc affine kernel then the mapping $ ( X, Y ) \mapsto K ( X, Y^* ) $ is a nc kernel.

\subsubsection{Completely positive global or nc kernels}

First, note that we need an order structure on the operators spaces $ \CW_{ 0, nc } $ and $ \CW_{ 1, nc } $ to talk about positivity of any elements in these spaces. Recall that a \textit{concrete operator system} $ \CS $ is a self-adjoint linear subspace of $ \CL ( \CH ) $ for some Hilbert space $ \CH $ containing the identity operator $ \mathrm{Id}_{\CH} $. Any concrete operator system $ \CS $, as well as, for all $ n\in \N $ the space $ \CS^{ n \times n } $ of $ n \times n $ matrices over $ \CS $ are equipped with a special cone -- the cone of poistive semi-definite elements -- inducing an ordering on $ \CS^{ n \times n } $ satisfying certain compatibility conditions (cf. \cite[pp. 176]{Paulsen}). Any abstract operator space with such a distinguished cone satisfying those compatibility conditions is called an \textit{abstract operator system}. The Choi-Effros Theorem (cf. \cite[Theorem 13.1]{Paulsen}) shows that any such abstract operator system is completely order isomorphic to a concrete operator system.

In this subsection, we assume that $ \CW_{ 0, nc } $ and $ \CW_{ 1, nc } $ are abstract operator systems equipped with an associated conjugate-linear involution $ P \mapsto P^* $ and an well-defined notion of positivity $ P \succeq 0 $.  

A nc kernel $ K : \Omega \times \Omega \rightarrow \CL ( \CW_{ 0, nc }, \CW_{ 1, nc } ) $ with $ \Omega \subset \CV_{ nc } $ is a \textit{completely positive noncommutative kernel} (cp nc kernel) if for all $ n \in \N $ and $ Z \in \Omega_n $,
\be \label{cp nc kernel}
K ( Z, Z ) ( P ) \succeq 0 ~~ \text{in} ~~ \CW_1^{ n \times n }, \quad \text{for all} ~~ P \succeq 0 ~~ \text{in} ~~ \CW_0^{ n \times n } . 
\ee
Whenever $ K $ is a global kernel satisfying Equation \eqref{cp nc kernel}, we say $ K $ is a \textit{cp global kernel}. In case, both $ \CW_{ 0, nc } $ as well as $ \CW_{ 1, nc } $ are $ C^* $-algebras, we have the following equivalent condition for completely positivity of a kernel function.

\begin{prop}\cite[Proposition 2.2]{Ball-Marx-Vinnikov-nc-rkhs} \label{equivalent condition for cp nc kernel}
The global or nc kernel $ K : \Omega \times \Omega \rightarrow \CL ( \CW_{ 0, nc }, \CW_{ 1, nc } ) $ is cp if and only if for any $ Z^j \in \Omega_{ n_j } $, $ P_j \in \CW_0^{ N \times n_j } $, $ Q_j \in \CW_1^{ n _j } $ for $ n_j \in \N $, $ 1 \leq j \leq N $ and $ N \in \N $, it holds that
$$ \sum_{ i, j = 1 }^N Q_i^* K ( Z^i, Z^j ) ( P_i^* P_j ) Q_j \succeq 0 . $$
\end{prop}

\subsubsection{Characterization of noncommutative kernels}

We now present the result cahracterizing cp nc kernels assuming that $ \CW_0 $ is a $ C^* $-algebra and $ \CW_{ 1, nc } $ is the concrete operator system consisting of all bounded linear operators on a Hilbert space.

\begin{thm}\cite[Theorem 3.1]{Ball-Marx-Vinnikov-nc-rkhs} \label{nc rkhs}
Suppose that $ \CV $ is an operator space, $ \Omega \subset \CV_{ nc } $ is a nc set, $ \CY $ is a Hilbert space, $ \CA $ is a $ C^* $-algebra and $ K : \Omega \times \Omega \rightarrow \CL ( \CA_{ nc }, \CL ( \CY )_{ nc } ) $ is a given function. Then the following are equivalent.

\begin{itemize}
\item[(1)] $ K $ is a cp nc kernel.

\item[(2)] There is a Hilbert space $ \CH ( K ) $ whose elements are nc functions $ f: \Omega \rightarrow \CL ( \CA, \CY )_{ nc } $ such that
\begin{itemize}
\item[(a)] For each $ W \in \Omega_m $, $ v \in \CA^{ 1 \times m } $, and $ y \in \CY^{ m \times 1 } $, the function 
$$ K_{ W, v, y } : \Omega_n \rightarrow \CL ( \CA, \CY )^{ n \times n } \cong \CL ( \CA^{ n \times 1 }, \CY^{ n \times 1 } ) ~~~ \text{defined by} $$
\be \label{kernel element} K_{ W, v, y } ( Z ) ( u ) : = K ( Z, W ) ( u v ) y, \ee
for all $ Z \in \Omega_n $, $ u \in \CA^{ n \times 1 } $ belongs to $ \CH ( K ) $. 
\item[(b)] The kernel elements $ K_{ W, v, y } $ as defined in Equation \eqref{kernel element} have the reproducing property: For $ f \in \CH ( K ) $, $ W \in \Omega_m $, $ v \in \CA^{ 1 \times m } $ and $ y \in \CY^{ m \times 1 } $,
\be \label{reproducing property}
\left\langle f, K_{ W, v, y } \right\rangle_{ \CH ( K ) } = \left\langle f ( W ) ( v^* ), y \right\rangle_{  \CY^{ m \times 1 } }.
\ee
\item[(c)] $ \CH ( K ) $ is equipped with a unital $ * $-representation $ \rho $ mapping $ \CA $ to $ \CL ( \CH ( K ) ) $ such that 
\be \label{* representation}
( \rho ( a ) f ) ( W ) ( v^* ) = f ( W ) ( v^* a )
\ee
for $ a \in \CA $, $ W \in \Omega_m $, $ v \in \CA^{ 1 \times m } $, with the action on the kernel elements given by
\be \label{action of A on kernel elements}
\rho ( a ) : K_{ W, v, y } \mapsto K_{ W, a v, y }.
\ee
\end{itemize}
\item[(3)] $ K $ has a Kolmogorov decomposition: there is a Hilbert space $ \CX $ equipped with a unital $*$-representation $ \rho : \CA \rightarrow \CL ( \CX ) $ together with a nc function $ H : \Omega \rightarrow \CL ( \CX, \CY )_{ nc } $ such that
\be \label{Kolmogorov decomposition}
K ( Z, W ) ( P ) = H ( Z ) ( \mathrm{Id}_{ \C^{ n \times m } } \otimes \rho ) ( P ) H ( W )^*
\ee 
for all $ Z \in \Omega_n $, $ W \in \Omega_m $, $ P \in \CA^{ n \times m } $.
\end{itemize}
\end{thm}

Furthermore, if $ \CH ( K ) $ consists of uniformly nc analytic functions on $ \Omega $ taking values in $ \CL ( \CA, \CY )_{ nc } $, then $ K $ turns out to be uniformly \textit{sesqui-analytic} nc function of order $ 1 $ as demonstrated in the following proposition. So we first provide the definition of \textit{uniformly anti-analytic nc function}.

\begin{defn} \label{definition of anti-analyticity}
Let $ \Omega \subset \CV_{ nc }$ be a nc domain which is both left and right admissible and $ \CW $ be a complex vector space. Then a nc function $ f : \Omega \rightarrow \CW_{ nc } $ is said to be uniformly anti-analytic nc function if the function $ f_* : \Omega^* \rightarrow \CW_{ nc } $ defined by 
$$ f_* ( W ) : = f ( W^* ) $$ is uniformly analytic nc function on $ \Omega $ where $ \Omega^* : = \{ W : W^* \in \Omega \} $. In this case, we define 
$$ _*\Delta_R f ( W^0, W^1 ) ( X ) := \Delta_R f_* ( ( W^0 )^*, ( W^1 )^* ) ( X^* ), $$
for $ W^0 \in \Omega_{n_0}, W^1 \in \Omega_{ n_1 }, X \in \CV^{ n_0 \times n_1 } $.
\end{defn}

\begin{prop} \label{properties of kernel elements}
Let $ \Omega \subset \CV_{ nc } $ and $ \CH ( K ) $ be a nc reproducing kernel Hilbert space with the cp nc reproducing kernel $ K : \Omega \times \Omega \rightarrow \CL ( \CA_{ nc }, \CL ( \CY )_{ nc } ) $. Assume that $ \CH ( K ) $ consists of $ \CL ( \CA, \CY )_{ nc } $ valued uniformly analytic nc functions on $ \Omega $, and that $ \Omega $ is both left and right admissible. Denote $ \Omega^* = \{ W : W^* \in \Omega \} $. Then 
\begin{itemize}

\item[(i)] $ K $ is a uniformly sesqui-analytic nc function of order $ 1 $ on $ \Omega \times \Omega $. In other words, the function $ K_* : \Omega \times \Omega^* \rightarrow \CL ( \CA_{ nc }, \CL ( \CY )_{ nc } ) $ defined by 
$$ K_* ( Z, W ) : = K ( Z, W^* ),~~~ Z \in \Omega,~~ W \in \Omega^*, $$
is a uniformly analytic nc function of order $ 1 $.

\item[(ii)] For any $ W^0 \in \Omega_{ n_0 } , \hdots, W^{ \ell } \in \Omega_{ n_\ell } $, $ X^1 \in \CV^{ n_0 \times n_1 }, \hdots, X^{ \ell } \in \CV^{ n_{ \ell-1 } \times n_{ \ell } } $, $ v \in \CA^{ 1 \times n_0 } $ and $ y \in \CY^{ n_{ \ell } \times 1 } $, 
$$ _{ 1 ^* } \Delta_R^{ \ell } K ( \cdot, W^0, \hdots, W^{ \ell } ) ( ( \cdot ) v, X^1, \hdots, X^{ \ell } ) y \in \CH ( K ) $$
whenever 
$$ 
\begin{bmatrix}
W^0 & X^1 & 0 & \cdots & 0 \\
0      & W^1 & X^2 & \cdots & 0 \\
\vdots & \vdots & \ddots & \ddots & 0 \\
0      &       0     & \cdots   & W^{ \ell -1 } & X^{ \ell } \\
0      &       0     & \cdots   &  0            &   W^{ \ell }
\end{bmatrix} \in \Omega_{ n_0 + \cdots + n_{ \ell } }.
$$
Furthermore, 
$$ _{ 1 ^* } \Delta_R^{ \ell } K ( \cdot, W^0, \hdots, W^{ \ell } ) ( ( \cdot ) v, X^1, \hdots, X^{ \ell } ) y = K_{ \left[ \begin{smallmatrix}
W^0 & 0 & 0 & \cdots & 0 \\
X^1     & W^1 & 0 & \cdots & 0 \\
0 & X^2 & W^2 &  \cdots & 0 \\
\vdots & \vdots & \ddots & \ddots & \vdots \\
0      &       0     & \hdots   & X^{ \ell } & W^{ \ell } 
\end{smallmatrix} \right], \left[ \begin{smallmatrix} v & 0 & \cdots & 0  \end{smallmatrix}\right],   \left[ \begin{smallmatrix} 0 \\ \vdots \\ 0 \\ y \end{smallmatrix}\right] }. $$

\item[(iii)] For every $ f \in \CH ( K ) $, $ W^0 \in \Omega_{ n_0 }, \hdots, W^{ \ell } \in \Omega_{ n_{ \ell } } $, $ X^1 \in \CV^{ n_0 \times n_1 }, \hdots , X^{ \ell } \in \CV^{ n_{ \ell-1 } \times n_{ \ell } } $, $ v \in \CA^{ 1 \times n_0 } $ and $ y \in \CY^{ n_{ \ell } \times 1 } $,
\begin{align} \label{reproducing property for derivative}
& \left\langle f, _{ 1 ^* } \Delta_R^{ \ell } K ( \cdot, W^0, \hdots, W^{ \ell } ) ( ( \cdot ) v, X^1, \hdots, X^{ \ell } ) y \right\rangle_{ \CH ( K ) }\\
\nonumber & = \left\langle \Delta_R^{ \ell } f ( W^{ \ell }, W^1, \hdots, W^{ \ell-1 }, W^0 ) ( X^{ \ell }, X^2, \hdots, X^{ \ell-1 }, X^1 ) v^*, y \right\rangle_{ \CY^{ n_{ \ell } } }. 
\end{align}
\end{itemize}
\end{prop}

\begin{proof}
(i) We begin by pointing out that since $ \CH ( K ) $ consists of $ \CL ( \CA, \CY )_{ nc } $ valued uniformly analytic nc functions on $ \Omega $, $ K $ is uniformly $ P $-locally bounded ( cf. \cite[Theorem 3.17]{Ball-Marx-Vinnikov-nc-rkhs}) and hence the proof follows from Theorem 7.53 in \cite[Chapter 7, pp 120]{Verbovetskyi-Vinnikov}.

(ii) Since $ K_* $ is a uniformly analytic nc function of order $ 1 $ on $ \Omega \times \Omega^* $, for $ Z \in \Omega_n $, $ W^0 \in \Omega_{ n_0 } $, $ W^1 \in \Omega_{ n_1 } $, $ X^1 \in \CV^{ n_0 \times n_1 } $, $ u \in \CA^{ n \times 1 } $, $ v \in \CA^{ 1 \times n_0 } $ and $ y \in \CY^{ n_1 \times 1 } $, recall (see \cite[Proposition 3.17, Equation (3.21), pp. 43]{Verbovetskyi-Vinnikov}) that
$$ \Delta_R K_* ( Z, ( W^0 )^*, ( W^1 )^* ) ( uv, ( X^1 )^* ) y  =  K_* \left( Z, \left[ \begin{smallmatrix} ( W^0 )^* & ( X^1 )^* \\ 0 & ( W^1 )^* \end{smallmatrix} \right] \right) \left( [ \begin{smallmatrix} uv & 0 \end{smallmatrix} ] \left[ \begin{smallmatrix} 0 \\ \mathrm{Id}_{ \CY^{ n_1 \times n_1 } } \end{smallmatrix} \right] \right) y . $$
Consequently, we have that
\Bea
\Delta_R K_* ( Z, ( W^0 )^*, ( W^1 )^* ) ( uv, ( X^1 )^* ) y & = & K_* \left( Z, \left[ \begin{smallmatrix} ( W^0 )^* & ( X^1 )^* \\ 0 & ( W^1 )^* \end{smallmatrix} \right] \right) \left( [ \begin{smallmatrix} uv & 0 \end{smallmatrix} ] \left[ \begin{smallmatrix} 0 \\ \mathrm{Id}_{ \CY^{ n_1 \times n_1 } } \end{smallmatrix} \right] \right) y\\
& = & K_* \left( Z, \left[ \begin{smallmatrix} ( W^0 )^* & ( X^1 )^* \\ 0 & ( W^1 )^* \end{smallmatrix} \right] \right) \left( [ \begin{smallmatrix} uv & 0 \end{smallmatrix} ] \left[ \begin{smallmatrix} 0 \\ y \end{smallmatrix} \right] \right) \\
& = & K \left( Z, \left[ \begin{smallmatrix} W^0 & X^1 \\ 0 & W^1 \end{smallmatrix} \right] \right) \left( [ \begin{smallmatrix} uv & 0 \end{smallmatrix} ] \left[ \begin{smallmatrix} 0 \\ y \end{smallmatrix} \right] \right) \\
& = & K_{ \left[ \begin{smallmatrix} W^0 & X^1 \\ 0 & W^1 \end{smallmatrix} \right],  \left[ \begin{smallmatrix} v & 0 \end{smallmatrix} \right],  \left[ \begin{smallmatrix} 0 \\ y \end{smallmatrix} \right]} ( Z ) ( u )
\Eea
completing the proof of part (ii) for $ \ell = 1 $ and the same proof with matrices of larger sizes yields the general case.

(iii) As in the part (ii) we demonstrate the proof of part (iii) only for $ \ell = 1 $ and the general case follows directly. Recall from Theorem \ref{nc rkhs} that, for any $ f \in \CH ( K ) $, $ W \in \Omega_m $, $ v \in \CA^{ 1 \times m } $, $ y \in \CY^{ m \times 1 } $ and $ m \in \N $,
$$ \langle f, K_{ W, v, y } \rangle_{ \CH ( K ) } = \langle f ( W ) ( v^* ), y \rangle_{ \CY^m } . $$
Consequently, this reproducing property together with part (ii) yield that
\Bea
\left\langle f, _{ 1^* } \Delta_R^l K ( \cdot, W^0, W^1 ) ( ( \cdot ) v, X^1 ) y \right\rangle_{ \CH ( K ) } & = & \left\langle f, K_{ \left[ \begin{smallmatrix} W^0 & 0 \\ X^1 & W^1 \end{smallmatrix} \right], \left[ \begin{smallmatrix} v & 0 \end{smallmatrix} \right], \left[ \begin{smallmatrix} 0 \\ y \end{smallmatrix} \right] } \right\rangle_{ \CH ( K ) } \\
& = & \left\langle f \left( \left[ \begin{smallmatrix} W^0 & 0 \\ X^1 & W^1 \end{smallmatrix} \right] \right) \left[ \begin{smallmatrix} v^* \\ 0 \end{smallmatrix} \right], \left[ \begin{smallmatrix} 0 \\ y \end{smallmatrix} \right] \right\rangle_{ \CH ( K ) } \\
& = & \left\langle \Delta_L f ( W^0, W^1 ) ( X^1 ) v^*, y \right\rangle_{ \CY^{ n_1 } } \\
& = & \left\langle \Delta_R f ( W^1, W^0 ) ( X^1 ) v^*, y \right\rangle_{ \CY^{ n_1 } }
\Eea
where the last equality follows from Equation \eqref{left vs right dd operators}.
\end{proof}

Let $ \CY_1 $ and $ \CY_2 $ be two Hilbert spaces, $ K_i : \Omega \times \Omega \to \CL ( \CA_{ nc }, ( \CY_i)_{ nc } ) $, $ i = 1, 2 $, be two cp nc kernels and $ S : \Omega \to \CL ( \CY_1, \CY_2 ) $ be a global function. Then following \cite[Section 4]{Ball-Marx-Vinnikov-nc-rkhs}, we define $ S $ to be a \emph{multiplier} if the operator $ M_S : \CH ( K_1 ) \to \CH ( K_2 ) $ defined by 
\begin{equation} \label{multiplier defn} 
( M_S ( f ) ) ( W ) = S( W ) f ( W ) 
\end{equation}
is bounded. Using closed graph theorem it can be seen that $ M_S $ is bounded whenever it is well defined. Moreover, the boundedness of $ M_S $ can be expressed as the completely positivity condition of certain nc kernels as demonstrated in the following result.

\begin{thm} \label{theorem on mult operator}
    Let $ K_i : \Omega \times \Omega \to \CL ( \CA_{ nc }, ( \CY_i)_{ nc } ) $, $ i = 1, 2 $, be two cp nc kernels and $ S : \Omega \to \CL ( \CY_1, \CY_2 ) $ be a global function. Then the following are equivalent.
    \begin{enumerate}
        \item The operator $ M_S $ is bounded.

        \item There exists $ C \geq 0 $ such that the kernel function $ K_S : \Omega \times \Omega \to \CL ( \CA_{ nc }, ( \CY_2 )_{ nc } ) $ defined by 
        $$ K_S ( Z, W ) ( P ) : = C K_1 ( Z, W ) ( P ) - S ( Z ) K_2 ( Z, W ) ( P ) S ( W )^* $$
        is a cp nc kernel.
    \end{enumerate}
    Moreover, if one these two happens, the norm $ \norm{ M_S } $ turns out to be the infimum of all positive constants $ C $ satisfying the condition (2).
\end{thm}

\begin{proof}
    Since the proof is verbatim to the proof of Theorem 4.1 in \cite{Ball-Marx-Vinnikov-nc-rkhs} with an exception that $ \norm{ M_S } $ is assumed to be at most $ 1 $ there, we omit it here.
\end{proof}

If $  M_S : \CH ( K_1 ) \to \CH ( K_2 ) $ as in the theorem above is a bounded linear operator, then the action of its adjoint $ ( M_S )^* : \CH ( K_2 ) \to \CH ( K_1 ) $ on any kernel element $ ( K_1 )_{ W, v, y } $, $ W \in \Omega_m, v \in \CA^{ 1 \times m } $, $ y \in \CY^{ m \times 1 } $, turns out to be
\begin{equation} \label{adjoint of mult op on ker element}
    ( M_S )^* ( ( K_1 )_{ W, v, y } ) = ( K_2 )_{ W, v, S ( W )^* y } . 
\end{equation}

Let $ K : \Omega \times \Omega \rightarrow \CL ( \CA_{ nc }, \CL ( \CY )_{ nc } ) $ be a cp nc kernel and $ \boldsymbol{ M } = ( M_{ Z_1 }, \hdots, M_{ Z_d } ) $ be the $ d $-tuple of multiplication operators by the nc coordinate functions on the space of all nc functions on $ \Omega $ taking values in the nc space over $ \CL ( \CA, \CY ) $ defined by $ f \mapsto M_{ Z_i } ( f ) $ and $ M_{ Z_i } ( f ) ( W ) = W_i f ( W ) $, $ W \in \Omega $. Then as a corollary of the preceding theorem, we have that the operator $ M_{ Z_j } $ defines a bounded linear operator on $ \CH ( K ) $ if and only if there exists $ C \geq 0 $ such that 
   $$ K^{ ( j ) } ( Z, W ) ( P ) : = C K_1 ( Z, W ) ( P ) - Z_j K_2 ( Z, W ) ( P ) W_j^* $$
   defines a cp nc kernel on $ \Omega $. Furthermore, we have the following corollary.

\begin{cor} \label{cond for membership in rkhs}
    A nc function $ f : \Omega \to \CL ( \CA, \CY )_{ nc } $ is in $ \CH ( K ) $ if and only if there exists $ C \geq 0 $ such that 
   $$ K_f ( Z, W ) ( P ) : = C K ( Z, W ) ( P ) - f ( Z ) P f ( W )^* $$
   defines a cp nc kernel on $ \Omega $. 
\end{cor}

\subsection{Noncommutative space over \texorpdfstring{$ \CH ( K ) $}{CH ( K ) }} \label{coordinate free presentation}

We now study the nc space over $ \CH ( K ) $ when $ \CA $ is a unital $ C^* $ algebra with the unit element $ \se $ and $ \CY = \C^r $, that is, the nc cp kernel $ K $ is now a function $ K : \Omega \times \Omega \to \CL ( \CA_{ nc }, \C^{ r \times r }_{ nc } ) $ and in consequence, the Hilbert space $ \CH ( K ) $ consists of nc functions on $ \Omega $ taking values in $ \CL ( \CA, \C^r )_{ nc } $. Recall that the nc space $ \CH ( K )_{ nc } $ over $ \CH ( K ) $ is, by definition (see Equation \eqref{nc space}), the graded space 
$$ \CH ( K )_{ nc } : = \coprod_{ m \geq 1 } \CH ( K )^{ m \times m } . $$
For each $ m \in \N $, the space $ \CH ( K )^{ m \times m } $ of $ m \times m $ matrices over $ \CH ( K ) $ -- which we denote by $ \CH_m $ -- carries a canonical bi-module structure over $ \C^{ m \times m } $, namely, for $ \SA, \SB, \SS \in \C^{ m \times m } $ and $ f = ( f_1, \hdots, f_r ) \in \CH ( K ) $, 
$$ \SA ( f \otimes \SS ) \SB = f \otimes ( \SA \SS \SB ) . $$
Moreover, the evaluation of an element $ f \otimes \SS = ( f_1 \otimes \SS, \hdots, f_r \otimes \SS ) \in \CH_m $ at any point $ W \in \Omega_m $ can be thought of as a linear mapping on $ \CA^{ m \times m } $ taking values in $ ( \C^{ m \times m } )^r $ as follows
$$ \left[ f ( W ) \otimes \SS \right] ( \TA ) = \left[ \begin{smallmatrix} f_1 ( W ) \TA \SS \\ \vdots \\ f_r ( W ) \TA \SS \end{smallmatrix} \right], ~~~  \TA = ( \! ( a_{ i j } ) \! )_{ i, j = 1 }^m \in \CA^{ m \times m } , $$
where the linear map $ f_j ( W ) : \CA^{ m \times 1 } \to \C^{ m \times 1 } $ acts column wise on $ \TA $. Finally, recalling the action of $ \CA $ on the kernel elements in \eqref{action of A on kernel elements}, we point out that there is a canonical action of $ \CA^{ m \times m } $ on $ \CH_m $ for each $ m \in \N $ defined as follows:
\be \label{action of amplified A}
\TA \cdot ( \! ( h_{ i j } ) \! )_{ i, j = 1 }^m : = \bigg( \!\!\! \bigg( \sum_{ k = 1 }^m \rho ( a_{ i k } ) h_{ k j } \bigg) \!\!\! \bigg)_{ i, j = 1 }^m, ~~~ \TA = ( \! ( a_{ i j } ) \! )_{ i, j = 1 }^m \in \CA^{ m \times m }, ~~~ ( \! ( h_{ i j } ) \! )_{ i, j = 1 }^m \in \CH_m.
\ee
In other words, $ \CH_m $ turns out to be a left module over $ \CA^{ m \times m } $ for every $ m \in \N $.
\begin{comment}
Before proceeding further, let us point out that Corollary \ref{cond for membership in rkhs} has a canonical matricial counterpart. For $ m \in \N $ and $ \CH ( K ) $ as above, a function $ F \in \CH_m $ if and only if there exists $ C \geq 0 $ such that 
 \begin{equation} \label{cond for membership in rkhs-matricial}
        C K \bigg( \bigoplus_{ i = 1 }^m Z, \bigoplus_{ i = 1 }^m Z \bigg) - \Pi F ( Z ) \otimes F ( Z )^* \Pi^* \succeq 0 , ~~~ Z \in \Omega_n, n \in \N ,
    \end{equation} 
    where $ (\!( \SA_{ i j })\!)_{ i, j = 1 }^s \otimes (\!( \SB_{ i j })\!)_{ i, j = 1 }^s : = (\!( \SA_{ i j } \otimes \SB_{ \ell k } )\!)^s_{ i, j , \ell, k = 1 } $, $ \SA_{ i j }, \SB_{ \ell k } \in \C^{ n \times n } $ for some $ n \in \N $, and $ \Pi $ is the projection onto the subspace of $ \C^{ m^2 \times m^2 } \otimes \C^{ n^2 \times n^2 } $ spanned by $ \{ ( \varepsilon_j \otimes \varepsilon_j ) \otimes \mathrm{ Id }_{ n^2 \times n^2 } : 1 \leq j \leq m \} $. 
\end{comment}
\subsubsection{Generalized kernel elements} \label{Generalized kernels} For $ m \in \N $, $ W \in \Omega_m $ and any vector $ \sigma \in \C^{ r \times 1 } $, define the uniformly analytic nc functions $ K ( \cdot, W ) \sigma : \Omega \rightarrow \CL ( \CA^m, ( \C^r )^m )_{ nc } $ by 
\be \label{coordinate free kernel}
K ( \cdot, W ) \sigma : = \sum_{ i, j =1 }^m  K_{ W, \se \otimes \varepsilon_i, \varepsilon_j^* \otimes \sigma } \otimes \varepsilon_i^* \varepsilon_j
\ee
where $ \{ \varepsilon_1, \hdots, \varepsilon_m \} $ is the standard ordered basis for $ \C^{ 1 \times m } $. Note that $ K ( \cdot, W ) \sigma $ is an $ m \times m $ matrix over $ \CH ( K ) $ whose $ i j $-th entry is the nc function $ K_{ W, \se \otimes \varepsilon_i, \varepsilon_j^* \otimes \sigma } $ on $ \Omega $ taking values in $ \CL ( \CA, \C^r )_{ nc } $. We call these vectors $ \{ K ( \cdot, W ) \sigma \in \CH_m : \sigma \in \C^{ r \times 1 } \} $ for $ W \in \Omega_m $ as the \textbf{generalized kernel elements}. For every $ \sigma \in \C^{ r \times 1 } $, $ W \in \Omega_m $, $ \SA \in \C^{ s \times m } $ and $ \SB \in \C^{ m \times t } $, observe that $ \SA [ K ( \cdot, W ) \sigma ] \SB \in \CH ( K )^{ s \times t } $ where by $ \SA [ K ( \cdot, W ) \sigma ] \SB $ we mean that
$$ \SA [ K ( \cdot, W ) \sigma ] \SB = \sum_{ i, j =1 }^m  K_{ W,  \se \otimes \varepsilon_i, \varepsilon_j^* \otimes \sigma } \otimes \SA ( \varepsilon_i^* \varepsilon_j ) \SB. $$
On the other hand, the expression for $ \SA [ K ( \cdot, W ) \sigma ] \SB $ has following equivalent presentation
$$ \SA [ K ( \cdot, W ) \sigma ] \SB = \sum_{ i =1 }^s \sum_{ j = 1 }^t  K_{ W, \se \otimes e_i \SA, \SB \epsilon_j^* \otimes \sigma } \otimes e_i^* \epsilon_j, $$
which in turn implies that 
\be \label{interpretation of kernel elements}
K_{ W,  \se \otimes e_{ \ell } \SA, \SB \epsilon_k^* \otimes \sigma } = e_{ \ell } ( \SA [ K ( \cdot, W ) \sigma ] \SB ) \epsilon_k^*, \quad 1 \leq \ell \leq s,~~~ 1 \leq k \leq t, ~~~ \sigma \in \C^r .
\ee
Here, $ \{ e_1, \hdots e_s \} $ and $ \{ \epsilon_1, \hdots, \epsilon_t \} $ are the standard ordered bases for $ \C^{ 1 \times s } $ and $ \C^{ 1 \times t } $, respectively. The left $ \CA^{ m \times m } $ module action on the generalized kernel elements turns out to be 
\be \label{C^* action on generalized kernel elements}
\TA \cdot K( \cdot, W ) \sigma =  \sum_{ i, j =1 }^m  K_{ W,  ( \se \otimes \varepsilon_i ) \TA, \varepsilon_j^* \otimes \sigma } \otimes \varepsilon_i^* \varepsilon_j, ~~~ \TA \in \CA^{ m \times m }.
\ee

Likewise the kernel elements (see \eqref{kernel element}), the generalized kernel elements give rise to nc functions on $ \Omega $ as shown in the following lemma.

\begin{lem}
Let $ \Omega^* = \{ W : W^* \in \Omega \} $ and $ \sigma \in \C^r $ be an arbitrary but fixed vector. Then the function $ \gamma : \Omega^* \rightarrow \CH ( K )_{ nc } $ defined by $ \gamma ( W ) : = K ( \cdot, W^* ) \sigma $ is uniformly analytic nc function.
\end{lem}

\begin{proof}
We begin by pointing out from the definition of $ K ( \cdot, W^* ) \sigma $ in \eqref{coordinate free kernel} that $ \gamma $ is a graded function. Next, we check that $ \gamma $ satisfies the intertwining. Let $ W \in \Omega^*_n, \widetilde{ W } \in \Omega^*_m $ and $ \SS \in \C^{ m \times n } $ so that $ \SS W = \widetilde{ W } \SS $. Also, denote the standard ordered bases for $ \C^{ 1 \times n } $ and $ \C^{ 1 \times m } $ as $ \{ \varepsilon_1, \hdots, \varepsilon_n \} $ and $ \{ \epsilon_1, \hdots, \epsilon_m \} $, respectively. Then note that
%\Bea 
%\gamma ( W ) \SS & = & [ K ( \cdot, W^* ) \sigma ] \SS \\
%& = & \sum_{ i, j = 1 }^n K_{ W^*, \varepsilon_i, \varepsilon_j } \otimes \varepsilon_i \varepsilon_j^* \SS \\
%& = &  \sum_{ i = 1 }^n \sum_{ j = 1 }^m K_{ W^*, \varepsilon_i, \SS \epsilon_j } \otimes \varepsilon_i \epsilon_j^* \\
%& = & \sum_{ i = 1 }^n \sum_{ j = 1 }^m K_{ \widetilde{ W }^*, \SS^* \varepsilon_i, \epsilon_j } \otimes \varepsilon_i \epsilon_j^* \\
%& = & \SS \sum_{ i, j = 1 }^m K_{ \widetilde{ W }^*, \epsilon_i, \epsilon_j } \otimes \epsilon_i \epsilon_j^* \\
%& = & \SS \gamma ( \widetilde{ W } )
%\Eea
$$ \sum_{ i, j = 1 }^n K_{ W^*, \se \otimes \varepsilon_i, \epsilon_j^* } \otimes \SS \varepsilon_i^* \epsilon_j  =   \sum_{ i = 1 }^n \sum_{ j = 1 }^m K_{ W^*, \se \otimes \varepsilon_i \SS, \epsilon_j^* } \otimes \varepsilon_i^* \epsilon_j = \sum_{ i = 1 }^n \sum_{ j = 1 }^m K_{ \widetilde{ W }^*, \se \otimes \varepsilon_i, \SS \epsilon_j^* } \otimes \varepsilon_i^* \epsilon_j $$
where the second equality holds since
\begin{align*} 
&K_{ W^*, \se \otimes \varepsilon_i \SS, \epsilon_j^* } ( Z ) u  =  K ( Z, W^* ) ( u ( \se \otimes \varepsilon_i \SS ) ) \epsilon_j^* \\
&=  K ( Z, \widetilde{ W }^* ) ( u ( \se \otimes \varepsilon_i ) ) \SS \epsilon_j^*  =  K_{ \widetilde{ W }^*, \se \otimes \varepsilon_i, \SS \epsilon_j^* } ( Z ) u, 
\end{align*}
for any $ Z \in \Omega_s $, $ u \in \CA^{ s \times 1 } $ and $ s \in \N $. Thus it verifies that
$$ \SS \gamma ( W ) =  \gamma ( \widetilde{ W } ) \SS. $$
Finally, uniformly analyticity of $ \gamma $ as a nc function follows from part (i) in Proposition \ref{properties of kernel elements}.
\end{proof}

We now discuss the reproducing property of these \textbf{generalized kernel elements} for which we  equip the space $ \CH_m $ of $ m \times m $ matrices over $ \CH ( K ) $, $ m \in \N $, with an hermitian product making it a Hilbert space as follows. For $ F =(\!( f_{ i j } )\!)_{ i, j = 1 }^m, \, G =(\!( g_{ i j } )\!)_{ i, j = 1 }^m \in \CH_m $, define
\be \label{inner product of matrix spaces}
\left\langle F, G \right\rangle_{HS} : = \mbox{tr} \left( F \ast G ^{ \top } \right). 
\ee
Here, $ F \ast G $ is an element in $ \CH_m $ whose $ i j $-th entry is defined as follows:
\be \label{product structure on matrix spaces}
( F \ast G )_{ i j } : = \sum_{ \ell =1 }^m \left\langle f_{ i \ell }, g_{ \ell j } \right\rangle.
\ee
Equivalently, for any $ f \otimes \SA, g \otimes \SB \in \CH_m $, the formula above yields that
$$ \left\langle f \otimes \SA, g \otimes \SB \right\rangle_{HS} = \langle f, g \rangle_{ \CH ( K ) } \langle \SA, \SB \rangle_{ \C^{ m \times m } }, $$
where $ \langle \cdot, \cdot \rangle_{ \C^{ m \times m } } $ is the usual Hilbert-Schmidt inner product on $ \C^{ m \times m } $. 

\begin{prop} 
Let $ m \in \N $, $ W \in \Omega_m $, $ \xi = ( \xi_1, \hdots, \xi_r ) \in \C^{ r \times 1 } $ and $ F = ( F_1, \hdots, F_r ) $ in $ \CH_m $. 
\begin{itemize}
\item[(i)] For any $ \SA, \SB \in \C^{ m \times m } $,
\be \label{generalized reproducing property 0}
\left\langle F, \SA [ K ( \cdot, W ) \xi ] \SB \right\rangle_{ HS }  = \sum_{ t = 1 }^r \overline{ \xi }_t \left\langle  F_t ( W ) ( \SA^* ), \SB \right\rangle_{ \C^{ m \times m } }.
\ee
In particular, for the standard ordered basis $ \{ \sigma_1, \hdots, \sigma_r \} $ of $ \C^r $, we have that
\be \label{generalized reproducing property}
\left\langle F, \SA [ K ( \cdot, W ) \sigma_t ] \SB \right\rangle_{ HS }  = \left\langle  F_t ( W ) ( \SA^* ), \SB \right\rangle_{ \C^{ m \times m } }, ~~~ 1 \leq t \leq r.
\ee
\item[(ii)] For $ \TA \in \CA^{ m \times m } $, $ \SX \in \C^{ m \times m } $,
\be \label{C^* action via reproducing property}
\left\langle F, \TA [ K ( \cdot, W ) \xi ] \SX \right\rangle_{ HS }  = \sum_{ t = 1 }^r \overline{ \xi }_t \left\langle  F_t ( W ) ( \TA^* ), \SX \right\rangle_{ \C^{ m \times m } }.
\ee
\end{itemize}
\end{prop}

\begin{proof}
(i) It is enough to verify the identity given in Equation \eqref{generalized reproducing property} for any elementary tensors $ f \otimes \SS \in \CH_m $. Note that 
\Bea
\left\langle f \otimes \SS,  \SA [ K ( \cdot, W ) \xi ] \SB \right\rangle_{ HS } & = & \left\langle f \otimes \SS, \sum_{ i, j = 1 }^m K_{ W, \varepsilon_i, \varepsilon_j^* \otimes \xi } \otimes \SA ( \varepsilon_i^* \varepsilon_j ) \SB \right\rangle_{ HS }\\
& = & \sum_{ i, j =1 }^m \left\langle f \otimes \SS, K_{ W, \varepsilon_i, \varepsilon_j^* \otimes \xi } \otimes \SA ( \varepsilon_i^* \varepsilon_j ) \SB \right\rangle_{ HS }\\
& = & \sum_{ i, j =1 }^m \left[ \left\langle f, K_{ W, \varepsilon_i, \varepsilon_j^* \otimes \xi } \right\rangle_{ \CH ( K ) } \left\langle \SS, \SA ( \varepsilon_i^* \varepsilon_j ) \SB \right\rangle_{ \C^{ m \times m } } \right]\\
& = & \sum_{ i, j =1 }^m \left[ \left\langle f ( W ) ( \varepsilon_i^* ), \varepsilon_j^* \otimes \xi \right\rangle_{ ( \C^r )^m } \left\langle \SA^* \SS \SB^*,  \varepsilon_i^* \varepsilon_j \right\rangle_{ \C^{ m \times m } } \right]\\
& = & \sum_{ i, j =1 }^m \left[ \sum_{ t = 1 }^r \overline{ \xi }_t \left\langle f_t ( W ) ( \varepsilon_i^* ), \varepsilon_j^*  \right\rangle_{ \C^m } \left\langle \SA^* \SS \SB^*,  \varepsilon_i^* \varepsilon_j \right\rangle_{ \C^{ m \times m } } \right]\\
& = & \sum_{ t = 1 }^r \overline{ \xi }_t \operatorname{trace} ( f_t ( W ) \SA^* \SS \SB^* )\\
& = & \sum_{ t = 1 }^r \overline{ \xi }_t \left\langle ( f_t ( W ) \otimes \SS ) ( \SA^* ), \SB \right\rangle_{ \C^{ m \times m } }
\Eea
completing the proof. The second statement follows directly from this computation with $ \xi = \sigma_t $, $ 1 \leq t \leq r $.

(ii) A similar computation with the help of Equation \eqref{C^* action on generalized kernel elements} yields the desired identity.
\end{proof}

\begin{rem} \label{remark on generalized kernel elements}
(i) When $ \CA = \C $, the linear span of the set of vectors 
$$ \{ \SA [ K ( \cdot, W ) \sigma_t ] \SB : W \in \Omega_m, \,\, \SA, \SB \in \C^{ m \times m },\, 1 \leq t \leq r \} $$ 
is dense in $ \CH_m $ for every $ m \in \N $. Indeed, it follows from Equation \eqref{generalized reproducing property} that only zero vector is perpendicular to $ \SA [ K ( \cdot, W ) \sigma_t ] \SB $ for all $ 1 \leq t \leq r $, $ \SA, \SB \in \C^{ m \times m } $ and $ W \in \Omega_m $.

(ii) For a general unital $ C^* $ algebra $ \CA $, it follows from part (ii) of the proposition above that the set 
$$ \{ \TA [ K ( \cdot, W ) \sigma_t ] \SX : W \in \Omega_m, \,\, \TA \in \CA^{ m \times m },\, \SX \in \C^{ m \times m }, \, 1 \leq t \leq r \} $$ turns out to be a dense subset of $ \CH_m $, $ m \in \N $.

(iii) Using the proposition above we have that 
$$ \TA K ( \cdot, W ) = 0 \quad \text{if and only if} \quad K ( W, W ) ( \TA^* \TA ) = 0 $$ 
for all $ W \in \Omega_m $, $ \TA \in \CA^{ m \times m } $ and $ m \in \N $. 
\end{rem}

\subsubsection{Adjoint of multiplication operators} \label{Action of amplified multiplication operator} In this subsection, we restrict ourselves to the case when $ \CV = \C^d $ and describe the action of the adjoint of the $ d $-tuple of multiplication operators by nc coordinate functions on $ \CH_m $ for $ m \in \N $. 

Recall that $ \boldsymbol{ M } = ( M_{ Z_1 }, \hdots, M_{ Z_d } ) $ is the $ d $-tuple of multiplication operators by the nc coordinate functions on $ \CH ( K ) $ with $ M_{ Z_i } : \CH ( K ) \rightarrow \CH ( K ) $ defined by $ f \mapsto M_{ Z_i } ( f ) $ and $ M_{ Z_i } ( f ) ( W ) = W_i f ( W ) $, $ W \in \Omega $. Here, we mean $ ( W_i f_1 ( W ), \hdots, W_i f_r ( W ) ) $ by $ W_i f ( W ) $ for $ i = 1, \hdots, d $ and $ f = ( f_1, \hdots, f_r ) \in \CH ( K ) $. Then $ \boldsymbol{ M } $ gives rise to a linear operator $ D_{ \boldsymbol{ M }^* - W^* } : \CH_m \rightarrow \CH_m ^{ \oplus d } $ for every $ W \in \Omega_m $ and  $ m \in \N $ defined by 
$$ D_{ \boldsymbol{ M }^* - W^* } : = ( M^*_{ Z_1 } \otimes \mathrm{Id}_{ \C^{ m \times m } } - \mathrm{Id}_{ \CH ( K ) } \otimes R_{ W _1^* }, \hdots, M^*_{ Z_d } \otimes \mathrm{Id}_{ \C^{ m \times m } } - \mathrm{Id}_{ \CH ( K ) } \otimes R_{ W _d^* } ). $$
%Recall (cf. \cite[Remark 4.2, pp. 1915]{Ball-Marx-Vinnikov-rkhs}) that $ M_{ Z_i }^* $ commutes with $ \rho ( a ) $ for each $ i = 1, \hdots, d $, $ a \in \CA $, and therefore, each $ M^*_{ Z_i } \otimes \mathrm{Id}_{ \C^{ m \times m } } - \mathrm{Id}_{ \CH ( K ) } \otimes R_{ W _i^* } $, $ 1 \leq i \leq d $, intertwines the action of $ \CA^{ m \times m } $ on $ \CH_m $ and the induced action on $ \CH_m^{ \oplus d } $. Consequently f
For $ 1 \leq j \leq d $, $ \xi \in \C^r $, $ m \in \N $, $ \TA \in \CA^{ m \times m } $ and $ \ST \in \C^{ m \times m } $, we have that
$$ ( M^*_{ Z_j } \otimes \mathrm{Id}_{ \C^{ m \times m } } - \mathrm{Id}_{ \CH ( K ) } \otimes R_{ W _j^* } ) ( \TA [ K ( \cdot, W ) \xi ] \ST ) = \TA ( K ( \cdot, W ) \xi ) [ W_j^*, \ST ] . $$
Indeed,
\Bea
M_{ Z_j }^* \otimes \mathrm{Id}_{ \C^{ m \times m } } ( \TA [K ( \cdot, W ) \xi ] \ST ) & = &  M_{ Z_j }^* \otimes \mathrm{Id}_{ \C^{ m \times m } } \left( \sum_{ i, k =1 }^m K_{ W,  ( \se \otimes \varepsilon_{ i } ) \TA, \ST \varepsilon_k^* \otimes \xi } \otimes \varepsilon_{ i }^* \varepsilon_k \right)\\
& = &  \sum_{ i, k =1 }^m K_{ W,  ( \se \otimes \varepsilon_{ i } ) \TA, W_j^*\ST \varepsilon_k^* \otimes \xi } \otimes \varepsilon_{ i }^* \varepsilon_k \\
& = &  \TA [ K ( \cdot, W ) \xi ] W_j^* \ST;
\Eea
and $ \mathrm{Id}_{ \CH ( K ) } \otimes R_{ W _j^* } ( \TA [ K ( \cdot, W ) \xi ] \ST ) = \TA [ K ( \cdot, W ) \xi ] \ST W_j^* $. Here, the validation of the second equality is guaranteed by Equation \eqref{adjoint of mult op on ker element}. It thus verifies the following inclusion: For $ m \in \N $ and $ W \in \Omega_m $,
\be \label{joint kernel of adjoint of mult operator}
N_{ W } : = \{ \TA K ( \cdot, W ) \xi : \TA \in \CA^{ m \times m }, \, \xi \in \C^r \} \subset \ker D_{ \boldsymbol{ M }^* - W^* }.
\ee
Note that $ N_W $ is a left module over $ \CA^{ m \times m } $ (in particular, over $ \C^{ m \times m } $) for each $ W \in \Omega_m $ and $ m \in \N $. 

As before, we now describe the action of $ D_{ \boldsymbol{ M }^* - W^* } $ on the derivatives of the \textbf{generalized kernel elements}. Since $ Z \mapsto K ( \cdot, Z ) \xi $, for any $ \xi \in \C^r $, is a uniformly anti-analytic nc function on $ \Omega $ taking values in $ ( \CH_r )_{ nc } $, it follows, for $ W^0 \in \Omega_{ n_0 }, \hdots, W^{ \ell } \in \Omega_{ n_{ \ell } } $, that the function 
$$ _{ 1 * } \Delta_R^{ \ell } K ( \cdot, W^0, \hdots, W^{ \ell } ) \xi : ( \C^d )^{ n_0 \times n_1 } \times \cdots \times ( \C^d )^{ n_{ { \ell } -1 } \times n_{ \ell } } \rightarrow \CH^{ n_0 \times n_{ \ell } } $$
is a multi-linear mapping over $ \C $. Consequently, in view of the identification made in Equation \eqref{coordinate free kernel}, when $ n_0 = \cdots = n_{ \ell } = m $ and $ W \in \Omega_m $, we have that
\be \label{derivative of coordinate free kernel} 
( _{ 1 * } \Delta_R^{ \ell } K ( \cdot, W, \hdots, W ) \xi ) ( X^1, \hdots, X^{ \ell } ) = \sum_{ i, j =1 }^{ m } K_{ _{ ( X^1, \hdots, X^{ \ell } ) } W^{ \oplus { \ell } }, \left[ \begin{smallmatrix} \se \otimes \varepsilon_i & 0 \end{smallmatrix}\right], \left[ \begin{smallmatrix} 0 \\ \varepsilon_j^* \otimes \xi \end{smallmatrix} \right] } \otimes \varepsilon_i^* \varepsilon_j
\ee
where $ _{ ( X^1, \hdots, X^{ \ell } ) } W^{ \oplus { \ell } } $ denotes the point $ \left[ \begin{smallmatrix}
W & 0 & 0 & \cdots & 0 \\
X^1     & W & 0 & \cdots & 0 \\
0 & X^2 & W &  \cdots & 0 \\
\vdots & \vdots & \ddots & \ddots & \vdots \\
0      &       0     & \cdots   & X^{ \ell } & W
\end{smallmatrix} \right]$. In a similar manner as in Equation \eqref{interpretation of kernel elements}, we also have the following identities involving derivatives:
\be \label{module action on derivative of kernel elements 1}
\SA  ( _{ 1 * } \Delta_R^{ \ell } K ( \cdot, W, \hdots, W ) \xi ) ( X^1, \hdots, X^{ \ell } ) \SB = \sum_{ i, j  =1 }^{ m } K_{ _{ ( X^1, \hdots, X^{ \ell } ) } W^{ \oplus { \ell } }, \left[ \begin{smallmatrix} \se \otimes \varepsilon_i &  0 \end{smallmatrix}\right], \left[ \begin{smallmatrix}  0 \\ \varepsilon_j^* \otimes \xi \end{smallmatrix} \right] } \otimes \SA ( \varepsilon_i^* \varepsilon_j ) \SB,
\ee
or equivalently,
\be \label{module action on derivative of kernel elements 2}
\SA  ( _{ 1 * } \Delta_R^{ \ell } K ( \cdot, W, \hdots, W ) \xi ) ( X^1, \hdots, X^{ \ell } ) \SB = \sum_{ i, j =1 }^{ m } K_{ _{ ( X^1, \hdots, X^{ \ell } ) } W^{ \oplus { \ell } }, \left[ \begin{smallmatrix}  \varepsilon_i \SA & 0 \end{smallmatrix}\right], \left[ \begin{smallmatrix}  0 \\ \SB \varepsilon_j^* \otimes \xi \end{smallmatrix} \right] } \otimes \varepsilon_i^* \varepsilon_j ,
\ee
for $ \xi \in \C^{ r \times 1 } $, $ \SA, \SB \in \C^{ m \times m } $, $ W \in \Omega_m $ and $ X^1, \hdots, X^{ \ell } \in \C^{ m \times m } $. The last equation thus yields the action of $ \C^{ m \times m } $ on these derivatives of generalized kernel elements. Finally, the action of $ \CA^{ m \times m } $ on these derivatives of generalized kernel elements show up as
\be \label{module action on derivative of kernel elements 3}
\TA  ( _{ 1 * } \Delta_R^{ \ell } K ( \cdot, W, \hdots, W ) \xi ) ( X^1, \hdots, X^{ \ell } ) \SB = \sum_{ i, j =1 }^{ m } K_{ _{ ( X^1, \hdots, X^{ \ell } ) } W^{ \oplus { \ell } }, \left[ \begin{smallmatrix} ( \se \otimes \varepsilon_i ) \TA & 0 \end{smallmatrix}\right], \left[ \begin{smallmatrix} 0 \\ \SB \varepsilon_j \otimes \xi \end{smallmatrix} \right] } \otimes \varepsilon_i^* \varepsilon_j .
\ee
In this set-up, the following proposition yields the reproducing property of the \textbf{generalized kernel elements} as well as the action of $ D_{ \boldsymbol{ M }^* - W^* } $ on them.

\begin{prop} \label{properties of generalized kernel elements}
Let $ \Omega \subset ( \C^d )_{ nc } $ and $ \CH ( K ) \subset \CN\CC\CO ( \Omega, \CL ( \CA, \C^r )_{ nc } ) $ be a nc reproducing kernel Hilbert space with the nc cp reproducing kernel $ K : \Omega \times \Omega \rightarrow \CL ( \CA_{ nc }, ( \C^r )_{ nc } ) $. Assume that $ \Omega $ is both left and right admissible. Denote $ \Omega^* = \{ W : W^* \in \Omega \} $. Then 
\begin{itemize}

\item[(i)] For $ m \in \N$, $ W \in \Omega_m $, $ \TA \in \CA^{ m \times m } $, $ \SB \in \C^{ m \times m } $, $ X^1, \hdots,  X^{ \ell } \in ( \C^d )^{ m \times m } $, $ F = ( F_1, \hdots, F_r ) \in \CH_m $ and the standard ordered basis $ \{ \sigma_t : 1 \leq t \leq r \} $ of  $ \C^{ r \times 1 } $,
\begin{align} \label{reproducing of derivatives}
& \left\langle F, \TA \, ( _{ 1 * } \Delta_R^{ \ell } K ( \cdot, W, \hdots, W ) \sigma_t ) ( X^1, \hdots, X^{ \ell } ) \, \SB \right\rangle_{ HS }\\
\nonumber & = \left\langle \Delta_L^{ \ell } F_t ( W, \hdots, W ) ( X^{ \ell }, X^2, \hdots, X^{ { \ell } - 1 }, X^1 ) ( \TA^* ), \SB \right\rangle_{ \C^{ m \times m } }, ~~~ t = 1, \hdots, r.
\end{align}

\item[(ii)] For every $ W \in \Omega_m $,  $ \TA \in \CA^{ m \times m } $, $ \SB \in \C^{ m \times m } $, $ X^1, \hdots, X^{ \ell } \in ( \C^d )^{ m \times m } $, $ \xi \in \C^{ r \times 1 } $ and $ 1 \leq k \leq d $,
\begin{align} \label{action of mult operator on derivative}
& \left( M_{ Z_k }^* \otimes \mathrm{Id}_{ m \times m } - \mathrm{Id}_{ \CH ( K ) } \otimes R_{ W_k^* } \right)  \left( \TA ( _{ 1 ^* } \Delta_R^{ \ell } K ( \cdot, W, \hdots, W ) \xi ) ( X^1, \hdots, X^{ \ell } ) \SB \right) \\
\nonumber & = ~ \TA ( _{1^*} \Delta_R^{ { \ell } - 1 } K ( \cdot, W, \hdots, W ) \xi ) ( X^1, \hdots, X^{ { \ell } -1 } )  ( X^{ \ell }_k )^* \SB \\
\nonumber & + ~ \TA ( _{1^*} \Delta_R^{ { \ell } } K ( \cdot, W, \hdots, W ) \xi ) ( X^1, \hdots, X^{ \ell } )  [ W_k^*, \SB ] .
\end{align}
\end{itemize}
\end{prop}

\begin{proof}
Since a similar proof as in the part (iii) in Proposition \ref{properties of kernel elements} together with the reproducing property of the generalized kernel elements given in Equation \eqref{generalized reproducing property} yield the part (i) we only prove part (ii).

Let $ f = ( f_1, \hdots, f_r ) \in \CH ( K ) $ and $ \SS \in \C^{ m \times m } $ so that $ f \otimes \SS \in \CH_m $. Then, for $ 1 \leq k \leq d $,
\begin{align*}
& \left\langle f \otimes \SS,  \left( M_{ Z_k }^* \otimes \mathrm{Id}_{ m \times m } - \mathrm{Id}_{ \CH ( K ) } \otimes R_{ W_k^* } \right)  \left( \TA ( _{ 1 ^* } \Delta_R^{ \ell } K ( \cdot, W, \hdots, W ) \xi ) ( X^1, \hdots, X^{ \ell } ) \SB \right) \right\rangle_{ HS } \\
& = \left\langle \! \! \left( M_{ Z_k } \otimes \mathrm{Id}_{ m \times m } \! - \!\mathrm{Id}_{ \CH ( K ) } \otimes R_{ W_k } \right) \! ( f \otimes \SS ), \! \sum_{ i, j =1 }^{ m } \! K_{\! _{ ( X^1, \hdots, X^{ \ell } ) } W^{ \oplus { \ell } }, \left[ \begin{smallmatrix} ( \se \otimes \varepsilon_i ) \TA & 0 \end{smallmatrix}\right], \left[ \begin{smallmatrix} 0 \\ \varepsilon_j^* \otimes \xi \end{smallmatrix} \right] } \otimes \varepsilon_i^* \varepsilon_j \SB \! \right\rangle_{ HS } \\
& = \sum_{ i, j = 1 }^m \left[ \left\langle M_{ Z_k } \otimes \mathrm{Id}_{ m \times m } ( f \otimes \SS ), K_{ _{ ( X^1, \hdots, X^{ \ell } ) } W^{ \oplus { \ell } }, \left[ \begin{smallmatrix}( \se \otimes \varepsilon_i ) \TA & 0 \end{smallmatrix}\right], \left[ \begin{smallmatrix} 0 \\ \varepsilon_j^* \otimes \xi \end{smallmatrix} \right] } \otimes \varepsilon_i^* \varepsilon_j \SB \right\rangle_{ HS } \right] \\
& - \sum_{ i, j = 1 }^m \left[ \left\langle \mathrm{Id}_{ \CH ( K ) } \otimes R_{ W_j } ( f \otimes \SS ), K_{ _{ ( X^1, \hdots, X^{ \ell } ) } W^{ \oplus { \ell } }, \left[ \begin{smallmatrix} ( \se \otimes \varepsilon_i ) \TA & 0 \end{smallmatrix}\right], \left[ \begin{smallmatrix} 0 \\ \varepsilon_j^* \otimes \xi \end{smallmatrix} \right] } \otimes \varepsilon_i^* \varepsilon_j \SB \right\rangle_{ HS } \right] \\
& = \sum_{ i, j = 1 }^m \left[ \left\langle _{ ( X_k^1, \hdots, X_k^{ \ell } ) } W_k^{ \oplus { \ell } } f \left( _{ ( X^1, \hdots, X^{ \ell } ) } W^{ \oplus { \ell } } \right) \left[ \begin{smallmatrix} \TA^* ( \se \otimes \varepsilon_i^* ) \\ 0 \end{smallmatrix}\right], \left[ \begin{smallmatrix} 0 \\ \varepsilon_j^* \otimes \xi \end{smallmatrix} \right] \right\rangle_{ ( \C^r )^{ \ell m } } \langle \SS, \varepsilon_i^* \varepsilon_j \SB \rangle_{ \C^{ m \times m } } \right] \\
& - \sum_{ i, j = 1 }^m \left[ \left\langle f \left( _{ ( X^1, \hdots, X^{ \ell } ) } W^{ \oplus { \ell } } \right) \left[ \begin{smallmatrix} \TA^* ( \se \otimes \varepsilon_i^* ) \\ 0 \end{smallmatrix}\right], \left[ \begin{smallmatrix} 0 \\ \varepsilon_j^* \otimes \xi \end{smallmatrix} \right] \right\rangle_{ ( \C^r )^{ { \ell } m } } \langle \SS W_k, \varepsilon_i^* \varepsilon_j \SB \rangle_{ \C^{ m \times m } } \right] \\
& = \sum_{ i, j = 1 }^m \left[ \left\langle X^{ \ell }_k \Delta_L^{ { \ell } - 1 } f ( W, \hdots W ) ( X^1, \hdots, X^{ { \ell } - 1 } ) ( \TA^* ( \se \otimes \varepsilon_i^* ) ) , \varepsilon_j^* \otimes \xi \right\rangle_{ ( \C^r )^m } \langle \SS, \varepsilon_i^* \varepsilon_j \SB \rangle_{ \C^{ m \times m } } \right] \\
& + \sum_{ i, j = 1 }^m \left[ \left\langle W_k \Delta_L^{ \ell } f ( W, \hdots, W ) ( X^1, \hdots, X^{ \ell } ) ( \TA^* ( \se \otimes \varepsilon_i^* ) ), \varepsilon_j^* \otimes \xi \right\rangle_{ ( \C^r )^m} \langle \SS, \varepsilon_i^* \varepsilon_j \SB \rangle_{ \C^{ m \times m } } \right] \\
& - \sum_{ i, j = 1 }^m \left[ \left\langle \Delta_L^{ \ell } f ( W, \hdots, W ) ( X^1, \hdots, X^{ \ell } ) ( \TA^* ( \se \otimes \varepsilon_i^* ) ), \varepsilon_j^* \otimes \xi \right\rangle_{ ( \C^r )^m } \langle \SS W_k, \varepsilon_i^* \varepsilon_j \SB \rangle_{ \C^{ m \times m } } \right] \\
& = \left\langle f \otimes \SS, \TA ( _{1^*} \Delta_R^{ { \ell } - 1 } K ( \cdot, W, \hdots, W ) \xi ) ( X^1, \hdots, X^{ { \ell } -1 } )  ( X^{ \ell }_k )^* \SB  \right\rangle_{ HS } \\
& +  \left\langle f \otimes \SS, \TA ( _{1^*} \Delta_R^{ { \ell } } K ( \cdot, W, \hdots, W ) \xi ) ( X^1, \hdots, X^{ \ell } )  [ W_k^*, \SB ]  \right\rangle_{ HS }  
\end{align*}
where the last equality holds due to the following identities
\begin{align*} 
& \sum_{ i, j = 1 }^m \left[ \left\langle W_k \Delta_L^{ \ell } f ( W, \hdots, W ) ( X^1, \hdots, X^{ \ell } ) ( \TA^* ( \se \otimes \varepsilon_i^* ) ), \varepsilon_j^* \otimes \xi \right\rangle_{ ( \C^r )^m} \langle \SS, \varepsilon_i^* \varepsilon_j \SB \rangle_{ \C^{ m \times m } } \right] \\
& -  \sum_{ i, j = 1 }^m \left[ \left\langle \Delta_L^{ \ell } f ( W, \hdots, W ) ( X^1, \hdots, X^{ \ell } ) ( \TA^* ( \se \otimes \varepsilon_i^* ) ), \varepsilon_j^* \otimes \xi \right\rangle_{ ( \C^r )^m } \langle \SS W_k, \varepsilon_i^* \varepsilon_j \SB \rangle_{ \C^{ m \times m } } \right] \\
& = \sum_{ t =1 }^r \overline{ \xi }_t \left\langle \left(\Delta_L^{ \ell } f_t ( W, \hdots, W ) ( X^1, \hdots, X^{ \ell } ) \otimes \SS \right) ( \TA^* ), [ W_k^*, \SB ] \right\rangle_{ \C^{ m \times m } } . 
\end{align*}
Since $ \CH_m $ is the dense linear span of elementary tensors of the form $ f \otimes \SS $ with $ f \in \CH ( K ) $ and $ \SS \in \C^{ m \times m } $, the proof of part (ii) follows from the computation above.
\end{proof}

\begin{cor} \label{nc generalized eigenspaces for mult op}
Let $ \Omega \subset \C^d_{ nc } $ be a uniformly open set and $ \CH ( K ) $ be a nc reproducing kernel Hilbert space with the cp nc kernel $ K : \Omega \times \Omega \to \CL ( \C_{ nc }, \C^{ r \times r }_{ nc } ) $. For $ m \in \N $ and $ W \in \Omega_m $, let $ \mathfrak{ L }_{ \ell } $ and $ \mathfrak{ B }_{ \ell } $, respectively, denote the left module and bi-module generated by the set $ \{ K ( \cdot, W ) \sigma_s, $ $ ( _{ 1 * } \Delta_R^j K ( \cdot, W, \hdots, W ) \sigma_t ) ( X^1, \hdots, X^j ) : 1 \leq j \leq { \ell }, 1 \leq s, t \leq r \} $ over $ \C^{ m \times m } $. Then
\begin{itemize}
\item[(i)]  $ \mathfrak{ L }_{ \ell } = \bigcap_{ j = 1 }^d \left( M^*_{ Z_j } \otimes \mathrm{Id}_{ \C^{ m \times m } } - \mathrm{Id}_{\C^{ m \times m } } \otimes R_{ W^*_j } \right)^{ - 1 } ( \mathfrak{ B }_{ { \ell } - 1 } )  $
where $ \mathfrak{ B }_0 $ is the bi-module generated by the set $ \{ K ( \cdot, W ) \sigma_s : 1 \leq s \leq r \} $.\\
\item[(ii)] For each $ \ell \in \N $, $ \mathfrak{ B }_{ \ell } $ is an invariant subspace of $ M^*_{ Z_j } \otimes \mathrm{Id}_{ \C^{ m \times m } } - \mathrm{Id}_{\C^{ m \times m } } \otimes R_{ W^*_j } $, $ 1 \leq j \leq d $.\\
\item[(iii)] $ \overline{ \bigvee_{ \ell \in \N \cup \{ 0 \} } \left\{ h_{ i j } : ( \! ( h_{ i j } ) \! )_{ i, j = 1 }^m \in \mathfrak{ B }_{ \ell } \right\} } = \CH ( K ) $.
\end{itemize}
\end{cor}

\begin{proof}
(i) First, we observe from Equation \eqref{action of mult operator on derivative} that 
$$ ( M^*_{ Z_j } \otimes \mathrm{Id}_{ \C^{ m \times m } } - \mathrm{Id}_{\C^{ m \times m } } \otimes R_{ W^*_j } ) ( \mathfrak{ L }_{ \ell } ) \subset \mathfrak{ B }_{ \ell - 1 }, $$ 
for every $ 1 \leq j \leq d $. Also, note that 
$$ \dim_{ \C } \left( \mathfrak{ L }_{ \ell } \right) = r m^2 \frac{ ( \ell +1 ) ( \ell + 2 ) }{ 2 } ~ \text{and} ~  \dim_{ \C } \left( \mathfrak{ B }_{ \ell } \right) = r m^2 \frac{ \ell ( \ell + 3 ) }{ 2 } . $$
Since the dimension of the joint kernel of the operators $ M^*_{ Z_j } \otimes \mathrm{Id}_{ \C^{ m \times m } } - \mathrm{Id}_{\C^{ m \times m } } \otimes R_{ W^*_j } $, $ 1 \leq j \leq d $ -- which is $ \ker D_{ \boldsymbol{ M }^* - W^* } $ -- is $ r m^2 $ as a vector space over $ \C $, it follows that  
$$ \dim_{ \C} ( \mathfrak{ L }_{ \ell } ) = \dim_{ \C } \left( \bigcap_{ j = 1 }^d \left( M^*_{ Z_j } \otimes \mathrm{Id}_{ \C^{ m \times m } } - \mathrm{Id}_{\C^{ m \times m } } \otimes R_{ W^*_j } \right)^{ - 1 } ( \mathfrak{ B }_{ \ell - 1 } ) \right) $$  
verifying the desired equality.

(ii) The proof follows from part (i) and Equation \eqref{action of mult operator on derivative}.

(iii) Let $ f \in \CH ( K ) $ be perpendicular to the subspace $ \overline{ \bigvee_{ l \in \N \cup \{ 0 \} } \left\{ h_{ i j } : ( \! ( h_{ i j } ) \! )_{ i, j = 1 }^m \in \mathfrak{ B }_{ \ell } \right\} } $. Then $ F : = f \otimes I_m $ is perpendicular to $ \mathfrak{ B }_{ \ell } $ for all $ \ell \in \N \cup \{ 0 \} $ and consequently, it follows from Equation \eqref{reproducing of derivatives} that $ F $ is identically zero on a uniformly-open neighbourhood of $ W $. Thus $ f $ is identically zero on a uniformly-open neighbourhood of $ W $ verifying that $ f = 0 $ in $ \CH ( K ) $. 
\end{proof}

\begin{rem}
    Note that the subspaces $ \mathfrak{ B }_{ \ell } $, for $ \ell \in \N $, can be viewed as matricial or nc analogues of the generalized eigenspaces of the tuple 
    $$ \big( M_{ Z_1 }^* \otimes \mathrm{Id}_{ \C^{ m \times m } }, \hdots,  M_{ Z_d }^* \otimes \mathrm{Id}_{ \C^{ m \times m } } \big) $$ 
    of order $ \ell $ associated to the matricial eigenvalue $ W^* = ( W_1^*, \hdots, W_d^* ) \in \Omega^*_m $, $ m \in \N $. We make use of this corollary later in subsection \ref{local operators} to study the local operators associated to an element in the nc Cowen-Douglas class.  
\end{rem}

\subsection{Hilbert \texorpdfstring{$ C^* $}{C^*} module structures on \texorpdfstring{$ \CH ( K ) \otimes \C^{ m \times m } $}{CH (K) \otimes \C^{ m \times m } }} \label{Hilbert module structure}
As in the preceding subsection, let us consider a cp nc kernel $ K : \Omega \times \Omega \to \CL ( \C_{ nc }, \C^{ r \times r }_{ nc } ) $ and $ m \in \N $. We now discuss canonical Hilbert $ C^* $ left and right module structures on the space $ \CH_m $ of $ m \times m $ matrices over $ \CH ( K ) $.

We define the mappings $ \langle \cdot, \cdot \rangle_L : \CH_m \times \CH_m \to \C^{ m \times m } $ and $ \langle \cdot, \cdot \rangle_R : \CH_m \times \CH_m \to \C^{ m \times m } $ by 
$$ \langle F, G \rangle_L : = \bigg( \!\! \bigg( \sum_{ \ell = 1 }^m \langle f_{ i \ell }, g_{ j \ell } \rangle_{ \CH ( K ) } \bigg) \!\! \bigg)_{ i, j = 1 }^m \quad \text{and} \quad \langle F, G \rangle_R : = \bigg( \! \! \bigg( \sum_{ \ell = 1 }^m \langle f_{ \ell j }, g_{ i \ell } \rangle_{ \CH ( K ) } \bigg) \!\! \bigg)_{ i, j = 1 }^m, $$
respectively, where $ F = (\!( f_{ ij } ) \! )_{ i, j = 1 }^m $ and $ G = (\!( g_{ ij } ) \! )_{ i, j = 1 }^m $ in $ \CH_m $. It can be seen that $ \langle \cdot, \cdot \rangle_L $ and $ \langle \cdot, \cdot \rangle_R $ equip $ \CH_m $ a left and right Hilbert $ C^* $ module over $ \C^{ m \times m } $, respectively. We also point out that these two Hilbert module structures give rise to the Hilbert space structure on $ \CH_m $ introduced in Equation \eqref{inner product of matrix spaces}, when we take their traces. Namely, for $ F = (\!( f_{ ij } ) \! )_{ i, j = 1 }^m $ and $ G = (\!( g_{ ij } ) \! )_{ i, j = 1 }^m $ in $ \CH_m $,
$$ \mathrm{trace} \langle F, G \rangle_L = \mathrm{trace} \langle F, G \rangle_R = \langle F, G \rangle_{HS} . $$
We now study left and right Hilbert $ C^* $ module structures on $ \CH_m $ separately in the following two subsections.

\subsubsection{Left Hilbert $ C^* $ module structures on $ \CH_m $} Define the \textbf{left generalized evaluation functional} $ \ev_W^L : \CH_m \to ( \C^{ m \times m } )^{ r \times 1 } $ as follows:
$$ \ev_W^L ( F ) : = \sum_{ j = 1 }^n \SA_j f_j ( W ), $$
where $ F = \sum_{ j = 1 }^n f_j \otimes \SA_j $ for some $ f_1, \hdots, f_n \in \CH ( K ) $ and $ \SA_1, \hdots, \SA_n \in \C^{ m \times m } $. Since elements in $ \CH ( K ) $ are nc functions on $ \Omega $ taking values in $ \C^r_{ nc } $, we consider, for each $ 1 \leq s \leq r $, the $ s $-th \textbf{left generalized evaluation functional} 
$$ \ev_{W, s}^L : \CH_m \to \C^{ m \times m } \quad \text{as} \quad \ev_{W, s}^L ( F ) : = \sum_{ j = 1 }^n \SA_j f_{ j, s } ( W ) , $$
where $ F \in \CH_m $ as above with each $ f_j = ( f_{ j, 1 }, \hdots, f_{ j, r } ) $, and $ f_{ j, s } $ is a nc function on $ \Omega $ taking values in $ \C_{ nc } $, for $ 1 \leq j \leq n $, $ 1 \leq s \leq r $. We now study the elementary properties of these left generalized evaluation functionals.

\begin{prop} \label{left evaluation}
    Let $ \ev_W^L : \CH_m \to ( \C^{ m \times m } )^{ r \times 1 } $ be as above.
\begin{enumerate}
    \item Reproducing property: For $ F \in \CH_m $, $ W \in \Omega_m $ and $ 1 \leq s \leq r $,
    $$ \langle F, K ( \cdot, W ) \sigma_s \rangle_L = \ev_{ W, s }^L ( F ) . $$
     \item For $ \boldsymbol{ \SA } = \begin{bmatrix} \SA_1 \\ \vdots \\ \SA_r \end{bmatrix} \in ( \C^{ m \times m } )^{ r \times 1 } $, 
    $ ( \ev_W^L )^* ( \boldsymbol{ \SA } ) = \sum_{ s = 1 }^r \SA_s [K ( \cdot, W ) \sigma_s ] .$ In particular, both $ \ev_W^L $ and $ ( \ev_W^L )^* $ are left module maps over $ \C^{ m \times m } $.
     \item $ \ev_W^L \circ ( \ev_W^L )^* : ( \C^{ m \times m } )^{ r \times 1 } \to ( \C^{ m \times m } )^{ r \times 1 } $ is the linear operator given by the formula:
    $$  \ev_W^L \circ ( \ev_W^L )^* ( \boldsymbol{ \SA } ) : = \boldsymbol{ \SA } \big( \! \big( \SE ( W )_{ s t } \big) \! \big)_{ s, t = 1 }^r : = \begin{bmatrix}
       \sum_{ s = 1 }^r \SA_s \SE ( W )_{ 1 s } \\ \vdots \\ \sum_{ s = 1 }^r \SA_s \SE ( W )_{ r s } 
       \end{bmatrix}, $$
       where each $ \SE ( W )_{ s t } $ is an $ m \times m $ matrix defined as
       $$ \SE ( W )_{ s t } : = \big( \! \big( \operatorname{trace} \big( K_{ s t } ( W, W ) ( E_{ i j }^* ) \big) \big) \! \big)_{ i, j = 1 }^m . $$
       In particular, when $ r = 1 $, $ \ev_W^L \circ ( \ev_W^L )^* ( \SA ) = \SA K ( W, W )^* ( I_m )^{ \top } $.
\end{enumerate}
\end{prop}

\begin{proof}
    Let $ f = ( f_1, \hdots, f_r ) \in \CH ( K ) $ and $ \SA \in \C^{ m \times m } $ so that $ f \otimes \SA \in \CH_m $. Then we compute, for $ 1 \leq s \leq r $ and $ W \in \Omega_m $,
   \begin{align*}
         &\langle f \otimes \SA, K ( \cdot, W ) \sigma_s \rangle_L = \sum_{ i, j = 1 }^m \big\langle f \otimes \SA, K_{ W, \varepsilon_i, \varepsilon_j^* \otimes \sigma_s } \otimes \varepsilon_i^* \varepsilon_j \big\rangle_L \\
         &= \sum_{ i, j = 1 }^m \big\langle f, K_{ W, \varepsilon_i, \varepsilon_j^* \otimes \sigma_s } \big\rangle_{ \CH ( K ) } \SA \varepsilon_j^* \varepsilon_i = \sum_{ \ell, k = 1 }^m \sum_{ i, j = 1 }^m \big\langle f ( W ) ( \varepsilon_i^* ), \varepsilon_j^* \otimes \sigma_s \big\rangle_{ \C^{ m r \times 1 } } a_{ \ell k } \varepsilon_{ \ell }^* \varepsilon_k \varepsilon_j^* \varepsilon_i \hspace{0.5in}\\
         &= \sum_{ \ell = 1 }^m \sum_{ i, j = 1 }^m \big\langle f ( W ) ( \varepsilon_i^* ), \varepsilon_j^* \otimes \sigma_s \big\rangle_{ \C^{ m r \times 1 } } a_{ \ell j } \varepsilon_{ \ell }^* \varepsilon_i = \sum_{ i, \ell = 1 }^m \big\langle f ( W ) ( \varepsilon_i^* ), \sum_{ j = 1 }^m  \overline{ a_{ \ell j } } \varepsilon_j^* \otimes \sigma_s \big\rangle_{ \C^{ m r \times 1 } } \varepsilon_{ \ell }^* \varepsilon_i \hspace{0.1in}\\
         &=  \sum_{ i, \ell = 1 }^m \big\langle f ( W ) ( \varepsilon_i^* ), \SA^* \varepsilon_{ \ell }^* \otimes \sigma_s \big\rangle_{ \C^{ m r \times 1 } } \varepsilon_{ \ell }^* \varepsilon_i = \sum_{ i, \ell = 1 }^m \big\langle \SA f ( W ) ( \varepsilon_i^* ), \varepsilon_{ \ell }^* \otimes \sigma_s \big\rangle_{ \C^{ m r \times 1 } } \varepsilon_{ \ell }^* \varepsilon_i \hspace{0.65in} \\
         &= \SA f_s ( W ) = \ev_{ W, s }^L ( f \otimes \SA ). \hspace{4.05in}
    \end{align*}
This completes the proof of part (1).

    For the proof of (2), let $ \boldsymbol{ \SA } \in ( \C^{ m \times m } )^{ r \times 1 } $ be as given and $ F = \sum_{ i = 1 }^n f_i \otimes \SB_i $ with $ \SB_1, \hdots, \SB_n \in \C^{ m \times m } $ and each $ f_i = \sum_{ s = 1 }^r f_{ i, s } \otimes \sigma_s $. Then we compute
    \begin{multline*}
        \langle \ev_W^L ( F ), \boldsymbol{ \SA } \rangle  =  \sum_{ i = 1 }^n \sum_{ s = 1 }^r \langle \SB_i f_{ i, s } ( W ), \SA_s \rangle 
         =  \sum_{ i = 1 }^n \sum_{ s = 1 }^r \SB_i f_{ i, s } ( W ) \SA_s^* 
         =  \sum_{ s = 1 }^r \big( \sum_{ i = 1 }^n \SB_i f_{ i, s } ( W ) \big)  \SA_s^* \\
         =  \sum_{ s = 1 }^r \ev_{ W, s }^R ( F ) \SA_s^* 
         =  \sum_{ s = 1 }^r \langle F, K ( \cdot, W ) \sigma_s \rangle_L \SA_s^* 
        % =  \sum_{ s = 1 }^r \langle F, \SA_s K ( \cdot, W ) \sigma_s \rangle_L 
         = \big\langle F, \sum_{ s = 1 }^r \SA_s K ( \cdot, W ) \sigma_s \big\rangle_L,
    \end{multline*}
    where the second last equality holds since $ \langle \cdot, \cdot \rangle_L $ defines a left Hilbert $ C^* $ module structure on $ \CH_m $. It thus verifies that
    $$ ( \ev_W^L )^* ( \boldsymbol{ \SA } ) = \sum_{ s = 1 }^r \SA_s [K ( \cdot, W ) \sigma_s ] . $$

    Finally, for the proof of part (3), we compute, for $ 1 \leq t \leq r $ and $ \boldsymbol{ \SA } \in ( \C^{ m \times m } )^{ r \times 1 } $,
    \begin{align*}
        & \big\langle \ev_W^L \circ ( \ev_W^L )^* ( \boldsymbol{ \SA } ), I_m \otimes \sigma_t \big\rangle 
        =  \big\langle ( \ev_W^L )^* ( \boldsymbol{ \SA } ), ( \ev_W^L )^* ( I_m \otimes \sigma_t ) \big\rangle_L \\
        & =  \sum_{ s = 1 }^r \big\langle \SA_s [K ( \cdot, W ) \sigma_s ], I_m [K ( \cdot, W ) \sigma_t ] \big\rangle_L 
        =  \sum_{ s = 1 }^r \SA_s \big\langle K ( \cdot, W ) \sigma_s, K ( \cdot, W ) \sigma_t \big\rangle_L \\
        & =  \sum_{ s = 1 }^r \sum_{ i, j = 1 }^m \sum_{ \ell, k = 1 }^m \langle K_{ W, \varepsilon_i, \varepsilon_j^* \otimes \sigma_s }, K_{ W, \varepsilon_{ \ell }, \varepsilon_k^* \otimes \sigma_t } \rangle_{ \CH ( K ) } \SA_s \varepsilon_i^* \varepsilon_j \varepsilon_k^* \varepsilon_{ \ell } \\
        & =  \sum_{ s = 1 }^r \sum_{ i, j = 1 }^m \sum_{ \ell = 1 }^m \langle K_{ W, \varepsilon_i, \varepsilon_j^* \otimes \sigma_s }, K_{ W, \varepsilon_{ \ell }, \varepsilon_j^* \otimes \sigma_t } \rangle_{ \CH ( K ) } \SA_s \varepsilon_i^* \varepsilon_{ \ell } \\
        & =  \sum_{ s = 1 }^r \sum_{ i, j = 1 }^m \sum_{ \ell = 1 }^m \langle K_{ W, \varepsilon_i, \varepsilon_j^* \otimes \sigma_s } ( W ) ( \varepsilon_{ \ell }^* ), \varepsilon_j^* \otimes \sigma_t \rangle  \SA_s \varepsilon_i^* \varepsilon_{ \ell } \\
        & =  \sum_{ s = 1 }^r \sum_{ i, j = 1 }^m \sum_{ \ell = 1 }^m \langle K ( W, W ) ( \varepsilon_{ \ell }^* \varepsilon_i ) ( \varepsilon_j^* \otimes \sigma_s ), \varepsilon_j^* \otimes \sigma_t \rangle \SA_s \varepsilon_i^* \varepsilon_{ \ell } \\
        & =  \sum_{ s = 1 }^r \SA_s \big[ \sum_{ i, \ell = 1 }^m \!\! \big( \sum_{ j = 1 }^m \langle K ( W, W ) ( \varepsilon_{ \ell }^* \varepsilon_i ) ( \varepsilon_j^* \otimes \sigma_s ), \varepsilon_j^* \otimes \sigma_t \rangle \big) \varepsilon_i^* \varepsilon_{ \ell } \big] \\
        & =  \sum_{ s = 1 }^r \SA_s \big[ \sum_{ i, \ell = 1 }^m  \text{trace} \big( K ( W, W )_{ t s } ( \varepsilon_{ \ell }^* \varepsilon_i ) \big) \varepsilon_i^* \varepsilon_{ \ell } \big] 
         =  \sum_{ s = 1 }^r  \SA_s \SE ( W )_{ t s }.
    \end{align*}
    This validates the desired formula for $  \ev_W^L \circ ( \ev_W^L )^* ( \boldsymbol{ \SA } ) $.
\end{proof}

\begin{cor}
    For any $ \SA, \SB \in \C^{ m \times m } $, $ W \in \Omega_m $ and $ 1 \leq s, t \leq r $, 
    $$ \langle \SA K ( \cdot, W ) \sigma_s, \SB K ( \cdot, W ) \sigma_t \rangle_L = \SA  \SE ( W )_{ s t } \SB^*  $$
    where $  \SE ( W )_{ s t } $ is as defined in Proposition \ref{left evaluation}.
\end{cor}

\begin{proof}
The proof follows directly from the proof of part (3) in the preceding proposition.    
\end{proof}

We now express certain non-degeneracy criteria of the cp nc kernel $ K $ in terms of the left generalized evaluation functional in the theorem below.

\begin{thm} \label{equivalent criterion for non-degeneracy}
    The following are equivalent for $ m \in \N $ and $ W \in \Omega_m $.
    \begin{enumerate}
        \item $ \ev^L_W : \CH_m \to ( \C^{ m \times m } )^{ 1 \times r } $ is a surjective.
        \item The matrix $$ \big( \! \big( K_{ t s } ( W, W )^* ( I_m ) \big) \! \big)_{ s, t = 1 }^r $$ is invertible, where $  K_{ t s } ( W, W )^* ( I_m ) $ is $ ( s, t ) $-th block, $ 1 \leq s, t \leq r $ and $ K_{ t s } ( W, W )^* $ is the adjoint of the linear mapping $  K_{ t s } ( W, W ): \C^{ m \times m } \to \C^{ m \times m } $ with respect to the Hilbert-Schmidt product on $ \C^{ m \times m } $. 
        \item $ K ( W, W ) : \C^{ m \times m } \to ( \C^{ m \times m } )^{ r \times r } $ satisfies: For $ P \in \C^{ m \times m } $ with $ P \geq 0 $,
        $$ K ( W, W ) ( P ) = 0 \hspace{0.1in} \text{if and only if} \hspace{0.1in} P = 0 . $$
        \item $ \CH_K $ has the following property:
        $$ \bigcap_{ f \in \CH_K } \ker f ( W ) = \{ 0 \} . $$
    \end{enumerate}
\end{thm}

\begin{proof}
 (1) $ \Leftrightarrow $ (2): This equivalence follows from part (3) of Proposition \ref{left evaluation} along with the following observation: For $ 1 \leq s, t \leq r $ and $ 1 \leq i, j \leq m $, the $ ( i, j ) $-th entry of the matrix $ \SE ( W )_{ st } $ is 
 \begin{eqnarray*} 
 \overline{ \text{trace} \big( K_{ s t } ( W, W ) ( \varepsilon_j^* \varepsilon_i ) \big) } & = & \sum_{ \ell = 1 }^m \langle K_{ W, \varepsilon_j, \varepsilon_{ \ell }^* \otimes \sigma_s }, K_{ W, \varepsilon_i, \varepsilon_{ \ell }^* \otimes \sigma_t } \rangle_{ \CH ( K ) } \\
 & = & \langle K_{ t s } ( W, W ) ( \varepsilon_i^* \varepsilon_j ), I_m \rangle_{ \C^{ m \times m } }, \end{eqnarray*}
which is the $ ( i, j ) $-th entry of the matrix $ K_{ t s } ( W, W )^* ( I_m ) $.

\noindent (3) $ \Leftrightarrow $ (4): First, we observe that 
$$ \bigcap_{ f \in \CH_K } \ker f ( W ) = \big\{ v^* \in \C^{ m \times 1 } : K_{ W, v, y } = 0, ~ \text{for all} ~ y \in \C^{ m r \times 1 } \big\} . $$
Indeed, if $ K_{ W, v, y } = 0 $ for all $ y \in \C^{ m r \times 1 } $, then $$ \langle f, K_{ W, v, y } \rangle_{ \CH_K } = \langle f ( W ) ( v^*), y \rangle_{ \C^{ m r \times 1 } } = 0 , $$
for all $ y \in \C^{ m r \times 1 } $ implying that $ v^* \in \ker f ( W ) $. Conversely, if $ v^* \in \ker f ( W ) $, the reverse to this argument shows that  $ K_{ W, v, y } = 0 $ for all $ y \in \C^{ m r \times 1 } $.

Now assuming (3) and applying it for rank $ 1 $ non-negative definite matrices, it can be seen that 
$$ \big\{ v^* \in \C^{ m \times 1 } : K_{ W, v, y } = 0, ~ \text{for all} ~ y \in \C^{ m r \times 1 } \big\} = \{ 0 \} $$
which verifies (4). Conversely, if $ K ( W, W ) ( P ) = 0 $ for $ P \geq 0 $, writing $ P = Q^* Q $, we notice that
\begin{align*} & K ( W, W ) ( P ) = K ( W,W ) ( Q^* Q ) \\ 
& = K ( W, W ) \left( Q^* \big( \sum_{ \ell = 1 }^m \varepsilon_{ \ell }^* \varepsilon_{ \ell } \big) Q \right)  = \sum_{ \ell = 1 }^m K ( W, W ) \big( Q^*\varepsilon_{ \ell }^* \varepsilon_{ \ell } Q \big) . \end{align*}
Consequently, for any $ y \in \C^{ mr \times 1 } $, we have that
\begin{multline*}
    \langle K ( W, W ) ( P ) y, y \rangle  = \langle K ( W,W ) ( Q^* Q ) y, y \rangle 
     =  \sum_{ \ell = 1 }^m \langle K ( W, W ) \big( Q^*\varepsilon_{ \ell }^* \varepsilon_{ \ell } Q \big) y, y \rangle \\
     =  \sum_{ \ell = 1 }^m \langle K_{ W, \varepsilon_{ \ell } Q, y }, K_{ W, \varepsilon_{ \ell } Q, y } \rangle 
     =  \sum_{ \ell = 1 }^m \norm{ K_{ W, \varepsilon_{ \ell } Q, y } }^2.
\end{multline*}
Since $ K ( W, W ) ( P ) = 0 $, it follows from the equation above that $ K_{ W, \varepsilon_{ \ell } Q, y } = 0 $ for all $ \ell = 1, \hdots, m $. Now applying (4), we see that $ Q = 0 $.

\noindent (1) $ \Leftrightarrow $ (4): Since $ \ev^L_W : \CH_m \to ( \C^{ m \times m } )^{ 1 \times r } $ is subjective, observe that its adjoint mapping $ ( \ev^L_W )^* : ( \C^{ m \times m } )^{ 1 \times r } \to \CH_m $ defined by $$ ( \ev_W^L )^* ( \boldsymbol{ \SA } ) = \sum_{ s = 1 }^r \SA_s [K ( \cdot, W ) \sigma_s ], \hspace{0.1in} \text{for} ~~ \boldsymbol{ \SA } = \begin{bmatrix} \SA_1 \\ \vdots \\ \SA_r \end{bmatrix} \in ( \C^{ m \times m } )^{ r \times 1 }, $$
is injective. We first show that $ K_{ W, v, y } = 0 $ for all $ y \in \C^{ rm \times 1 } $ only if $ v = 0 $. For $ v \in \C^{ 1 \times m } $, let $ K_{ W, v, y } = 0 $ for all $ y \in \C^{ rm \times 1 } $ and $ \SA \in \C^{ m \times m } $ be such that $ \varepsilon_1 \SA = v $ and $ \varepsilon_j \SA = 0 $ for $ j = 2, \hdots, m $. Then we consider the element $ \boldsymbol{ \SA } : = \SA \otimes \sigma_1 $ (recall that $ \{ \sigma_1, \hdots, \sigma_r \} $ is the standard ordered basis for $ \C^{ r \times 1 } $) and note that 
$$  ( \ev_W^L )^* ( \boldsymbol{ \SA } ) = \SA [K ( \cdot, W ) \sigma_1 ] = \begin{bmatrix}
    K_{ W, \varepsilon_1 \SA, \varepsilon_1^* \otimes \sigma_1 } & \cdots & K_{ W, \varepsilon_1 \SA, \varepsilon_m^* \otimes \sigma_1 } \\
    0 & \cdots & 0 \\
    \vdots & \ddots & \vdots \\
    0 & \cdots & 0
\end{bmatrix} = 0, $$
since $ K_{ W, \varepsilon_1 \SA, \varepsilon_j^* \otimes \sigma_1 } = K_{ W, v, \varepsilon_j^* \otimes \sigma_1 } $, for $ j =1, \hdots, m $, and $ K_{ W, v, y } = 0 $ for all $ y \in \C^{ rm \times 1 } $. Consequently, from the injectivity of $ ( \ev^L_W )^* $ we see that $ \SA = 0 $ implying that $ v = 0 $.

Conversely, we prove that $ ( \ev^L_W )^* $ is injective assuming that $ \bigcap_{ f \in \CH_K } \ker f ( W ) = \{ 0 \} $. So let $ \boldsymbol{ \SA } \in ( \C^{ m \times m } )^{ r \times 1 } $ and $  ( \ev^L_W )^*  ( \boldsymbol{ \SA } ) = 0 $. Then, for any $ F = f \otimes I_m \in \CH_m $ with $ f = \sum_{ s = 1 }^r f_s \otimes \sigma_s $, we have from the computation in the proof of part (2) in Proposition \ref{left evaluation} that 
$$ \langle F, ( \ev^L_W )^*  ( \boldsymbol{ \SA } ) \rangle = \sum_{ s = 1 }^r f_s ( W ) A_s^* = 0 . $$
In other words, the range of the adjoint of linear mapping $ \begin{bmatrix} A_1 & \cdots & A_r \end{bmatrix} : \C^{ mr } \to \C^m $ is contained in the kernel of the mapping $ \begin{bmatrix} f_1 ( W ) & \cdots & f_r ( W ) \end{bmatrix} : \C^{ mr } \to \C^m $, for every $ f \in \CH ( K ) $. Since $ \bigcap_{ f \in \CH_K } \ker f ( W ) = \{ 0 \} $, it follows that $ A_j = 0 $ for $ 1 \leq j \leq m $. It thus concludes that $ ( \ev^L_W )^* $ is injective. 
\end{proof}

\subsubsection{Right Hilbert $ C^* $ module structures on $ \CH_m $} Define the \textbf{right generalized evaluation functional} $ \ev_W^R : \CH_m \to ( \C^{ m \times m } )^{ r \times 1 } $ as follows:
$$ \ev_W^R ( F ) : = \sum_{ j = 1 }^n f_j ( W ) \SA_j , $$
where $ F = \sum_{ j = 1 }^n f_j \otimes \SA_j $ for some $ f_1, \hdots, f_n \in \CH ( K ) $ and $ \SA_1, \hdots, \SA_n \in \C^{ m \times m } $. Since elements in $ \CH ( K ) $ are nc functions on $ \Omega $ taking values in $ \C^r_{ nc } $, we consider, for each $ 1 \leq s \leq r $, the $ s $-th \textbf{right generalized evaluation functional} $ \ev_{W, s}^R : \CH_m \to \C^{ m \times m } $ as
$$ \ev_{W, s}^R ( F ) : = \sum_{ j = 1 }^n f_{ j, s } ( W ) \SA_j , $$
where $ F \in \CH_m $ as above with each $ f_j = ( f_{ j, 1 }, \hdots, f_{ j, r } ) $, and $ f_{ j, s } $ is an uniformly analytic nc function on $ \Omega $ taking values in $ \C_{ nc } $, for $ 1 \leq j \leq n $, $ 1 \leq s \leq r $. As in the preceding subsection, we now demonstrate some elementary properties of the right generalized evaluation functionals.

\begin{prop} \label{right evaluation}
    Let $ \ev_W^R : \CH_m \to ( \C^{ m \times m } )^{ r \times 1 } $ be as above. Then
\begin{enumerate}
    \item Reproducing property: For $ F \in \CH_m $, $ W \in \Omega_m $ and $ 1 \leq s \leq r $,
    $$ \langle F, K ( \cdot, W ) \sigma_s \rangle_R = \ev_{ W, s }^R ( F ) . $$
    \item For $ \boldsymbol{ \SA } = \begin{bmatrix} \SA_1 \\ \vdots \\ \SA_r \end{bmatrix} \in ( \C^{ m \times m } )^{ r \times 1 } $, 
    $ ( \ev_W^R )^* ( \boldsymbol{ \SA } ) = \sum_{ s = 1 }^r [K ( \cdot, W ) \sigma_s ] \SA_s .$ In particular, both $ \ev_W^R $ and $ ( \ev_W^R )^* $ are right module maps over $ \C^{ m \times m } $.
     \item $ \ev_W^R \circ ( \ev_W^R )^* : ( \C^{ m \times m } )^{ r \times 1 } \to ( \C^{ m \times m } )^{ r \times 1 } $ is the linear operator given by the formula:
    $$  \ev_W^R \circ ( \ev_W^R )^* ( \boldsymbol{ \SA } )  = K ( W, W ) ( I_m ) \boldsymbol{ \SA }  : = \begin{bmatrix}
       \sum_{ s = 1 }^r K_{ 1 s } ( W, W ) ( I_m ) \SA_s \\ \vdots \\ \sum_{ s = 1 }^r K_{ r s } ( W, W ) ( I_m ) \SA_s
       \end{bmatrix} $$
\end{enumerate}
\end{prop}

\begin{proof}
    Let $ f = ( f_1, \hdots, f_r ) \in \CH ( K ) $ and $ \SA \in \C^{ m \times m } $ so that $ f \otimes \SA \in \CH_m $. Then we compute, for $ 1 \leq s \leq r $ and $ W \in \Omega_m $,
    \begin{align*}
         &\langle f \otimes \SA, K ( \cdot, W ) \sigma_s \rangle_R = \sum_{ i, j = 1 }^m \big\langle f \otimes \SA, K_{ W, \varepsilon_i, \varepsilon_j^* \otimes \sigma_s } \otimes \varepsilon_i^* \varepsilon_j \big\rangle_R \\
         &= \sum_{ i, j = 1 }^m \big\langle f, K_{ W, \varepsilon_i, \varepsilon_j^* \otimes \sigma_s } \big\rangle_{ \CH ( K ) } \varepsilon_j^* \varepsilon_i \SA 
         = \sum_{ \ell, k = 1 }^m \sum_{ i, j = 1 }^m \big\langle f ( W ) ( \varepsilon_i^* ), \varepsilon_j^* \otimes \sigma_s \big\rangle_{ \C^{ m r \times 1 } } a_{ \ell k } \varepsilon_j^* \varepsilon_i \varepsilon_{ \ell }^* \varepsilon_k \\
        & = \sum_{ k = 1 }^m \sum_{ i, j = 1 }^m \big\langle f ( W ) ( \varepsilon_i^* ), \varepsilon_j^* \otimes \sigma_s \big\rangle_{ \C^{ m r \times 1 } } a_{ i k } \varepsilon_j^* \varepsilon_k 
         = \sum_{ j, k = 1 }^m \big\langle f ( W ) \big( \sum_{ i = 1 }^m a_{ i k } \varepsilon_i^* \big), \varepsilon_j^* \otimes \sigma_s \big\rangle_{ \C^{ m r \times 1 } } \varepsilon_j^* \varepsilon_k \\
        & = \sum_{ j, k = 1 }^m \big\langle f ( W ) ( \SA \varepsilon_k^* ), \varepsilon_j^* \otimes \sigma_s \big\rangle_{ \C^{ m r \times 1 } } \varepsilon_j^* \varepsilon_k
         = \sum_{ j, k = 1 }^m \big\langle f ( W ) \SA ( \varepsilon_k^* ), \varepsilon_j^* \otimes \sigma_s \big\rangle_{ \C^{ m r \times 1 } } \varepsilon_j^* \varepsilon_k \\
         &= f_s ( W ) \SA = \ev_{ W, s }^R ( f \otimes \SA ).
    \end{align*}
    This completes the proof of part (1).

    For the proof of (2), let $ \boldsymbol{ \SA } \in ( \C^{ m \times m } )^{ r \times 1 } $ be as given and $ F = \sum_{ i = 1 }^n f_i \otimes \SB_i $ with $ \SB_1, \hdots, \SB_n \in \C^{ m \times m } $ and each $ f_i = \sum_{ s = 1 }^r f_{ i, s } \otimes \sigma_s $. Then we compute
    \begin{multline*}
        \langle \ev_W^R ( F ), \boldsymbol{ \SA } \rangle = \sum_{ i = 1 }^n \sum_{ s = 1 }^r \langle f_{ i, s } ( W ) \SB_i, \SA_s \rangle 
         =  \sum_{ i = 1 }^n \sum_{ s = 1 }^r \SA_s^* f_{ i, s } ( W ) \SB_i 
         =  \sum_{ s = 1 }^r \SA_s^* \big( \sum_{ i = 1 }^n f_{ i, s } ( W ) \SB_i \big) \\
         =  \sum_{ s = 1 }^r \SA_s^* \ev_{ W, s }^R ( F ) 
         =  \sum_{ s = 1 }^r \SA_s^* \langle F, K ( \cdot, W ) \sigma_s \rangle_R 
        %& = & \sum_{ s = 1 }^r \langle F, K ( \cdot, W ) \sigma_s \SA_s \rangle_R \\
         =  \big\langle F, \sum_{ s = 1 }^r K ( \cdot, W ) \sigma_s \SA_s \big\rangle_R,
    \end{multline*}
    where the second last equality holds since $ \langle \cdot, \cdot \rangle_R $ defines a right Hilbert $ C^* $ module structure on $ \CH_m $. Consequently, we have
    $$ ( \ev_W^R )^* ( \boldsymbol{ \SA } ) = \sum_{ s = 1 }^r [K ( \cdot, W ) \sigma_s ] \SA_s . $$

    Finally, for $ m \in \N $ and $ W \in \Omega_m $, the computation below proves the claim in part (3).

    \begin{align*}
        & \big\langle \ev_W^R \circ ( \ev_W^R )^* ( \boldsymbol{ \SA } ), I_m \otimes \sigma_t \big\rangle = \big\langle ( \ev_W^R )^* ( \boldsymbol{ \SA } ), ( \ev_W^R )^* ( I_m \otimes \sigma_t ) \big\rangle_R \\
        & = \sum_{ s = 1 }^r \big\langle [K ( \cdot, W ) \sigma_s ] \SA_s, [K ( \cdot, W ) \sigma_t ] I_m \big\rangle_R = \sum_{ s = 1 }^r \big\langle K ( \cdot, W ) \sigma_s, K ( \cdot, W ) \sigma_t \big\rangle_R \SA_s \\
        & = \sum_{ s = 1 }^r \sum_{ i, j = 1 }^m \sum_{ \ell, k = 1 }^m \langle K_{ W, \varepsilon_i, \varepsilon_j^* \otimes \sigma_s }, K_{ W, \varepsilon_{ \ell }, \varepsilon_k^* \otimes \sigma_t } \rangle_{ \CH ( K ) } \varepsilon_k^* \varepsilon_{ \ell } \varepsilon_i^* \varepsilon_j \SA_s \\
        & = \sum_{ s = 1 }^r \sum_{ i, j = 1 }^m \sum_{ k = 1 }^m \langle K_{ W, \varepsilon_i, \varepsilon_j^* \otimes \sigma_s }, K_{ W, \varepsilon_i, \varepsilon_k^* \otimes \sigma_t } \rangle_{ \CH ( K ) } \varepsilon_k^* \varepsilon_j \SA_s \\
        & = \sum_{ s = 1 }^r \sum_{ i, j = 1 }^m \sum_{ k = 1 }^m \langle K_{ W, \varepsilon_i, \varepsilon_j^* \otimes \sigma_s } ( W ) ( \varepsilon_i^* ), \varepsilon_k^* \otimes \sigma_t \rangle \varepsilon_k^* \varepsilon_j \SA_s \\
        & = \sum_{ s = 1 }^r \sum_{ i, j = 1 }^m \sum_{ k = 1 }^m \langle K ( W, W ) ( \varepsilon_i^* \varepsilon_i ) ( \varepsilon_j^* \otimes \sigma_s ), \varepsilon_k^* \otimes \sigma_t \rangle \varepsilon_k^* \varepsilon_j \SA_s \\
        & = \sum_{ s = 1 }^r \sum_{ j, k = 1 }^m \langle K ( W, W ) ( I_m ) ( \varepsilon_j^* \otimes \sigma_s ), \varepsilon_k^* \otimes \sigma_t \rangle \varepsilon_k^* \varepsilon_j \SA_s\\
        & = \sum_{ s = 1 }^r \sum_{ j, k = 1 }^m \langle K ( W, W )_{ t s } ( I_m ) \varepsilon_j^*, \varepsilon_k^* \rangle \varepsilon_k^* \varepsilon_j \SA_s = \sum_{ s = 1 }^r K ( W, W )_{ t s } ( I_m ) \SA_s.
    \end{align*}
\end{proof}

\begin{cor}
    For any $ \SA, \SB \in \C^{ m \times m } $, $ W \in \Omega_m $ and $ 1 \leq s, t \leq r $, 
    $$ \langle \SA K ( \cdot, W ) \sigma_s, \SB K ( \cdot, W ) \sigma_t \rangle_R = K_{ t s } ( W, W ) ( \SB^* \SA )  $$
    where $ K_{ t s } ( W, W ) ( \SB^* \SA ) $ is the $ ( t, s ) $-th block of the block matrix $ K ( W, W ) ( \SB^* \SA ) $.
\end{cor}

\begin{proof}
    The proof follows directly from the proof of part (3) in the preceding proposition.
\end{proof}

In the following theorem, we obtain certain strict positivity criteria of the nc cp kernel $ K $ in terms of the right generalized evaluation functional.

\begin{thm} \label{equivalent criterion for strict positivity}
    The following are equivalent for $ m \in \N $ and $ W \in \Omega_m $.
    \begin{enumerate}
        \item $ \ev^R_W : \CH_m \to ( \C^{ m \times m } )^{ 1 \times r } $ is a surjective.
        \item $ K ( W, W ) ( I_m ) \in ( \C^{ m \times m } )^{ r \times r } $ is an invertible matrix.
        \item $ K ( W, W ) : \C^{ m \times m } \to ( \C^{ m \times m } )^{ r \times r } $ satisfies: For $ P \in \C^{ m \times m } $, 
        $$ P > 0 \hspace{0.1in} \text{implies} \hspace{0.1in} K ( W, W ) ( P ) > 0 . $$
        \item $ \CH_K $ has the following property:
        $$ \bigcap_{ f \in \CH_K } \ker f ( W )^* = \{ 0 \} . $$
    \end{enumerate}
\end{thm}

\begin{proof}
    (1) $ \Leftrightarrow $ (2): This equivalence follows from part (3) of Proposition \ref{right evaluation}. 

    (2) $ \Leftrightarrow $ (3): Note that (3) trivially implies (2). So we prove $ K ( W, W ) ( P ) > 0 $ with $ P > 0 $, assuming that $ K ( W, W ) ( I_m ) $ is an invertible matrix. For $ P > 0 $, we have that
    $$ P \geq \norm{ P^{ - 1 } }^{ - 1 } I_m > 0, $$
    since being strictly positive, $ P $ is invertible. Consequently, 
    $$ K ( W, W ) ( P ) \geq K ( W, W ) \big( \norm{ P^{ - 1 } }^{ - 1 } I_m \big) = \norm{ P^{ - 1 } }^{ - 1 } K ( W, W ) ( I_m ) > 0 . $$

    (2) $ \Leftrightarrow $ (4): We first prove that the matrix $ K ( W, W ) ( I_m ) $ is an invertible if and only if 
    $$ \big\{ y \in \C^{ mr \times 1 } : K_{ W, \varepsilon_j, y } = 0, ~ 1 \leq j \leq m \big\} = \{ 0 \} . $$
    If $ K_{ W, \varepsilon_j, y } = 0 $ for all $ 1 \leq j \leq m $, then for any $ x \in \C^{ mr \times 1 } $, we have that
    $$ \langle K ( W, W ) ( I_m ) y, x \rangle_{ \C^{ mr \times 1 } } = \sum_{ j = 1 }^m \langle K_{ W, \varepsilon_j, y }, K_{ W, \varepsilon_j, x } \rangle_{ \CH ( K ) } = 0 . $$
    Consequently, $ K ( W, W ) ( I_m ) y = 0 $, implying that $ y = 0 $ since $ K ( W, W ) ( I_m ) $ is invertible.

    On the other hand, if $ K ( W, W ) ( I_m ) y = 0 $ for some $ y \in \C^{ mr \times 1 } $, then 
    $$ \sum_{ j = 1 }^m \langle K_{ W, \varepsilon_j, y }, K_{ W, \varepsilon_j, y } \rangle_{ \CH ( K ) } = \langle K( W, W ) ( I_m ) y, y \rangle_{ \C^{ mr \times 1 } } = 0 . $$
    But this implies that $ \norm{ K_{ W, \varepsilon_j, y } } = 0 $ for all $ 1 \leq j \leq m $. Therefore, $ y = 0 $. 

    Finally, we observe the following realization
    $$ \big\{ y \in \C^{ mr \times 1 } : K_{ W, \varepsilon_j, y } = 0, ~ 1 \leq j \leq m \big\} = \bigcap_{ f \in \CH_K } \ker f ( W )^* . $$
    Indeed, the condition, $ K_{ W, \varepsilon_j, y } = 0 $ for all $ 1 \leq j \leq m $, is equivalent to the fact that
    $$ \langle f ( W ) ( \varepsilon_j^* ), y \rangle_{ \C^{ mr \times 1 } } = \langle f, K_{ W, \varepsilon, y } \rangle_{ \CH ( K ) } = 0, $$
    for all $ 1 \leq j \leq m $ and $ f \in \CH ( K ) $. In other words, $ y \perp \text{range} f ( W ) $, for all $ f \in \CH ( K ) $, verifying the claim.
\end{proof}

%\subsection{Hilbert $ C^* $ modules of matrices over an nc RKHS} 

\section{Noncommutative Cowen-Douglas class and vector bundles} \label{Noncommutative Cowen-Douglas class and vector bundles}
In this section, we introduce a class of noncommuting tuples of bounded linear operators on a complex separable Hilbert space that exhibit properties reminiscent of those in the classical Cowen-Douglas class. We refer to this class as the noncommutative (nc) Cowen-Douglas class. We then explore various examples that illustrate and justify our definition.

Next, we define noncommutative (nc) Hermitian holomorphic vector bundles over a nc domain and, using Theorem \ref{closed range operators general version}, establish that every tuple in the nc Cowen-Douglas class naturally gives rise to such a nc vector bundle. Furthermore, the unitary invariants of operator tuples in this class are entirely determined by the invariants of the associated nc vector bundles. This connection paves the way for a deeper study of invariants of nc vector bundle, which we wish to pursue in subsequent work.

\subsection{Noncommutative Cowen-Douglas class} \label{nc CD class}

We begin by fixing some notations. Recall that $ \CH $ is a complex separable Hilbert space. Throughout this subsection we let $ \CV = \C^d $ endowed with an operator space structure and $ \Omega \subset \C^d_{ nc } $ be a nc domain in $ \C^d_{ nc } $. 

Suppose $ \boldsymbol{ T } = ( T_1, \hdots, T_d ) $ is a $ d $-tuple of bounded linear operators  on $ \CH $. For $ m \in \N $ and $ W = (W_1, \hdots, W_d) \in \Omega_m $, we define the associated operator 
$$ D_{ \boldsymbol{T} - W } : \CH^{ m \times m} \rightarrow \CH^{ m \times m } \oplus \cdots \oplus \CH^{ m \times m } $$ by 
\be \label{D_T} 
D_{ \boldsymbol{T} - W } := ( T_1 \otimes \mathrm{Id}_{ \C^{ m \times m } } - \mathrm{Id}_{ \CH } \otimes R_{ W_1 }, \hdots, T_d \otimes \mathrm{Id}_{ \C^{ m \times m } } - \mathrm{Id}_{ \CH } \otimes R_{ W_d } ) . 
\ee
Note that the kernel of $ D_{ \boldsymbol{T} - W } $ is nothing else but the joint kernel of the operators $ T_1 \otimes \mathrm{Id}_{ \C^{ m \times m } } - \mathrm{Id}_{ \CH } \otimes R_{ W_1 }, \hdots, T_d \otimes \mathrm{Id}_{ \C^{ m \times m } } - \mathrm{Id}_{ \CH } \otimes R_{ W_d } $.

\begin{defn} \label{Definition of nc CD class}
A $ d $-tuple $ \boldsymbol{ T } = ( T_1, \hdots, T_d ) $ of  bounded linear operators  on $ \CH $ is said to be in the noncommutative (nc) Cowen-Douglas class $ \mathrm B_r ( \Omega )_{nc} $ of rank $ r $ over $ \Omega $ if for all $ m \in \N $ and $ W = (W_1, \hdots, W_d) \in \Omega_m $, the operators  $ D_{ \boldsymbol{T} - W } $ satisfy the following three conditions:
\begin{itemize}
\item[(i)] $ \dim_{ \C } ( \ker D_{ \boldsymbol{T} - W } )= r m^2 $,
\item[(ii)] $ \overline{ \bigvee_{ W \in \Omega_m, m \in \N} \left\{ h_{ i j } : ( \! ( h_{ i j } ) \! )_{ i, j = 1 }^m \in \ker D_{ \boldsymbol{T} - W } \right\} } = \CH $,
\item[(iii)] range of $ D_{ \boldsymbol{T} - W } $ is closed in $ \CH^{ m \times m } \oplus \cdots \oplus \CH^{ m \times m } $. 
\end{itemize}
\end{defn}

\begin{rem}
(i) Note that unlike the classical Cowen-Douglas class over a domain in $ \C^m $, the tuples of operators in the nc Cowen-Douglas class do not have the same dimensional joint eigenspaces. However, they do have a constant dimensional kernel of the associated operator $ D_{ \boldsymbol{T} - W } $ associated to all $ W \in \Omega_m $ for each fixed $ m \in \N $, which is essentially due to the free noncommutative phenomenon. 

(ii) Also, observe that for any bounded domain $ \CD \subset \C^d $, every $ d $-tuple of operators in the classical Cowen-Douglas class of rank $ r $ over $ \CD $ induces an operator tuple in the nc Cowen-Douglas class of rank $ r $ over the nc domain $ \Omega = \coprod_{ n = 1 }^{ \infty } \CD^{ \oplus n } $, where $ \CD^{ \oplus n } = \{ z^1 \oplus \cdots \oplus z^n : z^1, \hdots, z^n \in \CD \} $ and $ z \oplus w $ is the $ d $-tuple of $ 2 \times 2 $ diagonal matrix $ \left[ \begin{smallmatrix} z & 0 \\ 0 & w \end{smallmatrix} \right] $.
\end{rem}

We now discuss some non-trivial examples of tuples of operators in the nc Cowen-Douglas class which naturally appear in the study of noncommutative function theory on the noncommutative balls and poly-discs. We call the unit ball in the row operator Hilbert space $ \C^d $ as nc unit ball in $ \C^d_{ nc } $ which is, by definition,
$$ \B^d_{ nc } : = \coprod_{ m = 1 }^{ \infty } \left\{ ( X_1, \hdots, X_d ) \in \left( \C^{ m \times m } \right)^d : \sum_{ i = 1 }^d X_i X_i^* < I_m \right\} . $$
Similarly, we define the nc unit poly-disc in $ \C^d_{ nc } $ to be the unit ball in $ \C^d $ equipped with the maximum norm operator space structure -- $ \| ( X_1, \hdots, X_d ) \|_{ \infty } = \mathrm{max} \{ \| X_1 \|, \hdots, \| X_d \| \} $, $ \| \cdot \| $ is the usual matrix norm on $ \C^{ m \times m } $, $ m \in \N $ -- namely,
$$ \D^d_{ nc } : = \coprod_{ m = 1 }^{ \infty } \left\{ ( X_1, \hdots, X_d ) \in \left( \C^{ m \times m } \right)^d : \| X_i \| < 1,~ i = 1, \hdots, d \right\} . $$

\begin{ex} \label{nc Hardy space}
Let $ \CG^d $ be the free monoid (unital semi-group) of all words in the $ d $ letters $ \{ 1, \hdots, d \} $ with the empty word $ \emptyset $ (the word containing no letters) as the unit. For a noncommuting indeterminate $ x = ( x_1, \hdots, x_d ) $ and a word $ \alpha = \alpha_1 \cdots \alpha_k \in \CG^d $, set $ x^{ \alpha } = x_{ \alpha_1 } \cdots x_{ \alpha_k } $. Also, write $ | \alpha | $ for the length of the word $ \alpha \in \CG^d $. Then we define the \textit{noncommutative $ \ell^2 $-space} $ \ell^2 ( \CG^d ) $ over $ \CG^d $ as 
$$ \ell^2 ( \CG^d ) : = \bigg\{ f : \CG^d  \rightarrow \C : \| f \|_2^2 : = \sum_{ \alpha \in \CG^d } | f ( \alpha ) |^2 < + \infty \bigg\} . $$
As in the commutative case $ \ell^2 ( \CG^d ) $ turns out to be a Hilbert space with the hermitian product 
$$ \langle f, g \rangle_{ \ell^2 ( \CG^d ) } : = \sum_{ \alpha \in \CG^d } f ( \alpha ) \overline{ g ( \alpha ) } . $$
The set $ \{ e_{ \alpha } : \alpha \in \CG^d \} $ then turns out to be the canonical orthonormal basis for $ \ell^2 ( \CG^d ) $, where $ e_{ \alpha } : \CG^d \rightarrow \C $ is defined as $ e_{ \alpha } ( \beta ) = \delta_{ \alpha \beta } $. The \textit{right creation operators} $ R_1, \hdots, R_d $ are the linear operators on $ \ell^2 ( \CG^d ) $ acting on the orthonormal basis $ \{ e_{ \alpha } : \alpha \in \CG^d \} $ as follows:
$$ R_j ( e_{ \alpha } ) = e_{ \alpha j }, \quad 1 \leq j \leq d,~~~ \alpha \in \CG^d . $$
Note that the action of their adjoints $ R_1^*, \hdots, R_d^* $ on $ \{ e_{ \alpha } : \alpha \in \CG^d \} $ become
$$ R_j^* ( e_{ \alpha_1 \cdots \alpha_k } ) = \left\{ \begin{array}{lll} e_{ \alpha_1 \cdots \alpha_{ k - 1 } } & \mbox{if} & \alpha_d = j \\ 0 &  & \mbox{otherwise} \end{array} \right. , ~~~ 1 \leq j \leq d, ~~~ \alpha = \alpha_1 \cdots \alpha_k \in \CG^d, k \in \N . $$

For $ n \in \N $ and $ X = ( X_1, \hdots, X_d ) \in ( \B^d_{ nc } )_n $, it can be seen that 
$$ \{ X^{ \alpha } \}_{ \alpha \in \CG^d } : = \sum_{ \alpha \in \CG^d } e_{ \alpha } \otimes X^{ \alpha } \in \left( \ell^2 ( \CG^d ) \right)^{ n \times n } .$$ 
Consequently, for each $ 1 \leq j \leq d $, we have 
$$ R_j^* \otimes \mathrm{Id}_{ \C^{ n \times n } } \bigg( \sum_{ \alpha \in \CG^d } e_{ \alpha } \otimes X^{ \alpha } \bigg) = \sum_{ \alpha \in \CG^d } e_{ \alpha } \otimes X^{ \alpha } X_j . $$
A straightforward computation further reveals that 
\be \label{joint kernel of right shift} 
\ker ( D_{ \boldsymbol{R}^* - X } ) = \bigg\{ \SA \bigg( \sum_{ \alpha \in \CG^d } e_{ \alpha } \otimes X^{ \alpha } \bigg)  : \SA \in \C^{ n \times n } \bigg\}, 
\ee
where $ \boldsymbol{ R }^* =( R^*_1, \hdots, R^*_d ) $. Thus, for $ X \in ( \B^d_{ nc } )_n $ and $ n \in \N $, it confirms 
$$ \dim \ker ( D_{ \boldsymbol{ R }^* - X } ) = n^2 . $$  
Moreover, the operator $ D_{ \boldsymbol{ R }^* - X } $ has closed range for all $ X \in ( \B^d_{ nc } )_n $ and $ n \in \N $. The $ d $-tuple of operators $ \boldsymbol{ R }^* =( R^*_1, \hdots, R^*_d ) $ thus satisfies all three condition of Definition \ref{Definition of nc CD class} verifying that $ \boldsymbol{R}^* \in \mathrm B_1 ( \B^d_{ nc } )_{ nc }  $.

The Hilbert space $ \ell^2 ( \CG^d ) $ can also be viewed as Full Fock space $ \CH^2 ( \B^d_{ nc } ) $ over $ \B^d_{ nc } $, which by definition is 
$$ \CH^2 ( \B^d_{ nc } ) : = \bigg\{ f : \B^d_{ nc } \rightarrow \C_{ nc } : f ( X ) = \sum_{ \alpha \in \CG^d } f_{ \alpha } X^{ \alpha }, \{ f_{ \alpha } \}_{ \alpha \in \CG^d } \in \ell^2 ( \CG^d ) \bigg\} . $$
The inner product on $  \CH^2 ( \B^d_{ nc } ) $ is given by the formula: For $ f, g \in  \CH^2 ( \B^d_{ nc } ) $ with $ f ( X ) = \sum_{ \alpha \in \CG^d } f_{ \alpha } X^{ \alpha } $ and $ g ( X ) = \sum_{ \alpha \in \CG^d } g_{ \alpha } X^{ \alpha } $ on $ \B^d_{ nc } $,
$$ \langle f, g \rangle_{ \CH^2 ( \B^d_{ nc } ) } : =  \sum_{ \alpha \in \CG^d } f ( \alpha ) \overline{ g ( \alpha ) } . $$
With this inner product, $ \CH^2 ( \B^d_{ nc } ) $ turns out to be a nc reproducing kernel Hilbert space (cf. Theorem \ref{equivalent condition for cp nc kernel}) of uniformly analytic nc functions on $ \B^d_{ nc } $ with the cp nc reproducing kernel $$ K :  \B^d_{ nc } \times \B^d_{ nc } \rightarrow \coprod_{ n, m = 1 }^{ \infty }\CL ( \C^{ n \times m }, \C^{ n \times m } ) $$ defined, for any $ Z \in ( \B^d_{ nc } )_n $, $ W \in ( \B^d_{ nc } )_m $ and $ \SP \in \C^{ n \times m } $, by
$$ K ( Z, W ) ( P ) : = \sum_{ \alpha \in \CG^d } Z^{ \alpha } P ( W^{ \alpha } )^* . $$
In this set up, it is seen that the vectors $  \{ X^{ \alpha } \}_{ \alpha \in \CG^d } \in \ell^2 ( \CG^d ) $ defined above correspond to the \textit{generalized kernel elements} in $ \CH^2 ( \B^d_{ nc } )_{ nc } $ (cf. Section \ref{coordinate free presentation}). Moreover, the $ d $-tuple $ ( M_{ Z_1 }^*, \hdots, M_{ Z_d }^* ) $ of adjoint of the left multiplication operators by the nc coordinate functions on  $ \CH^2 ( \B^d_{ nc } )_{ nc } $ becomes the model for the tuple $ \boldsymbol{ R }^* $ as described in Subsection \ref{model}.
\end{ex} 

\begin{ex}
Other important examples of noncommuting tuples of operators in the nc Cowen-Douglas class are the $ d $-tuple $ \boldsymbol{ M }^* = ( M_{ Z_1 }^*, \hdots, M_{ Z_d }^* ) $ of the adjoint of left multiplication operators $ M_{ Z_1 }, \hdots, M_{ Z_d } $ by the nc coordinate functions on nc Hardy space over $ \CB_{ \B^d }$ and $ \CB_{ \D^d } $ (cf. \cite{Popa-Vinnikov}). In particular, it can be shown that $ \boldsymbol{ M }^* $ belongs to $ \mathrm B_1 \big( \frac{ 1 }{ \sqrt{ d } } \B^d_{ nc } \big)_{ nc } $ and $ \mathrm B_1 \big( \frac{ 1 }{ \sqrt{ d } } \D^d_{ nc } \big)_{ nc } $, respectively (cf. \cite[Remark 3.10, pp 960]{Popa-Vinnikov}).
\end{ex}

\begin{rem}
From Equation \eqref{joint kernel of right shift} it is evident that the kernel of the operator $ D_{ \boldsymbol{ R }^* - X } $ with $ X \in ( \B^d_{ nc } )_n $ is not only a complex vector space of dimension $ n^2 $, but also a free left module of rank $ 1 $ over the ring $ \C^{ n \times n } $. 
\end{rem}

The property of being a left module over $ \C^{ n \times n } $ is not an isolated phenomenon for the kernel of $ D_{ \boldsymbol{ R }^* - X } $ with $ X \in ( \B^d_{ nc } )_n $. Indeed, for any $ d $-tuple of operators $ \boldsymbol{ T } = ( T_1, \hdots, T_d ) \in \mathrm B_r ( \Omega )_{ nc } $ and $ X \in \Omega_n $, the kernel of $ D_{ \boldsymbol{ T } - X } $ turns out to be free left module of rank $ r $ over $ \C^{ n \times n } $ as is proved below (see Corollary \ref{joint kernel}). In accomplishing our goal we recall the following non-trivial result in the study of spaces of matrices of low rank.

\begin{thm}[Flanders' theorem] \cite[Theorem 1]{Flanders}\label{Flanders' theorem}
Let $ n, p, r \in \N $. If all matrices in a subspace of $ \C^{ n \times p } $ have rank less than or equal to $ r $ then the subspace has dimension less than or equal to $ nr $. 
\end{thm}

\begin{prop} \label{application of Flanders' theorem}
Let $ n, r \in \N $, $ \CH $ be a complex separable Hilbert space and $ \CS $ be a left submodule of the module $ \CH^{ n \times n } $ over $ \C^{ n \times n } $. Assume that $ \CS $ has dimension $ n^2 r $ over $ \C $ as a vector subspace of $ \CH^{ n \times n } $. Then $ \CS $ is a free left submodule of $ \CH^{ n \times n } $ of rank $ r $. 
\end{prop}

\begin{proof}
We provide two proofs of this proposition, beginning with the first:

Consider the case when $ r =1 $. Suppose, for every $ H \in \CS \subseteq \CH^{ n \times n } $, that there exists a nonzero matrix $ \SC \in \C^{ n \times n } $ (depending on $ H $) such that 
$$ \SC H = 0. $$ 
Since the entries of all the matrices in $ \CS $ lie in a finite dimensional subspace of $ \CH $, we can pick up a basis thereof $ h_1, \ldots, h_N $ and identify $ \CS $ with a subspace of $ ( \C^{ n \times n } )^N \cong \C^{ n \times nN } $ (where $ ( \SA_1, \ldots, \SA_N ) $ corresponds to $ \SA_1 h_1 + \cdots + \SA_N h_N $). Every matrix in this subspace necessarily has a nontrivial left kernel (given by a nonzero row of the corresponding matrix $ \SC $), implying that it's rank is at most $ n - 1 $. By Theorem \ref{Flanders' theorem}, the dimension of the subspace is at most $ n ( n - 1 ) = n^2 - n < n^2 $. Consequently, there exists $ H \neq 0 $ in $ \CS $ such that $ \SC H = 0 $ if and only if $ \SC = 0 $ for $ \SC \in \C^{ n \times n } $. 

Note that the dimension of the subspace $ \CE = \{ \SC H : \SC \in \C^{ n \times n } \} \subset \CS $ as a complex vector space is $ n^2 $. Indeed, thinking of $ H $ as an element in $ ( \CH^{ \oplus n } )^{ n \times 1 } $ as $ H = ( R_1 \cdots R_n )^{ \top } $, observe that $ \{ B_{ i j } : 1 \leq i, j \leq n \} $ is a basis of $ \CE $ over $ \C $ where for any $ 1 \leq i, j \leq n $, $ B_{ i j } $ is the matrix in $ \CH^{ n \times n } $ whose only non-zero row is the $ i $-th row which is $ R_j $. This implies that $ \CS = \CE  $.

For a general $ r $, we proceed as above to obtain a non-zero element $ H_1 \in \CS $ such that the subspace $ \CE_1 := \{ \SC H_1 : \SC \in \C^{ n \times n } \} $ has complex dimension $ n^2 $. Since $ \C^{ n \times n } $ is a semi-simple ring, so is every $ \C^{ n \times n } $ module, verifying that $ \CS $ is a semi-simple left $ \C^{ n \times n } $ module. Consequently, there exists a left $ \C^{ n \times n } $ submodule $ \CS' $ of $ \CS $ such that $ \CE_1 \oplus \CS' = \CS $ as left $\C^{ n \times n } $ module. But $ \CS' $ is then a vector subspace of $ \CS $ as well and the complex dimension of $ \CS' $ is $ n^2 r - n^2 = n^2 ( r - 1 ) $. Since $ \CS $ is a left semi-simple module over $ \C^{ n \times n } $, so is $ \CS / \CE_1 \cong \CS' $. Therefore, we run this process until the complex dimension of $ \CS' $ becomes $ 0 $ to obtain a decomposition of $ \CS $ as direct sum of $ r $ many left $ \C^{ n \times n } $ submodules $ \CE_1, \hdots, \CE_r $ where $\CE_j := \{ \SC H_j : \SC \in \C^{ n \times n } \} $ for some $ H_j \neq 0 $ in $ \CS $ for $ j = 1, \hdots, r $. Finally, observe that $ \{ H_1, \hdots, H_r \} $ forms a free left module basis for $ \CS $ over $ \C^{ n \times n } $. This completes the proof.

The second proof proceeds as follows: 

Any left $ \C^{ n \times n } $ module is isomorphic to a direct sum of finitely many copies of $ \C^n $, say $ \ell $ many copies. Equivalently, any representation of $ \C^{ n \times n } $ decomposes into a direct sum of copies of the standard representation. Counting the dimension of $ \CS $ as a vector space over $ \C $, we see that 
$$ \CS \simeq \bigoplus_1^{ nr } \C^n \simeq \bigoplus_1^r \bigg( \bigoplus_1^n { \C }^n \bigg). $$ 
Since $ \bigoplus_1^n \C^n $ is isomorphic to $ \C^{n \times n} $, which serves as the left regular representation and is of course a free left $ \C^{ n \times n } $ module of rank 1, it follows that $ \CS $ is a free left module of rank $ r $, completing the proof.
\end{proof}

Applying this result to a tuple of operators in $ \mathrm B_r ( \Omega )_{ nc } $ and using the definition of nc Cowen-Douglas class we have the following corollary.

\begin{cor} \label{joint kernel}
For $ \boldsymbol{T} \in \mathrm B_r ( \Omega )_{nc} $ and $ X \in \Omega_n $, the kernel of the operator $ D_{\boldsymbol{T} - X } : \CH^{ n \times n } \rightarrow ( \CH^{ n \times n } )^{ \oplus d } $ in $ \CH^{ n \times n } $ is a free left submodule of $ \CH^{ n \times n } $ of rank $ r $.
\end{cor}

\subsection{Noncommutative vector bundles}

As mentioned earlier in this section, the objective of the present subsection is to introduce a concept of noncommutative hermitian holomorphic vector bundles over a nc uniformly open subset of $ \C^d_{ nc } $. We begin with the notion of the \textit{trivial noncommutative (nc) vector bundle} over a nc domain.

\begin{defn} \label{trivial nc bundle}
Let $ \Omega \subset \C^d_{ nc } $ be a nc domain and $ \CV $ be an operator space. Then a \textit{trivial noncommutative (nc) vector bundle} $ \mathfrak{ T }_{ \Omega } ( \CV ) $ over $ \Omega $ is a uniformly open subset of $ ( \C^d \times \CV )_{ nc } $ equipped with a uniformly analytic nc surjective mapping $ \pi : \mathfrak{ T }_{ \Omega } ( \CV ) \rightarrow \Omega $ such that $ \pi^{ - 1 } ( Z ) \simeq \CV^{ n \times n } $ both as a vector space as well as a free left $ \C^{ n \times n } $ module for all $ Z \in \Omega_n $ and $ n \in \N $. 

Such a trivial nc vector bundle is denoted by $ \mathfrak{ T }_{ \Omega } ( \CV ) \overset{ \pi }{ \rightarrow } \Omega $, the mapping $ \pi $ is called the canonical projection, $ \Omega $ is the base, and the set $ \pi^{ - 1 } ( Z ) $ is called the fibre of the bundle $ \mathfrak{ T }_{ \Omega } ( \CV ) \overset{ \pi }{ \rightarrow } \Omega $ over $ Z $ denoted by $ \mathfrak{ T }_Z $. The \textit{rank} of a trivial nc vector bundle is defined as the dimension of $ \CV $. When $ \dim \CV = 1 $, we call it as \textit{trivial nc line bundle over $ \Omega $} and denote it by $ \mathfrak{ L }_{ \Omega } $.
The total space of a trivial nc vector bundle $ \mathfrak{ T }_{ \Omega } ( \CV ) \overset{ \pi }{ \rightarrow } \Omega $ is the set 
$ \mathfrak{ T } : = \prod_{ n = 1 }^{ \infty } \Omega_n \times \CV^{ n \times n } $.
\end{defn}

For $ X \in \Omega_n $ and $ Y \in \Omega_m $, we denote the direct sum of the fibres over $ X $ and $ Y $ as 
$$ \mathfrak{ T }_X \oplus \mathfrak{ T }_Y : = \left\{ \bigg( \begin{bmatrix} X & 0 \\ 0 & Y \end{bmatrix} , \begin{bmatrix} A & 0 \\ 0 & B \end{bmatrix} \bigg) : A \in \CV^{ n \times n },~ B \in \CV^{ m \times m } \right\}, $$
and observe that 
$$ \mathfrak{ T }_X \oplus \mathfrak{ T }_Y \overset{ \iota }{\hookrightarrow } \mathfrak{ T }_{ X \oplus Y } = \{ X \oplus Y \} \times \CV^{ ( n + m ) \times ( n + m ) } . $$
Further, for each $ n \in \N $, $ \SS \in \text{GL} ( n, \C ) $ and $ X \in \Omega_n $, whenever $ \SS X \SS^{ - 1 } \in \Omega_n $, the mapping $ \mathfrak{ S } : \mathfrak{ T }_X \rightarrow \mathfrak{ T }_{ \SS X \SS^{ - 1 } } $ defined by
$$ \mathfrak{ S } ( X, A ) = ( S X S^{ - 1 }, S A S^{ - 1 } ) $$
turns out to be an isomorphism. 

We now introduce the notion of noncommutative (nc) vector bundles over 
$ \Omega $. Given that the presheaf of nc functions on an open set is not necessarily a sheaf, we choose to define nc vector bundles as those that are embedded in a trivial nc vector bundle over $ \Omega $. Moreover, we note that in the free noncommutative setting, the notions of local boundedness and analytic functions are equivalent (recall Theorem \ref{locally bounded = nc analytic}). As a result, the only relevant notion in this context is that of holomorphic vector bundles unlike the classical case.

\begin{defn} \label{nc vector bundle}
Given a trivial nc holomorphic vector bundle $ \mathfrak{ T }_{ \Omega } ( \CV ) \overset{ \pi }{ \rightarrow } \Omega $, let $ E $ be a nc subset of $ \mathfrak{ T }_{ \Omega } ( \CV ) $ equipped with the relative uniformly-open topology. Then we say that $ E $ is a \textit{noncommutative holomorphic vector bundle} of rank $ r $ over $ \Omega $ if
\begin{itemize}
\item[(1)]  $ \pi |_{ E } : E \rightarrow \Omega $ is surjective;
\item[(2)] there exists an open cover $ \{ \CU_{ \alpha } \}_{ \alpha \in \Lambda } $ of $ \Omega $ by uniformly open sets $ \CU_{ \alpha } \subset \Omega $ for $ \alpha \in \Lambda $ and uniformly analytic nc bijective function $ \phi_{ \alpha } : \mathfrak{ T }_{ \CU_{ \alpha } } ( \C^r ) \rightarrow ( \pi |_{ E } )^{ - 1 } ( \CU_{ \alpha } ) $ satisfying
\begin{itemize}
\item[(2a)] $ \pi|_E \circ \phi_{ \alpha } = \pi $;
\item[(2b)] for every $ s \in \N $ and $ Z \in \CU_{ \alpha } \cap \C^{ s \times s } $, $ \phi_{ \alpha } $ on the fibre $ \pi^{ - 1 } \{ Z \} $ is linear as well as free left $ \C^{ s \times s } $ module isomorphism onto its image.
\end{itemize}
We denote a nc holomorphic vector bundle by $ E \overset{ \pi }{ \rightarrow } \Omega $. The collection $ \{ ( \CU_{ \alpha }, \phi_{ \alpha } ) \}_{ \alpha \in \Lambda } $ is called a local trivialization of the bundle $ E \overset{ \pi }{ \rightarrow } \Omega $.
\end{itemize}
\end{defn}
Since each $ \phi_{ \alpha } $ is locally uniformly bounded on $ \mathfrak{ T }_{ \CU_{ \alpha } } ( \C^r ) $, it is evident that so is the mapping $ \phi_{ \alpha }^{ - 1 } $ with respect to the relative uniformly-open topology on $ E $. 
%{\color{blue} }
\begin{rem}
Note that the conditions (2a) and (2b) together with Proposition \ref{application of Flanders' theorem} imply that each fibre of a nc vector bundle $ E \overset{ \pi }{ \rightarrow } \Omega $ of rank $ r $ over a point in $ \Omega_n $ is not only a complex vector space of dimension $ n^2 r $ but also it is a free left module of rank $ r $ over the ring $ \C^{ n \times n } $. 

We also point out that existence of local trivialization $ ( \CU_{ \alpha }, \phi_{ \alpha } ) $ is equivalent to the fact that there exist uniformly analytic nc functions $ \gamma_1, \hdots, \gamma_r : \CU_{ \alpha } \rightarrow \CV_{ nc } $ such that for every $ Z \in \CU_{ \alpha } \cap \C^{ s \times s } $, $ \{ \gamma_1( Z ), \hdots, \gamma_r ( Z ) \} $ forms a free left $ \C^{ s \times s } $ module basis of the fibre $ ( \pi |_{ E } )^{ - 1 } \{ Z \} $ over $ Z $. 
\end{rem}

\subsubsection{Sections and homomorphisms of nc vector bundles} \label{section and homomorphism}

Let $ E \overset{ \pi }{ \rightarrow } \Omega $ be a nc holomorphic vector bundle of rank $ r $ embedded in $ \mathfrak{ T }_{ \Omega } ( \CV ) $. Then a \textit{nc section} $ \sigma $ of $ E \overset{ \pi }{ \rightarrow } \Omega $ is a mapping $ \sigma : \Omega \rightarrow E $ satisfying following properties:
\begin{itemize}
\item[(i)] $ \pi \circ \sigma = \mathrm{Id}_{ \Omega } $ where $ \mathrm{Id}_{ \Omega } : \Omega \rightarrow \Omega $ is the nc identity mapping on $ \Omega $;
\item[(ii)] $ \sigma ( X \oplus Y ) = \sigma ( X ) \oplus \sigma ( Y ) $ for all $ X, Y \in \Omega $;
\item[(iii)] $ \sigma ( \SS X \SS^{ - 1 } ) = \mathfrak{ S } (  \sigma ( \SS X \SS^{ - 1 } ) ) = \SS \sigma ( X ) \SS^{ - 1 } $ for all $ X \in \Omega_n $, $ \SS \in \text{GL} ( n, \C ) $ and $ n \in \N $.
\end{itemize} 
A nc section of $ E \overset{ \pi }{ \rightarrow } \Omega $ of the above kind is said to be uniformly analytic if it, viewed as a mapping $ \sigma : \Omega \rightarrow \CV_{ nc } $, is locally bounded.

 A local \textit{uniformly analytic ordered nc frame} of a nc holomorphic vector bundle $ E \overset{ \pi }{ \rightarrow } \Omega $ is an ordered $ r $-tuple $ \sigma_1, \hdots, \sigma_r $ of uniformly analytic nc sections of $ E \overset{ \pi }{ \rightarrow } \Omega $ over an uniformly open subset $ \CU $ of $ \Omega $ such that for every $ Z \in \CU_n $, $ \{ \sigma_1 ( Z ), \hdots, \sigma_r ( Z ) \} $ is free left $ \C^{ n \times n } $ module basis for the fibre $ E_Z : = \pi^{ - 1 } \{ Z \} $ of $ E \overset{ \pi }{ \rightarrow } \Omega $ over $ Z $.

\begin{defn} \label{nc homomorphism}
Suppose that $ E \overset{ \pi }{ \rightarrow } \Omega $ and $ E' \overset{ \pi' }{ \rightarrow } \Omega' $ are two nc holomorphic vector bundles. A \textit{nc vector bundle homomorphism} from $ E $ to $ E' $ is a uniformly analytic nc mapping $ \widetilde{ f } : E \rightarrow E' $ --- that is, $ \widetilde{ f } $ is locally uniformly bounded with respect to the relative uniformly-open topology on $ E $ and $ E' $ --- which descends to a uniformly analytic nc map $ f : \Omega \rightarrow \Omega' $ and satisfies 
$$ \pi' \circ \widetilde{ f } = f \circ \pi , $$
along with the property that the restriction $ \widetilde{ f }|_{ E_Z } : E_Z \rightarrow E_{ f ( Z ) } $ is a homomorphism as complex vector spaces and left modules over $ \C^{ n \times n } $ for all $ Z \in \Omega_n $ and $ n \in \N $.

Further, $ \widetilde{ f } $ is called an isomorphism of nc vector bundles if, in addition, $ f^{ - 1 } $ exists and is locally uniformly bounded with respect to the relative uniformly-open topology on $ E' $ and $ E $, and the restriction $ \widetilde{ f }|_{ E_Z } : E_Z \rightarrow E_{ f ( Z ) } $ is an isomorphism as complex vector spaces and left modules over $ \C^{ n \times n } $ for all $ Z \in \Omega_n $ and $ n \in \N $.
\end{defn}

%$$ \mathfrak{ S } ( X, ( A_1, \hdots, A_r ) ) : = ( \SS A_1 \SS^{ - 1 }, \hdots, \SS A_r \SS^{ - 1 } ) $$
%$$ \pi_1 : \mathfrak{ T } \rightarrow \Omega ~~~ \text{defined by} ~~~ \pi_1 ( X, ( A_1, \hdots, A_r ) ) = X . $$ 
%

%Also following \cite[Chapter 9]{Verbovetskyi-Vinnikov}, recall that the \textit{direct summands extension} of a similarity invariant nc set $ \Omega $ is the nc set 
%$$ \Omega_{ \text{d.s.e} } : = \{ X \in \C^d_{ nc } : X \oplus Y \in \Omega ~~~ \text{for some} ~~~ Y \in \C^d_{ nc } \} . $$
%We call a nc set $ \Omega $ \textit{stable} if $ \Omega = \Omega_{ \text{d.s.e} } $ (see also \cite[Section 11]{Voiculescu-II}). Note that any nc uniformly-open subset $ \CU $ of $ \Omega $ can be assume without loss of generality to be stable. Indeed, since $ \Omega $ is invariant under similarity, the similarity envelope $ \widetilde{ \CU } $ of $ \CU $ must be contained in $ \Omega $ and being stable, $ \Omega $ also contains $ \widetilde{ \CU }_{ \text{d.s.e} } $. In the remaining part of this section, we only consider stable nc uniformly open sets.

\subsubsection{Transition data for nc vector bundles} \label{transition data}

Suppose $ E \overset{ \pi }{ \rightarrow } \Omega $ is a nc holomorphic vector bundle of rank $ r $ embedded in $ \mathfrak{ T }_{ \Omega } ( \CV ) $. From the definition above (Definition \ref{nc vector bundle}) it follows that there exists an open cover $ \{ \CU_{ \alpha } \}_{ \alpha \in \Lambda } $ of $ \Omega $ by uniformly open subsets that trivialize the bundle $ E \overset{ \pi }{ \rightarrow } \Omega $. That is, for each $ \CU_{ \alpha } $, there exists a nc uniformly bi-holomorphic mapping 
$$ \varphi_{ \alpha } : E |_{ \CU_{ \alpha } } \rightarrow \mathfrak{ T }_{ \CU_{ \alpha } } ( \C^r ) $$ 
such that $ \pi \circ \varphi_{ \alpha } =\pi $ on $ E|_{ \CU_{ \alpha } } $. Furthermore, the restriction 
$$ \varphi_{ \alpha }|_{ E_Z } : E_Z \to \mathfrak{ T }_{ \CU_{ \alpha} } ( \C^r )_Z $$ 
to the corresponding fibres over $ Z $ is an isomorphism as complex vector spaces as well as of left modules over $ \C^{ m \times m } $ for all $ m \in \N $ and $ Z \in \CU_{ \alpha } \cap \C^{ m \times m } $. 

For $ \alpha, \beta \in \Lambda $, the transition map
$$ \varphi_{ \alpha \beta } : = \varphi_{ \alpha } \circ \varphi_{ \beta }^{ - 1 } : \mathfrak{ T }_{ \CU_{ \alpha } \cap \CU_{ \beta } } ( \C^r) \rightarrow \mathfrak{ T }_{ \CU_{ \alpha } \cap \CU_{ \beta } } ( \C^r ), $$
is a nc uniformly bi-holomorphic mapping satisfying $ \pi_1 \circ \varphi_{ \alpha \beta } = \pi_1 $, which ensures that $ \varphi_{ \alpha \beta } $ preserves the fibres, mapping 
$$ \{ Z \} \times ( \C^{ n \times n } )^r \to \{ Z \} \times ( \C^{ n \times n } )^r , $$ 
for each $ n \in \N $ and $ Z \in ( \CU_{ \alpha } \cap \CU_{ \beta } ) \cap \C^{ n \times n } $. Moreover, the restriction of $ \varphi_{ \alpha \beta } $ to each fibre induces an isomorphism of vector spaces and free left modules. When $ Z \in ( \CU_{ \alpha } \cap \CU_{ \beta } ) \cap \C^{ n \times n } $, such an isomorphism necessarily takes the form 
$$ ( Z, ( A_1, \hdots, A_r ) ) \mapsto \left( Z, \begin{bmatrix} A_1 & \cdots & A_r \end{bmatrix} g_{ \alpha \beta } ( Z ) \right) \quad \text{for all} \quad A_1, \hdots, A_r \in \C^{ n \times n } $$
for a uniquely determined $ g_{ \alpha \beta } ( Z ) \in \text{GL} ( nr, \C ) $, where $ g_{ \alpha \beta } ( Z ) $ is viewed as $ r \times r $ block matrix with $ n \times n $ blocks acting via right matrix multiplication. This shows that the transition map $ \varphi_{ \alpha \beta } $ is explicitly given by 
$$  ( Z, ( A_1, \hdots, A_r ) ) \overset{ \varphi_{ \alpha \beta } }{ \longmapsto } \left( Z, \begin{bmatrix} A_1 & \cdots & A_r \end{bmatrix} g_{ \alpha \beta } ( Z ) \right), $$
for $ Z \in ( \CU_{ \alpha } \cap \CU_{ \beta } )_n $, $ A_1, \hdots, A_r \in \C^{ n \times n } $, $ n \in \N $. The map $ \varphi_{ \alpha \beta } $ is thus entirely determined by
$$ g_{ \alpha \beta } : \CU_{ \alpha } \cap \CU_{ \beta } \rightarrow  \coprod_{ n = 1 }^{ \infty } \mathrm{GL} ( nr, \C ) . $$ 
Since $ \varphi_{ \alpha \beta } $ is nc uniformly bi-holomorphic mapping, it follows that $ g_{ \alpha \beta } $ is uniformly analytic nc mapping. We refer to these nc mappings as \textit{transition mappings}.

Given a nc holomorphic vector bundle $ E \overset{ \pi }{ \rightarrow } \Omega $ of rank $ r $ with $ E \subset \mathfrak{ T }_{ \Omega } ( \C^r ) $, we thus obtain a cover $ \{ \CU_{ \alpha } \}_{ \alpha \in \Lambda } $ of $ \Omega $ by uniformly open sets and a collection of uniformly analytic nc mappings 
$$ \left\{ g_{ \alpha \beta } : \CU_{ \alpha } \cap \CU_{ \beta } \rightarrow \coprod_{ n = 1 }^{ \infty } \text{GL} ( nr, \C ) \right\}_{ \alpha, \beta \in \Lambda } .$$
These mappings satisfy following properties:
\begin{itemize}
\item[(i)] $ g_{ \alpha \alpha } \equiv \mathrm{Id}_{nc } $, the nc constant function taking the value as the identity matrix of correct size over $ \C^{ r \times r } $, since $ g_{ \alpha \alpha } \equiv \varphi_{ \alpha } \circ \varphi_{ \alpha }^{ - 1 } $ is the identity homomorphism of the nc trivial bundle of rank $ r $ over $ \CU_{ \alpha } $;
\item[(ii)] $ g_{ \alpha \beta } g_{ \beta \alpha } \equiv \mathrm{Id}_{ nc } $, since $ \varphi_{ \alpha \beta } \varphi_{ \beta \alpha } $ is the identity homomorphism of the nc trivial bundle of rank $ r $ over $ \CU_{ \alpha } \cap \CU_{ \beta } $;
\item[(iii)] $ g_{ \alpha \beta } g_{ \beta \gamma } g_{ \gamma \alpha } \equiv  \mathrm{Id}_{ nc } $, since $ \varphi_{ \alpha \beta } \varphi_{ \beta \gamma } \varphi_{ \gamma \alpha } $ is 
the identity homomorphism of the nc trivial bundle of rank $ r $ over $ \CU_{ \alpha } \cap \CU_{ \beta } \cap \CU_{ \gamma } $.
\end{itemize}
Note that the last condition itself derives the other two. So it is enough to have only condition, namely, the last one for getting all of the conditions above. We refer to the last condition -- condition (iii) -- as \textit{nc (\v{C}ech) cocycle condition} for the family $ \{ g_{ \alpha \beta } \} $ of uniformly analytic maps relative to the uniformly open cover $ \{ \CU_{ \alpha } \} $ of $ \Omega $. The family of maps $ \{ g_{ \alpha \beta } \} $ is then called a \textit{nc cocycle} with values in $  \coprod_{ n = 1 }^{ \infty } \text{GL} ( nr, \C ) $ relative to the cover $ \{ \CU_{ \alpha } \} $. 

Next, consider the direct summand similarity extension $ \mathfrak{ T }_{ \Omega } ( \CV )_{ \text{d.s.e} } $ of the trivial nc holomorphic bundle $ \mathfrak{ T }_{ \Omega } ( \CV ) $ defined as:
$$ \mathfrak{ T }_{ \Omega } ( \CV )_{ \text{d.s.e} } : = \{ ( X, A ) \in ( \C^d \times \CV )_{ nc } : ( X \oplus Y, A \oplus B ) \in \mathfrak{ T }_{ \Omega } ( \CV ) ~ \text{for some} ~ ( Y, B ) \in ( \C^d \times \CV )_{ nc } \} . $$
Since $ \mathfrak{ T }_{ \Omega } ( \CV ) $ is a uniformly open subset of $ ( \C^d \times \CV )_{ nc } $, so is $  \mathfrak{ T }_{ \Omega } ( \CV )_{ \text{d.s.e} }  $ (cf. \cite[Proposition 9.1]{Verbovetskyi-Vinnikov}). The uniformly analytic nc surjection $ \pi : \mathfrak{ T }_{ \Omega } ( \CV ) \rightarrow \Omega $ uniquely extends to a uniformly analytic nc surjective mapping (cf. \cite[Proposition 9.2 (I), Proposition A.3]{Verbovetskyi-Vinnikov})
$$ \widetilde{ \pi } : \mathfrak{ T }_{ \Omega } ( \CV )_{ \text{d.s.e} } \rightarrow \Omega_{ \text{d.s.e} } . $$
For $ Z \in ( \Omega_{ \text{d.s.e} } )_n $, $ n \in \N $, the fibre $ \widetilde{ \pi }^{ - 1 } ( Z )$ is isomorphic to $ \CV^{ n \times n } $ both as complex vector space and free left $ \C^{ n \times n } $ module, yielding the trivial nc vector bundle
$$ \mathfrak{ T }_{ \Omega } ( \CV )_{ \text{d.s.e} } \overset{ \widetilde{ \pi } }{ \rightarrow } \Omega_{ \text{d.s.e} } . $$
Indeed, for $ Z \in ( \Omega_{ \text{d.s.e} } )_n $, there exist $ m \in \N $ and $ W \in \Omega_m $ such that $ Z \oplus W \in \Omega_{ n + m } $. Consequently, we have the following inclusion:
$$ \CV^{ n \times n } \oplus \CV^{ m \times m } \subset \CV^{ ( n + m ) \times ( n + m ) } \simeq \pi^{ - 1 } ( Z \oplus W ) . $$
This implies that $ \{ Z \oplus W \} \times ( \CV^{ n \times n } \oplus \CV^{ m \times m } ) $ is contained in $ \mathfrak{ T }_{ \Omega } ( \CV ) $ verifying $ \{ Z \} \times \CV^{ n \times n } \subset \mathfrak{ T }_{ \Omega } ( \CV )_{ \text{d.s.e} } $. We record this discussion in the following proposition.

\begin{prop} \label{d.s.e trivial bundle}
    Given the trivial nc holomorphic vector bundle $ \mathfrak{ T }_{ \Omega } ( \CV ) \overset{ \pi }{ \rightarrow } \Omega $, the direct summand similarity extension $ \mathfrak{ T }_{ \Omega } ( \CV )_{ \text{d.s.e} } $ of $ \mathfrak{ T }_{ \Omega } ( \CV ) $ is the trivial nc holomorphic vector bundle over $ \Omega_{ \text{d.s.e} } $. That is, $$ \mathfrak{ T }_{ \Omega } ( \CV )_{ \text{d.s.e} } =  \mathfrak{ T }_{ \Omega_{ \text{d.s.e} } } ( \CV ) $$
\end{prop}

We now observe in the proposition below that the direct summand similarity extension $ E_{ \text{d.s.e} } $ of the nc subset $ E \subset \mathfrak{ T }_{ \Omega } ( \CV ) $ turns out to be a nc holomorphic vector bundle of rank $ r $ over $ \Omega_{ \text{d.s.e} } $ which we define to be the \emph{direct summand similarity extension bundle} of the given nc vector bundle $ E \overset{ \pi }{ \rightarrow } \Omega $. 

\begin{prop} \label{d.s.e nc vector bundle}
    Given a nc holomorphic vector bundle $ E \overset{ \pi }{ \rightarrow } \Omega $ of rank $ r $ embedded in $ \mathfrak{ T }_{ \Omega } ( \CV ) $, the direct summand similarity extension $ E_{ \text{d.s.e} } $ of $ E $ is an nc holomorphic vector bundle over $ \Omega_{ \text{d.s.e} } $ of rank $ r $.
\end{prop}

\begin{proof}
    Since $ \pi|_{ E } : E \rightarrow \Omega $ is surjective, so is the induced mapping $ \widetilde{ \pi }|_{ E_{ \text{d.s.e} } } : E_{ \text{d.s.e} } \rightarrow \Omega_{ \text{d.s.e} } $. Let $ \{ ( \CU_{ \alpha }, \phi_{ \alpha } ) \}_{ \alpha \in \Lambda } $ be a local trivialization of $ E \overset{ \pi }{ \rightarrow } \Omega $. By \cite[Proposition 9.1, Proposition 9.2]{Verbovetskyi-Vinnikov}, the collection $ \{ ( \CU_{ \alpha } )_{ \text{d.s.e} }  \}_{ \alpha \in \Lambda } $ becomes an open cover for $ \Omega_{ \text{d.s.e} } $ and, for each $ \alpha \in \Lambda $, the uniformly analytic bijective mapping 
    $$ \phi_{ \alpha } : \mathfrak{ T }_{ \Omega } ( \C^{ r } ) \to ( \pi|_E )^{ - 1 } ( \CU_{ \alpha } ) $$ 
    extends uniquely to the uniformly analytic bijective mapping 
$$ \widetilde{ \phi_{ \alpha } } : \mathfrak{ T }_{ \Omega_{ \text{d.s.e} } } ( \C^{ r } ) \to ( \widetilde{ \pi }|_{ E_{ \text{d.s.e} } } )^{ - 1 } ( ( \CU_{ \alpha } )_{ \text{d.s.e} } ), $$ 
which satisfies $ \widetilde{ \pi }|_{ E_{ \text{d.s.e} } } \circ \widetilde{ \phi_{ \alpha } } = \widetilde{ \pi } $; here $ \widetilde{ \pi } $ is the canonical projection $ \widetilde{ \pi } : \mathfrak{ T }_{ ( \CU_{ \alpha } )_{ \text{d.s.e} } } ( \C^{ r } ) \to ( \CU_{ \alpha } )_{ \text{d.s.e} } $. 

Moreover, from $ \widetilde{ \phi_{ \alpha } } $ in \cite[Proposition 9.2]{Verbovetskyi-Vinnikov}, for any $ Z \in ( \Omega_{ \text{d.s.e} } )_n $, we have 
$$ \widetilde{ \phi_{ \alpha } }\big|_{ \widetilde{ \pi }^{ - 1 } \{ Z \} } \equiv \bigg( \phi_{ \alpha } \big|_{ \pi^{ - 1 } \{ Z \oplus W \} } \bigg) \bigg|_{ ( \C^r )^{ n \times n } \times \{ 0 \} }, $$
for some $ m \in \N $ and $ W \in ( \C^d )^{ m \times m } $. This definition of $ \widetilde{ \phi_{ \alpha } }\big|_{ \widetilde{ \pi }^{ - 1 } \{ Z \} } $ is independent of the choices of $ m $ and $ W $ (see \cite[Proposition 9.2]{Verbovetskyi-Vinnikov}). 

Since $ \phi_{ \alpha } $  is a linear and free left $ \C^{ ( n + m ) \times ( n + m ) } $ module isomorphism on the fibre $ \pi^{ - 1 } \{ Z \oplus W \} $, $ \widetilde{ \phi_{ \alpha } } $ on $ \widetilde{ \pi }^{ - 1 } \{ Z \} $ is also a linear and free left $ \C^{ n \times n } $ module isomorphism. As a consequence, $ \{ ( ( \CU_{ \alpha } )_{ \text{d.s.e} }, \widetilde{ \phi_{ \alpha } } ) \}_{ \alpha \in \Lambda } $ defines a local trivialization of $ E_{ \text{d.s.e} } $ over $ \Omega_{ \text{d.s.e} } $, thereby establishing that $ E_{ \text{d.s.e} } $ is a nc vector bundle over $ \Omega_{ \text{d.s.e} } $ of rank $ r $. This completes the proof.
\end{proof}

Finally, it can be seen that the equivalence classes of nc holomorphic vector bundles up to isomorphism are preserved under the direct summand similarity extensions. In particular, we have the following result.

\begin{thm} \label{d.s.e of vector bundles}
Two nc holomorphic vector bundles $ E \overset{ \pi }{ \rightarrow } \Omega $ and $ F \overset{ \pi }{ \rightarrow } \Omega $ embedded into a trivial nc vector bundle $ \mathfrak{ T }_{ \Omega } ( \CV ) \overset{ \pi }{ \rightarrow } \Omega $ are isomorphism if and only if so are the corresponding direct summand similarity extensions $ E_{ \text{d.s.e} } \overset{ \widetilde{ \pi } }{ \rightarrow } \Omega_{ \text{d.s.e} } $ and $ F_{ \text{d.s.e } } \overset{ \widetilde{ \pi } }{ \rightarrow } \Omega_{ \text{d.s.e} } $ of $ E $ and $ F $, respectively.
\end{thm}

\begin{proof}
    Let $ E \overset{ \pi }{ \rightarrow } \Omega $ and $ F \overset{ \pi }{ \rightarrow } \Omega $ be isomorphic as nc vector bundles. Then there is a bijective nc mapping $ f : E \to F $ such that both $ f $ and $ f^{ - 1 } $ are locally uniformly bounded with respect to the relative uniformly open topologies on $ E $ and $ F $ induced from that on $ ( \C^d \times \CV )_{ nc } $, respectively, and $ f $ satisfies
    $$ \pi \circ f = \pi  $$
    along with the property that the restriction $ f|_{ E_Z } : E_Z \to F_Z $ is an isomorphism as complex vector spaces as well as left modules over $ \C^{ m \times m } $ for $ Z \in \Omega_m $ and $ m \in \N $. 
    
    Following \cite[Proposition 9.2]{Verbovetskyi-Vinnikov}, we now define the bijective nc mapping $ \widetilde{ f } : E_{ \text{d.s.e} } \to F_{ \text{d.s.e} } $ extending the given isomorphism $ f : E \to F $ such that both $ \widetilde{ f } $ and $ \widetilde{ f }^{ - 1 } $ are uniformly analytic. Evidently, $ \widetilde{ f } $ satisfies the identity 
    $$ \widetilde{ \pi } \circ \widetilde{ f } = \widetilde{ \pi } . $$

    Finally, for $ Z \in ( \Omega_{ \text{d.s.e} } )_n $, there exists $ W \in ( \C^d )^{ m \times m } $ for some $ m \in \N $, such that $ Z \oplus W \in \Omega $. But this implies that $ W \in ( \Omega_{ \text{d.s.e} } )_m $ as well. Consequently, we have from the construction of $ E_{ \text{d.s.e} } $ described above that $ E_{ \text{d.s.e} } |_Z \oplus E_{ \text{d.s.e} } |_W \subset E_{ Z \oplus W } $ where $ E_Z $ and $ E_{ \text{d.s.e} } |_Z $ are the fibres of $ E $ and $ E_{ \text{d.s.e} } $ over $ Z $, respectively. Moreover, it follows from the definition of $ \widetilde{ f } $ that
    $$ \widetilde{ f } \big |_{ E_{ \text{d.s.e} } |_Z } \equiv ( f |_{ E_{ Z \oplus W } } ) \big|_{ E_{ \text{d.s.e} } |_Z } . $$
    Therefore, the restriction $ \widetilde{ f } |_{ E_{ \text{d.s.e} } |_Z } : E_{ \text{d.s.e} } |_Z \to F_{ \text{d.s.e} } |_Z $ is an isomorphism as complex vector spaces as well as left modules over $ \C^{ n \times n } $ for $ Z \in ( \Omega_{ \text{d.s.e} } )_n $ and $ n \in \N $.
\end{proof}

In view of Theorem \ref{d.s.e of vector bundles}, one can define an equivalence relation on nc cocycles over the same uniformly open cover $ \{ \CU_{ \alpha } \} $ of $ \Omega $ by saying $ \{ g_{ \alpha \beta } \} $ is \emph{equivalent} or \emph{cohomologous} to another nc cocycle $ \{ h_{ \alpha \beta } \} $ if there exist uniformly analytic functions $ \lambda_{ \alpha } : ( \CU_{ \alpha } )_{\text{d.s.e}} \rightarrow  \coprod_{ n = 1 }^{ \infty } \text{GL} ( nr, \C ) $ satisfying
\begin{equation} \label{cohomologous cocycles}
\widetilde{ g_{ \alpha \beta } } ( Z ) \lambda_{ \beta } ( Z ) = \lambda_{ \alpha } ( Z ) \widetilde{ h_{ \alpha \beta } } ( Z ), \quad Z \in ( \CU_{ \alpha } )_{\text{d.s.e}} \cap ( \CU_{ \beta } )_{\text{d.s.e}}, 
\end{equation}
for every $ \alpha, \beta \in \Lambda $ such that $ ( \CU_{ \alpha } )_{\text{d.s.e}} \cap ( \CU_{ \beta } )_{\text{d.s.e}} \neq \emptyset $, where $ \{ \widetilde{ g_{ \alpha \beta } } \} $ and $ \{ \widetilde{ h_{ \alpha \beta } } \} $ are nc cocycles with respect to $ \{ ( \CU_{ \alpha } )_{\text{d.s.e}} \} $ obtained by extending the cocycles $ \{ g_{ \alpha \beta } \} $ and $ \{ h_{ \alpha \beta } \} $, respectively, to the direct summand similarity extension $ \{ ( \CU_{ \alpha } )_{\text{d.s.e}} \} $ of the cover $ \{ \CU_{ \alpha } \} $. A cocycle $ \{ g_{ \alpha \beta } \} $ is said to be cohomologous to the trivial cocycle if there exists uniformly analytic functions $ \lambda_{ \alpha } : ( \CU_{ \alpha } )_{ \text{d.s.e} } \rightarrow  \coprod_{ n = 1 }^{ \infty } \text{GL} ( nr, \C ) $ satisfying 
$$ \widetilde{ g_{ \alpha \beta } } (Z ) = \lambda_{ \alpha } ( Z ) \circ \lambda_{ \beta } ( Z )^{ - 1 } \quad \text{for} ~~ Z \in ( \CU_{ \alpha } )_{\text{d.s.e}} \cap ( \CU_{ \beta } )_{\text{d.s.e}} . $$ 
It turns out that equivalent nc cocycles give rise to equivalent nc vector bundles, as recorded in the theorem below.

\begin{thm} \label{Equivalent cocycle vs bundle isomorphism}
Let $ \Omega $ be a uniformly open nc set which is invariant under similarity and direct summands. Suppose $ E \overset{ \pi }{ \rightarrow } \Omega $ and $ F \overset{ \pi }{ \rightarrow } \Omega $ are two nc holomorphic vector bundles embedded in the trivial bundle $ \mathfrak{ T }_{ \Omega } ( \CV ) \overset{ \pi }{ \rightarrow } \Omega $. Then $ E $ and $ F $ are isomorphic if and only if there exists a uniformly open cover $ \{ \CU_{ \alpha } \} $ -- consisting of uniformly open sets invariant under direct summands and similarities -- of $ \Omega $ trivializing both $ E $ and $ F $ such that the cocycles associated to the transition functions with respect to this cover are equivalent. 
\end{thm}

\begin{proof}
Since $ \Omega $ is invariant under the similarity and direct summands, we may assume, without loss of generality, that the uniformly open sets $ \CU_{ \alpha } $ in any nc trivialization $ \{ ( \CU_{ \alpha }, \varphi_{ \alpha } ) \} $ of a nc vector bundle over $ \Omega $ are also invariant under these operations.  

Indeed, given a local trivialization $ \varphi : ( \pi \vert_E )^{ - 1 } ( \CU ) \rightarrow \mathfrak{ T }_{ \CU } ( \C^r ) $ over a uniformly open subset $ \CU $ of $ \Omega $, we first consider the direct summand extension $ \CU_{ \text{d.s.e} } $ of $ \CU $, followed by the similarity envelop $ \widetilde{ \CU_{ \text{d.s.e} } } $ of $ \CU_{ \text{d.s.e} } $. It follows from the proof of \cite[Proposition 9.2 (i) and Proposition A.3 ]{Verbovetskyi-Vinnikov} that there exists unique trivialization $ \widetilde{ \varphi_{ \text{d.s.e} } } : ( \pi \vert_E )^{ - 1 } ( \widetilde{ \CU_{ \text{d.s.e} } } ) \rightarrow \mathfrak{ T }_{ \widetilde{ \CU_{ \text{d.s.e} } } } ( \C^r ) $ extending $ \varphi $.

Suppose that $ E $ and $ F $ are isomorphic nc holomorphic vector bundles. Let $ \{ \CU_i \} $ and $ \{ \widetilde{ \CU }_j \} $ be coverings of $ \Omega $ consisting of uniformly open subsets invariant under direct summands and similarities, which trivialize $ E $ and $ F $, respectively. The common refinement 
$$ \CU : = \{ \CU_i \cap \widetilde{ \CU }_j \} $$ 
remains a covering of $ \Omega $ consisting of uniformly open subsets that are similarly invariant under direct summands and similarities. Furthermore, the restrictions of local trivializations for $ E $ and $ F $ to this refinement trivialize them. Denote these nc trivializations for $ E $ by $ \{ ( \CU_{ \alpha }, \varphi_{ \alpha } ) \} $ and for $ F $ by $ \{ ( \CU_{ \alpha }, \psi_{ \alpha } ) \} $.

Since $ E $ and $ F $ are isomorphic as nc vector bundles, there is a uniformly bi-holomorphic mapping $ f : E \rightarrow F $ that preserves the fibres. For each $ n \in \N $, consider a point $ ( Z, ( A_1, \hdots, A_r ) ) \in ( \CU_{ \alpha } \cap \C^{ n \times n } ) \times ( \C^{ n \times n } )^r $. Applying $ \psi_{ \alpha } \circ f \circ \varphi_{ \alpha }^{ - 1 } $ to it, we obtain the point 
$$ \psi_{ \alpha } \circ f \circ \varphi_{ \alpha }^{ - 1 } ( Z, ( A_1, \hdots, A_r ) ) \in ( \CU_{ \alpha } \cap \C^{ n \times n } ) \times ( \C^{ n \times n } )^r . $$  Because each of these nc maps preserves the fibres and their restrictions to the fibres are free left module (over $ \C^{ n \times n } $) isomorphisms, it follows as in the discussion in Subsection \ref{transition data} that there exists a uniformly analytic function $ \lambda_{ \alpha } : \CU_{ \alpha } \rightarrow  \coprod_{ n = 1 }^{ \infty } \text{GL} ( nr, \C ) $ such that  
$$ \psi_{ \alpha } \circ f \circ \varphi_{ \alpha }^{ - 1 } ( Z, ( A_1, \hdots, A_r ) ) = ( Z, \left[ \begin{smallmatrix} A_1 & \cdots & A_r \end{smallmatrix} \right] \lambda_{ \alpha } ( Z ) ) . $$ 
Now on a non-empty intersection $ \CU_{ \alpha } \cap \CU_{ \beta } $, we have (with a slight abuse of notation)
$$ \lambda_{ \alpha } ( Z ) g^E_{ \alpha \beta } ( Z )  = \psi_{ \alpha } \circ f \circ \varphi_{ \alpha }^{ - 1 } \varphi_{ \alpha } \circ \varphi_{ \beta }^{ - 1 } = \psi_{ \alpha } \circ f \circ \varphi_{ \beta }^{ - 1 }, $$
and 
$$ g^F_{ \alpha \beta } ( Z ) \lambda_{ \beta } ( Z ) = \psi_{ \alpha } \circ \psi_{ \beta }^{ - 1 } \psi_{ \beta } \circ f \circ \varphi_{ \beta }^{ - 1 } = \psi_{ \alpha } \circ f \circ \varphi_{ \beta }^{ - 1 } $$
verifying that the nc cocycles $ \{ g^E_{ \alpha \beta } \} $ for $ E $ and $ \{ g^F_{ \alpha \beta } \} $ for $ F $ over the same uniformly open cover $ \{ \CU_{ \alpha } \} $ of $ \Omega $ are equivalent.

Conversely, suppose that $ E $ and $ F $ have equivalent nc cocycles $ \{ g^E_{ \alpha \beta } \} $ and $ \{ g^F_{ \alpha \beta } \} $ of uniformly analytic transition functions over some cover $ \{ \CU_{ \alpha } \} $ of $ \Omega $ consisting of direct summands and similarities invariant uniformly open sets trivializing both $ E $ and $ F $. Then by the definition of equivalent cocycles (cf. Equation \eqref{cohomologous cocycles}), it follows that there exists a family of uniformly analytic functions $ \big\{ \lambda_{ \alpha } : \CU_{ \alpha } \rightarrow  \coprod_{ n = 1 }^{ \infty } \text{GL} ( nr, \C ) \big\} $ such that 
$$ \lambda_{ \alpha } ( Z ) g^E_{ \alpha \beta } ( Z ) = g^F_{ \alpha \beta } ( Z ) \lambda_{ \beta } ( Z ) , \quad Z \in \CU_{ \alpha } \cap \CU_{ \beta } . $$

Let $ \{ ( \CU_{ \alpha }, \varphi_{ \alpha } ) \} $ and $ \{ ( \CU_{ \alpha }, \psi_{ \alpha } ) \} $ be nc trivializations for $ E $ and $ F $, respectively. On each $ \CU_{ \alpha } $, we consider the mapping $ f_{ \alpha } : E \vert_{ \CU_{ \alpha } } \rightarrow F \vert_{ \CU_{ \alpha } } $ defined by 
$$ f_{ \alpha } : = \psi^{ - 1 }_{ \alpha } \circ \lambda_{ \alpha } \circ \varphi_{ \alpha } . $$
Here, by the composition with $ \lambda_{ \alpha } $, we mean the map $$ ( Z, ( A_1, \hdots, A_r ) ) \mapsto \big( Z, \begin{bmatrix} A_1 & \cdots & A_r \end{bmatrix} \lambda_{ \alpha } ( Z ) \big) $$ for $ ( Z, ( A_1, \hdots, A_r ) ) \in ( \CU_{ \alpha } \cap \C^{ n \times n } ) \times ( \C^{ n \times n } )^r $, $ n \in \N $. It thus defines an isomorphism between the nc holomorphic vector bundles $ E $ and $ F $ restricted to the uniformly open set $ \CU_{ \alpha } $.

Note that since the nc cocycles $ \{ g^E_{ \alpha \beta } \} $ and $ \{ g^F_{ \alpha \beta } \} $ are equivalent via $ \{ \lambda_{ \alpha } \} $, $ f_{ \alpha } $ agrees with $ f_{ \beta } $ on $ \CU_{ \alpha } \cap \CU_{ \beta } $, whenever it is non-empty. Consequently, the association $ f : E \rightarrow F $ given by 
$$ f \vert_{ \CU_{ \alpha } }. : = f_{ \alpha } $$
defines a well defined function on $ E $. We now show that $ f $ is a nc function.

Evidently, $ f $ is graded. Let $ P \in E \vert_{ \CU_{ \alpha } } $ and $ Q \in E \vert_{ \CU_{ \beta } } $. Assume that $ P \oplus Q \in E \vert_{ \CU_{ \gamma } } $. Because the uniformly open sets $ \{ \CU_{ \alpha } \} $ are invariant under direct summands and similarities and $ \pi $ is a nc function, so are the uniformly open sets $ \{ E \vert_{ \CU_{ \alpha } } = \pi^{ - 1 } ( \CU_{ \alpha } ) \} $. Therefore, both $ P $ and $ Q $ are in $ E \vert_{ \CU_{ \gamma } } $. Consequently,
$$ f ( P \oplus Q ) = f_{ \gamma } ( P \oplus Q ) = f_{ \gamma } ( P ) \oplus f_{ \gamma } ( Q ) = f_{ \alpha } ( P ) \oplus f_{ \beta } ( Q ) = f ( P ) \oplus f ( Q ) $$
implying that $ f $ respects the direct sums. Since each $ \CU_{ \alpha } $ is invariant under similarities and each $ f_{ \alpha } $ is a nc function, it follows that $ f $ respects the similarities as well. Furthermore, since each $ f_{ \alpha } : E \vert_{ \CU_{ \alpha } } \rightarrow F \vert_{ \CU_{ \alpha } } $ is an isomorphism of nc holomorphic vector bundles, so is the nc function $ f $. 
\end{proof}

\begin{cor}
    Any nc holomorphic vector bundle with trivial cocycle is isomorphic to the trivial nc vector bundle.
\end{cor}

\begin{proof}
    This is a direct consequence of the preceding theorem.
\end{proof}

\subsubsection{Noncommutative hermitian structures}

Given a nc vector bundle, we now define the notion of a hermitian metric on each of its fibres what we call \textit{nc hermitian structure}. Recall the notation: $ \CL ( \C_{ nc }, \C_{ nc } ) = \coprod_{ n, m \geq 1 } \CL ( \C^{ n \times m } ) $.

\begin{defn} \label{hermitian metric}
A nc hermitian structure on a nc holomorphic vector bundle $ E \overset{ \pi }{ \rightarrow } \Omega $ of rank $ r $ over $ \Omega $ with $ E \subset \mathfrak{ T }_{ \Omega } ( \CV ) $ is an assignment 
$$ \mathsf{ Her } : E \times E \rightarrow \CL ( \C_{ nc }, \C_{ nc } ), $$ 
satisfying following properties:
\begin{itemize}
\item[(1)] For every $ n, m \in \N $, $ Z \in \Omega_n $, $ W \in \Omega_m $, and $ f \in E_Z $, $ g \in E_W $,
$$ \mathsf{ Her } ( f, g ) \in \CL ( \C^{ n \times m }, \C^{ n \times m } ) . $$
\item[(2)] $ \mathsf{ Her } ( f, g ) $ is $ \C $-linear in $ f $ for every $ g \in E_W $ and $ W \in \Omega $.
\item[(3)] For every uniformly open subset $ \CU \subset \Omega $ and any two uniformly analytic nc sections $ \sigma_1, \sigma_2 : \CU \rightarrow E $, the nc function 
$$ \mathsf{ Her } ( \sigma_1 ( \cdot ), \sigma_2 ( \cdot ) ) : \CU \times \CU \rightarrow \CL ( \C_{ nc }, \C_{ nc } ) $$ 
of order $ 1 $ is a cp nc kernel.
\end{itemize}
A nc holomorphic vector bundle $ E \overset{ \pi }{ \rightarrow } \Omega $ equipped with a nc hermitian structure $ \mathsf{ Her } $ is called a nc hermitian holomorphic vector bundle.
\end{defn}

\begin{rem} \label{remark on nc hermitian structure}
For a trivial nc holomorphic vector bundle $ \mathfrak{ T }_{ \Omega } ( \CH ) \overset{ \pi } { \rightarrow } \Omega $ with some Hilbert space $ \CH $, note that $ \mathfrak{ T }_{ \Omega } ( \CH ) \overset{ \pi } { \rightarrow } \Omega $ turns out to be a nc hermitian holomorphic vector bundle equipped with the canonical hermitian structure obtained from the matrix sesquilinear form (cf. \cite[Chapter 3, Section 3.5]{Effros-Ruan}): The given sesquilinear form $ \langle \cdot, \cdot \rangle $ on $ \CH $ induces a \textit{matrix sesquilinear form} $ \langle \langle \cdot, \cdot \rangle \rangle : \CH^{ n \times n } \times \CH^{ m \times m } \rightarrow \C^{ n \times n } \otimes \C^{ m \times m } $  for all $ n, m \in \N $ as follows:
\be \label{matrix sesquilinear form}
 ( A, B ) \mapsto \langle \! \langle A, B \rangle \! \rangle : = \left( \! \left( \langle a_{ i j }, b_{ l k } \rangle \right) \! \right).
\ee
In consequence, a nc holomorphic vector bundle $  E \overset{ \pi }{ \rightarrow } \Omega $ embedded in $ \mathfrak{ T }_{ \Omega } ( \CH ) $ inherits a canonical hermitian structure from the trivial bundle $ \mathfrak{ T }_{ \Omega } ( \CH ) \overset{ \pi } { \rightarrow } \Omega $.

In general, for any uniformly open subset $ \CU \subset \Omega $ and a local frame $ \gamma_1, \hdots, \gamma_r : \CU \rightarrow E $, note that the nc hermitian structure on $ E \overset{ \pi }{ \rightarrow } \Omega $ gives rise to the cp nc kernel 
$$ K_{ \mathsf{ Her } } : \CU \times \CU \rightarrow \CL ( \C_{ nc }, ( \C^{ r \times r } )_{ nc } ) $$ 
defined as
$$ K_{ \mathsf{ Her } } ( Z, W ) ( P ) = \left( \! \left( \mathsf{ Her } ( \gamma_i ( Z ), \gamma_j ( W ) ) ( P ) \right) \! \right)_{ i, j = 1 }^r , $$
for $ Z \in \CU \cap \C^{ n \times n } $, $ W \in \CU \cap \C^{ m \times m } $, $ P \in \C^{ n \times m } $ and $ n, m \in \N $.
\end{rem}

\begin{defn} \label{automorphism of hermitian bundle}
Two nc hermitian holomorphic vector bundles $ E \overset{ \pi }{ \rightarrow } \Omega $ and $ E' \overset{ \pi' }{ \rightarrow } \Omega' $ with hermitian structures $ \mathsf{ Her } $ and $ \mathsf{ Her }' $, respectively, are said to be isomorphic if there exists a nc vector bundle isomorphism $ \widetilde{ \psi } : E \rightarrow E' $ which also preserves the hermitian structures: For every $ n, m \in \N $, $ Z \in \Omega_n $, $ W \in \Omega_m $, and $ f \in E_Z $, $ g \in E_W $, $ \widetilde{ \psi } $ satisfies 
$$ \mathsf{ Her }' \bigg( \widetilde{ \psi } ( f ), \widetilde{ \psi } ( g ) \bigg) = \mathsf{ Her } ( f, g ) . $$
\end{defn}

Let $ \widetilde{ \psi } : E \rightarrow E' $ be a nc vector bundle isomorphism preserving the hermitian structures and $ \CU \subset \Omega $ be a uniformly open subset over which $ \gamma_1, \hdots, \gamma_r $ is a local frame for $ E $. Then $ \psi ( \CU ) $ is a uniformly open subset of $ \Omega' $, and assume that $ \gamma'_1, \hdots, \gamma'_r $ is  a local frame for $ E' $. For $ 1 \leq i \leq r $ and $ Z \in \CU \cap \C^{ n \times n } $,
$$ \widetilde{ \psi } ( \gamma_i ( Z ) ) = \Psi_{ i 1 } ( Z ) \gamma'_1 ( Z ) + \cdots + \Psi_{ i r } ( Z ) \gamma'_r ( Z ) $$
implying that $ \widetilde{ \psi } $ gives rise to a uniformly analytic nc mapping $ \Psi : \CU \rightarrow ( \C^{ r \times r } )_{ nc } $. Since $ \widetilde{ \psi } $ is a nc vector bundle isomorphism, it follows that $ \Psi $ is a $  \coprod_{ n = 1 }^{ \infty } \text{GL} ( nr, \C ) $ valued uniformly analytic nc function on $ \CU $. Thus, under a local trivialization the condition for $ \widetilde{ \psi } $ to preserve the hermitian structures of $ E $ and $ E' $ turns out to be the existence of a uniformly analytic nc function $ \Psi : \CU \rightarrow  \coprod_{ n = 1 }^{ \infty } \text{GL} ( nr, \C ) $ such that
$$ K_{ \mathsf{ Her }' } ( \psi ( Z ), \psi ( W ) ) ( P ) = \Psi ( Z ) K_{ \mathsf{ Her } } ( Z, W ) ( P ) \Psi ( W )^*  , $$
for $ Z \in \CU \cap \C^{ n \times n }, W \in \CU \cap \C^{ m \times m } $ and $ P \in \C^{ n \times m } $, where $ K_{ \mathsf{ Her } } $ and $ K_{ \mathsf{ Her }' } $ are the cp nc kernels obtained from the hermitian structures of $ E \overset{ \pi }{ \rightarrow } \Omega $ and $ E' \overset{ \pi' }{ \rightarrow } \Omega' $, respectively, with respect to the local frames $ \{ \gamma_i : 1 \leq i \leq r \} $ and $ \{ \gamma'_i : 1 \leq i \leq r \} $.

\subsection{Noncommutative Cowen-Douglas bundles}

In this subsection, we first show that every tuple in the nc Cowen-Douglas class $ \mathrm{ B } ( \Omega )_{ nc } $ gives rise to a nc hermitian holomorphic vector bundle over $ \Omega $. Then we prove that the unitary equivalence class of any nc Cowen-Douglas tuple is the same as the unitary equivalence class of the associated nc hermitian holomorphic vector bundle. We begin with the following technical lemma.

\begin{lem} \label{continuity of frame}
Let $ \CH_r $ be a row Hilbert operator space obtained from a complex separable Hilbert space $ \CH $ and for $ n \in \N $, $ \CM_n $ be the module $ \CH^{ n \times n } $ over $ \C^{ n \times n } $ with the module structure obtained by the left multiplication by matrices in $ \C^{ n \times n } $. Assume that a subset $ \{ H_1, \hdots, H_k \} $ of $ \CM_s $ is linearly independent over $ \C^{ s \times s } $ for some $ k, s \in \N $. Then there exists $ \delta > 0 $ such that the set $ \{ A_1, \hdots, A_k \} $ with $ A_j \in B_{nc} ( H_j, \delta )_{ ms }, 1 \leq j \leq k $, is linearly independent over $ \C^{ ms \times ms } $ for all $ m \in \N $.
\end{lem}

\begin{proof}
First, we prove the desired result for $ k = 1 $. Denote $ H_1 = H $ and let $ R^j $ be the $ j $-th row of the matrix $ H $, for $ 1 \leq j \leq s $. Note that the linearly independence of the set $ \{ H \} $ of $ \CM_s $ over $ \C^{ s \times s } $ is equivalent to the linearly independence (over $ \C $) of the set of vectors $ \{ R^j : 1 \leq j \leq s \} $ in $ \CH^{ 1 \times s } \cong \CH \oplus \cdots \oplus \CH $. In other words, $ \{ H \} $ is linearly independent over $ \C^{ s \times s } $ if and only if the Gramian matrix $ G_0 $ of the set of vectors $ \{ R^j : 1 \leq i \leq k, 1 \leq j \leq s \} $ in $ \CH^{ 1 \times s } $ is invertible. 

Since $ \{ H \} $ is linearly independent over $ \C^{ s \times s } $, $ G_0 $ is invertible and hence is a strictly positive matrix. Let $ \lambda > 0 $ be the minimum eigenvalue of $ G_0 $ and observe that $ \lambda $ remains be the minimum eigenvalue of $ G_0^{ \oplus m } $, which is the Gramian matrix of the set of rows of the $ m $-th amplification $ H^{ \oplus m } $ of $ H $. Now, recall from the definition of the row operator space norm that
$$ \norm{ G_0^{ \oplus m } } = \norm{ H^{ \oplus m } }^2_{ ms, \CH_r }, ~~~ m \in \N . $$
Consequently, for any $ 0 < \delta < \frac{ \sqrt{ \lambda } }{ 2 } $, the Gramian of the rows of any matrix $ A \in B_{ nc } ( H, \delta ) $ turns out to be invertible verifying that the set $ \{ A \} $ with $ A \in B_{ nc } ( H, \delta )_{ ms } $ is linearly independent over $ \C^{ ms \times ms } $ for all $ m \in \N $.

For general $ k > 1 $, a similar argument with $ \CH^{ k \times 1 } \cong \CH \oplus \cdots \oplus \CH $ ($k$ times) instead of $ \CH $ together with the observation that $ ( \CH^{ k \times 1 } )_r \cong \CH_r^{ k \times 1 } $ yield the desired result. 
\end{proof}

Next, we use this lemma in the following proposition to prove the existence of local frame of the desired vector bundle associate to an element in the nc Cowen-Douglas class.

\begin{prop} \label{existence of local nc frame}
Let $ \boldsymbol{T} \in \mathrm B_r ( \Omega )_{nc} $ and $ Y \in \Omega_s $ be a semi-simple point. Then there exist a uniformly open neighbourhood $ \Omega_0 \subseteq \Omega $ containing $ Y $ and uniformly analytic functions $ \gamma_1, \hdots, \gamma_r : \Omega_0 \rightarrow ( \CH_r )_{nc} $ such that, for each $ m \in \N $ and $ W \in ( \Omega_0 )_{ms} $, $ \{ \gamma_1 ( W ), \hdots, \gamma_r ( W ) \} $ is a basis for the free left submodule $ \ker 
D_{ \boldsymbol{T} - W } $ of the module $ \CH^{ ms \times ms } $ over $ \C^{ms \times ms } $.
\end{prop}

\begin{proof}
First, we observe from  Corollary \ref{joint kernel} that $ \ker D_{ \boldsymbol{ T } - W } $ is a free left module over $ \C^{ n \times n } $ for $ W \in \Omega_n $, $ n \in \N$. Fix a free left module basis $ \{ H_1, \hdots, H_r \} $ of $ \ker D_{ \boldsymbol{ T } - Y } $ over $ \C^{ s \times s } $. We now use Theorem \ref{closed range operators}, when $ \CX = \CB ( \CH ) $ endowed with the natural operator space structure, $ \CY = \CZ = \CH $ with the row operator Hilbert space structure, the uniformly analytic nc functions $ \CT : \Omega \to \CB ( \CH )_{ nc } $ and $ \CF : \Omega \to \CH_{ nc } $, respectively, are given by 
$$ \CT ( W ) = D_{ \boldsymbol{ T } - W } \quad \text{and} \quad \CF ( W ) = 0 , \quad W \in \Omega . $$

Since $ Y \in \Omega_s $ is a semi-simple point, it follows from Theorem \ref{closed range operators} that there exist $ \rho > 0 $ and uniformly analytic functions $ \gamma_1, \hdots, \gamma_r : B_{ nc} ( Y, \rho ) \rightarrow ( \CH_r )_{ nc } $ such that 
$$ \gamma_j ( Y ) = H_j, ~ 1 \leq j \leq r \quad \text{and} \quad \{ \gamma_1 ( W ), \hdots, \gamma_r ( W ) \} \subset \ker D_{ \boldsymbol{ T } - W }, $$ 
for $ W \in B_{ nc } ( Y, \rho )_{ m s } $, $ m \in \N $. Also, we have from Lemma \ref{continuity of frame} that there exists $ \delta > 0 $ such that the set $ \{ A_1, \hdots, A_r : A_j \in B_{ nc } ( H_j, \delta )_{ ms }, 1 \leq j \leq r \} $ are linearly independent over $ \C^{ ms \times ms } $ for all $ m \in \N $. The theorem then follows for any uniformly open nc set $ \Omega_0 \subset \cap_{ j = 1 }^r \gamma_j^{ - 1 } ( B_{ nc } ( H_j, \delta ) ) $.
\end{proof}

We now prove the existence of nc hermitian holomorphic vector bundle associated to an element $ \boldsymbol{ T } \in \mathrm B_r ( \Omega )_{ nc } $, for which the following technical result is needed.

\begin{lem} \label{open cover of bi-full set}
Let $ \Omega \subset \C^d_{ nc } $ be a bi-full uniformly open nc set and  $ \{ Y_{ \alpha } \in \Omega_{ m_{ \alpha } } : \alpha \in \Lambda \} $ be the set of all semi-simple points in $ \Omega $, $ \Lambda $ be an indexing set and for every $ \alpha \in \Lambda $, $ B_{ nc} ( Y_{ \alpha }, \rho_{ \alpha } ) $ be a nc open ball in $ \Omega $. Then $ \left\{ \widetilde{ B_{ nc } ( Y_{ \alpha }, \rho_{ \alpha } ) } : \alpha \in \Lambda \right\} $ forms an nc open cover of $ \Omega $ where $ \widetilde{ B_{ nc } ( Y_{ \alpha }, \rho_{ \alpha } ) } $ is the similarity envelop of $ B_{ nc } ( Y_{ \alpha }, \rho_{ \alpha } ) $ for each $ \alpha \in \Lambda $.
\end{lem}

\begin{proof}
Let $ W \in \Omega_s $ be a reducible point. Then $ W $ is similar to $ \left[ \begin{smallmatrix} A & B \\ 0 & C \end{smallmatrix} \right] $ for some $ A \in ( \C^{ s_1 \times s_1 } )^d, C \in ( \C^{ s_2 \times s_2 } )^d $ and $ B \in ( \C^{ s_1 \times s_2 } )^d $. Assume without loss of generality that both $ A $ and $ C $ are irreducible so that $ W' : = \left[ \begin{smallmatrix} A & 0 \\ 0 & C \end{smallmatrix} \right] $ is a semi-simple point. Being bi-full $ \Omega $ is similarity invariant and therefore, $ W' = Y_{ \alpha } $ for some $ \alpha \in \Lambda $. Further, for $ \rho > 0 $, note that $ W $ is similar to $ W_{ \rho } : = \left[ \begin{smallmatrix} A & \rho B \\ 0 & C \end{smallmatrix} \right] $ as $ W_{ \rho } = \SS_{ \rho } W \SS_{ \rho }^{ - 1 } $ with $ \SS_{ \rho } = \left[ \begin{smallmatrix} \rho & 0 \\ 0 & 1 \end{smallmatrix} \right] $. Thus, there exists $ \epsilon > 0 $ such that $ W_{ \epsilon } \in B_{ nc } ( Y_{ \alpha }, \rho_{ \alpha } ) $ because $ W_{ \rho } \rightarrow W' = Y_{ \alpha } $ as $ \rho \rightarrow 0 $. Consequently, $ W \in  \widetilde{ B_{ nc } ( Y_{ \alpha }, \rho_{ \alpha } ) } $. 
\end{proof} 

\begin{thm} \label{Existence of Cowen-Douglas bundle}
Let $ \Omega \subset \C^d_{ nc } $ be a uniformly open bi-full nc domain, $ \boldsymbol{ T } \in \mathrm B_r ( \Omega )_{ nc } $ and $ \mathfrak{ T }_{ \Omega } ( \CH_r ) \overset{ \pi }{ \rightarrow } \Omega $ be the trivial nc hermitian holomorphic vector bundle. Then 
\be \label{Cowen-Douglas bundle}
E_{ \boldsymbol{ T } } : = \{ ( Z, F ) \in \mathfrak{ T }_{ \Omega } ( \CH_r ) : F \in \ker D_{ \boldsymbol{ T } - Z } \},
\ee
along with the canonical projection $ \pi $ restricted to $ E_{ \boldsymbol{ T } } $ defines a nc hermitian holomorphic vector bundle $ E_{ \boldsymbol{ T } } \overset{ \pi }{ \rightarrow } \Omega $ of rank $ r $ over $ \Omega $.
\end{thm}

\begin{proof}
First, observe from the definition of the nc Cowen-Douglas class $ \mathrm B_r ( \Omega )_{ nc } $ that the canonical projection $ \pi : E_{ \boldsymbol{ T } } \rightarrow \Omega $ is surjective. Note that the fibre of $ E_{ \boldsymbol{ T } } \overset{ \pi }{ \rightarrow } \Omega $ over $ Z \in \Omega $ is $ \ker D_{ \boldsymbol{ T } - Z } $. 

Let $ W \in \Omega_s $ be a semi-simple point for some $ s \in \N $. It follows from Proposition \ref{existence of local nc frame} that there exists an uniformly open neighbourhood $ \CU_W \subset \Omega $ around $ W $ and uniformly analytic functions $ \gamma_1, \hdots, \gamma_r : \CU_W \rightarrow ( \CH_r )_{ nc } $ such that for each $ m \in \N $ and $ Z \in \CU_W \cap \C^{ ms \times ms } $, $ \{ \gamma_1 ( W ), \hdots, \gamma_r ( W ) \} $ forms a free left module basis for $ \ker D_{ \boldsymbol{ T } - Z } $ over $ \C^{ ms \times ms } $. Consequently, for any $ F \in \ker D_{ \boldsymbol{ T } - Z } $ with $ Z \in \CU_W \cap \C^{ ms \times ms } $, there exist unique $ \Xi_1, \hdots, \Xi_r \in \C^{ ms \times ms } $ such that
$$ F = \Xi_1 \gamma_1 ( Z ) + \cdots + \Xi_r \gamma_r ( Z ) . $$
Now define $ \varphi_W : \pi^{ - 1 } ( \CU_W ) \rightarrow \mathfrak{ T }_{ \CU_W } ( \C^r ) $ as 
$$ \varphi_W ( Z, F ) : = ( Z, \Xi_1, \hdots, \Xi_r ), $$
where $ \mathfrak{ T }_{ \CU_W } ( \C^r ) \overset{ \pi_1 }{ \rightarrow } \CU_W $ is the trivial bundle of rank $ r $. Note that $ \pi_1 \circ \varphi_W = \pi $ and since both $ \pi $ and $ \pi_1 $ are uniformly analytic nc functions, so is $ \varphi_W $. For each $ Z \in \CU_W \cap \C^{ ms \times ms } $, $ \varphi_W ( Z, . ) : \ker D_{ \boldsymbol{ T } - Z } \rightarrow ( \C^{ ms \times ms } )^r $ is $ \C $-linear as well as free left $ \C^{ ms \times ms } $ module isomorphism. 

Since $ \Omega $ is bi-full, it follows from Lemma \ref{open cover of bi-full set} that the family $$ \{ \widetilde{ \CU_W } : W ~ \text{is semi-simple},~ W \in \Omega \} $$ forms an open cover for $ \Omega $. In consequence, $ \{ ( \widetilde{ \CU_W }, \widetilde{ \phi_W } ) \} $ constitutes a nc local trivialization for $ E_{ \boldsymbol{ T } } \overset{ \pi }{ \rightarrow } \Omega $. Here, $ \widetilde{ \CU_W } $ is the similarity envelop of $ \CU_W $ and $ \widetilde{ \phi_W } $ is the unique extension of $ \phi_W $ to $ \widetilde{ \CU_W } $.

Further, $ E_{ \boldsymbol{ T } } \overset{ \pi }{ \rightarrow } \Omega  $ is equipped with the hermitian structure induced from the canonical one on the trivial bundle $ \mathfrak{ T }_{ \Omega } ( \CH_r ) \overset{ \pi }{ \rightarrow } \Omega $ (cf. Remark \ref{remark on nc hermitian structure}). It thus verifies that $ E_{ \boldsymbol{ T } } \overset{ \pi }{ \rightarrow } \Omega $ is a nc hermitian holomorphic vector bundle of rank $ r $ over $ \Omega $.
\end{proof}

We conclude this section by proving that the nc vector bundles associated to the nc Cowen-Douglas tuples completely determine them up to unitary equivalence. Before continuing with this result we record the following basic property of nc Cowen-Douglas class.

\begin{prop} \label{local property of CD class}
Let $ \Omega \subset \C^d_{nc} $ be a bi-full uniformly open nc domain and $ Y \in \Omega_s $ be a semi-simple point. Then $ \mathrm B_r ( \Omega )_{nc} \subset \mathrm B_r ( B_{nc} ( Y, \rho ) )_{nc} $, whenever $ B_{nc} ( Y, \rho ) \subset \Omega $.
\end{prop}

\begin{proof}
From the definition of nc Cowen-Douglas class it is enough to show that 
$$ \CH_Y : = \bigvee \{ h_{ i j } : ( \!( h_{ i j } ) \! )_{ i, j = 1 }^n \in \ker D_{ \boldsymbol{ T } - W }, W \in B_{ nc } ( Y, \rho ) \} = \CH . $$
If possible suppose that $ h \in \CH $ is perpendicular to $ \CH_Y $. Let $ \{ Y_{ \alpha } \in \Omega_{ m_{ \alpha } } : \alpha \in \Lambda \} $ be the set of all semi-simple points in $ \Omega $, $ \Lambda $ be an indexing set and for every $ \alpha \in \Lambda $, $ B_{ nc } ( Y_{ \alpha }, \rho_{ \alpha } ) $ be a nc open ball in $ \Omega $ on which there is a local nc frame $ \{ \gamma_1^{ \alpha }, \hdots, \gamma_r^{ \alpha } \} $ for $ E_{ \boldsymbol{ T } } \overset{ \pi }{ \rightarrow } \Omega $. Since $ \Omega $ is bi-full, it follows from Lemma \ref{open cover of bi-full set} that $ \{ B_{ nc } ( Y_{ \alpha }, \rho_{ \alpha } ) : \alpha \in \Lambda \} $ forms an nc open cover of $ \Omega $. Set 
$$ \CH_{ \alpha } : =  \bigvee_{ W \in B_{ nc} ( Y_{ \alpha }, \rho_{ \alpha } ) } \left\{ \gamma_t ( W )_{ i j } : 1 \leq t \leq r \right\}, ~~~ \alpha \in \Lambda $$
and $ \Lambda' : = \{ \alpha \in \Lambda : h \perp \CH_{ \alpha } \} $. Then consider the uniformly open subsets 
$$ \Omega' : = \bigcup_{ \alpha \in \Lambda' } B_{ nc } ( Y_{ \alpha }, \rho_{ \alpha } ) \quad \text{and} \quad \Omega'' : = \bigcup_{ \alpha \in \Lambda \setminus \Lambda' } B_{ nc } ( Y_{ \alpha }, \rho_{ \alpha } ) $$ 
of $ \Omega $. Since $ \Omega $ is connected, $ \Omega' \cap \Omega'' \neq \emptyset $. Let $ \hat{ \alpha } \in \Lambda \setminus \Lambda' $ be such that $ B_{ nc } ( Y_{ \hat{ \alpha } }, \rho_{ \hat{ \alpha } } ) \cap \Omega' \neq \emptyset $. We now show that $ h \perp \CH_{ \hat{ \alpha } } $.

Since $ h \perp \CH_Y $, it follows that 
$$ \langle \! \langle h \otimes \mathrm{Id}_{ \C^{ ms \times ms } }, \gamma_t ( W ) \rangle \! \rangle = 0, \quad 1 \leq t \leq r, ~~ W \in B_{ nc } ( Y, \rho )_{ ms }, ~~ m \in \N , $$
where $ \langle \! \langle \cdot, \cdot \rangle \! \rangle $ is the matrix sesquilinear form associated to $ \CH $ (cf. Equation \eqref{matrix sesquilinear form}). Then for each $ 1 \leq t \leq r $, the uniformly nc analytic function $ F_t : B_{ nc } ( Y_{ \hat{ \alpha } }, \rho_{ \hat{ \alpha } } ) \rightarrow \C_{ nc } $ defined by 
$$ W \mapsto  \langle \! \langle \gamma_t^{ \alpha } ( W ), h \otimes \mathrm{Id}\rangle \! \rangle $$ 
is identically zero on $ B_{ nc } ( Y_{ \hat{ \alpha } }, \rho_{ \hat{ \alpha } } ) \cap \Omega' $. Consequently, by the identity theorem for nc analytic functions, $ F_t $ vanishes on $ B_{ nc } ( Y_{ \hat{ \alpha } }, \rho_{ \hat{ \alpha } } ) $, $ 1 \leq t \leq r $. But this contradicts the choice of $ \hat{ \alpha } $ unless $ h = 0 $. Thus, $ \CH_Y = \CH $.
\end{proof}

\begin{thm} \label{unitary equivalence}
Let $ \Omega \subset \C^d_{ nc } $ be a uniformly open bi-full nc domain and $ \boldsymbol{ T }, \widetilde{ \boldsymbol{ T } } \in \mathrm B_r ( \Omega )_{ nc } $. Then $ \boldsymbol{ T } $ and $\widetilde{ \boldsymbol{ T } } $ are unitarily equivalent if and only if $ E_{ \boldsymbol{ T } } \overset{ \pi }{ \rightarrow } \Omega $ is isomorphic to $ E_{ \widetilde{ \boldsymbol{ T } } } \overset{ \pi }{ \rightarrow } \Omega $ as nc hermitian holomorphic vector bundles.
\end{thm}

\begin{proof}
Let $ \boldsymbol{ T } \in \CB ( \CH ) $ and $ \widetilde{ \boldsymbol{ T } } \in \CB ( \widetilde{ \CH } ) $ be unitarily equivalent via a unitary operator $ U : \CH \rightarrow \widetilde{ \CH } $. Then for each $ m \in \N $ and $ W \in \Omega_m $, $ U \otimes \mathrm{Id}_{ \C^{ m \times m } } $ intertwines $ D_{ \boldsymbol{ T } - W } $ and $ D_{ \widetilde{ \boldsymbol{ T } } - W } $. Furthermore, since $ U $ is unitary, 
$$ ( U \otimes \mathrm{Id}_{ \C^{ m \times m } } ) ( \ker D_{ \boldsymbol{ T } - W } ) = \ker D_{ \widetilde{ \boldsymbol{ T } } - W }, \quad W \in \Omega_m, ~~ m \in \N . $$ Thus, we have a nc vector bundle isomorphism $ \Phi : \mathfrak{ T }_{ \Omega } ( \CH ) \rightarrow \mathfrak{ T }_{ \Omega } ( \widetilde{ T } ) $ defined by 
$$ ( W, H ) \mapsto ( W, ( U \otimes \mathrm{Id}_{ \C^{ m \times m } } ) ( H ) ), \quad W \in \Omega_m, ~~ H \in \mathfrak{ T }_{ \Omega } ( \CH )|_W, ~~ m \in \N, $$
and $ \Phi |_{ E_{ \boldsymbol{ T } } } : E_{ \boldsymbol{ T } } \rightarrow E_{ \widetilde{ \boldsymbol{ T } } } $ is onto. 

For $ Z \in \Omega_n $, $ W \in \Omega_m $ and $ A \in \mathfrak{ T }_{ \Omega } ( \CH )|_Z $, $ B \in \mathfrak{ T }_{ \Omega } ( \CH )|_W $, note that
\Bea
\widetilde{ \mathsf{ Her } } \left( \Phi ( Z, A ), \Phi ( W, B ) \right) & = & \widetilde{ \mathsf{ Her } } \left( ( U \otimes \mathrm{Id}_{ \C^{ n \times n } }) ( A ), ( U \otimes \mathrm{Id}_{ \C^{ n \times n } }) ( B ) \right) \\
& = & \left( \! \left( \langle U ( a_{ i j } ), U ( b_{ l k } ) \rangle_{ \widetilde{ \CH } } \right) \! \right) \\
& = &  \left( \! \left( \langle a_{ i j } , b_{ l k }  \rangle_{ \CH } \right) \! \right) \\
& = & \mathsf{ Her } ( A, B )
\Eea
verifying that $ \Phi |_{ E_{ \boldsymbol{ T } } } : E_{ \boldsymbol{ T } } \rightarrow E_{ \widetilde{ \boldsymbol{ T } } } $ is a nc hermitian holomorphic vector bundle isomorphism.

Conversely, suppose that $ \Phi : E_{ \boldsymbol{ T } } \rightarrow E_{ \widetilde{ \boldsymbol{ T } } } $ is a nc hermitian holomorphic vector bundle isomorphism and $ Y $ be a semi-simple point in $ \Omega_s $, $ s \in \N $. It follows from Proposition \ref{existence of local nc frame} that there exists a local nc frame $ \gamma_1, \hdots, \gamma_r $ for the nc bundle $ E_{ \boldsymbol{ T } } \overset{ \pi }{ \rightarrow } \Omega $ over a nc ball $ B_{ nc } ( Y, \rho ) $ at $ Y $. Then $ \Phi \circ \gamma_1, \hdots, \Phi \circ \gamma_r $ form a frame for $ E_{ \boldsymbol{ \widetilde{ T } } } \overset{ \pi }{ \rightarrow } \Omega $ over $ B_{ nc } ( Y, r ) $. Also, for every $ W \in \Omega $, note that
$$ \Phi \left( \ker D_{ \boldsymbol{ T } - W } \right) = \ker D_{ \boldsymbol{ \widetilde{ T } } - W } . $$

Set $ \CH_Y : = \bigvee \{ h_{ i j } : ( \!( h_{ i j } ) \! )_{ i, j = 1 }^n \in \ker D_{ \boldsymbol{ T } - W }, W \in B_{ nc } ( Y, \rho ) \} $ and similarly $ \widetilde{ \CH }_Y $. Observe that
$$ \CH_Y = \bigvee_{ W \in B_{ nc } ( Y, r ) } \left\{ \gamma_t ( W )_{ i j } : 1 \leq t \leq r \right\} \quad \text{and} \quad \widetilde{ \CH }_Y = \bigvee_{ W \in B_{ nc } ( Y, r ) } \left\{ \Phi ( \gamma_t ( W ) )_{ i j } : 1 \leq t \leq r \right\} . $$
We now define a linear map $ U : \CH_Y \rightarrow \widetilde{ \CH }_Y $ by
$$ U \left( \gamma_t ( W )_{ i j } \right) : = \Phi ( \gamma_t ( W ) )_{ i j }, \quad 1 \leq t \leq r, ~~ W \in B_{ nc } ( Y, \rho ) . $$
Being a nc hermitian holomorphic vector bundle isomorphism, $ \Phi $ preserves the hermitian structures on the fibres. In particular, for $ Z \in \Omega_n $ and $ W \in \Omega_m $, $ 1 \leq s, t \leq r $, we have that 
$$ \mathsf{ Her } ( \gamma_s ( Z ), \gamma_t ( W ) ) = \widetilde{ \mathsf{ Her } } ( \Phi( \gamma_s ( Z ) ), \Phi( \gamma_t ( W ) ) ) . $$
Consequently, it follows from the definition of the nc hermitian structure (see Definition \ref{hermitian metric}) that 
$$ \langle \gamma_s ( Z )_{ i j }, \gamma_t ( W )_{ \ell k } \rangle_{ \CH } = \langle \Phi ( \gamma_s ( Z ) )_{ i j }, \Phi ( \gamma_t ( W ) )_{ \ell k } \rangle_{ \widetilde{ \CH } } , $$
verifying that $ U : \CH_Y \to \widetilde{ \CH }_Y $ is unitary.

So it remains to show that $ U $ intertwines $ \boldsymbol{ T } $ and $ \widetilde{ \boldsymbol{ T } } $. Since for $ W \in B_{ nc } ( Y, \rho )_{ m s } $ and $ m \in \N $, $ D_{ \boldsymbol{ \widetilde{ T } } - W } ( \Phi ( \gamma_t ( W ) ) ) = 0 $, $ 1 \leq t \leq r $, we have for each $ 1 \leq i \leq d $ and $ 1 \leq j, k \leq ms $,
\begin{multline*}
\left[ \left( \widetilde{ T }_i \otimes \mathrm{Id}_{ \C^{ ms \times ms } } \right) ( \Phi ( \gamma_t ( W ) ) ) \right]_{ j k }  =  \left[ \left( \mathrm{Id}_{ \CH } \otimes R_{ W_i } \right) ( \Phi ( \gamma_t ( W ) ) ) \right]_{ j k } 
 =  \left( \Phi ( \gamma_t ( W ) ) W_i \right)_{ j k }\\
 =  \sum_{ l = 1 }^{ ms } \Phi ( \gamma_t ( W ) )_{ j l } W^{ l k }_i 
 =  \sum_{ l = 1 }^{ ms } U ( \gamma_t ( W )_{ j l } ) W^{ l k }_i 
 =  U \left( ( \gamma_t ( W ) W_i )_{ j k } \right) 
 =  U \left( T_i ( \gamma_t ( W ) )_{ j k } \right)
\end{multline*}
where $ W_i = \big( \! \big( W^{ j k }_i \big)  \! \big)_{ j, k = 1 }^{ ms } $. As a result, for every $ 1 \leq t \leq r $, $ 1 \leq i \leq d $ and $ 1 \leq j, k \leq ms $, we have that
\begin{align*}
  & \widetilde{ T }_i \big( U \big( \gamma_t ( W )_{ j k } \big) \big) = \widetilde{ T }_i \big( \Phi \big( \gamma_t ( W ) )_{ j k } \big) \big) \\
  & =  \bigg[ \bigg( \widetilde{ T }_i \otimes \mathrm{Id}_{ \C^{ ms \times ms } } \bigg) \big( \Phi ( \gamma_t ( W ) ) \big) \bigg]_{ j k } = U \big( T_i ( \gamma_t ( W ) )_{ j k } \big) .  
\end{align*}
%$$ \widetilde{ T }_i \left( U \left( \gamma_t ( W )_{ j k } \right) \right) = \widetilde{ T }_i \left( \Phi \left( \gamma_t ( W ) )_{ j k } \right) \right)  =  \left[ \left( \widetilde{ T }_i \otimes \mathrm{Id}_{ \C^{ ms \times ms } } \right) ( \Phi ( \gamma_t ( W ) ) ) \right]_{ j k } = U \left( T_i ( \gamma_t ( W ) )_{ j k } \right) . $$
It thus verifies that $ U $ intertwines $ \boldsymbol{ T } $ and $ \widetilde{ \boldsymbol{ T } } $ on $ \CH_Y $. Finally, since $ \Omega $ is bi-full, the result follows from Proposition \ref{local property of CD class}.
\end{proof}

\begin{rem}
    In light of Theorem \ref{unitary equivalence}, the unitary invariants of a tuple in the noncommutative Cowen-Douglas class correspond to those of the associated noncommutative hermitian holomorphic vector bundle. This naturally leads to the investigation of invariants of noncommutative hermitian holomorphic vector bundles, which we intend to explore in future work.
\end{rem}

\section{Model for noncommutative Cowen-Douglas tuples} \label{model and local operators}

The objective of the first half of this section is to model a $ d $-tuple of operators in the nc Cowen-Douglas class $ \mathrm B_r ( \Omega )_{ nc } $ ($ \Omega \subset \C^d_{ nc} $) over a nc reproducing kernel Hilbert space, while in the later half, we introduce the notion of ``local operators" associated to a tuple in $ \mathrm B_r ( \Omega )_{ nc } $ for each $ W \in \Omega $. 

\subsection{Construction of model spaces} \label{model} Given an element $ \boldsymbol{ T } = ( T_1, \hdots, T_d ) $ in $ \mathrm B_r ( \Omega )_{ nc } $ with $ T_1, \hdots, T_d \in \CB ( \CH ) $, we construct a nc functional Hilbert space consisting of uniformly analytic nc functions on some uniformly open subset of $ \Omega $ taking values in $ \CL ( \C, \C^r )_{ nc } $ as demonstrated below. Here, we equip $ \CH $ with the row operator Hilbert space structure and denote it as $ \CH_r $ as before.

Since $ \boldsymbol{T} \in \mathrm B_r ( \Omega )_{nc} $, it follows from Proposition \ref{existence of local nc frame} that there exists a uniformly open sudomain $ \Omega_0 \subset \Omega $ and a uniformly analytic functions $ \gamma_j : \Omega_0 \rightarrow ( \CH_r )_{nc} $, $ j = 1, \hdots, r $ such that $ \{ \gamma_1 ( X ), \hdots, \gamma_r ( X ) \} $ is a basis for $ \ker D_{ \boldsymbol{T} - X } $ viewing $ \ker D_{ \boldsymbol{ T } - W } $ as a left $ \C^{ m \times m } $ module for all $ X \in ( \Omega_0 )_m $, $ m \in \N $, where $ D_{ \boldsymbol{ T } - W } $ is as defined in Equation \eqref{D_T}.

First, we define the uniformly analytic nc mapping $ \Gamma : \Omega_0 \rightarrow \CL ( \C^r, \CH )_{nc} $ as follows: For each $ m \in \N $, consider the mapping $ \Gamma_m : ( \Omega_0 )_m \rightarrow \CL ( ( \C^{ m \times 1 } )^{ r \times 1 }, \CH^{ \oplus m } ) $ defined by the formula 
$$ \Gamma_m ( X ) : ( \C^{ m \times 1 } )^{ r \times 1 } \rightarrow \CH^{ \oplus m } \quad \text{as} \quad \left[ \begin{smallmatrix} \xi_1 \\ \vdots \\ \xi_r \end{smallmatrix} \right] \mapsto \sum_{ j = 1 }^r \gamma_j ( X ) \xi_j, $$ $ \xi_1, \hdots, \xi_r \in \C^{ m \times 1 } $ and define $ \Gamma|_{ ( \Omega_0 )_m } : = \Gamma_m $. 

Evidently, $ \Gamma_m ( X ) $ is a bounded linear map for each $ m \in \N $ and $ X \in \Omega_0$, as it follows from the definition of the row operator Hilbert space structure on $ \CH $ that 
$$ \norm{ \sum_{ j = 1 }^r \gamma_j ( X ) \xi_j } = \norm{ \left[ \begin{smallmatrix} \gamma_1 ( X ) & \cdots & \gamma_r ( X ) \end{smallmatrix} \right] \cdot \xi } \leq \norm{ \left[ \begin{smallmatrix} \gamma_1 ( X ) & \cdots & \gamma_r ( X ) \end{smallmatrix} \right] } \cdot \norm{ \xi } , $$
where $ \xi \in ( \C^{ m \times 1 } )^{ r \times 1 } $ is the vector $ \begin{bmatrix} \xi_1 & \cdots & \xi_r \end{bmatrix}^{ \top } $.

Denote $ \Omega^*_0 := \{ W : W^* \in \Omega_0 \} $ and consider the Hilbert space $ \C^{ m \times m } $ endowed with the Hilbert-Schmidt inner product 
$$ \langle A, B \rangle_{ HS } = \mathrm{trace} ( A B^* ) , $$ 
denoted by $ ( \C^{ m \times m }, \mathrm{ tr } ) $. For $ m \in \N $, define $ \Lambda_m : ( \Omega^*_0 )_m \rightarrow \CL ( ( \C^{ m \times m }, \mathrm{tr} )^{ \oplus r }, \CH ) $ as follows:
$$ \Lambda_m ( W ) ( ( y_1 v_1, \hdots, y_r v_r ) ) := v_1 \gamma_1 ( W^* ) y_1 + \cdots + v_r \gamma_r ( W^* ) y_r, $$ for all $ y_1, \hdots, y_r \in \C^{ m \times 1 } $ and $ v_1, \hdots, v_r \in \C^{ 1 \times m } $. Observe, for each $ m \in \N $ and $ W \in \Omega^*_0 $, that $ \Lambda_m ( W ) $ is a bounded linear operator. Indeed, it follows from the definition of the row operator space structure on $ \CH $ that 
$$ \norm{ v_1 \gamma_1 ( W^* ) y_1 + \cdots + v_r \gamma_r ( W^* ) y_r } \leq M \sum_{ j = 1 }^r \norm{ y_j v_j }_{ HS } \leq r M \norm{  ( y_1 v_1, \hdots, y_r v_r ) }_{ ( \C^{ m \times m }, \text{tr} )^{ \oplus r } }, $$ 
where $ M $ is the maximum of $ \norm{ \gamma_1 ( W^* ) }, \hdots, \norm{ \gamma_r ( W^* ) } $, and we note, for any $ y \in \C^{ m \times 1 } $, $ v \in \C^{ 1 \times m } $, that $ \norm{ y v }_{ HS } = \norm{ y } \norm{ v } $. Finally, we define the bounded linear operator 
$$ U_{ \Gamma } : \CH \rightarrow \CN\CC\CO ( \Omega^*_0, ( \C^r )_{ nc } ) \hspace{0.1in} \text{as} \hspace{0.1in} h \mapsto U_{ \Gamma } ( h ), $$ 
where $ U_{ \Gamma } ( h ) : \Omega^*_0 \rightarrow ( \C^r )_{nc} $ is defined, for each $ m \in \N $ and $ W \in ( \Omega^*_0 )_m $, by
\be \label{unitary map for model}
U_{ \Gamma } ( h ) ( W ) := \Lambda_m ( W )^* ( h ).
\ee
Here, we think of $ f \in \CN\CC\CO ( \Omega^*_0, ( \C^r )_{ nc } ) $ as a nc function on $ \Omega^*_0 $ taking values in the nc space over $ \CL ( \C, \C^r ) \cong \C^r $. 

Note that $ U_{ \Gamma } $ is an injective linear map. Indeed, if $ U_{ \Gamma } ( h ) = 0 $, then $ h $ is perpendicular to $ v \gamma_j ( W ) y $ for all $ 1 \leq j \leq r $, $ W \in ( \Omega_0 )_m $, $ v \in \C^{ 1 \times m } $, $ y \in \C^{ m \times 1 } $ and $ m \in \N $. Consequently, it follows from part (ii) in Definition \ref{Definition of nc CD class} that $ h = 0 $. Let $ \CH_{ \Gamma } $ be the range of $ U_{ \Gamma } $. We now carry forward the hermitian product on $ \CH $ via $ U_{ \Gamma } $ to $ \CH_{ \Gamma } $ making $ U_{ \Gamma } $ an unitary operator from $ \CH $ onto $ \CH_{ \Gamma } $.

Staring with a noncommutative $ d $-tuple of bounded linear operators $ \boldsymbol{T} = ( T_1, \hdots, T_d ) $ in the nc Cowen-Douglas class $ \mathrm B_r ( \Omega )_{nc} $, we thus have constructed a Hilbert space $ \CH_{ \Gamma } $ of uniformly nc analytic $ ( \C^r )_{ nc } $ valued functions on some subdomain $ \Omega^*_0 \subset \Omega^* $ unitarily equivalent to $ \CH $ via the unitary operator $ U_{ \Gamma } $. Moreover, $ \CH_{ \Gamma } $ turns out to be a nc reproducing kernel Hilbert space and $ U_{ \Gamma } $ intertwines the operators $ T_1, \hdots, T_d $ on $ \CH $ and the adjoint $ M_{ Z_1 }^*, \hdots, M_{ Z_d }^* $ of the left multiplication operators by the nc coordinate functions on $ \CH_{ \Gamma } $ simultaneously as shown in the theorem below.

\begin{thm} \label{model space}
Let $ \Omega \subset \C^d_{ nc } $ be a uniformly open nc domain, $ \CH $ be a complex separable Hilbert space and $ T_1, \hdots, T_d \in \CB ( \CH ) $. Assume that the $ d $-tuple of operators $ \boldsymbol{ T } = ( T_1, \hdots, T_d ) $ is in $ \mathrm B_r ( \Omega )_{ nc } $. Then there exists a nc uniformly open subdomain $ \Omega_0 \subset \Omega $, a nc reproducing kernel Hilbert space $ \CH ( K ) \subset \CN\CC\CO ( \Omega_0^*, \CL ( \C, \C^r )_{ nc } ) $ with a cp nc kernel $ K : \Omega_0^* \times \Omega_0^* \rightarrow \CL ( \C_{nc }, \C^r_{ nc } ) $ and a unitary operator $ U : \CH \rightarrow \CH ( K ) $ such that 
\be \label{model operator}
U \circ T_j = M_{ Z_j }^* \circ U, ~~~ j = 1, \hdots, d
\ee
where $ \Omega_0^* = \{ W : W^* \in \Omega_0 \} $ and $ M_{ Z_1 }, \hdots M_{ Z_d } $ are the left multiplication operators by the nc coordinate functions on $ \CH ( K ) $.
\end{thm}

\begin{proof}
Let $ \CH ( K ) : = \CH_{ \Gamma } $ and $ U : = U_{ \Gamma } $ as above. Fix an arbitrary $ 1 \leq j \leq d $ and recall from Definition \ref{Definition of nc CD class} that 
$$ ( T_j \otimes \mathrm{Id}_{ \C^{ m \times m } } ) ( \gamma_i ( W ) ) = \gamma_i ( W ) W_j, $$
for all $ 1 \leq i \leq r $, $ W = ( W_1, \hdots, W_d ) \in (\Omega_0^*)_m $ and $ m \in \N $. Consequently, it follows for any $ v \in \C^{ 1 \times m } $, $ y \in \C^{ m \times 1 } $ and $ W \in ( \Omega_0^* )_m $, that
\Bea
T_j \circ \Lambda_m ( W ) ( ( y_1 v_1, \hdots, y_r v_r ) ) & = & T_j ( v_1 \gamma_1 ( W^* ) y_1 + \cdots + v_r \gamma_r ( W^* ) y_r )\\
& = & v_1 \gamma_1 ( W^* ) W^*_j y_1 + \cdots + v_r \gamma_r ( W^* ) W^*_j y_r\\
& = & \Lambda_m ( W ) ( ( W^*_j y_1 v_1, \hdots, W_j^* y_r v_r )\\
& = & \Lambda_m ( W ) \circ L_{ W^*_j } ( ( y_1 v_1, \hdots, y_r v_r ) ), 
\Eea
where $ L_{ W^*_j } $ is the left multiplication operator on $ ( \C^{ m \times m }, \text{tr} )^{ \oplus r } $ by $ W^*_j $. So it verifies the following identity 
$$ T_j \circ \Lambda_m ( W ) = \Lambda_m ( W ) \circ L_{ W^*_j }, \quad W \in \Omega^*_0. $$
Now observe, for $ h \in \CH $ and $ W \in \Omega^*_0 $, that
\begin{multline*}
( U_{ \Gamma } ( T_j^* ( h ) ) ( W )  =  \Lambda_m ( W )^* ( T_j^* ( h ) )
 =  ( T_j \circ \Lambda_m ( W ) )^* ( h )
 =  ( \Lambda_m ( W ) \circ L_{ W^*_j } )^* ( h )\\
 =  L_{ W_j } \circ \Lambda_m ( W )^* ( h )
 =  W_j U_{ \Gamma } ( h )
 =  ( M_{ Z_j } ( U_{ \Gamma } ( h ) ) ( W )
\end{multline*}
verifying the identity given in Equation \eqref{model operator}. It remains to show that $ \CH ( K ) $ possesses a nc reproducing kernel.

For $ m \in \N $ and $ W \in ( \Omega^*_0 )_m $, consider the evaluation mapping $ \mathsf{E}_W : \CH_{ \Gamma } \rightarrow ( \C^{ m \times m } )^r $ defined by 
$$ \mathsf{E}_W ( U_{ \Gamma } ( h ) ) := U_{ \Gamma } ( h ) ( W ) = \Lambda_m ( W )^* ( h ). $$ 
Since $ \Lambda_m ( W ) $ is bounded and $ U_{ \Gamma } $ is unitary, $ \mathsf{E}_W $ is bounded for $ W \in ( \Omega^*_0 )_m $ and $ m \in \N $. Consequently, the evaluation functionals $ \mathsf{E}_{ W, v, y } : \CH_{ \Gamma } \rightarrow \C $ given by 
\be \label{evaluation functional}
\mathsf{E}_{ W, v, y } ( U_{ \Gamma } ( h ) ) := \langle U_{ \Gamma } ( h ) ( W ) v^*, y \rangle_{ ( \C^r )^m }
\ee
is bounded for every $ v \in \C^{ 1 \times m } $, $ y \in ( \C^r )^{ m \times 1 } $, $ W \in ( \Omega^* )_m $ and $ m \in \N $. Therefore, by Riesz representation theorem, for each $ m \in \N $, $ W \in ( \Omega^*_0 )_m $ and $ v \in \C^m $, $ y \in ( \C^r )^m $, there exists unique vector $ K_{ W, v, y } \in \CH_{ \Gamma } $ such that 
\be \label{reproducing formula}
\langle U_{ \Gamma } ( h ), K_{ W, v, y } \rangle_{ \CH_{ \Gamma } } = \mathsf{E}_{ W, v, y } ( U_{ \Gamma } ( h ) ) = \langle U_{ \Gamma } ( h ) ( W ) v^*, y \rangle_{ ( \C^r )^m }.
\ee
verifying that $ \CH_{ \Gamma } $ is a nc reproducing kernel Hilbert space of uniformly analytic $ \CL ( \C, \C^r )_{nc} $ valued functions on $ \Omega^*_0 $ (see \cite[Theorem 3.3]{Ball-Marx-Vinnikov-nc-rkhs}). Moreover, the reproducing kernel associated to $ \CH_{ \Gamma } $ turns out to be the cp nc reproducing kernel $ K^{ \Gamma } $ (see Theorem \ref{nc rkhs}) defined as follows:
$$ K^{ \Gamma } : \Omega^*_0 \times \Omega^*_0 \rightarrow \CL ( \C_{nc}, ( \C^{ r \times r } )_{nc} ) \quad \text{where, for} ~ n, m \in \N, $$ 
$$ K^{ \Gamma } : ( \Omega^*_0 )_n \times ( \Omega^*_0 )_m \rightarrow \CL ( \C^{ n \times m }, ( \C^{ r \times r } )^{ n \times m } ) \quad \text{is the mapping} $$
\be \label{nc reproducing kernel} K^{ \Gamma } ( Z, W ) ( P ) := \Gamma ( Z^* )^* ( P ) \Gamma ( W^* ), \quad P \in \C^{ n \times m }
\ee
We set $ K = K^{ \Gamma } $. Indeed, for any $ m \in \N $, $ W \in ( \Omega^*_0 )_m $, $ h \in \CH $, $ v \in \C^{ 1 \times m } $ and writing $ y \in ( \C^r )^{ m \times 1 } $ as $ y = ( y_1 \cdots y_r )^{ \top } $ with $ y_1, \hdots, y_r \in \C^{ m \times 1 } $, we have that
\Bea
\langle U_{ \Gamma } ( h ), U_{ \Gamma } ( v \gamma_1 ( W^* ) y_1 + \cdots + v \gamma_r ( W^* ) y_r ) \rangle_{ \CH_{ \Gamma } } & = & \langle h, v \gamma_1 ( W^* ) y_1 + \cdots + v \gamma_r ( W^* ) y_r \rangle_{ \CH }\\
& = & \langle h, \Lambda_m ( W ) ( ( y_1 v, \hdots, y_r v ) ) \rangle_{ \CH }\\
& = & \langle \Lambda_m ( W )^* ( h ), ( y_1 v, \hdots, y_r v )  \rangle_{ ( \C^{ m \times m } )^r }\\
& = & \langle U_{ \Gamma } ( h ) ( W ), ( y_1 v, \hdots, y_r v )  \rangle_{ ( \C^{ m \times m } )^r }\\
& = & \langle U_{ \Gamma } ( h ) ( W ) v^*, y \rangle_{ ( \C^m )^r }\\
& = & \langle U_{ \Gamma } ( h ), K_{ W, v, y } \rangle_{ \CH_{ \Gamma } }
\Eea
verifying that 
$$
U_{ \Gamma } ( v \gamma_1 ( W^* ) y_1 + \cdots + v \gamma_r ( W^* ) y_r ) = U_{ \Gamma } ( v \Gamma ( W^* ) ( y ) ) = K_{ W, v, y }.
$$

Also for any $ m, n \in \N $, $ W \in ( \Omega^*_0 )_m $, $ Z \in ( \Omega^*_0 )_n $, $ u \in \C^{ 1 \times n } $, $ x \in ( \C^r )^{ n \times 1 } $ and $ v \in \C^{ 1 \times m } $, $ y \in ( \C^r )^{ m \times 1 } $, note that
\Bea
\langle K_{ W, v, y }, K_{ Z, u, x } \rangle_{ \CH_{ \Gamma } } & = & \langle U_{ \Gamma } ( v \Gamma ( W^* ) ( y ) ), U_{ \Gamma } ( u \Gamma ( Z^* ) ( x ) )\rangle_{ \CH_{ \Gamma } }\\
& = & \langle v \Gamma ( W^* ) ( y ) , u \Gamma ( Z^* ) ( x ) \rangle_{ \CH }\\
& = & \langle ( u^* v ) \Gamma ( W^* ) ( y ), \Gamma ( Z^* ) ( x ) \rangle_{ \CH^{ \oplus n } }\\
& = & \langle  \Gamma ( Z^* )^*  ( u^* v ) \Gamma ( W^* ) ( y ), x \rangle_{ ( \C^n )^r }\\
& = & \langle K^{ \Gamma } ( Z, W ) ( u^* v ) y, x \rangle_{ ( \C^n )^r }.
\Eea
This completes the proof.
\end{proof}

\begin{rem} \label{no canonical model}
In general, there is no canonical representation of the Hilbert space of nc uniformly analytic functions on which a given tuple $ \boldsymbol{ T } \in \mathrm{ B }_r ( \Omega )_{ nc } $ is modeled as the tuple of adjoint of the left multiplication operators by the nc coordinate functions on it. This is because there is no canonical choice of local frames $ \{ \gamma_1, \hdots, \gamma_r \} $ of the associated vector bundle $ E_{ \boldsymbol{ T } } \overset{ \pi }{ \to } \Omega $. However, two such frames give rise to unitarily equivalent models spaces.

Let $ \{ \gamma_1, \hdots, \gamma_r \} $ and $ \{ \gamma'_1, \hdots, \gamma'_r \} $ be two local frames over a same nc subdomain $ \Omega_0 \subset \Omega $. As in the discussion following Definition \ref{automorphism of hermitian bundle}, we then have there exists a uniformly analytic nc mapping $ \Psi : \Omega_0 \to \coprod_{ n = 1 }^{ \infty } $ such that 
$$ \gamma'_i ( Z ) = \Psi_{ i 1 } \gamma_1 ( Z ) + \cdots + \Psi_{ i r } ( Z ) \gamma_r ( Z ), \quad Z \in \Omega_0, ~ 1 \leq i \leq r , $$
where $ \Psi = ( \! ( \Psi_{ i j } ) \! )_{ i, j = 1 }^r $ is an $ r \times r $ matrix of $ \C_{ nc } $ valued uniformly analytic nc functions on $ \Omega_0 $ so that, for any $ Z \in (\Omega_0)_n $, $ \Psi ( Z ) \in \mathrm{ GL } ( nr, \C ) $. Consequently, the corresponding cp nc kernels --- denoted by $ K $ and $ K' $ --- obtained from the construction above are related in the following manner: For $ n, m \in \N $, $ Z \in ( \Omega_0^* )_n $, $ W \in ( \Omega_0^* )_m $ and $ P \in \C^{ n \times m } $,
$$ K' ( Z, W ) ( P ) = \Psi ( Z^* )^* K ( Z, W ) ( P ) \Psi ( W^* ) . $$
In other words, the associated Hilbert spaces $ \CH ( K ) $ and $ \CH ( K' ) $ are unitarily equivalent via the unitary 
$$ U : = M_{ \Psi^* } : \CH ( K ) \to \CH ( K' ) \quad \text{defined by} \quad f ( Z ) \mapsto \Psi^* ( Z ) f ( Z ), $$
where $ \Psi^* : \Omega_0^* \to \coprod_{ n = 1 }^{ \infty } $ is the uniformly analytic nc function defined as 
$$ Z \mapsto \Psi^* ( Z ) : = \Psi ( Z^* )^* . $$
Thus, the model spaces turn out to be unique up to unitary equivalence. 

In certain special cases, there are canonical choices of the local frame. For instance, the case of the tuple $ \boldsymbol{ R }^* = ( R_1^*, \hdots, R_d^* ) $ of adjoint of the right creation operators on $ \ell^2 ( \CG_d ) $ as described in Example \ref{nc Hardy space}.
\end{rem}

\subsection{Local operators} \label{local operators} The \textit{local operators} associated to an element $ \boldsymbol{ T } = ( T_1, \hdots, T_d ) \in \mathrm B_r ( \Omega )_{ nc} $ are defined as follows. 
%Let $ \CG^d $ be the free monoid of all words in the $ d $ letters $ \{ 1, \hdots, d \} $ with the empty word $ \emptyset $ as the unit. For $ \alpha = \alpha_1 \cdots \alpha_k \in \CG^d $, $ \boldsymbol{ T } = ( T_1, \hdots, T_d ) \in \mathrm B_r ( \Omega )_{ nc } $ and $ W \in \Omega_m $, set
%$$ D^{ \alpha }_{ \boldsymbol{ T } - W } : = ( T_{ \alpha_1 } \otimes \mathrm{Id}_{ \C^{ m \times m } } - \mathrm{Id}_{ \CH } \otimes R_{ W_{ \alpha_1 } } ) \cdots ( T_{ \alpha_k } \otimes \mathrm{Id}_{ \C^{ m \times m } } - \mathrm{Id}_{ \CH } \otimes R_{ W_{ \alpha_k } } ) . $$
First, for each $ \ell \in \N $ and $ W \in \Omega $, we construct a family of subspaces $ \left \{ \CN_{ \ell, W } \right \}_{ \ell = 1 }^{ \infty } $. If $ W \in \Omega_m $ define $ \CN_{ 1, W } $ to be the $ \C^{ m \times m } $ bi-module generated by $ \ker D_{ \boldsymbol{ T } - W } $. Suppose that we have defined $ \CN_{ \ell, W } $ for $ 1 \leq \ell \leq p $. Then $ \CN_{ p + 1, W } $ is defined as $ \C^{ m \times m } $ bi-module generated by 
\be \label{nc generalized eigenspaces} 
( D_{ \boldsymbol{ T } - W } )^{ - 1 } \big( \CN_{ p, W }^{ \oplus d } \big) : = \big\{ h \in \CH^{ m \times m } : D_{ \boldsymbol{ T } - W } ( h ) \in \CN_{ p, W }^{ \oplus d } \big\} 
\ee
where $ \CN_{ p, W }^{ \oplus d } = \CN_{ p, W } \oplus \cdots \oplus \CN_{ p, W } $ ($ d $ times). Note that $ \CN_{ \ell, W } $ is an invariant subspace of $ T_j \otimes \mathrm{Id}_{ \C^{ m \times m } } - \mathrm{Id}_{ \CH } \otimes R_{ W_j } $, $ 1 \leq j \leq d $, for each $ \ell \in \N $ as recorded in the following proposition.

\begin{prop}
The subspaces $ \CN_{ \ell, W } $, $ \ell \in \N $ as defined in \eqref{nc generalized eigenspaces} are invariant under the action of the operators $ T_j \otimes \mathrm{Id}_{ \C^{ m \times m } } - \mathrm{Id}_{ \CH } \otimes R_{ W_j } $, $ 1 \leq j \leq d $, for any $ \boldsymbol{ T } \in \mathrm B_r ( \Omega )_{ nc } $, $ W \in \Omega_m $ and $ m \in \N $.
\end{prop}

\begin{proof}
Recall from Theorem \ref{model space} that the operators $ T_j \otimes \mathrm{Id}_{ \C^{ m \times m } } - \mathrm{Id}_{ \CH } \otimes R_{ W_j } $ are simultaneously unitarily equivalent to the operators $ M^*_{ Z_j } \otimes \mathrm{Id}_{ \C^{ m \times m } } - \mathrm{Id}_{ \CH } \otimes R_{ W^*_j } $, $ 1 \leq j \leq d $, respectively, on a nc reproducing kernel Hilbert space $ \CH ( K ) $ with the nc reproducing kernel $ K : \Omega_0^* \times \Omega_0^* \rightarrow \CL ( \C_{ nc }, \C^r_{ nc } ) $ where $ \Omega_0 \subset \Omega $ is a nc subdomain and $ \Omega_0^* = \{ W : W^* \in \Omega_0 \} $. Consequently, the proof follows from part (ii) of Corollary \ref{nc generalized eigenspaces for mult op}.
\end{proof}

\begin{rem}
A direct proof of this proposition without referring to the model space of the given tuple $ \boldsymbol{ T } $ of operators can be obtained with the help of the following lemma.
\end{rem}

\begin{lem}
Let $ \boldsymbol{ T } = ( T_1, \hdots, T_d ) \in \mathrm B_r ( \Omega )_{ nc } $ and $ \gamma_1, \hdots, \gamma_r $ be a local uniformly analytic nc frame of the nc hermitian holomorphic vector bundle $ E_{ \boldsymbol{ T } } \overset{ \pi }{ \rightarrow } \Omega $ associated to $ \boldsymbol{ T } $ over a uniformly open subdomain $ \Omega_0 \subset \Omega $. Then, for every $ m \in \N $, $ W \in ( \Omega_0 )_m $, $ \SA, \SB \in \C^{ m \times m } $, $ X^1, \hdots, X^{ \ell } \in ( \C^d )^{ m \times m } $, $ \xi \in \C^r $ and $ 1 \leq k \leq d $, $ 1 \leq j \leq r $,
\begin{align} \label{action of CD operator on derivative}
& \big( T_k \otimes \mathrm{Id}_{ m \times m } - \mathrm{Id}_{ \CH } \otimes R_{ W_k } \big)  \big( \SA ( _{ 1 } \Delta_R^{ \ell } \gamma_j ( W, \hdots, W ) ( X^1, \hdots, X^{ \ell } ) \SB \big) \\
\nonumber & = ~ \SA ( _1\Delta_R^{ { \ell } - 1 } \gamma_j ( W, \hdots, W ) ( X^1, \hdots, X^{ \ell -1 } ) X^{ \ell }_k \SB  + \SA ( _1 \Delta_R^{ \ell } \gamma_j ( W, \hdots, W ) ( X^1, \hdots, X^{ \ell } )  [ W_k, \SB ] .
\end{align}
\end{lem}

\begin{proof}
Since $ \boldsymbol{ T } \in \mathrm B_r ( \Omega )_{ nc } $ and $ \gamma_1, \hdots, \gamma_r $ is a uniformly analytic nc frame for the associated nc vector bundle $ E_{ \boldsymbol{ T } } \overset{ \pi }{ \rightarrow } \Omega_0 $, it follows for $ m \in \N $, $ W \in ( \Omega_0 )_m $, $ 1 \leq j \leq r $ and $ 1 \leq k \leq d $, that
$$ T_k \otimes \mathrm{Id}_{ \C^{ m \times m } } ( \gamma_j ( W ) ) = \gamma_j ( W ) R_k . $$
Then the desired identity is obtained by differentiating the equation above repeatedly using the nc Leibnitz rule for nc different-differentials.
\end{proof}

We now define the \textit{local operator} $ N^{ ( \ell ) }_{ \boldsymbol{ T }, W } $ associated to $ \boldsymbol{ T } $ at $ W \in \Omega_m $ of order $ \ell $ as 
\be \label{local operator}
N^{ ( \ell ) }_{ \boldsymbol{ T }, W } : = D_{ \boldsymbol{ T } - W } \left \vert_{ \CN_{ \ell + 1, W } } \right. : \CN_{ \ell + 1, W } \rightarrow \underbrace{ \CN_{ \ell + 1, W } \oplus \cdots \oplus \CN_{ \ell + 1, W } }_{ d-\text{times}} . 
\ee

\begin{thm} \label{Thm 6.5}
Let $ \Omega \subset ( \C^d )_{ nc } $ be a nc domain and $ \boldsymbol{ T }, \widetilde{ \boldsymbol{ T } } \in \mathrm B_r ( \Omega )_{ nc } $. Assume that there is a unitary operator $ U : \CH \rightarrow \CH $ satisfying $ ( U \otimes \mathrm{Id}_{ \C^{ m \times m } } ) ( \CN_{ \ell, W } ) = \CN_{ \ell, W } $, $ \ell \in \N $ and
$$ ( U \otimes \mathrm{Id}_{ \C^{ m \times m } } ) \left\vert_{ \CN_{ \ell + 1, W } } \right. N^{ ( \ell ) }_{ \widetilde{ \boldsymbol{ T } }, W } = N^{ ( \ell ) }_{ \boldsymbol{ T }, W } ( U \otimes \mathrm{Id}_{ \C^{ m \times m } } )  \left\vert_{ \CN_{ \ell + 1, W } } \right., $$
for some $ W \in \Omega_m $. Then $ U $ intertwines $ \boldsymbol{ T } $ and $ \widetilde{ \boldsymbol{ T } } $.
\end{thm}

\begin{proof}
Since $ U \otimes I_{ \C^{ m \times m } } $ intertwines $ N_{ \boldsymbol{ T }, W }^{ ( \ell ) } $ and $ N_{ \widetilde{ \boldsymbol{ T } }, W }^{ ( \ell ) } $ for all $ \ell \in \N $, it follows that $ U $ intertwines $ \boldsymbol{ T } $ and $ \widetilde{ \boldsymbol{ T } } $ on the closed subspace 
$$ \overline{ \bigvee \{ h_{ i j } : ( \! ( h_{ i j } ) \! )_{ i, j = 1 }^m \in \CN_{ \boldsymbol{ T }, W }^{ ( \ell ) }, \ell \in \N \} }, $$ 
which can be shown to be equal to $ \CH $ as in the Corollary \ref{nc generalized eigenspaces for mult op}. 
\end{proof}

\begin{rem}
It is worth noting that the ideal hypothesis for Theorem \ref{Thm 6.5} would be the existence of a family of unitary operators $ U_{ \ell } : \CN_{ \ell, W } \to \CN_{ \ell, W } $ intertwining the local operators $ N^{ ( \ell ) }_{ \boldsymbol{ T }, W } $ and $ N^{ ( \ell ) }_{ \widetilde{ \boldsymbol{ T } }, W } $, for $ \ell \in \N $. However, it is not clear, in that case, if a global unitary $ U : \CH \to \CH $ as given in the theorem above can be manufactured from the given family of unitary operators. 
\end{rem}
%In fact, this is not clear even in the case of the classical Cowen-Douglas operators. 
%In particular, for the existence of such a global unitary from a given family $ U_{ \ell } : \CN_{ \ell, W } \to \CN_{ \ell, W } $ of unitaries as above, it is enough to have the property that the subspace $ \CN_{ \ell, W } $ is an invariant subspace of $ U_j $ whenever $ \ell \leq j $. Unlike the classical case, it is not clear if such a family of unitaries possess this property as the action 

\section{A noncommutative Gleason Problem} 

In this section, we present a noncommutative analogue of the classical Gleason problem: For a bounded domain $ D \subset \C^n $ and a Banach algebra $ \CA ( D ) $, the problem is to study the existence of functions $ f_1, \hdots, f_n \in \CA ( D ) $ such that
$$ f ( z ) = \sum_{ j = 1 }^n ( z_j - p_j ) f_j ( z ), ~~~ \text{for all} ~~~ z \in D, $$
whenever $ f \in \CA ( D ) $ with $ f ( p ) = 0 $ for some $ p  = ( p_1, \hdots, p_n ) \in D $. However, Gleason problem can be studied from different points of view. For instance, given a reproducing kernel Hilbert space $ \CH $ of holomorphic functions on $ D $ so that the multiplication operators $ M_{ z_1 }, \hdots, M_{ z_n } $ by the coordinate functions on $ \CH $ are bounded, one studies the existence of $ f_1, \hdots, f_n \in \CH $ satisfying
$$ f ( z ) = \sum_{ j = 1 }^n ( z_j - p_j ) f_j ( z ), $$
on some connected neighbourhood of $ p $ in $ D $, for any function $ f \in \CH $ with $ f ( p ) = 0 $. 

Paraphrasing we can restate this problem as the following identity
$$ \ker ( \mathrm ev_p ) = \mathrm{ ran } \left( D^*_{ M^* - \bar{ p } \mathrm{Id}_{ n \times n } } \right), $$
where $ M^* - \bar{ p } \mathrm{Id}_{ n \times n } = ( M_{z_1}^* - \bar{ p }_1 \mathrm{Id}_{ n \times n }, \hdots, M_{z_n }^* - \bar{ p }_n \mathrm{Id}_{ n \times n }) $ and $ D^*_{ M^* - \bar{ p } \mathrm{Id}_{ n \times n } } : \CH \rightarrow \CH \oplus \cdots \oplus \CH $ is defined by $ h \mapsto ( ( M_{z_1}^* - \bar{ p }_1 \mathrm{Id}_{ n \times n } ) h, \hdots, ( M_{z_n }^* - \bar{ p }_n \mathrm{Id}_{ n \times n }) h ) $. In the following subsection, we introduce a noncommutative analogue of the Gleason problem for uniformly analytic nc functions and show that such a problem in the nc category is always solvable uniquely, unlike the classical case. We then make use of this in the later subsection to obtain a characterization of nc reproducing Hilbert spaces of uniformly analytic nc functions on a nc domain in $ \C^d_{ nc } $ so that the adjoint of the $ d $-tuple of left multiplication operators by the nc coordinate functions are in the nc Cowen-Douglas class.

\subsection{Noncommutative Gleason problem} \label{On nc Gleason problem}
Let $ \CR $ be a commutative unital ring, $ \CN $ be a module over $ \CR $ and $ \Omega $ be a nc set in $ \CR^d_{ nc } $. Consider a bimodule $ \M $ of nc functions $ f : \Omega \to \mathrm{ Hom } ( \CN, \CR )_{ nc } $ over the free algebra $ \CR \langle z_1, \hdots, z_d \rangle $ and observe, for $ s \in \N $, that the nc space $ ( \mathrm{ Hom } ( \CN, \CR )^{ s \times s } )_{ nc } $ can be identified with $ \mathrm{ Hom } ( \CN, \CR )_{ nc } \otimes \CR^{ s \times s } $, and therefore, any nc function $ F : \Omega \to ( \mathrm{ Hom } ( \CN, \CR )^{ s \times s } )_{ nc } $ can be viewed as 
$$ F = \sum_{ j = 1 }^n f^j \otimes A^j $$
for some nc functions $ f^1, \hdots, f^n : \Omega \rightarrow \mathrm{ Hom } ( \CN, \CR )_{ nc } $ and $ A^1, \hdots, A^n \in \CR^{ s \times s } $. For $ s \in \N $ and $ W \in \Omega_s $, we define an evaluation mapping $ \mathrm{ev}_W : \M \otimes \CR^{ s \times s } \to \mathrm{ Hom } ( \CN, \CR )^{ s \times s } $ by
\be \label{evaluation mapping} 
\mathrm{ev}_W ( F ) : = \sum_{ j = 1 }^n A^j f^j ( W ),
\ee
where $ F = \sum_{ j = 1 }^n f^j \otimes A^j $. Observe that the kernel $ \M_W $ of $ \mathrm{ev}_W $ turns out to be a right $ \CR \langle z_1, \hdots, z_d \rangle $ - left $ \CR^{ s \times s } $ submodule of $ \M $. Then we ask the following question.\\

\noindent\emph{Given a nc set $ \Omega \subset \CR^d_{nc} $, a module $ \CN $ over $ \CR $ and a bimodule of nc functions on $ \Omega $ taking values in $ \CN_{ nc } $, for a point $ W \in \Omega_s $, if the right $ \CR \langle z_1, \hdots, z_d \rangle $ - left $ \CR^{ s \times s } $ submodule $$ \M_W : = \{ F \in \M \otimes \CR^{ s \times s } : \mathrm{ev}_W ( F ) = 0 \} $$ coincides with the direct sum of the submodules $$ ( L_{ Z_j } \otimes \mathrm{Id}_{ \CR^{ s \times s } } - \mathrm{Id}_{ \M } \otimes R_{ W_j } ) ( \M \otimes \CR^{ s \times s } ), ~ j= 1, \hdots, d, $$ where $ L_{ Z_j } $ and $ R_{ W_j } $ are the left and right multiplication by $ Z_j $ and $ W_j $, $ j = 1, \hdots, d $, respectively. }\\

For our purpose in this article, we study the nc Gleason problem in the case of the $ \C \langle z_1, \hdots, z_d \rangle $-bimodule $ \M $ of uniformly analytic nc functions on a nc domain $ \Omega \subset \C^d_{ nc } $ taking values in $ \CL ( \CV, \C )_{ nc } $, and prove that the nc Gleason problem is always solvable uniquely, unlike the classical case. However, the same proof can be adapted to solve the nc Gleason problem in full generality as introduced, provided the spaces of square matrices $ \CR^{ n \times n } $ over $ \CR $, $ n \in \N $, are equipped with a ``Hilbert-Schmidt type" pairing, that is, a function $ \langle \cdot, \cdot \rangle : \CR^{ n \times n } \times \CR^{ n \times n } \to \CR $ defined by $$ \langle A, B \rangle = \mathrm{trace} ( A B^{ \top } ) $$ with the property that $ \langle A, X \rangle = 0 $ for all $ X \in \CR^{ n \times n } $ if and only if $ A = 0 $.

Let $ \Omega $ be a nc domain in $ \C^d_{nc} $, a normed linear space $ \CV $ and $ F : \Omega \rightarrow ( \CL ( \CV, \C )^{ s \times s } )_{nc} $ be a uniformly analytic nc function for some $ s \in \N $. Since the nc space $ ( \CL ( \CV, \C )^{ s \times s } )_{ nc } $ can be identified with $ \CL ( \CV, \C )_{ nc } \otimes  \C^{ s \times s }$, we can also view $ F : \Omega \rightarrow ( \CL ( \CV, \C )^{ s \times s } )_{ nc } $ as 
$$ F = \sum_{ j = 1 }^n f^j \otimes \SA^j $$
for some uniformly analytic nc functions $ f^1, \hdots, f^n : \Omega \rightarrow \CL ( \CV, \C )_{ nc } $ and $ \SA^1, \hdots, \SA^n \in \C^{ s \times s } $. Then we say $ F $ \emph{vanishes} at $ W \in \Omega_s $ if $$ \mathrm{ev}_W ( F ) = 0 .$$

On the other hand, viewing $ F ( W ) $ as the linear mapping $ F ( W ) : \CV^{ s \times s } \rightarrow \C^{ s \times s } $ given by 
$$ X \mapsto F ( W ) ( X ) : = \sum_{ j = 1 }^n f^j ( W ) X \SA^j , $$ 
we note that $ \mathrm{ev}_W ( F ) = 0 $ if and only if 
$$ \mathrm{trace} ( F ( W ) ( X ) ) = 0 , \quad \text{for all} ~~ X \in \CV^{ s \times s }. $$
Indeed, it follows from the observation
$$ \mathrm{trace} ( F ( W ) ( X ) ) = \sum_{ j = 1 }^n\mathrm{trace} ( f^j ( W ) X \SA^j ) = \sum_{ j = 1 }^n\mathrm{trace} (  \SA^j f^j ( W ) X ) . $$
\begin{comment}
For $ W \in \Omega_s $, $ F ( W ) \in ( \CL ( \CV, \C )^{ s \times s } )^{ s \times s } $. So $ F ( W ) $ can be thought of as a linear mapping $ F ( W ) : \CV^{ s \times s } \rightarrow \C^{ s \times s } $ with $ X \mapsto F ( W ) ( X ) $. We then say that $ F $ \textit{vanishes at} $ W $ if 
$$ \text{trace} ~ ( F ( W ) ( X ) ) = 0 ~~~ \text{for all} ~~~ X \in \CV^{ s \times s } .$$
In other words, this condition can be paraphrased with respect to the Hilbert-Schmidt inner product $ \langle \cdot, \cdot \rangle_{ \C^{ s \times s } } $ on $ \C^{ s \times s } $ as follows:
$$ \langle F ( W ) ( X ), \mathrm{Id}_{ s \times s } \rangle_{ \C^{ s \times s } } = 0 ~~~ \text{for all} ~~~ X \in \CV^{ s \times s } . $$
Since the nc space $ ( \CL ( \CV, \C )^{ s \times s } )_{ nc } $ can be identified with $ \CL ( \CV, \C )_{ nc } \otimes  \C^{ s \times s }$, we can also view $ F : \Omega \rightarrow ( \CL ( \CV, \C )^{ s \times s } )_{ nc } $ as 
$$ F = \sum_{ j = 1 }^n f^j \otimes \SA^j $$
for some uniformly nc analytic functions $ f^1, \hdots, f^n : \Omega \rightarrow \CL ( \CV, \C )_{ nc } $ and $ \SA^1, \hdots, \SA^n \in \C^{ s \times s } $. In this setting, it turns out that $ F $ vanishes at $ W $ if and only if 
$$ \sum_{ j = 1 }^n \text{trace} ( A^j f^j ( W ) ( X ) ) = 0 ~~~ \text{for all} ~~~ X \in \CV^{ s \times s } . $$
\end{comment}
Thus, solving the \textit{nc Gleason problem} in the case when $ \CR = \C $ and $ \CN $ is a normed linear space $ \CV $ amounts to solving the following problem:\\

\noindent \textit{Given a nc domain $ \Omega \subset \C^d_{ nc } $, a uniformly analytic nc function $ F : \Omega \rightarrow ( \CL ( \CV, \C )^{ s \times s } )_{ nc } $ for some $ s \in \N $ and a point $ W \in \Omega_s $, assume that $ F $ vanishes at $ W $ as described above. Then the problem of interest is to ask if $ F $ can be factorized as 
\be \label{Gleason solution} 
F( Z ) = \sum_{ i = 1 }^d ( L_{ Z_i} \otimes I_{ s } - I_{ms} \otimes R_{ W_i } ) F_i ( Z ),~~~Z \in \Omega_{ ms }, m \in \N,
\ee
for some uniformly analytic nc functions $ F_1, \hdots, F_d $ on $ \Omega $ taking values in the nc space over $ ( \CL ( \CV, \C )^{ s \times s } ) $. If such factorization holds, which conditions on $ F $ make $ F_1, \hdots, F_d $ to be unique?}\\

We first solve this problem algebraically without involving any topology on the spaces, and therefore, the functions in this set-up are only nc functions; no assumptions on the analyticity of corresponding functions are imposed here.

\begin{thm} \label{uniqueness of Gleason solution}
Let $ \Omega \subset \C^d_{ nc } $ be a right admissible nc set, $ F : \Omega \rightarrow ( \CL ( \CV, \C )^{ s \times s } )_{ nc } $ be a nc function and $ W \in \Omega_s $ for some $ s \in \N $. If there exist nc functions $ F_1, \hdots, F_d : \Omega \rightarrow ( \CL ( \CV, \C )^{ s \times s } )_{ nc } $ such that 
$$ F( Z ) = \sum_{ i = 1 }^d ( L_{ Z_i} \otimes I_{ s } - I_{m} \otimes R_{ W_i } ) F_i ( Z ),~~~Z \in \Omega_m, m \in \N ,$$
then such nc functions $ F_1, \hdots, F_d $ are unique.
\end{thm}

\begin{proof}
We begin by pointing out that $ \Omega = \big( \coprod_{ m = 1 }^{ \infty } \Omega_{ ms } \big)_{\text{d.s.e}} $, and therefore, by \cite[Proposition 9.2 (i) and Proposition A.3 ]{Verbovetskyi-Vinnikov} it is enough to find the required nc functions on $ \coprod_{ m = 1 }^{ \infty } \Omega_{ ms } $.

Assume that there are nc functions $ F_1, \hdots, F_d $ on $ \Omega $ taking values in $ ( \CL ( \CV, \C )^{ s \times s } )_{ nc } $ such that the equation \eqref{Gleason solution} holds on $ \coprod_{ m \geq 1 } \Omega_{ m s } $. Then we show the uniqueness of the nc functions $ F_1, \hdots, F_d $ by comparing their Taylor-Taylor coefficients with those of $ F $ from their respective Taylor-Taylor formula incorporated in Equation \eqref{Gleason solution}.

Let $ \CG_d $ be the free monoid over $ d $ letters and $ \beta \in \CG_d $ with $ | \theta | = \ell $, $ \ell \in \N $. Then we consider the Taylor-Taylor formula for $ F $ and $ F_1, \hdots, F_d $ around the point $ W \in \Omega_s $ up to the order $ \ell + 2 $. In other words, following Theorem \ref{TT formula}, there exist families of multilinear mappings $ \{ F_{ \beta } : ( ( \C^d )^{ s \times s } )^k \rightarrow ( \CL ( \CV, \C )^{ s \times s } )^{ s \times s } : \beta \in \CG_d, ~ | \beta | = k, k \leq \ell + 1 \} $, $ \{ F_{ i, \alpha } : ( ( \C^d )^{ s \times s } )^{ \ell } \rightarrow ( \CL ( \CV, \C )^{ s \times s } )^{ s \times s } : \alpha \in \CG_d, ~ | \alpha | = k, k \leq \ell + 1 \} $, $ i =1, \hdots, d $, each family satisfying the canonical intertwining conditions at $ W $, $ \text{CIC} ( W ) $ (see Definition \ref{canonical intertwining conditions}), such that 
\begin{multline} \label{TT series of F} 
F ( Z ) = \sum_{ k = 0 }^{ \ell + 1 } \sum_{ | \beta | = k } \left ( Z - I_m \otimes W \right )^{ \odot_s \beta } F_{ \beta } \\ + \sum_{ | \beta | = \ell + 2 }  \left ( Z - I_m \otimes W \right )^{ \odot_s \beta } \Delta^{ \beta^{ \top } }_R F ( W, \hdots, W, Z ) , ~~ Z \in \Omega_{ms}; 
\end{multline}
\begin{multline} \label{TT series of F_i} 
F_i ( Z ) = \sum_{ k = 0 }^{ \ell + 1 } \sum_{ | \alpha | = k } \left ( Z - I_m \otimes W \right )^{ \odot_s \alpha } F_{ i, \alpha } \\ + \sum_{ | \alpha | = \ell + 2 }  \left ( Z - I_m \otimes W \right )^{ \odot_s \alpha } \Delta^{ \alpha^{ \top } }_R F_i ( W, \hdots, W, Z ), ~ Z \in \Omega_{ms},~ i = 1, \hdots, d. 
\end{multline}
Also, writing $ Z_i = Z_i - I_m \otimes W_i + I_m \otimes W_i $ for each $ i = 1, \hdots, d $, we have that 
\be \label{variable adjustment}
 Z_i \otimes I_{ s } - I_{ ms } \otimes W_i = Z_i \otimes I_{ s } - I_m \otimes W_i \otimes I_{ s } + I_m \otimes W_i \otimes I_{ s } - I_{ms} \otimes W_i, ~~~ i = 1, \hdots, d. 
 \ee
Substituting $ F ( Z ) $, $ F_i ( Z ) $ and $ Z_i \otimes I_{ s } - I_{ ms } \otimes W_i $ for $ i =1, \hdots, d $ from Equations \eqref{TT series of F}, \eqref{TT series of F_i} and \eqref{variable adjustment}, respectively in Equation \eqref{Gleason solution}, and then comparing the coefficients associated to the word $ \beta $ with $ | \beta | = \ell $ in the resulting equation we have that 
\begin{align} \label{general coefficients}
\nonumber & ( X_1 \odot_s \cdots \odot_s X_{ \ell + 1 } ) F_{ g_j \beta } - ( X_1 \otimes I_{ s } ) ( X_2 \odot_s \cdots \odot_s X_{ \ell + 1 } ) F_{ j, \beta }\\
& =  \sum_{ i = 1 }^d I_m \otimes ( W_i \otimes I_{ s } - I_{ s } \otimes W_i ) ( X_1 \odot_s \cdots \odot_s X_{ \ell + 1 } ) F_{ i, g_j \beta },
\end{align}
for all $ j = 1, \hdots, d $, $ X_1, \hdots, X_{ \ell + 1 } \in \C^{ ms \times ms } $. Here, each of the terms $ ( X_1 \odot_s \cdots \odot_s X_{ \ell + 1 } ) F_{ g_j \beta }, ( X_1 \odot_s \cdots \odot_s X_{ \ell + 1 } ) F_{ i, g_j \beta } $ and $ ( X_2 \odot_s \cdots \odot_s X_{ \ell + 1 } ) F_{ j, \beta } $ belongs to the space $ ( \CL ( \CV, \C )^{ s \times s } )^{ ms \times ms } $, which is isomorphic to the space $ ( ( \CL ( \CV, \C )^{ s \times s } )^{ s \times s } )^{ m \times m } $. Thus each of these elements, being an $ m \times m $ block matrix with entries in $ ( \CL ( \CV, \C )^{ s \times s } )^{ s \times s } $, is viewed as a linear operator from $( \CV^{ s \times s } )^{ m \times 1 } $ into $ ( \C^{ s \times s } )^{ m \times 1 } $. Also, writing $ X_1 = ( \! ( X_1^{ ij } ) \! )_{ i, j = 1 }^{ m } $ with $ X_1^{ ij } \in \C^{ s \times s } $ for $ 1 \leq i, j \leq m $, we interpret the matrices $ X_1 \otimes I_s = ( \! ( X_1^{ ij } \otimes I_s ) \! )_{ i, j = 1 }^{ m } $ and $ I_m \otimes ( W_i \otimes I_{ s } - I_{ s } \otimes W_i ) $ as elements of the space $ ( \C^{ s \times s } )^{ ms \times ms } $, which is isomorphic to $ ( ( \C^{ s \times s } )^{ s \times s } )^{ m \times m } $. As such, these matrices are viewed as $ m \times m $ block matrices with entries in $ ( \C^{ s \times s } )^{ s \times s } $ and act linearly on $( \C^{ s \times s } )^{ m \times 1 } $. Consequently, it follows from the equation above that for any $ X_1, \hdots, X_{ \ell + 1 } \in \C^{ ms \times ms } $, $ T \in ( \CV^{ s \times s } )^{ m \times 1 } $ and $ 1 \leq j \leq d $,
\begin{align} \label{general defining equation} 
\nonumber & \langle ( X_1 \odot_s \cdots \odot_s X_{ \ell + 1 } ) F_{ g_j \beta } ( T ), \varepsilon_1 \otimes I_{ s } \rangle_{ ( \C^{ s \times s } )^{ m \times 1 } } \\
& =  \langle ( X_2 \odot_s \cdots \odot_s X_{ \ell + 1 } ) F_{ j, \beta } ( T ), X_1^* ( \varepsilon_1 \otimes I_{ s } )\rangle_{ ( \C^{ s \times s } )^{ m \times 1 } } 
\end{align}
where $ ( \C^{ s \times s } )^{ m \times 1 } $ is the Hilbert space obtained by taking $ m $-th direct sum of $ \C^{ s \times s } $ equipped with the Hilbert-Schmidt inner product, and $ ( \varepsilon_1 \otimes I_{ s } ) $ is the column vector $ ( I_s, 0, \hdots, 0 )^{ \text{tr} } $ in $ ( \C^{ s \times s } )^{ m \times 1 } $. In particular for $ m = 1 $, we have that
\bea \label{eqn1}
\langle F_{ g_j \beta } ( X_1, \hdots, X_{ \ell + 1 } )  ( T ), I_{ s } \rangle_{ \C^{ s \times s } } & = & \langle F_{ j, \beta } ( X_2 , \hdots, X_{ \ell + 1 } ) ( T ), X_1^* \rangle_{ \C^{ s \times s } } .
\eea
Note that the equation above with $ m = 1 $ defines the Taylor-Taylor coefficients of $ F_i $ for $ i = 1, \hdots, d $ uniquely.
\end{proof}

\begin{rem} \label{an alternative proof of existence}
We point out that Equation \eqref{eqn1} can be used to prove the local existence of uniformly analytic nc functions $ F_1, \hdots, F_d $ satisfying Equation \eqref{Gleason solution} whenever $ F $ is given to be uniformly analytic. 

For $ \beta \in \CG_d $ and $ 1 \leq j \leq d $, define the multilinear mappings $ F_{ j, \beta } : ( ( \C^d )^{ s \times s } )^{ \ell } \rightarrow ( \CL ( \CV, \C )^{ s \times s } )^{ s \times s } $ as in Equation \eqref{eqn1} and consider the formal Taylor-Taylor series $ F_j ( Z ) $ centered at $ W $ obtained from this family of multilinear mappings. Using the canonical intertwining conditions for $ F $ along with Equation \eqref{eqn1}, it can be seen that the family of multilinear mappings $ \{ F_{ j, \beta } : \beta \in \CG_d \} $ satisfy Equation \eqref{general coefficients}. Consequently, $ F_j ( Z ) $, $ 1 \leq j \leq d $, satisfy Equation \eqref{Gleason solution} whenever they define uniformly analytic nc functions. Then another use of the canonical intertwining conditions for $ F $ at $ W $ and Equation \eqref{eqn1} yields that the family of multilinear mappings $ \{ F_{ j, \beta } : \beta \in \CG_d \} $, $ 1 \leq j \leq d $ satisfy the canonical intertwining conditions. Finally, for $ X_2, \hdots, X_{ l + 1 } \in \C^{ ms \times ms } $, $ X \in ( \C^{ s \times s } )^{ m \times 1 } $ with $ \vert \vert X \vert \vert_{ ( \C^{ s \times s } )^{ m \times 1 } } = 1 $ and $ T \in ( \CV^{ s \times s } )^{ m \times 1 } $ with  $ \vert \vert T \vert \vert_{ ( \CV^{ s \times s } )^{ m \times 1 } } = 1 $, take $ X_1 \in \C^{ ms \times ms } $ to be the matrix whose first row is $ X^* $ and all other entries are the $ s \times s $ zero matrix. Then observe from \eqref{general defining equation} that
\begin{align*}
& \vert \vert ( X_2 \odot_s \cdots \odot_s X_{ \ell + 1 } ) F_{ j, \beta } \vert \vert_{ m s^2 } \\
& = \sup_{ \vert \vert T \vert \vert_{ ( \CV^{ s \times s } )^{ m \times 1 } } = \vert \vert X \vert \vert_{ ( \C^{ s \times s } )^{ m \times 1 } } = 1 } \left \vert \left \langle ( X_2 \odot_s \cdots \odot_s X_{ \ell + 1 } ) F_{ j, \beta } ( T ), X \right \rangle_{ ( \C^{ s \times s } )^{ m \times 1 } } \right \vert \\
& = \sup_{ \vert \vert T \vert \vert_{ ( \CV^{ s \times s } )^{ m \times 1 } } = \vert \vert X \vert \vert_{ ( \C^{ s \times s } )^{ m \times 1 } } = 1 } \left \vert \left \langle ( X_2 \odot_s \cdots \odot_s X_{ \ell + 1 } ) F_{ j, \beta } ( T ), X_1^* ( \varepsilon_1 \otimes I_{ s } ) \right \rangle_{ ( \C^{ s \times s } )^{ m \times 1 } } \right \vert \\
& = \sup_{ \vert \vert T \vert \vert_{ ( \CV^{ s \times s } )^{ m \times 1 } } = \vert \vert X \vert \vert_{ ( \C^{ s \times s } )^{ m \times 1 } } = 1 } \left \vert \left \langle ( X_1 \odot_s \cdots \odot_s X_{ \ell + 1 } ) F_{ g_j \beta } ( T ), \varepsilon_1 \otimes I_{ s } \right \rangle_{ ( \C^{ s \times s } )^{ m \times 1 } } \right \vert \\
%& \leq  \sqrt{ s } \sup_{ \vert \vert T \vert \vert_{ ( \CV^{ s \times s } )^{ m \times 1 } } = \vert \vert X \vert \vert_{ ( \C^{ s \times s } )^{ m \times 1 } } = 1 } \vert \vert ( X_1 \odot_s \cdots \odot_s X_{ l + 1 } ) F_{ g_j \beta } ( T ) \vert \vert_{ ms^2 } \\ 
%& \leq \sqrt{ s } \sup_{ \vert \vert X \vert \vert_{ ( \C^{ s \times s } )^{ m \times 1 } } = 1 } \vert \vert ( X_1 \odot_s \cdots \odot_s X_{ l + 1 } ) F_{ g_j \beta } \vert \vert_{ m s^2 } \\
& \leq \sqrt{ s } \sup_{ \vert \vert X \vert \vert_{ ( \C^{ s \times s } )^{ m \times 1 } } = 1 } \vert \vert F_{ g_j \beta } \vert \vert_{ \CL^{ \ell + 1 }_{ \text{cb} } } \vert \vert X_1 \vert \vert_{ ms } \vert \vert X_2 \vert \vert_{ ms } \cdots \vert \vert X_{ \ell + 1 } \vert \vert_{ ms } \\
& \leq \sqrt{ s } \frac{1}{\delta^{ \ell + 1 } } \vert \vert X_2 \vert \vert_{ ms } \cdots \vert \vert X_{ \ell + 1 } \vert \vert_{ ms } 
\end{align*}
where the first inequality is obtained by using Cauchy-Schwarz inequality on the Hilbert space $ ( \C^{ s \times s } )^{ m \times 1 } $ and the second one holds since $ F $ has Taylor-Taylor series expansion on $ B_{nc} ( W, \delta ) $ and due to the fact that
$$ \vert \vert X_1 \vert \vert_{ ms } \leq \sqrt{ \mathrm{trace} ( X_1 X_1^* ) } = \vert \vert X \vert \vert_{ ( \C^{ s \times s } )^{ m \times 1 } } \leq 1 .$$
Thus we have that 
$$ \limsup_{ \ell \rightarrow \infty } \sqrt[ \ell ]{ \norm { F_j }_{ \CL^{ \ell }_{ \text{cb} } } } \leq \limsup_{ \ell \rightarrow \infty } \sqrt[ \ell ]{ \sqrt{ s } \frac{1}{\delta^{ \ell + 1 } } } < \frac{ \sqrt[ 2 \ell ]{ s } }{ \delta } < + \infty, ~~~ 1 \leq j \leq d $$
verifying that $ F_1, \hdots, F_d $ are uniformly analytic functions on $ B_{ nc } ( W, r ) $ with $ 0 < r < \frac{ \delta }{ 2 \sqrt[ 2 \ell ]{ s } } $.

However, the next theorem shows that $ F_1, \hdots, F_d $ can be constructed explicitly from the given $ F $ and thus, they are actually defined globally on $ \coprod_{ m \geq 1 } \Omega_{ ms } $.
\end{rem}

We now proceed to prove the existence of the nc functions $ F_1, \hdots, F_d $ on $ \Omega $, which take values in $ ( \CL ( \CV, \C )^{ s \times s } )_{ nc } $ and solve the nc Gleason problem.

%We describe the construction of nc functions $ F_1, \hdots, F_d $ satisfying the factorization of $ F $ given in \eqref{Gleason solution} whenever it vanishes at $ W $.

\begin{thm} \label{constructive solution for nc Gleason problem-algebraic version}
Let $ \Omega \subset \C^d_{ nc } $ be a right admissible nc set, $ F : \Omega \rightarrow ( \CL ( \CV, \C )^{ s \times s } )_{ nc } $ be a nc function and $ W \in \Omega_s $ for some $ s \in \N $. Then $ F $ vanishes at $ W $ if and only if there exists nc functions $ F_1, \hdots, F_d : \Omega \rightarrow ( \CL ( \CV, \C )^{ s \times s } )_{ nc } $ such that 
$$ F( Z ) = \sum_{ j = 1 }^d ( L_{ Z_j} \otimes I_{ s } - I_{ms} \otimes R_{ W_j } ) F_j ( Z ),~~~Z \in \Omega_{ n }, n \in \N . $$
Furthermore, the nc functions $ F_1, \hdots, F_d $ can be determined from the difference-differential quotient of $ F $ by the formulae:
$$ F_j ( Z ) : = \sum_{ i = 1 }^n \sum_{ \ell, k = 1 }^s \big( \varepsilon^*_{ \ell } \otimes I_{ ( \CV^* )^t } \big) \Delta_{ R, j } f^i ( W, Z ) \big( \varepsilon_k \otimes I_{ ( \CV^* )^t } \big) \otimes \SA^i E_{ \ell k } , ~~~ Z \in \Omega_t, $$
where $ E_{ \ell k } $'s are $ s \times s $ elementary matrices, $ F = \sum_{ i = 1 }^n f^i \otimes \SA^i $, $ f^i : \Omega \to \CL ( \CV, \C )_{ nc } $, and $ \SA^i \in \C^{ s \times s } $.
\end{thm}

\begin{proof}
We begin by pointing out that $ \Omega = \big( \coprod_{ m = 1 }^{ \infty } \Omega_{ ms } \big)_{\text{d.s.e}} $, and therefore, by \cite[Proposition 9.2 (i) and Proposition A.3 ]{Verbovetskyi-Vinnikov}, it is enough to find the required nc functions on $ \coprod_{ m = 1 }^{ \infty } \Omega_{ ms } $. 

Since the existence of nc functions $ F_1, \hdots, F_d : \coprod_{ m = 1 }^{ \infty } \Omega_{ ms } \rightarrow ( \CL ( \CV, \C )^{ s \times s } )_{ nc } $ satisfying Equation \eqref{Gleason solution} trivially verifies that $ \mathrm{trace} F ( W ) ( X ) = 0 $ for all $ X \in \CV^{ s \times s } $, we only prove the forward direction.

Let $ F = \sum_{ i = 1 }^n f^i \otimes \SA^i $ for some nc functions $ f^1, \hdots, f^n : \Omega \rightarrow \CL ( \CV, \C )_{ nc } $ and $ \SA^1, \hdots, \SA^n \in \C^{ s \times s } $. Then for each $ j = 1, \hdots, d $ and $ Z \in \Omega_t $, $ t \in \N $,
$$ \Delta_{ R, j } F ( W, Z ) = \sum_{ i = 1 }^n  \Delta_{ R, j } f^i ( W, Z ) \otimes \SA^i . $$
Also for each $ i = 1, \hdots, n $ and $ j = 1, \hdots, d $, note that $ \Delta_{ R, j } f^i ( W, . ) $ can be thought of as a $ s \times s $ matrix of nc functions on $ \{ W \} \times \Omega $ taking values in $ \CL ( \CV, \C )_{ nc } $ as follows
$$ \Delta_{ R, j } f^i ( W, Z ) = \sum_{ \ell, k = 1 }^s E_{ \ell k } \otimes f^i_{ \ell k } ( Z ), Z \in \Omega_t, $$
where we view $ \Delta_{ R, j } f^i ( W, Z ) $ as a linear mapping from $ \C^{ s \times t } $ into $ \CL ( \CV, \C )^{ s \times t } $ which acts on a rank one matrix $ u \cdot v $ with $ u \in \C^{ s \times 1 } $ and $ v \in \C^{ 1 \times t } $ as
$$ \Delta_{ R, j } f^i ( W, Z ) ( u \cdot v ) = \sum_{ \ell, k = 1 }^s E_{ \ell k } ( u \cdot v ) f^i_{ \ell k } ( Z ), $$
and $ \{ \varepsilon_1, \hdots, \varepsilon_s \} $ is the standard ordered basis for $ \C^{ s \times 1 } $. Here, we have used the notation 
$$ f^i_{ \ell k } ( Z ) : = \big( \varepsilon^*_{ \ell } \otimes I_{ ( \CV^* )^t } \big) \Delta_{ R, j } f^i ( W, Z ) \big( \varepsilon_k \otimes I_{ ( \CV^* )^t } \big) , \,\,\, Z \in \Omega_t , $$
and $ I_{ ( \CV^* )^t } $ is the identity mapping on $ ( \CV^* )^t $.
\begin{comment}
    Also for each $ i = 1, \hdots, n $ and $ j = 1, \hdots, d $, note that $ \Delta_{ R, j } f^i ( W, . ) $ can be thought of as a $ s \times s $ matrix of uniformly analytic nc functions on $ \{ W \} \times \Omega $ taking values in $ \CL ( \CV, \C )_{ nc } $ as follows
$$ \Delta_{ R, j } f^i ( W, Z ) = \sum_{ \ell, k = 1 }^s E_{ \ell k } \otimes [ ( \varepsilon^*_{ \ell } \otimes I_n ) \Delta_{ R, j } f^i ( W, Z ) ( \varepsilon_k \otimes I_n ) ], Z \in \Omega_n, $$
where we view $ \Delta_{ R, j } f^i ( W, Z ) $ as a linear mapping on $ \CL ( \CV, \C )^{ s \times n } $ which acts on a rank one matrix $ u \cdot v $ with $ u \in \CL ( \CV, \C )^{ s \times 1 } $ and $ v \in \CL ( \CV, \C )^{ 1 \times n } $ as
$$ \Delta_{ R, j } f^i ( W, Z ) ( u v ) = ( I_s \otimes v ) [ \Delta_{ R, j } f^i ( W, Z ) ] ( u \otimes I_n ) =  \sum_{ \ell, k = 1 }^s E_{ \ell k } u \otimes v f^i_{ \ell k } ( Z ), $$
and $ \{ \varepsilon_1, \hdots, \varepsilon_s \} $ is the standard ordered basis for $ \C^{ s \times s } $. Here, we have introduced the notation 
$$ f^i_{ \ell k } ( Z ) : = ( \varepsilon^*_{ \ell } \otimes I_n ) \Delta_{ R, j } f^i ( W, Z ) ( \varepsilon_k \otimes I_n ) , \,\,\, Z \in \Omega_n . $$
\end{comment}
For the rest of the proof, we thus write the nc function $  \Delta_{ R, j } f^i ( W, . ) $ as 
$$ \Delta_{ R, j } f^i ( W, . ) = \sum_{ \ell, k = 1 }^s E_{ \ell k } \otimes f^i_{ \ell k } , $$
and consequently, $  \Delta_{ R, j } F ( W, Z ) $, for $ Z \in \Omega $, turns out to be
$$  \Delta_{ R, j } F ( W, Z ) = \sum_{ i = 1 }^n \left[ \sum_{ \ell, k = 1 }^s E_{ \ell k } \otimes f^i_{ \ell k } ( Z ) \right] \otimes \SA^i . $$
We now define the nc functions $ F_1, \hdots, F_d : \coprod_{ m = 1 }^{ \infty } \Omega_{ ms } \rightarrow ( \CL ( \CV, \C )^{ s \times s } )_{ nc } $ by the following formulae:
\be \label{Gleason factors}
F_j ( Z ) : = \sum_{ i = 1 }^n \sum_{ \ell, k = 1 }^s f^i_{ \ell k } ( Z ) \otimes \SA^i E_{ \ell k } ,
\ee
and check that $ F_1, \hdots, F_d $ indeed satisfy Equation \eqref{Gleason solution}. 

It follows from Equation \eqref{finite difference formula} that
\be \label{finite difference formula for F}
( \SS \otimes I_s ) F( Z ) - F( W ) ( \SS \otimes I_s ) = \sum_{ j = 1 }^d \Delta_{ R, j } F ( W, Z ) ( \SS Z_j - W_j \SS ),
\ee
for $m \in \N $, $ Z \in \Omega_{ ms } $ and $ \SS \in \C^{ s \times ms } $. Here, we view $ \SS $ as an $ 1 \times m $ row vector over $ \C^{ s \times s } $ and write 
$ \SS = \left[ \begin{smallmatrix} \SS_1 & \cdots & \SS_m \end{smallmatrix} \right] = \sum_{ p = 1 }^m \epsilon^*_p \otimes \SS_p $ with $ \SS_1, \hdots, \SS_m \in \C^{ s \times s } $ and the standard ordered basis $ \{ \epsilon_1, \hdots, \epsilon_m \} $ for $ \C^m $. 

For $ m \in \N $, $ Z \in \Omega_{ ms } $ and $ T \in \CV^{ ms \times s } $ viewing as $ m \times 1 $ column vector 
$$ T = \begin{bmatrix} T_1 & \cdots & T_m \end{bmatrix}^{ \top } = \sum_{ q = 1 }^m \epsilon_q \otimes T_q $$ 
with $ T_1, \hdots, T_m \in \CV^{ s \times s } $, write $ f^i ( Z ) = \big( \! \big( f^i ( Z )_{ p q } \big) \! \big)_{ p, q = 1 }^m $, $ i = 1, \hdots, n $ and observe that
\begin{align*}
& \bigg\langle [ ( \SS \otimes I_s ) F( Z ) - F( W ) ( \SS \otimes I_s )] ( T ), I_s \bigg\rangle_{ \C^{ s \times s } } \\
& = \bigg\langle \bigg( \begin{bmatrix} \SS_1 & \cdots & \SS_m \end{bmatrix} \sum_{ i = 1 }^n ( \! ( f^i ( Z )_{ p q } ) \! )_{ p, q = 1 }^m \otimes A^i \bigg) ( T ), I_s  \bigg\rangle_{ \C^{ s \times s } } \\
& = \sum_{ i = 1 }^n \big\langle \begin{bmatrix} \sum_{ p = 1 }^m S_p f^i ( Z )_{ p 1 } \otimes A^i & \cdots & \sum_{ p = 1 }^m S_p f^i ( Z )_{ p m } \otimes A^i \end{bmatrix} T, I_s \big\rangle_{ \C^{ s \times s } } \\
& =  \sum_{ i = 1 }^n \bigg\langle \sum_{ q' = 1 }^m \varepsilon^*_{ q' } \otimes \bigg(  \sum_{ p = 1 }^m S_p f^i ( Z )_{ p q' } \otimes A^i \bigg) \bigg( \sum_{ q = 1 }^m \epsilon_q \otimes T_q \bigg), I_s \bigg\rangle_{ \C^{ s \times s } } \\
& =  \sum_{ i = 1 }^n \bigg\langle \sum_{ p, q = 1 }^m S_p f^i ( Z )_{ p q } (T_q A^i ), I_s \bigg\rangle_{ \C^{ s \times s } } \\
& =  \sum_{ p = 1 }^n \bigg\langle \sum_{ i, q = 1 }^m f^i ( Z )_{ p q } (T_q A^i ), \SS_p^* \bigg\rangle_{ \C^{ s \times s } } \\
& =  \bigg\langle  \sum_{ p = 1 }^n \varepsilon_p \otimes \bigg[ \sum_{ i, q = 1 }^m f^i ( Z )_{ p q } (T_q A^i ) \bigg], \sum_{ p' = 1 }^m \varepsilon_{ p' } \otimes \SS_{ p' }^* \bigg\rangle_{ \C^{ ms \times s } } \\
& = \langle F ( Z ) ( T ), \SS^* \rangle_{ \C^{ ms \times s } }
\end{align*}
where $ \SS^* = \sum_{ p = 1 }^m \epsilon_p \otimes \SS^*_p = \left[ \begin{smallmatrix} \SS^*_1 & \cdots & \SS^*_m \end{smallmatrix} \right]^{ \top } $. Here, the first equality holds due the following observation
\begin{align*}
&\left\langle F ( W ) ( \SS \otimes I_s ) ( T ), I_s \right\rangle_{ \C^{ s \times s } }  =  \sum_{ q = 1 }^m \sum_{ i = 1}^n \mathrm{trace} (f^i ( W ) ( \SS_q T_q ) A^i ) \\
 &=   \sum_{ q = 1 }^m \sum_{ i = 1}^n \mathrm{trace} ( A^i f^i ( W ) ( \SS_q T_q ) ) 
 =  0,
\end{align*}
as $ \mathrm{trace} ( F ( W ) ( X ) ) = 0 $. We have thus shown that
\be \label{LHS of 8.14}
\big\langle [ ( \SS \otimes I_s ) F( Z ) - F( W ) ( \SS \otimes I_s )] ( T ), I_s \big\rangle_{ \C^{ s \times s } }  = \langle F ( Z ) ( T ), \SS^* \rangle_{ \C^{ ms \times s } } .
\ee
Now for each $ 1 \leq j \leq d $, $ Z \in \Omega_{ ms } $ and $ T = \left[ \begin{smallmatrix} T_1 & \cdots & T_m \end{smallmatrix} \right]^{ \top } \in \CV^{ ms \times s } $, note that
\begin{align*}
& \Delta_{ R, j } F ( W, Z ) ( \SS Z_j ) ( T ) \\
& = \bigg[ \bigg( \sum_{ i = 1 }^n \bigg[ \sum_{ \ell, k = 1 }^s E_{ \ell k } \otimes f^i_{ \ell k } ( Z ) \bigg] \otimes A^i \bigg) \begin{bmatrix} \SS_1 & \cdots & \SS_m \end{bmatrix} \big( \! \big( Z_j^{ pq } \big) \! \big)_{ p, q = 1 }^m \bigg] \begin{bmatrix} T_1 \\ \vdots \\ T_m \end{bmatrix} \\
& = \bigg( \sum_{ i = 1 }^n \bigg[ \sum_{ \ell, k = 1 }^s E_{ \ell k } \begin{bmatrix} \SS_1 & \cdots & \SS_m \end{bmatrix} \big( \! \big( Z_j^{ pq } \big) \! \big)_{ p, q = 1 }^m \big( \! \big( ( f^i_{ \ell k } ( Z ) )_{ p'q' } \big) \! \big)_{ p', q' = 1 }^m \bigg] \otimes A^i \bigg) \begin{bmatrix} T_1 \\ \vdots \\ T_m \end{bmatrix} \\
& =  \sum_{ i = 1 }^n \bigg[ \sum_{ \ell, k = 1 }^s \sum_{ q = 1 }^m \bigg(  \sum_{ p = 1 }^m E_{ \ell k } \SS_p \bigg( \sum_{ t = 1 }^m Z_j^{ pt } ( f^i_{ \ell k } ( Z ) )_{ tq } \bigg) ( T_q ) A^i \bigg) \bigg] \\
& = \sum_{ i = 1 }^n \bigg[  \sum_{ \ell, k = 1 }^s \sum_{ t = 1 }^m \bigg(  \sum_{ p = 1 }^m E_{ \ell k } \SS_p \bigg( \sum_{ q = 1 }^m Z_j^{ pq } \big( f^i_{ \ell k } ( Z ) \big)_{ qt } \bigg) ( T_t ) A^i \bigg) \bigg] 
\end{align*} 
and consequently, we have that
\begin{align*}
& \big\langle  \Delta_{ R, j } F ( W, Z ) ( \SS Z_j ) ( T ), I_s \big\rangle_{ \C^{ s \times s } } \\
& = \sum_{ i = 1 }^n \sum_{ \ell, k = 1 }^s \sum_{ t = 1 }^m \sum_{ p = 1 }^m \sum_{ q = 1 }^m \text{trace} \big( E_{ \ell k } \SS_p  Z_j^{ pq } ( f^i_{ \ell k } ( Z ) )_{ qt }  ( T_t ) A^i \big) \\
& = \sum_{ p = 1 }^m \sum_{ i = 1 }^n \sum_{ \ell, k = 1 }^s \sum_{ q, t = 1 }^m \text{trace} \big( Z_j^{ pq } ( f^i_{ \ell k } ( Z ) )_{ qt } ( T_t ) A^i E_{ \ell k } \SS_p \big) \\
& = \sum_{ p = 1}^m \bigg\langle \sum_{ i = 1 }^n \sum_{ \ell, k = 1 }^s \sum_{ q, t = 1 }^m Z_j^{ pq } ( f^i_{ \ell k } ( Z ) )_{ qt }  ( T_t ) A^i E_{ \ell k }, \SS_p^* \bigg\rangle_{ \C^{ s \times s } }\\
& = \big\langle ( Z_j \otimes I_s ) F_j ( Z ) ( T ), \SS^* \big\rangle_{ \C^{ ms \times s } }.
\end{align*}
Here, the last equality holds since for each $ 1 \leq j \leq d $,
\begin{align*}
& \big( Z_j \otimes I_s \big) F_j ( Z ) ( T ) \\
& = \big( \! \big( Z_j^{ p q } \otimes I_s \big) \! \big)_{ p, q = 1 }^m \bigg( \sum_{ p', q' = 1 }^m E_{ p' q' } \otimes \bigg[ \sum_{ i = 1 }^n \sum_{ \ell, k = 1 }^s ( f_{ \ell k }^i ( Z ) )_{ p' q' } \otimes A^i E_{ \ell k } \bigg] \bigg) \bigg( \sum_{ t = 1 }^m \epsilon_t \otimes T_t \bigg) \\
& = \sum_{ p = 1 }^m \epsilon_p \otimes \bigg[ \sum_{ i = 1 }^n \sum_{ \ell, k = 1 }^s \sum_{ q, t = 1 }^m Z_j^{ pq } ( f^i_{ \ell k } ( Z ) )_{ qt }  ( T_t ) A^i E_{ \ell k } \bigg] . 
\end{align*}
Thus, it has been shown, for each $ 1 \leq j \leq d $, that
\be \label{RHS of 8.14 part 1}
\big\langle  \Delta_{ R, j } F ( W, Z ) ( \SS Z_j ) ( T ), I_s \big\rangle_{ \C^{ s \times s } } = \big\langle ( Z_j \otimes I_s ) F_j ( Z ) ( T ), \SS^* \big\rangle_{ \C^{ ms \times s } }.
\ee
A similar computation also shows, for every $ 1 \leq j \leq d $, that
\be \label{RHS of 8.14 part 2}
\big\langle  \Delta_{ R, j } F ( W, Z ) ( W_j \SS ) ( T ), I_s \big\rangle_{ \C^{ s \times s } } = \big\langle ( I_{ ms } \otimes W_j ) F_j ( Z ) ( T ), \SS^* \big\rangle_{ \C^{ ms \times s } }.
\ee
Finally, Equation \eqref{finite difference formula for F}, Equation \eqref{LHS of 8.14}, Equation \eqref{RHS of 8.14 part 1} and Equation \eqref{RHS of 8.14 part 2} together verify the validation of Equation \eqref{Gleason solution} with the nc functions $ F_1, \hdots, F_d $ as constructed above.
\end{proof}

\begin{rem}
When $ \Omega $ is uniformly open nc domain and $ F $ is uniformly analytic, the uniqueness of the functions $ F_1, \hdots, F_d $ is followed from Equations \eqref{RHS of 8.14 part 1} or \eqref{RHS of 8.14 part 2}, provided $ \Omega $ contains at least one invertible point $ Z $ or $ W $ is invertible, respectively. Indeed, any nc functions $ G_1, \hdots, G_d $ satisfying Equation \eqref{Gleason solution} must satisfy \eqref{RHS of 8.14 part 1} and \eqref{RHS of 8.14 part 2}. Consequently, we have 
$$  ( Z_j \otimes I_s ) F_j ( Z ) ( T ) =  ( Z_j \otimes I_s ) G_j ( Z ) ( T ) $$ and
$$\ ( I_{ ms } \otimes W_j ) F_j ( Z ) ( T ) = ( I_{ ms } \otimes W_j ) G_j ( Z ) ( T ) ,  $$
for all $ 1 \leq j \leq d $, $ Z \in \Omega_{ ms } $, $ \SS \in \C^{ s \times ms } $ and $ T \in \CV^{ ms \times s } $, and therefore, with the help of the added hypothesis it verifies the claim.
\end{rem}

In the following theorem, we solve the nc Gleason problem when $ \Omega $ is a uniformly open nc set and $ F $ is uniformly analytic.

\begin{thm} \label{analytic solution for nc Gleason problem}
Let $ F : \Omega \rightarrow ( \CL ( \CV, \C )^{ s \times s } )_{ nc } $ be a uniformly analytic nc function and $ W \in \Omega_s $ for some $ s \in \N $. Then $ F $ vanishes at $ W $ if and only if there exists uniformly analytic nc functions $ F_1, \hdots, F_d : \Omega \rightarrow ( \CL ( \CV, \C )^{ s \times s } )_{ nc } $ such that Equation \eqref{Gleason solution} holds. Moreover, the functions $ F_1, \hdots, F_d $ satisfying Equation \eqref{Gleason solution} are unique.
\end{thm}

\begin{proof}
Since $ \Omega = \big( \coprod_{ m = 1 }^{ \infty } \Omega_{ ms } \big)_{\text{d.s.e}} $, by \cite[Proposition 9.2 (i) and Proposition A.3 ]{Verbovetskyi-Vinnikov}, it is enough to find the required nc functions on $ \coprod_{ m = 1 }^{ \infty } \Omega_{ ms } $. 

In view of Theorem \ref{uniqueness of Gleason solution} and Theorem \ref{constructive solution for nc Gleason problem-algebraic version}, it is enough to show that the functions $ F_1, \hdots, F_d $ obtained in Theorem \ref{constructive solution for nc Gleason problem-algebraic version} are uniformly analytic nc functions whenever so is $ F $. Recall from Equation \eqref{Gleason factors} that $ F_1, \hdots, F_d : \coprod_{ m = 1 }^{ \infty } \Omega_{ ms } \rightarrow ( \CL ( \CV, \C )^{ s \times s } )_{ nc } $ are nc functions given by
    $$ F_j ( Z ) : = \sum_{ i = 1 }^n \sum_{ \ell, k = 1 }^s f^i_{ \ell k } ( Z ) \otimes \SA^i E_{ \ell k },  $$
where $ F = \sum_{ i = 1 }^n f^i \otimes \SA^i $ and for $ 1 \leq \ell, k \leq s $, the functions $ f^i_{ \ell k } $ are determined by the equation 
$$ \Delta_{ R, j } f^i ( W, Z ) = \sum_{ \ell, k = 1 }^s E_{ \ell k } \otimes f^i_{ \ell k } ( Z ), ~ ~ Z \in \Omega . $$
Since $ F $ is uniformly analytic nc function, so are the order $ 1 $ nc functions  (cf. \cite[Theorem 7.44, pp 114]{Verbovetskyi-Vinnikov}) 
$$ \Delta_{ R, j } F ( \cdot, \cdot ) : \Omega \times \Omega \to \CL \big( \C^d_{ nc }, ( \CL ( \CV, \C )^{ s \times s } )_{ nc } \big), \quad 1 \leq j \leq d . $$ 
In particular, for a fixed $ W \in \Omega_s $, the functions $ \Delta_{ R, j } F ( W, \cdot ) $, $ 1 \leq j \leq d $, turn out to uniformly analytic nc functions on $ \coprod_{ m \geq 1 } \Omega_{ m s } $ taking values in $ ( \CL ( \CV, \C )^{ s \times s } )_{ nc } $. Consequently, $ F_1, \hdots, F_d $ are uniformly analytic nc functions completing the proof.
\end{proof}

\begin{rem}
    Note that there is no speciality of the choice of uniformly-open topology in Theorem \ref{analytic solution for nc Gleason problem}. In fact, if $ \Omega $ is finitely open (respectively, open in disjoint open topology) and the nc function $ F $ is locally bounded on slices (respectively, locally bounded in disjoint open topology), the same line of argument as in Theorem \ref{analytic solution for nc Gleason problem} along with Theorem \ref{constructive solution for nc Gleason problem-algebraic version} and Theorem \ref{uniqueness of Gleason solution} yield locally bounded solutions $ F_1, \hdots, F_d $ to the Gleason problem in finitely open and disjoint open topologies, respectively. We refer \cite[Chapter 7, pp 83, 89]{Verbovetskyi-Vinnikov} to the readers for the definitions of finitely open and disjoint open topologies.
\end{rem}

\subsection{Application to the noncommutative Cowen-Douglas theory} \label{application of nc Gleason problem}

In this subsection, we obtain a sufficient condition on a cp nc reproducing kernel Hilbert space of uniformly analytic nc functions on a nc domain $ \Omega \subset \C^d_{ nc } $ so that the adjoint of the left multiplication operators by the nc coordinate functions on this Hilbert space is in the nc Cowen-Douglas class over $ \Omega $. 

First, recall from Equation \eqref{evaluation mapping} that the evaluation mapping $ \sev_W : \CH_s \to ( \C^{ m \times m } )^r $ at $ W \in \Omega_s $ is, by definition, 
$$ \sev_W ( F ) = \sum_{ j = 1 }^n \SA^j f^j ( W ) ,$$
where $ F = \sum_{ j = 1 }^n f^j \otimes \SA^j \in \CH_s $. Evidently, $ \sev_W $ is a bounded linear operator. We now compute the kernel of this operator in the following lemma.

\begin{lem} \label{kernel of generalized evaluation}
For $ W \in \Omega $, the kernel of the evaluation mapping $ \sev_W : \CH_s \to ( \C^{ m \times m } )^r $ at $ W \in \Omega_s $ is 
$$ \ker \sev_W = N^{ \perp }_W = \{ S K ( \cdot, W ) \sigma_t : S \in \C^{ m \times m },\, 1 \leq t \leq r \}^{ \perp } . $$
\end{lem}

\begin{proof}
Let $ W \in \Omega_m $ for some $ m \in \N $ and $ F \in \ker \sev_W $. Recall from the preceding subsection that $ \sev_W ( F ) = 0 $ if and only if $ \mathrm{ trace } F ( W ) ( \SX ) = 0 $ for all $ \SX \in \C^{ s \times s } $. Thus, we have that
$$  \big\langle F, \SX K ( \cdot, W ) \sigma_t \big\rangle_{ HS } = \mathrm{trace} ( ( F ( W ) \sigma_t ) ( \SX ) )  = \big\langle ( F ( W ) \sigma_t ) ( \SX ), \mathrm{Id}_{ \C^{ m \times m } } \big\rangle_{ \C^{ m \times m } } = 0 $$
verifying, in particular for $ \SX = E_{ i j } $, $ 1 \leq i, j \leq m $, that $ \ker \sev_W \subset N_W^{ \perp } $. The converse inclusion also follows from the similar computation as above.
\end{proof}

\begin{thm} \label{complete description of CD tuple}
Let $ \Omega $ be a nc domain in $ \C^d_{ nc } $ and $ \CH ( K ) $ be a nc reproducing kernel Hilbert space of uniformly analytic nc functions on $ \Omega^* = \{ W : W^* \in \Omega \} $ taking values in $ ( \C^r )_{ nc } $ with the cp nc reproducing kernel $ K : \Omega^* \times \Omega^* \rightarrow \CL ( \C_{ nc }, ( \C^{ r \times r } )_{ nc } ) $ which is invariant under the action of the multiplication operators by the nc coordinate functions from left. Assume that
\begin{itemize}
\item[(a)] For any $ s \in \N $ and $ W \in \Omega^*_s $, the generalized evaluation functional $ \sev_W $ is bounded from $ \CH_s $ onto $ ( \C^{ s \times s } )^r $;
\item[(b)] for any $ s \in \N $, $ \CH_s $ satisfies the nc Gleason property, that is, every $ F \in \CH_s $ with $ \sev_W ( F ) = 0 $ can be expressed as 
\be \label{nc Gleason property}
F = \sum_{ i =1 }^d \left( M_{ Z_i } \otimes \mathrm{Id}_{ \C^{ s \times s } } - \mathrm{Id}_{ \CH ( K ) } \otimes R_{ W_i } \right) F_i,
\ee
for some $ F_1, \hdots, F_d \in \CH_s $.
\end{itemize}
Then $ \boldsymbol{M}^* = ( M^*_{ Z_1 }, \hdots, M^*_{ Z_d } ) $ lies in $ \mathrm B_r ( \Omega )_{ nc } $.
\end{thm}

\begin{proof}
Since $ \CH ( K ) $ is invariant under the action of the multiplication operators by the nc coordinate functions from left, it follows from Closed Graph theorem that they are bounded operators on $ \CH ( K ) $. For $ s \in \N $ and $ W \in \Omega^*_s $, recall that $ \boldsymbol{M}^*_W : \CH_s \rightarrow \CH_s^{ \oplus d } $ is the linear mapping 
$$ \boldsymbol{ M }^*_W : = ( M_{ Z_1 }^* \otimes \mathrm{Id}_{ \C^{ s \times s } } - \mathrm{Id}_{ \CH ( K ) } \otimes R_{ W _1^* }, \hdots, M_{ Z_d }^* \otimes \mathrm{Id}_{ \C^{ s \times s } } - \mathrm{Id}_{ \CH ( K ) } \otimes R_{ W _d^* } ). $$
So the adjoint $ ( \boldsymbol{M}^*_W )^*$ of this operator turns out to be the map $ ( \boldsymbol{M}^*_W )^* : \CH_s^{ \oplus d } \rightarrow \CH_s $ defined by 
$$ ( F_1, \hdots, F_d ) \mapsto \sum_{ i =1 }^d \left( M_{ Z_i } \otimes \mathrm{Id}_{ \C^{ s \times s } } - \mathrm{Id}_{ \CH ( K ) } \otimes R_{ W_i } \right) F_i. $$
Thus, the condition (b) in the hypothesis turns out to be  
$$ \ker \sev_W = \mathrm{range} ( \boldsymbol{M}^*_W )^*.$$
Since $  \mathrm{range} ( \boldsymbol{M}^*_W )^* = ( \ker \boldsymbol{M}^*_W )^{ \perp } $ and $ \sev_W $ is onto $ ( \C^{ s \times s } )^r $ (from condition (a)), it follows that the dimension of $ \ker \boldsymbol{M}^*_W $ is $ r s^2 $. Also, the same identity implies that the range of $ ( \boldsymbol{M}^*_W )^* $ is closed and consequently, so is the range of $ \boldsymbol{M}^*_W $. Finally, recall from Lemma \ref{kernel of generalized evaluation} that $ \ker \sev_W^{ \perp } = \{ \SS K ( \cdot, W ) \sigma_t : \SS \in \C^{ s \times s }, \, 1 \leq t \leq r \} $ and observe that
\begin{align*} 
& \overline { \bigvee_{  s \in \N, W \in \Omega_s^*, v, y \in \C^s, 1 \leq t \leq r } \left\{ v^* ( \SS K( \cdot, W ) \sigma_t ) y : \SS \in \C^{ s \times s } \right\} }\\
& = \overline { \bigvee_{ s \in \N, W \in \Omega_s^*, v \in \C^s, \mathsf{y} \in ( \C^r )^s } \left\{ K_{ W, v, \mathsf{y}} \right\} } = \CH ( K ). 
\end{align*}
This completes the proof.
\end{proof}

\begin{rem}
    From Theorem \ref{constructive solution for nc Gleason problem-algebraic version}, we observe that the condition in part (b) of the theorem holds for all nc reproducing kernel Hilbert spaces $ \CH ( K ) $ which are closed under all partial right difference-differential operators: For any $ f \in \CH ( K ) $, $ s \in \N $, $ W \in \Omega^*_s $ and $ 1 \leq j \leq d $, $ \Delta_{ R, j } f ( W, \cdot ) \in \CH_s $. In general, we need this property only for a certain class of functions in $ \CH ( K ) $, namely, the family of all functions in $ \CH ( K ) $ which are entries of all functions $ F \in \CH_s $ with $ \sev_W ( F ) = 0 $. However, note that this condition needs to be satisfied by all elements in $ \CH ( K ) $ whenever $ \CH ( K ) $ contains nc constant functions. In other words, the tuple of adjoints of the multiplication operators by the nc coordinate functions on a cp nc reproducing kernel Hilbert space $ \CH ( K ) $ is in $ \mathrm{ B }_r ( \Omega )_{ nc } $ if for every $ s \in \N $, $ W \in \Omega^*_s $, $ 1 \leq i, j \leq s $, the nc functions $ f_{ W, \varepsilon_i, \varepsilon_j^* } : \Omega^* \to \C^r_{ nc } $ defined by 
    $$ Z \mapsto f_{ W, \varepsilon_i, \varepsilon_j^* } ( Z ) : \C^{ 1 \times n } \to \C^{ 1 \times n }, \quad u \mapsto \varepsilon_j^* \Delta_R f ( W, Z ) ( \varepsilon_i u ) , $$
    belong to $ \CH ( K ) $ whenever $ f \in \CH ( K ) $. Also, observe that using Corollary \ref{cond for membership in rkhs} this condition can be reformulated as the following inequality: For $ 1 \leq i, j \leq s $, there exist $ c_{ i j } > 0 $ such that
    $$ c_{ i j } K \big( Z, Z \big) -  f_{ W, \varepsilon_i, \varepsilon_j^* } ( Z ) \otimes f_{ W, \varepsilon_i, \varepsilon_j^* } ( Z )^* \succeq 0 , \quad Z \in \Omega^*_n, n \in \N . $$
    In fact, one can choose a positive constant $ c > 0 $ such that this inequality turns out to be equivalent to the following:
    $$ c K \big( Z, Z \big) -  f_{ W, \varepsilon_i, \varepsilon_j^* } ( Z ) \otimes f_{ W, \varepsilon_i, \varepsilon_j^* } ( Z )^* \succeq 0 , \quad Z \in \Omega^*_n,~ n \in \N,~ 1 \leq i, j \leq s. $$
\end{rem}

\bigskip

\noindent\textit{Acknowledgments.} The authors wish to thank Professor Joseph A. Ball for fruitful discussions on nc reproducing kernel Hilbert spaces.


\begin{thebibliography}{10}

\bibitem{Agler-McCarthy1}
J. Agler and J. E. McCarthy.
\newblock Aspects of non-commutative function theory.
\newblock {\em Concr. Oper.}, 3 (2016), no. 1, 15--24.

\bibitem{Agler-McCarthy2}
J. Agler and J. E. McCarthy.
\newblock Non-commutative functional calculus.
\newblock {\em J. Anal. Math.}, 137 (2019), no. 1, 211--229.

\bibitem{Agler-McCarthy-Young1}
J. Agler, J. E. McCarthy, and N.~J. Young.
\newblock Non-commutative manifolds, the free square root and symmetric
  functions in two non-commuting variables.
\newblock {\em Trans. London Math. Soc.}, 5 (2018), no. 1, 132--183.

\bibitem{Agler-McCarthy-Young2}
J. Agler, J. E. McCarthy, and N. J. Young.
\newblock {\em Operator analysis -- Hilbert space methods in complex analysis},
  {\em Cambridge Tracts in Mathematics}, 219.
\newblock Cambridge University Press, Cambridge, 2020.

\bibitem{Alpay-Verbovetskyi}
D. Alpay and D. S. Kalyuzhny\u{\i}-Verbovetzki\u{\i}.
\newblock On the intersection of null spaces for matrix substitutions in a
  non-commutative rational formal power series.
\newblock {\em C. R. Math. Acad. Sci. Paris}, 339 (2004), no. 8, 533--538.

\bibitem{Ball-Groenewald-Malakorn1}
J. A. Ball, G. Groenewald, and T. Malakorn.
\newblock Structured noncommutative multidimensional linear systems.
\newblock {\em SIAM J. Control Optim.}, 44 (2005), no. 4, 1474--1528.

\bibitem{Ball-Groenewald-Malakorn2}
J.~A. Ball, G. Groenewald, and T. Malakorn.
\newblock Bounded real lemma for structured noncommutative multidimensional
  linear systems and robust control.
\newblock {\em Multidimens. Syst. Signal Process.}, 17 (2006), no. 2-3, 119--150.

\bibitem{Ball-Groenewald-Malakorn3}
J. A. Ball, G. Groenewald, and T. Malakorn.
\newblock Conservative structured noncommutative multidimensional linear
  systems.
\newblock 161 (2006), 179--223.

\bibitem{Ball-Verbovetskyi}
J. A. Ball and D. S. Kalyuzhny\u{\i}-Verbovetzki\u{\i}.
\newblock Conservative dilations of dissipative multidimensional systems: the commutative and non-commutative settings.
\newblock {\em Multidimens. Syst. Signal Process.}, 19 (2008), no. 1, 79--122.

\bibitem{Ball-Marx-Vinnikov-nc-rkhs}
J. A. Ball, G. Marx, and V. Vinnikov.
\newblock Noncommutative reproducing kernel {H}ilbert spaces.
\newblock {\em J. Funct. Anal.}, 271(7):1844--1920, 2016.

\bibitem{Ball-Marx-Vinnikov-interpolation}
J. A. Ball, G. Marx, and V. Vinnikov.
\newblock Interpolation and transfer-function realization for the noncommutative {S}chur-{A}gler class.
\newblock {\em Oper. Theory Adv. Appl.} 262 (2018), 23--116.

\bibitem{Ball-Marx-Vinnikov-nc-hereditary kernel}
J. A. Ball, G. Marx, and V. Vinnikov.
\newblock Free noncommutative hereditary kernels: {J}ordan decomposition, {A}rveson extension, kernel domination.
\newblock {\em Doc. Math.}, 27 (2022), 1985--2040.

\bibitem{Ball-Vinnikov-formal-kernel}
J. A. Ball and V. Vinnikov.
\newblock Formal reproducing kernel {H}ilbert spaces: the commutative and
  noncommutative settings.
\newblock 143 (2003), 77--134.

\bibitem{Ball-Vinnikov}
J. A. Ball and V. Vinnikov.
\newblock Lax-{P}hillips scattering and conservative linear systems: a
  {C}untz-algebra multidimensional setting.
\newblock {\em Mem. Amer. Math. Soc.}, 178 (2005), no. 837, iv+101 pp.

\bibitem{Berstel-Reutenauer}
J. Berstel and C. Reutenauer.
\newblock {Les s\'{e}ries rationnelles et leurs langages}.
\newblock {\em \'{E}tudes et Recherches en Informatique. [Studies and Research in
  Computer Science]}. Masson, Paris, 1984.
  
\bibitem{Blecher-Merdy}
D. P. Blecher and C. Le Merdy.
\newblock {Operator algebras and their modules -- an operator space
  approach}, volume 30 of {\em London Mathematical Society Monographs. New
  Series}.
\newblock The Clarendon Press, Oxford University Press, Oxford, 2004.
\newblock Oxford Science Publications.

\bibitem{Cowen-Douglas1}
M. J. Cowen and R. G. Douglas.
\newblock Complex geometry and operator theory.
\newblock {\em Acta Math.}, 141 (1978), no. 3-4, 187--261.

\bibitem{Cowen-Douglas2}
M. J. Cowen and R. G. Douglas.
\newblock Operators possessing an open set of eigenvalues.
\newblock 35 (1983), 323--341.

\bibitem{Cowen-Douglas-connection}
M. J. Cowen and R. G. Douglas.
\newblock Equivalence of connections.
\newblock {\em Adv. in Math.}, 56 (1985), no. 1, 39--91.

\bibitem{Curto-Salinas}
R. E. Curto and N. Salinas.
\newblock Generalized {B}ergman kernels and the {C}owen-{D}ouglas theory.
\newblock {\em Amer. J. Math.}, 106 (1984), no. 2, 447--488.

\bibitem{Effros-Ruan}
E. G. Effros and Z. J. Ruan.
\newblock {Operator spaces}, volume 23 of {\em London Mathematical Society
  Monographs. New Series}.
\newblock The Clarendon Press, Oxford University Press, New York, 2000.

\bibitem{Eilenberg}
S. Eilenberg.
\newblock {Automata, languages, and machines. {V}ol. {A}}.
\newblock {\em Pure and Applied Mathematics}, Vol. 58. Academic Press [Harcourt Brace
  Jovanovich, Publishers], New York, 1974.

\bibitem{Flanders}
H. Flanders.
\newblock On spaces of linear transformations with bounded rank.
\newblock {\em J. London Math. Soc.}, 37 (1962), 10--16.

\bibitem{Gleason}
A. M. Gleason.
\newblock Finitely generated ideals in {B}anach algebras.
\newblock {\em J. Math. Mech.}, 13 (1964), 125--132.

\bibitem{Hadwin}
D. W. Hadwin.
\newblock Continuous functions of operators; a functional calculus.
\newblock {\em Indiana Univ. Math. J.}, 27 (1978), no. 1, 113--125.

\bibitem{Hadwin-Kaonga-Mathes}
D. Hadwin, L. Kaonga, and B. Mathes.
\newblock Noncommutative continuous functions.
\newblock {\em J. Korean Math. Soc.}, 40 (2003), no. 5, 789--830.

\bibitem{Helton-Klep-McCullough-free-proper}
J. W. Helton, I. Klep, and S. McCullough.
\newblock Proper analytic free maps.
\newblock {\em J. Funct. Anal.}, 260 (2011), no. 5, 1476--1490.

\bibitem{Helton-Klep-McCullough-convex-Positivstellensatz}
J. W. Helton, I. Klep, and S. McCullough.
\newblock The convex {P}ositivstellensatz in a free algebra.
\newblock {\em Adv. Math.}, 231 (2012), no. 1, 516--534.

\bibitem{Helton-Klep-McCullough-convexity-semidefinite-programming}
J. W. Helton, I. Klep, and S. McCullough.
\newblock Convexity and semidefinite programming in dimension-free matrix
  unknowns.
\newblock In {\em Handbook on semidefinite, conic and polynomial optimization},
  volume 166 of {\em Internat. Ser. Oper. Res. Management Sci.}, 377--405. Springer, New York, 2012.

  \bibitem{Helton-Klep-McCullough-Free-convex-algebraic-geometry}
J. W. Helton, I. Klep, and S. McCullough.
\newblock Free convex algebraic geometry.
\newblock In {\em Semidefinite optimization and convex algebraic geometry},
{\em MOS-SIAM Ser. Optim.}, SIAM, Philadelphia, PA, 13 (2013), 341--405.

  \bibitem{Helton-Klep-McCullough-Slinglend-nc-ball-maps}
J. W. Helton, I. Klep, S. McCullough and N. Slinglend.
\newblock Noncommutative ball maps.
\newblock {\em J. Funct. Anal.}, 257 (2009), no. 1, 47--87.

\bibitem{Helton-Klep-McCullough-Volcic-spectrahedra}
J. W. Helton, I. Klep, S. McCullough, and J. Vol\v{c}i\v{c}.
\newblock Bianalytic free maps between spectrahedra and spectraballs.
\newblock {\em J. Funct. Anal.}, 278 (2020), no. 11, 108472, 61 pp.

\bibitem{Helton-McCullough-convex-free-basic-semi-algebraic-set}
J. W. Helton and S. McCullough.
\newblock Every convex free basic semi-algebraic set has an {LMI}
  representation.
\newblock {\em Ann. of Math. (2)}, 176 (2012), no. 2, 979--1013.

\bibitem{Helton-McCullough-Putinar-Vinnikov}
J. W. Helton, S. McCullough, M. Putinar, and V. Vinnikov.
\newblock Convex matrix inequalities versus linear matrix inequalities.
\newblock {\em IEEE Trans. Automat. Control}, 54 (2009), no. 5, 952--964.

\bibitem{Helton-McCullough-Vinnikov}
J. W. Helton, S. A. McCullough, and V. Vinnikov.
\newblock Noncommutative convexity arises from linear matrix inequalities.
\newblock {\em J. Funct. Anal.}, 240 (2006), no. 1, 105--191.

\bibitem{Verbovetskyi-Vinnikov}
D. S. Kaliuzhnyi-Verbovetskyi and V. Vinnikov.
\newblock {Foundations of free noncommutative function theory}, volume 199
  of {\em Mathematical Surveys and Monographs}.
\newblock American Mathematical Society, Providence, RI, 2014.

\bibitem{Kalman-Falb-Arbib}
R. E. Kalman, P. L. Falb, and M. A. Arbib.
\newblock {\em Topics in mathematical system theory}.
\newblock McGraw-Hill Book Co., New York-Toronto, Ont.-London, 1969.

\bibitem{Kleene}
S. C. Kleene.
\newblock Representation of events in nerve nets and finite automata.
\newblock 1956, 3--41.

\bibitem{Muhly-Solel1}
P. S. Muhly and B. Solel.
\newblock Hardy algebras, {$W^\ast$}-correspondences and interpolation theory.
\newblock {\em Math. Ann.}, 330 (2004), no. 2, 353--415.

\bibitem{Muhly-Solel2}
P. S. Muhly and B. Solel.
\newblock Hardy algebras associated with {$W^*$}-correspondences (point
  evaluation and {S}chur class functions).
\newblock In {\em Operator theory, systems theory and scattering theory:
  multidimensional generalizations}, {\em Oper. Theory Adv.
  Appl.}, Birkh\"{a}user, Basel, 157 (2005), 221--241.

\bibitem{Muhly-Solel3}
P. S. Muhly and B. Solel.
\newblock Schur class operator functions and automorphisms of {H}ardy algebras.
\newblock {\em Doc. Math.}, 13 (2008), 365--411.

\bibitem{Muhly-Solel4}
P. S. Muhly and B. Solel.
\newblock Matricial function theory and weighted shifts.
\newblock {\em Integral Equations Operator Theory}, 84 (2016), no. 4, 501--553.

\bibitem{Paulsen}
V. Paulsen.
\newblock {Completely bounded maps and operator algebras}, volume 78 of {\em Cambridge Studies in Advanced Mathematics}.
\newblock Cambridge University Press, Cambridge, 2002.

\bibitem{Pisier}
G. Pisier.
\newblock {Introduction to operator space theory}, volume 294 of {\em
  London Mathematical Society Lecture Note Series}.
\newblock Cambridge University Press, Cambridge, 2003.

\bibitem{Popa-Vinnikov}
M. Popa and V. Vinnikov.
\newblock {$\rm H^2$} spaces of non-commutative functions.
\newblock {\em Complex Anal. Oper. Theory}, 12 (2018), no. 4, 945--967.

\bibitem{Popescu1}
G. Popescu.
\newblock Free holomorphic functions on the unit ball of {$B(H)^n$}.
\newblock {\em J. Funct. Anal.}, 241 (2006), no. 1, 268--333.

\bibitem{Popescu2}
G. Popescu.
\newblock Operator theory on noncommutative domains.
\newblock {\em Mem. Amer. Math. Soc.}, 205 (2010), no. 364, vi+124 pp.

\bibitem{Popescu3}
G. Popescu.
\newblock Composition operators on noncommutative {H}ardy spaces.
\newblock {\em J. Funct. Anal.}, 260 (2011), no. 3, 906--958.

\bibitem{Popescu4}
G. Popescu.
\newblock Free biholomorphic functions and operator model theory.
\newblock {\em J. Funct. Anal.}, 262 (2012), no. 7, 3240--3308.

\bibitem{Popescu5}
G. Popescu.
\newblock Noncommutative multivariable operator theory.
\newblock {\em Integral Equations Operator Theory}, 75 (2013), no. 1, 87--133.

\bibitem{Popescu6}
G. Popescu.
\newblock Unitary invariants on the unit ball of {$B(H)^n$}.
\newblock {\em Trans. Amer. Math. Soc.}, 365 (2013), no. 12, 6243--6267.

\bibitem{Popescu7}
G. Popescu.
\newblock Free biholomorphic functions and operator model theory, {II}.
\newblock {\em J. Funct. Anal.}, 265 (2013), no. 5, 786--836.

\bibitem{Popescu8}
G. Popescu.
\newblock Bergman spaces over noncommutative domains and commutant lifting.
\newblock {\em J. Funct. Anal.}, 280 (2021), no. 8, 108943, 89 pp.

\bibitem{Schutzenberger1}
M. P. Sch\"{u}tzenberger.
\newblock On the definition of a family of automata.
\newblock {\em Information and Control}, 4 (1961), 245--270.

\bibitem{Schutzenberger2}
M. P. Sch\"{u}tzenberger.
\newblock Certain elementary families of automata.
\newblock 1963 139--153.

\bibitem{Taylor1}
J. L. Taylor.
\newblock A general framework for a multi-operator functional calculus.
\newblock {\em Advances in Math.}, 9 (1972), 183--252.

\bibitem{Taylor2}
J. L. Taylor.
\newblock Functions of several noncommuting variables.
\newblock {\em Bull. Amer. Math. Soc.}, 79 (1973), 1--34.

\bibitem{Voiculescu-free-probability}
D. Voiculescu.
\newblock {Lectures on free probability theory} in {\em Lectures on probability theory and statistics ({S}aint-{F}lour, 1998), Lecture Notes in Math.}, Springer, Berlin.
\newblock 1738 (2000), 279--349, 2000.

\bibitem{Voiculescu-I}
D. Voiculescu.
\newblock Free analysis questions. {I}. {D}uality transform for the coalgebra
  of {$\partial_{X\colon B}$}.
\newblock {\em Int. Math. Res. Not.}, 16 (2004), 793--822.

\bibitem{Voiculescu-II}
D. V. Voiculescu.
\newblock Free analysis questions {II}: the {G}rassmannian completion and the
  series expansions at the origin.
\newblock {\em J. Reine Angew. Math.}, 645 (2010), 155--236.

\end{thebibliography}
\end{document}